\documentclass[3p,times]{elsarticle}

\usepackage{amsmath}
\usepackage{amsfonts}
\usepackage{mathtools}
\usepackage{amssymb}
\usepackage{amsthm}
\usepackage{graphicx,xcolor}
\usepackage{bbm}
\usepackage{url}
\usepackage[varvw]{newtxmath}

\usepackage[figuresright]{rotating}

\numberwithin{equation}{section}
\newtheorem{theorem}{Theorem}[section]
\newtheorem{lemma}[theorem]{Lemma}
\newtheorem{proposition}[theorem]{Proposition}
\newcommand{\norm}[1]{\Vert #1\Vert}
\newcommand{\R}{{\mathord{\mathbb R}}}
\newcommand{\rmd}{{\rm d}}
\newcommand{\rme}{{\rm e}}
\newcommand{\vep}{\varepsilon}
\newcommand{\ol}{\overline}
\newcommand{\cf}{\mathbbm{1}}
\newcommand{\mcL}{\mathcal{L}}
\newcommand{\tmcL}{\tilde{\mathcal{L}}}
\newcommand{\tmcV}{\tilde{\mathcal{V}}}
\newcommand{\tmcK}{\tilde{\mathcal{K}}}
\newcommand{\tGamma}{\tilde{\Gamma}}
\newcommand{\tg}{\tilde{\gamma}}
\newcommand{\thg}{\tilde{g}}
\newcommand{\mcV}{\mathcal{V}}
\newcommand{\mcK}{\mathcal{K}}
\newcommand{\tr}{\tilde{r}}
\newcommand{\tI}{\tilde{I}}
\newcommand{\te}{\tilde{e}}
\newcommand{\tA}{\tilde{A}}
\newcommand{\tF}{\tilde{F}}
\newcommand{\tQ}{\tilde{Q}}
\newcommand{\tpsi}{\tilde{\psi}}
\newcommand{\vphi}{\varphi}
\newcommand{\bn}{\bar{n}}
\newcommand{\bvphi}{\bar{\varphi}}

\newcommand{\bg}{\bar{g}}

\newcommand{\mcJ}{\mathcal{J}}

\newcommand{\psiep}{\psi^{\vep}}

\newcommand{\hw}{\bar{h}}
\newcommand{\ftot}{f_{\rm tot}}
\newcommand{\freg}{f_{\rm reg}}
\newcommand{\fbec}{f_{\rm BE}}
\newcommand{\ftbec}{\tilde{f}_{\rm BE}}
\newcommand{\gbec}{g_{\rm BE}}
\newcommand{\gtbec}{\tilde{g}_{\rm BE}}
\newcommand{\Rbe}{R_\beta}
\newcommand{\feql}{f_{\beta,0}}

\begin{document}

\begin{frontmatter}




\title{Smoothing Properties of a Linearization of the Three-Wave Collision Operator in the bosonic Boltzmann--Nordheim Equation}


\author{Jogia Bandyopadhyay\corref{cor1}}
\ead{jogiab@gmail.com}
\cortext[cor1]{Corresponding author}

\author[2]{Jani Lukkarinen}
\ead{jani.lukkarinen@helsinki.fi}
\address[2]{University of Helsinki, Department of Mathematics and Statistics}


\begin{abstract}
		We consider the kinetic theory of a three-dimensional fluid of weakly
	interacting bosons in a non-equilibrium state which includes both normal fluid and a condensate. More precisely, we look at the previously postulated nonlinear Boltzmann--Nordheim equations for such systems, in a spatially homogeneous state which has an isotropic momentum distribution, and we linearize the equation around an equilibrium state which has a condensate. We study the most singular part of the linearized operator coming from the three-wave collision operator for supercritical initial data. The operator has two types of singularities, one of which is similar to the marginally smoothing operator defined by the symbol $\ln(1+p^2)$. Our main result in this context is that  for initial data in a certain Banach space of functions satisfying a H\"{older} type condition, at least for some finite time, evolution determined by the linearized operator improves the H\"{o}lder regularity. The main difficulty in this problem arises from the combination of a point singularity and a line singularity present in the linear operator, and we have to use certain fine-tuned function spaces in order to carry out our analysis.
\end{abstract}


\end{frontmatter}


\section{Introduction}\label{section: introduction}
An experimentally observed and also widely theoretically studied phenomenon of
cold quantum fluids is {\em Bose condensation\/}: a macroscopic number of
fluid particles form a condensate whose mechanical properties are very
different from the properties of the same fluid at higher temperatures.  For
instance, the resistance of the condensate can vanish.  Many books have been 
written on the topic, for example, there is a recent review of results including physics of partial condensation by Griffin, Nikuni
and Zaremba \cite{BECbook09}.

Mathematically, to study Bose condensation is a challenge; even the definition
of the corresponding equilibrium states for ideal Bose gas requires some
effort (see, e.g., chapter 5.2.5.\ in \cite{bratteli:ope2}).   For interacting
Bose gases, rigorous results on condensation have only recently started to
appear. They mainly concern the case of {\em total condensation\/}, or zero
temperature, where all of the available particles lie in the condensate
\cite{LiebSY00,LiebSeir06}.
In the mean-field limit, the system can then be well-described by a factored
state (i.e., the particles are not correlated) determined by a single
wave-function whose dynamics follow the Gross-Pitaevskii
equation \cite{ESchY06}.

As far as we are aware, there are no rigorous results about the dynamics of
non-equilibrium states in other than the mean-field limit with total
condensation.  In condensed matter physics, one commonly used tool in the study of
the time-evolution of bosonic quantum fluids is the bosonic Boltzmann-Nordheim
equation (also called Uehling-Uhlenbeck equation).  It describes the evolution
of the ``phase space'' density of the particles, $f(r,v,t)\ge 0$, with
$r\in\R^3$ denoting position, $v\in\R^3$ velocity, and $t\ge 0$ time, such
that
\begin{align}\label{eq:C4nonlin}
\partial_t f(r,v,t) + v\cdot  \nabla_{\! r} f(r,v,t) = \mathcal{C}_{4}[f(r,\cdot,t)](v)\, .
\end{align}
The collision operator is given by
\begin{align}\label{eq:C4def}
& \mathcal{C}_{4}[h](v_0) = 4\pi\int_{(\R^{3})^3} \!\!\rmd v_1 \rmd v_2 \rmd v_3\,
\delta(v_0+v_1-v_2-v_3)
\delta(\omega_0+\omega_1-\omega_2-\omega_3)
\left[\tilde{h}_0 \tilde{h}_1 h_2 h_3 - h_0 h_1\tilde{h}_2 \tilde{h}_3 \right]\, ,
\end{align}
where $\delta(\cdot)$ denote Dirac $\delta$-distributions: the first one is
simply a shorthand notation for a convolution integral, and the second
enforces conservation of kinetic energy in the ``collisions''; we have also
used the shorthand notations
\begin{align}
h_j=h(v_j),\quad \tilde{h}_j = 1+h_j,\quad \omega_j = E^{\text{kin}}_j = \frac{1}{2} v_j^2\, .
\end{align}

As we discuss in Section \ref{sec:physics} and Appendix \ref{appendix: Linearization}, the main contribution to the evolution of a suitably scaled perturbation $\psi_t$
of an equilibrium solution containing a condensate density $\bar{n}>0$
is obtained by the linearization of a certain three-wave collision operator.  This results in an evolution equation
\begin{align}\label{eq:psievolxvariable}
\frac{d}{dt}\psi_t = -  \bn L_3 \psi_t ,
\end{align}
	where the operator $L_3$  is defined, for functions $\psi$ which have sufficient regularity and are not too badly behaved when the argument is close to the origin or tends to infinity, as follows:
\begin{align}\label{eq:L3fulldef}
L_{3}\psi(x) &=\int_0^{\infty} \rmd y\ K_3(x,y)(\psi(x)-\psi(y)),\\
K_3(x, y) &=   4\hw(x)^2 x y \rme^{-\min(x, y)}\fbec(|x-y|)\ftbec\left(\max(x, y)\right)\ftbec(x+y)\left( 1 + \rme^{-\max(x, y)} \right),\\
\hw(x) &= \left(x^{5/2}\fbec(x)\ftbec(x)\right)^{-\frac{1}{2}}\,, \qquad
\fbec (x) := \frac{1}{\rme^x-1}\,.
\end{align}
The integral kernel $K_3$ has two types of singularities: there is boundary ``point singularity'' when $x\to 0$, and a $1/|x-y|$ -type ``line singularity'' as $x\to y$.  The defining integral is absolutely convergent for $C^{(1)}$-functions with sufficiently fast decay at $x\to 0$ and $x\to \infty$, but it is not clear if such properties could be preserved by solutions to the equation.

In this paper, we clarify two issues about the evolution equation 
(\ref{eq:psievolxvariable}).  We first define the action of the operator, $L_3$, by the integral (\ref{eq:L3fulldef}) for 
H\"older continuous functions $\psi$, with a weight ensuring that the integral is absolutely convergent. We then show that this operator is non-negative in a certain weighted $L^2$-space, and we prove that it has a unique Friedrichs extension into a self-adjoint, non-negative operator on this space. The kernel of the operator is one-dimensional, and we prove that it has a spectral gap.  Hence the semigroup generated by the self-adjoint operator is contractive in the orthocomplement of the zero subspace.

We then study pointwise solutions to the Duhamel-integrated evolution equation which  the difference functions $\Delta_t(x,y):=\psi_t(x)-\psi_t(y)$
would satisfy if they were sufficiently regular.
For initial data, sufficient regularity is guaranteed by picking them from a certain weighted Banach space of functions which satisfy a H\"older-type condition. We obtain unique solutions for the difference equation to these initial data in this Banach space. However, in order to prove the smoothing, we then have to show that the solutions obtained thus, are in fact differences of the form  $\psi_t(x)-\psi_t(y)$. The difficulty here comes from the presence of the line singularity in the kernel, because of which (\ref{eq:psievolxvariable}) cannot be solved directly, unless one has some control over the difference function. In order to get around this problem, we study, for similar H\"older initial data, a version of the evolution equation for differences where the line singularity has been regularized. We obtain unique solutions to the Duhamel-integrated regularized equation that are also difference functions, take a limit to remove the regularization, and prove that in this limit the difference function becomes identical to the $\Delta_t$ solution obtained earlier. Finally, we show that for our class of H\"older initial data, the pointwise solutions to the regularized version of the original problem  coincide with the solutions of the evolution equation (\ref{eq:psievolxvariable}) given by the $L^2$-semigroup. Our result proved for the $\Delta_t$ solutions then imply that H\"older regularity is improved by the time evolution almost everywhere for some finite time. It also implies that  the semigroup is smoothing, at least for sufficiently regular initial data.


The main technical difficulty in the study of the linearized operator mentioned above is that the evolution equation has two competing singularities, where the singularity connected with smoothing is marginal.  Thus we have to exercise considerable care in defining function spaces that remain invariant under the evolution at least for some finite time. 

To illustrate the point, let us first consider the semigroup
generated by $-L_0$, where $L_0$ denotes the positive operator on $L^2(\R^d)$
corresponding to multiplication with $\ln(1 + p^2)$ in the Fourier
space.  An explicit computation shows that the operator then acts on Schwarz functions as $\int \rmd y\, K(x,y) (\psi(x)-\psi(y))$
where the integral kernel $K$ has the same line singularity, $1/|x-y|$, as $L_3$.
Then for $t\ge 0$ the semigroup operator $\rme^{-tL_0}$ is given by
multiplication with $(1+p^2)^{-t}$ in the Fourier space, and thus
the semigroup provides slow smoothing of solutions: it maps the
Sobolev space $H^s$ to $H^{s+2t}$ for any $s$.  This behavior is different
from the standard case of semigroup of the Laplacian which has the symbol
$\rme^{-t p^2}$, and thus immediately produces smooth functions. In fact, it will become apparent later that the linearized operator, after a change of variables to make it act on a weighted $L^2(\R)$ space, in the present case closely
resembles $V_0 L_0$, where $V_0$ denotes multiplication with $\rme^{-u/2}$, $u\in \R$.
However, this space is too large to be used directly to study the full nonlinear problem.  For example, it contains unphysical solutions of infinite mass and some regularity is needed to make sense of the nonlinear collision integrals.

Our original motivation for studying the problem was to complete a nonlinear perturbation argument, and it was clear there that some smoothing property of the linearized semigroup would be needed to control the evolution.
Indeed, this programme has been taken up the Escobedo, et al., in a series of papers and preprints \cite{CE2020,escobedo:TM,escobedo:FM}.
An explicitly treatable asymptotic version of the linearized operator was considered in \cite{escobedo:TM} where its solutions were generated via a Green's function method, yielding an integral formula with an controllable kernel acting on sufficiently regular initial data.  The control of the kernel leads to estimates which are consistent with the smoothing which we prove here.

The full linearized operator has recently been studied in the preprint \cite{escobedo:FM} where solutions to the evolution equation (\ref{eq:psievolxvariable}) are generated by a perturbation argument on the kernel function derived in \cite{escobedo:TM}.
Uniqueness and possible semigroup properties are not considered.  Therefore, the present works gives at least an partial answer, in the sense of $L^2$ spaces.  Since the explicit form
of the linearized operator is slightly different from the one in (\ref{eq:L3fulldef}), we have explained in Appendix \ref{appendix: Linearization} how they nevertheless can be connected by a change of variables.

Let us now briefly describe how the rest of this paper is organized. As explained earlier, we study the linearized problem in a weighted $L^2$ space as well as a certain Banach space. For this reason, and also because the Banach space analysis is more involved and has a more-or-less free-standing structure, it seems natural to present our analyses in two separate sections. Thus we devote Section  \ref{section: resultsinL2} to the analysis of the problem in the weighted $L^2$ space. We first define Friedrichs extensions of the operator $L_3$ and a sequence of regularized approximating operators denoted by $L_3^{\vep}$. We then show that these Friedrichs-extended operators generate contractive semigroups and prove some properties for the corresponding semigroup solutions. Finally we prove the most important result of this section, namely Theorem \ref{theorem: L2limitfunction}, which shows that the semigroup solutions corresponding to the approximating operators converge to the $L^2$ solution for our original operator. Subsection $2.1$ contains the definitions of the Friedrichs extensions and the proof of the existence of the respective semigroups. Subsection $2.2$ is reserved for the convergence result mentioned above and the lemmas leading up to it. The proofs of our lemmas in Subsection $2.2$ rely on some results obtained in a Banach space. This does not lead to any circularity of argument because our Banach space results are independent of those obtained in this subsection. To emphasize this, we explicitly mention it whenever we need to use any Banach space result in Section \ref{section: resultsinL2}.

Section \ref{section: mainresultsBanach} contains all of our Banach space results. The H\"{o}lder regularity is obtained in Theorem \ref{theorem: DeltaSolutionsY} as an existence-uniqueness result in a certain Banach space of functions of two variables.
We then show that these functions can be identified with differences of the $L^2$ solutions obtained in Section \ref{section: resultsinL2}, at least for some finite time.  In order to do this we have to first obtain results such as Theorem \ref{theorem: rDpsiSoln}  and Proposition \ref{theorem: rDpsiDeltaApprox} for the regularized approximating operators. The main smoothing result Theorem \ref{theorem: MainResultPaper} follows as a straightforward consequence of our Banach space results and Theorem \ref{theorem: L2limitfunction}.

We devote the rest of this section to a discussion of the physical connection and derivation of the linearized three waves collision operator. The necessary computational details of this derivation can be found in Appendix \ref{appendix: Linearization}.

\subsection{Physical motivation of the nonlinear problem and the proposed linearization}
\label{sec:physics}

	The Bolztmann--Nordheim equation (\ref{eq:C4nonlin})--(\ref{eq:C4def})
has not yet been rigorously derived from evolution of a bosonic quantum system 
but the following
conjecture can be generalized from the discussion in \cite{ls09}.  Consider a
{\em translation invariant quasi-free} initial state on the bosonic Fock space
determined by a two-point correlation function $g_0(r_1-r_2)$ where $g_0$ is a
rapidly decreasing function.
Suppose that the $N$-particle dynamics is given by a Hamiltonian with weak
pair interactions,
$H_N = \sum_{i=1}^N \frac{1}{2} p_i^2 + \lambda \frac{1}{2} \sum_{i\ne j}
V(r_i-r_j)$, $0<\lambda\ll 1$.  If $V$ is well-localized and $\int_{\R^3}\!
\rmd r\, V(r)=1$, then up to times $O(\lambda^{-2})$ the state should be
translation invariant and quasi-free apart from small corrections, and the
following limit of the {\em Fourier transform\/} of the time-evolved two-point
function should exist: $\widehat{g}_t(v)|_{t=\lambda^{-2} \tilde{t}} \to
W(v,\tilde{t})$ as $\lambda\to 0^+$ for any, not too large, $\tilde{t}\ge 0$.
In addition, this limit should be well approximated by solutions to the
equation
$\partial_t W(v,t) = \mathcal{C}_{4}[W(\cdot,t)](v)$, with initial data
$W(v,0)=\widehat{g}_0(v)$.  (In \cite{ls09}, for technical reasons, only a
discretized version of this problem is considered.  The above conjecture is a
generalization of Conjecture 5.1 in the bosonic case ``$\theta=1$'' there.)

Although the conjecture remains unproven at the moment, it relies on a
perturbative expansion which has been successfully applied to a related
problem of equilibrium time-correlations for a discrete nonlinear
Schr\"{o}dinger equation \cite{NLS09} (the connection to the above conjecture
is outlined in Sec.\ 2.3 there)
and, more recently, also for its continuum version \cite{DH21}.
The proof in \cite{NLS09} uses quite heavily a
property analogous to $\sup_v |\widehat{g}_0(v)|<\infty$
which in the above setup would be a consequence of the assumed
sufficently fast decay of correlations for large spatial separation.  For a
critical ideal Bose fluid at equilibrium one has $W(v)$ equal to the critical
Bose-Einstein distribution which blows up as $|v|^{-2}$ near $v=0$.  Thus this
condition is violated in a critical system, and one can justifiably
question the validity of the perturbative derivation for any system with
critical and supercritical densities.

If $W$ is homogeneous and {\em isotropic\/}, it depends only on
$x=\omega(v)=\frac{1}{2}v^2$, and the evolution equation for $f(x,t):=
W(\sqrt{2 x} \hat{e}_1,t)$ can be rewritten for $x,t\ge 0$ as
$\partial_t f(x,t) = \mathcal{C}_{4}[f(\cdot,t)](x)$ after (with a slight
abuse of notation and neglecting an overall numerical constant) we define
\begin{align}
\mathcal{C}_{4}[f](x_0) := \frac{1}{\sqrt{x_0}}
\int_{\R_+^{2}} \!\!\!\rmd x_2 \rmd x_3\, \cf(x_1\ge 0) \min_{j=0,1,2,3} \sqrt{x_j}
\left[\tilde{f}_0 \tilde{f}_1 f_2 f_3 - f_0 f_1\tilde{f}_2 \tilde{f}_3 \right] \, ,
\end{align}
where $x_1=x_2+x_3-x_0$, and $f_i=f(x_i)$.  (To show the connection is not
completely straightforward.  It has been done in Appendix A of \cite{STk97}
for $W$ which are Schwartz functions.)  The numerical solutions to this
equation were studied by Semikoz and Tkachev in \cite{STk97}, and they found a
finite time blowup at $x=0$ for smooth, but supercritical, initial data.  In
light of the doubts about the perturbative derivation for singular cases, it
is not clear how the solutions should be continued beyond the formation of the
singularity.

One possibility is that the perturbative kinetic argument simply becomes
inapplicable, and one has to go back to the original time-evolution in Fock
space to resolve the issue.  Another possibility is that nothing very special
happens, and the equation continues to hold in its original (pointwise) sense,
only restricted to $x>0$.  The equation
$\partial_t f= \mathcal{C}_{4}[f]$, $x>0$, with $f(x,0)\sim x^{-7/6}$, was
studied by Escobedo, Mischler, and Vel\'{a}zquez in \cite{EMV08}. (The value
$\frac{7}{6}$ is related to Kolmogorov theory of wave turbulence
\cite{DNPZ92}.)
They show the existence of solutions locally in time, preserving the
$x^{-7/6}$ singularity.  However, these solutions do not conserve total mass
of particles (which is obviously conserved by the microscopic dynamics).  If
one thinks that the extra mass is exchanged with the condensate, this way of
``extending'' the solution would correspond to adding the condensate mass as
an extra degree of freedom {\em with no backreaction\/} to the normal fluid.

A different extension was considered by Lu in \cite{Lu04,Lu05}.  He considers
{\em weak\/} solutions, positive measures $\mu_t(\rmd x)$ such that $t\mapsto
\mu_t$ is weakly continuously differentiable and $\partial_t \mu_t=
\mathcal{C}_{4}[\mu_t]$ in the sense of distributions.  The existence of such
solutions is proven in \cite{Lu04} wherein the precise meaning of how the
measures are solutions is defined on page 1611.   In \cite{Lu05} he also proves
that the solutions can be chosen so that they conserve both mass and energy and
converge to the physically
expected equilibrium distribution as $t\to \infty$.  This occurs even in the supercritical case,
and it is shown that then
a portion of the total mass condenses to $x=0$
(at least asymptotically as $t\to \infty$).  All of these results are
deduced from subsequences of approximating solutions to a regularized problem,
and as such leave open the uniqueness of these solutions, and are not amenable
for numerical treatment.

In \cite{STk97}, Semikoz and Tkachev proposed a different method of continuing
the solution after a condensate has formed.  On physical grounds, they
postulated that the solution would be a positive measure of the form
$\ftot (x,t)\sqrt{x}\rmd x = \freg (x,t)\sqrt{x}\rmd x+n(t)\delta(x)\rmd x$,
which corresponds to
putting a mass $n(t)\ge 0$ into a $\delta$-distribution at the origin
$v=0\in \R^3$
and allowing only for a regular distribution for $|v|>0$.  (The square roots
are explained by the identity $v^2 \rmd |v| = \sqrt{2 x} \rmd x$.)  With this
ansatz they arrived at coupled equations of the form
\begin{align}\label{eq:fregdot}
\partial_t \freg(t) = \mathcal{C}_{4}[\freg(t)] + n(t) \mathcal{C}_{3}[\freg(t)]\, , \quad
\frac{\rmd}{\rmd t} n(t) = - n(t) \rho[\mathcal{C}_{3}[\freg(t)]]\, ,
\end{align}
where $ \mathcal{C}_{3}$ is a new collision operator and $\rho[f]$
denotes the mass functional.
Since these equations involve $n(t)$, their solutions do
not coincide with the singular solutions studied in \cite{EMV08}.  The
equations were again solved numerically, and convergence towards the expected
equilibrium was found in \cite{STk97}.
But even this set of equations is somewhat problematic from the physical point of view: it does not answer
how a condensate can be generated (if $n(0)=0$, it remains zero for all times),
and for regular functions $f$ one can prove that $\rho[\mathcal{C}_{3}[f]]\ge 0$, so it
seems that $n(t)$ can only decrease.  A more detailed analysis of the ansatz was made
by Spohn in \cite{spohn08}, and
among other things, he showed that if $\freg(x,t)\simeq a(t) x^{-1}$, as would
be the case for a critical equilibrium distribution, then the second equation is equal to
\begin{align}\label{eq:ntdot}
\frac{\rmd}{\rmd t} n(t) = - n(t) (2 \rho[\sqrt{x}\freg(x,t)]-c_0 a(t)^2)\, , \quad \text{with }c_0 = \frac{1}{3} \pi^2\, .
\end{align}
It follows that the critical $\freg$ are stationary solutions, and there can be
an exchange of mass with either sign between the regular fluid and the
condensate.  This computation and those in \cite{EMV08} clearly illustrate
that, for singular data, the meaning of the Boltzmann equation has to be
carefully specified.

\newcommand{\mutot}{\mu^{\mathrm{tot}}}

In this paper we look at a slight modification of the above
evolution equations in which, at least in principle,
condensate can be freely created and destroyed.
We use the same equations as before for the {\em regular part\/}, but do
not try to form a differential equation for $n(t)$.  Instead, as motivated by
Lu's results, we impose a strict conservation of mass for all times and
use this as a {\em definition\/} of $n(t)$. Our aim is to study the linearization of the three waves collision operator $\mathcal{C}_3$ around a stationary solution.  To this end, 
we consider the full distribution function for a Bose fluid at time $t$, in
the presence of a condensate, given by the measure
$\mutot_t(\rmd x) = \freg (x,t)\sqrt{x}\rmd x+n(t)\delta_0(\rmd x)$ on $\mathbb{R}_{+}:=[0,\infty)$.
Here $\delta_0$ denotes the unit measure concentrated at $x=0$, and $\freg$ denotes a function on $\mathbb{R}_{+}$
with finite mass and energy, respectively defined by the functionals
\begin{align*}
\rho[f] := \int_0^\infty \rmd x \sqrt{x} f(x),\quad
e[f] := \int_0^\infty \rmd x \sqrt{x} x f(x)\, .
\end{align*}
Including the contribution from the condensate, the total mass and energy are then defined as
\begin{align*}
M(t)= \rho [\freg  (t)] +n(t)\, ,\quad E(t) = e[\freg  (t)]\, .
\end{align*}
To summarize the conventions made so far:
to get the ``density'' in the original 3-dimensional
Boltzmann-Nordheim equation, we use ``$f(r,v,t)\rmd v$'' $=$ $\sqrt{2} \mutot_t(\rmd (v^2/2)) \rmd \Omega$, where $\rmd \Omega$ denotes the standard integration over the angular variables of $v$.  The presence of the scaling factor $\sqrt{2}$ here guarantees that ``$\tilde{f}(r,v,t)\rmd v$'' $=$ $\sqrt{2} \tilde{\mu}^{\mathrm{tot}}(\rmd (v^2/2)) \rmd \Omega$ with
$\tilde{\mu}^{\mathrm{tot}}(\rmd x):=(1+\freg (x,t))\sqrt{x}\rmd x+n(t)\delta_0(\rmd x)$.

Suppose that the initial data is determined by $n(0)=n_0\ge 0$ and $\freg (x,0)=f_0(x)$, where  $f_0\ge 0$ is suitably regular. In particular, we assume that
\begin{align*}
e_0 := e[f_0],\quad \rho_0 := \rho[f_0]
\end{align*}
are finite.
If $e_0=0$, then $f_0=0$ almost everywhere, and defining $\freg(x,t)=0$, $n(t)=n_0$, will yield a solution, corresponding to total condensation.  This case is not of interest here, let us only point out that it is in line with the previous results on total condensation: the homogeneous solutions to the Gross-Pitaevskii equation are wave-functions with a constant magnitude, thus leading to particle densities which do not vary in time.

Let us thus assume that $e_0>0$ when also $\rho_0>0$.  For such initial data, and with weak interactions, one would physically expect the system to relax to an equilibrium distribution which is very close to that of free particles.  As discussed in the introduction, the Boltzmann-Nordheim equation is believed to arise in a scaling limit with the strength of the interaction going to zero, thus its stationary solutions should be given by the ideal gas equilibrium distributions.
For subcritical initial data these depend on two parameters: an inverse temperature $\beta>0$ and
the chemical potential $\mu\le 0$.  A particular case is when $\mu=0$ and the corresponding distribution is called
{\em the critical Bose-Einstein distribution\/},
\begin{align}
\feql(x)= \fbec (\beta x) \, , \quad \text{ where }\fbec (x) := \frac{1}{\rme^x-1}\, ,
\end{align}
and the subcritical distributions are given by $f_{\beta,\mu}(x):=\fbec(\beta(x-\mu))$, $\mu<0$.
Since $e_0>0$, there obviously is a unique $\beta>0$ such that $ e_0 = e[\feql]$,
and this is determined by
\begin{align}
\beta := \left(\frac{e[\fbec]}{e_0} \right)^{\frac{2}{5}}\, .
\end{align}
We assume now that the initial state is {\em supercritical\/}, $M(0)>\rho[\feql]$, i.e.,
\begin{align*}
\frac{n_0 + \rho_0}{e_0^{3/5}} > \frac{\rho[\fbec]}{e[\fbec]^{\frac{3}{5}}} \, .
\end{align*}
If this is true, there are no Bose-Einstein distributions which would have the right energy and particle densities.
These results can be found in many references,
for instance, see Section 2 in \cite{spohn08} and Section 6 in \cite{Lu04}.

On physical grounds,
it is expected that the normal fluid relaxes towards the corresponding critical distribution and that the additional particles are forming the condensate: since the particles in the condensate do not contribute to energy density, this allows having the same mass and energy in the initial and the equilibrium state.  Motivated by this physical discussion, we
define the {\em equilibrium condensate density\/} by
\begin{align}
\bn := M(0)-\rho[\feql] =
n_0 + \rho_0-\rho[\feql]\, .
\end{align}
With these definitions, the assumption of a supercritical initial state is equivalent to assuming $\bn>0$.

We postulate that the time-evolution of the system conserves total mass, i.e., that
\begin{align}\label{eq:defnt}
n(t) = M(0)-\rho[\freg (t)] = \bn - \rho[\freg (t)-\feql]\, .
\end{align}
so the evolution equation is
\begin{align}\label{oeqn}
\frac{d}{dt}\freg  (x,t) = \mathcal{C}_{4}[\freg  (\cdot,t)](x)+n(t)\mathcal{C}_{3}[\freg  (\cdot,t)](x),
\end{align}
which becomes a closed equation for $\freg$, after we insert the mass conservation law in (\ref{eq:defnt}).
Here the first collision operator is given by
\begin{align}
& \mathcal{C}_{4}[f](x_0) =
\frac{1}{\sqrt{x_0}}
\int_{\R_+^{2}} \!\!\!\rmd x_2 \rmd x_3\, \cf(x_1\ge 0) I(x)
\left[\tilde{f}_0 \tilde{f}_1 f_2 f_3 - f_0 f_1\tilde{f}_2 \tilde{f}_3 \right]_{x_1=x_2+x_3-x_0}
\end{align}
where $f_i:= f(x_i)$, $\tilde{f}_i=1+f_i$, and
\begin{align*}
I(x) := \min_{j=0,1,2,3} \sqrt{x_j}\, .
\end{align*}
We have also ignored an overall explicit constant which can be recovered by rescaling time.
From this form, a formal substitution of the ansatz yields for the interaction
term with the condensate
\begin{align}\label{defn: C3op}
& \mathcal{C}_{3}[f](x) =
\frac{2}{\sqrt{x}}  \int_0^x\! \rmd y\,
\left[\tilde{f}(x) f(x-y) f(y) - f(x) \tilde{f}(x-y)\tilde{f}(y)\right]
\nonumber \\ & \qquad
- \frac{4}{\sqrt{x}}  \int_x^\infty\! \rmd y\,
\left[ \tilde{f}(y) f(y-x) f(x) - f(y) \tilde{f}(y-x)\tilde{f}(x) \right] \, .
\end{align}

As shown in \cite{spohn08}, for $f_0=\feql$ one has $\mathcal{C}_{4}[f_0]=0=\mathcal{C}_{3}[f_0]$.  Thus $f(x,t)=\feql(x)$
yields a stationary solution of this equation. (Note that then by definition also $n(t)=\bn=n_0$ is constant.)  Here we inspect if ``small'' perturbations of such states lead to solutions which relax towards a state of the same type (the parameters of the new state do not need to be the same as those of the original one).  A convenient way to define the perturbation is to consider
\begin{align*}
\psi_t(x) := \frac{f(x,t)-\feql(x)}{\Rbe (x)}\, , \quad \text{ where } \Rbe (x) = \beta x \feql(x) (1+\feql(x))
= R_1(\beta x)\, .
\end{align*}
As here $R_1\simeq x^{-1}$, for solutions of the type considered by Spohn, $f(x,t)\simeq a(t) x^{-1}$, one would have $a(t)=(\psi_t(0)+1)/\beta$.  Thus if continuous solutions with such asymptotics exist, then $\psi_t$ would be continuous also at $0$.

Now we can solve the dependence on $\beta$, by scaling $x'=\beta^{-1} x$.  Then $\mathcal{C}_{4}$ gains a factor of $\beta^{2}$ and $\mathcal{C}_{3}$ a factor of $\beta^{1/2}$.  Thus if we have a solution $f_1(x,t;\bn)$ for $\beta=1$,
a solution to the generic case is given by $f(x,t;\bn,\beta):= f_1(\beta x,\beta^{-2} t;\beta^{-3/2}\bn)$.  Therefore, it suffices to consider the case $\beta=1$, when $\feql=\fbec$, and
\begin{align}
\psi_t(x) = \frac{f(x,t)-\fbec(x)}{R(x)}\, , \quad \text{ where } R (x) = R_1(x)=  x \fbec(x) \ftbec(x)\, .
\end{align}
Here and in the following, we  employ a shorthand notation $\ftbec := 1+\fbec$.  
The time-evolution of $\psi_t(x)$ is thus determined by
\begin{align*}
\frac{d}{dt}\psi_t (x)&= \frac{1}{R(x)} \Bigl(
\mathcal{C}_{4}[\fbec+R \psi_t](x)
+ (\bn - \rho[R \psi_t])\, \mathcal{C}_{3}[\fbec+R \psi_t](x)\Bigr)\, ,
\end{align*}

Now let $-\ol\mcL_i$ denote the linearization of $\mathcal{C}_{i}$ around $\fbec$, and $\mathcal{Q}_i$ the corresponding remainder:
$\mathcal{Q}_i[h] := \mathcal{C}_{i}[\fbec+h]+\ol\mcL_i h$.  Since $\mathcal{C}_{i}[\fbec]=0$, $\mathcal{Q}_i[h]$ is quadratic in $h$.
Thus the evolution equation for $\psi_t$ can be written in the form
\begin{align*}
\frac{d}{dt}\psi_t = - L \psi_t + Q[\psi_t]\, ,\text{ where}
\end{align*}
\begin{align*}
\ L & = L_4 +  \bn L_3\, ,\quad & & L_i = R^{-1} \ol\mcL_i R\, ,\\
Q[\psi] & = Q_4[\psi] + \bn Q_3[\psi] - \rho[R \psi] Q_3[\psi] +
\rho[R \psi] L_3 \psi\, ,\quad & &  Q_i[\psi] := R^{-1} \mathcal{Q}_i[R \psi]\, .
\end{align*}

It turns out that the linearized three waves collision operator is the most singular and can thus be thought of as dominant.   Following the computations presented in Appendix \ref{appendix: Linearization}, we arrive at the evolution equation (\ref{eq:psievolxvariable}) with the form (\ref{eq:L3fulldef}) for the operator $L_3$. Although it is not shown in this paper, our main result can be extended without difficulty to cover also the subdominant operator $L_4$, leading to a smoothing result for the full linearized operator.

\section{Solutions in a Weighted $L^2$-space}\label{section: resultsinL2}

In this section we consider, in a certain weighted $L^2$ space of functions, the linear operator appearing in the evolution equation (\ref{eq:psievolxvariable}) as well as a certain sequence of related linear operators. We first change variables $x\to u=\ln(\rme^x -1),$ $x$ being the energy variable. Thus, while the old variable $x$ was in $\R_+$, our new variable $u$ takes values in $\R,$ and we look at  (\ref{eq:psievolxvariable})  in a weighted $L^2$ space of functions on $\R$. Also, we will henceforth include the prefactor $\bn$ in the definition of $L_3$ (while continuing to use the same symbol for the operator), so that the right hand side of (\ref{eq:psievolxvariable}) now reads $(-L_3\psi_t)$.

As discussed in more detail in Appendix \ref{appendix: Linearization},
the resulting operator then naturally acts on a weighted $L^2$ space.
The weight function $\nu$ is
\begin{align*}
\nu(\rmd w) =\nu(w)\rmd w= \rme^{-w}\left(\ln ( 1 + \rme^w)\right)^{\frac{5}{2}}\rmd w.
\end{align*}
In this space the linearized operator acts on any suitably regular function $\psi$ (say, compactly supported and smooth)  in its domain $D(L_3)$ as follows:
\begin{align}\label{defn: OrigL2Operator}
(L_3\psi)(u) &= \int_{\R}\nu(\rmd v)\ol K_3(u, v)\big(\psi(u) - \psi(v)\big),\\
\ol K_3(u, v) &= 4\bn\left[(\ln(1 + \rme^u))(\ln(1 + \rme^v))\right]^{-\frac{3}{2}}\frac{\rme^{-|u-v|}}{1 + \rme^{-\min(u, v)} + \rme^{-\max(u, v)}}\ \frac{2 + \rme^{\max(u, v)}}{1 - \rme^{-|u-v|}}.\nonumber
\end{align}
The associated sesquilinear form is given by
\begin{align*}
\tQ\big( \phi,\psi \big) = \frac{1}{2}\int_{\R^2}(\nu\times\nu)(\rmd u,\rmd v)\ol K_3(u, v) \big(\phi(u) - \phi(v)\big)^{*}\big( \psi(u) -\psi(v)\big).
\end{align*}
$\tQ$ is evidently symmetric and non-negative.  We now extend the form domain $D(\tQ)$ to cover all $\psi$ for which $\tQ\big(\psi,\psi \big)<\infty$ in the sense the above integral is convergent (note that in this case the integrand is real and non-negative).
The Cauchy--Schwartz inequality then implies that $\tQ(\phi, \psi)$ is defined for all $\phi,\psi\in D(\tQ)$ as an absolutely convergent integral.

We will also look at a sequence of regularized linear operators $L_3^{\vep}$, defined for all $\vep_0>0$ in the following way: 
\begin{align}\label{defn: ApproxL2Operator}
(L_3^{\vep}\psi)(u) &= \int_{\R}\nu(\rmd v)\ol K_3^{\vep}(u, v)\big(\psi(u) - \psi(v)\big),\ \forall\psi\in D(L_3^{\vep}),\ \text{where,}\nonumber\\
\ol K_3^{\vep}(u, v) &= \ol K_3(u, v) \frac{1 - \rme^{-\min(\vep(u, v), |u-v|)}}{1 - \rme^{-\vep(u, v)}},\ \vep(u, v) =\vep(\max(u, v))= \vep_0\exp\left(-\frac{\mu'}{\gamma_0}\max(a_1, \max(u, v))\right).
\end{align}
Here $\mu',\ \gamma_0,$ and $\ a_1$ are positive-valued parameters. For results connected with only the $L^2$-solutions, the values of these parameters in $\R_+$ do not matter. However, for our main result, proved in a certain Banach space, these parameters have to be restricted to certain intervals in $\R_+$, as described in (\ref{defn: tGammaGamma0}) and in Subsection \ref{subsection: regDpsiDiff}. The sesquilinear form $\tQ_{\vep}$ corresponding to the regularized linear operator is again non-negative and symmetric.

In the rest of this section we will construct Friedrichs extensions of the operators $L_3$ and $L_3^{\vep}$ in $L^2(\nu)$, show that these generate contractive semigroups in $L^2(\nu)$, and finally, we prove that the semigroup solutions associated with the Friedrichs extension of $L_3^{\vep}$ approximate those for the extension of $L_3$ as $\vep_0\to 0.$

\subsection{Friedrichs Extensions and Their Properties}

We will describe in detail the construction of the Friedrichs extension for $L_3$. The extensions for $L_3^{\vep}$, $\forall\vep>0$ can be constructed similarly.

Starting from the sesquilinear form $\tQ$, we first define the following inner product \begin{align*}
Q(\phi, \psi) = \tQ(\phi, \psi) + (\phi, \psi)\,,
\end{align*}
on the previously defined domain
\begin{align*}
D(Q) = D(\tQ) = \{ \psi\in L^2(\nu): Q(\psi, \psi)<\infty \}.
\end{align*}
This domain is in fact already complete, i.e., suitable for the Friedrichs extension, as the next result shows.

\begin{lemma}\label{D_Q_completeness}
	$D(Q)$ is complete under the inner product $Q$.
\end{lemma}
\begin{proof}
	In order to prove completeness of $D(Q)$,
	consider an arbitrary
	sequence $\psi_n\in D(Q),$ such that $\psi_n$ is Cauchy under $Q$, i.e., $Q(\psi_n-\psi_m, \psi_n-\psi_m)\to 0$ as $m,n\to\infty$.
	Since the assumptions imply that $(\psi_n)$ is also Cauchy in $L^2(\nu)$, there is $\psi\in L^2(\nu)$ such that $\psi_n\to\psi$.
	Thus we only need to show that i) $\psi\in D(Q),$ and,	ii) $Q(\psi_n-\psi, \psi_n-\psi)\to 0$ as $n\to\infty$.
	
	Let us first note that, for any $r>0$, we can define the bounded kernel
	$K_r(u, v) = \cf(\ol K_3(u, v)< 1/r) \ol K_3(u, v)$, so that, $K_r\leq \ol K_3$ and $K_r$ converges pointwise to $\ol K_3$ as $r\to 0^+$. Then the corresponding sesquilinear forms satisfy
	\begin{align*}
	\lim_{r\to0^+} Q_r(\phi, \psi) = Q(\phi, \psi),\ \forall \phi,\psi\in D(Q),\ \text{by dominated convergence}.
	\end{align*}
	Also, for any $r>0$, there exist positive constants $C,\ C_1<\infty$, such that
	\begin{align*}
	|Q_r(\phi,\psi)|&\leq |(\phi,\psi)| + \int_{\R^2}(\nu\times\nu)(\rmd u,\rmd v)K_r(u, v)|\big(\phi(u) - \phi(v)\big)||\big( \psi(u) -\psi(v)\big)|  \leq \big( 1 + \frac{C}{r}\big)\norm{\psi}_{L^2}\norm{\phi}_{L^2},\ \\
	& \text{and,}\ \norm{\psi}_{L^2}^2\leq Q_r(\psi, \psi)\leq \big( 1 + \frac{C_1}{r}\big)\norm{\psi}_{L^2}^2.
	\end{align*}
	The above inequalities imply that $D(Q_r)= L^2(\nu)$ and that the norm defined by $Q_r$ is equivalent to the $L^2$-norm. This means $\psi_n\to\psi\in D(Q_r)$ and $Q_r(\psi_n, \psi_n)\to Q_r(\psi, \psi)$ as $n\to\infty.$ 
	
	Since $\psi_n\in D(Q)$ is $Q-$Cauchy, it is $Q-$bounded. Thus there exists some positive constant $C'<\infty$, such that $Q_r(\psi_n, \psi_n)\leq Q(\psi_n, \psi_n)\leq C'$ for all $n$, and for all $ r>0.$ This means $Q_r(\psi,\psi)\leq C',\ $ for all $r>0$ and $Q_r(\psi, \psi)\nearrow Q(\psi, \psi)\leq C'$ by monotone convergence.
	Thus $\psi\in D(Q),$ and condition i) is fulfilled.
	
	We can now prove ii). Since $\psi_n\overset{L^2}{\longrightarrow}\psi\in D(Q),$ the sequence $\phi_n=\psi_n -\psi$ is such that $\phi_n\in D(Q)\ $ for all $n$ and $\phi_n\to0\in L^2(\nu)$.
	Choose $\ol\vep>0.$ Then, since the sequence $\phi_n$ is Cauchy in $D(Q)$,  there exists $n_0$ such that $Q(\phi_n-\phi_m, \phi_n-\phi_m)<\ol\vep^2,\ $ for all $m, n>n_0.$
	Now $Q_r(\phi_n, \phi_n)\nearrow Q(\phi_n, \phi_n)$, and thus for $\ol\vep>0$ and $n>n_0,$ there exists $r=r(n,\ol\vep)$ such that $0\le Q(\phi_n, \phi_n) - Q_r(\phi_n, \phi_n)<\ol\vep.$
	Then we have, for all $ m, n>n_0,$ the following:
	\begin{align*}
	Q_r(\phi_n, \phi_n)&= Q_r(\phi_n-\phi_m, \phi_n) - Q_r(\phi_m, \phi_m-\phi_n) + Q_r(\phi_m, \phi_m)\\
	&\leq |Q_r(\phi_n-\phi_m, \phi_n) | + | Q_r(\phi_m, \phi_m-\phi_n)| + Q_r(\phi_m, \phi_m)\\
	&\leq 4\sqrt{C'}\Big(Q(\phi_n-\phi_m, \phi_n-\phi_m)\Big)^{1/2} + Q_r(\phi_m, \phi_m)\\
	&\leq 4\sqrt{C'}\ol\vep + Q_r(\phi_m,\phi_m).
	\end{align*}
	We can now let $m\to\infty$, so that $Q_r(\phi_m, \phi_m)\to 0.$ Thus
	\begin{align*}
	Q(\phi_n, \phi_n)< Q_r(\phi_n, \phi_n) + \ol\vep\leq (4\sqrt{C'}+2)\ol\vep,
	\end{align*}
	which means $Q(\phi_n, \phi_n)\to 0$ as $n\to\infty,$ proving ii).
\end{proof}

In order to construct the Friedrichs extension we start from a simpler version of the operator $L_3$, defined on the space $C_c^{0,\alpha}$ of compactly supported $\alpha$-H\"{o}lder continuous functions on $\R$. Clearly this is a subspace of $D(Q)$ and dense in $L^2(\nu)$. Then for $\psi\in C_c^{0,\alpha}(\R)$, we define
\begin{align*}
(L_3^R\psi)(u) &= \int_{\R}\nu(\rmd v)\ol K_3(u, v)\big( \psi(u) - \psi(v) \big),
\end{align*}
where the integral is absolutely convergent for all $u$, and it yields a function in $L^2(\nu)$.
We can use Fubini's theorem to conclude that for all $\psi\in C_c^{0,\alpha}(\R)$ and $\phi\in D(Q)$
\begin{align*}
(\phi, L_3^R\psi) = Q(\phi, \psi) - (\phi, \psi),
\end{align*}
so the form domain of $L_3^R$ is contained in $D(Q).$ 
We can then conclude that the Friedrichs extension $\ol L_3$ of $L_3^R$
is given by
\begin{align*}
(\phi, \ol L_3\psi) &= Q(\phi, \psi) - (\phi, \psi),\ \forall \psi\in D(\ol L_3),\ \phi\in D(Q),\\
D(\ol L_3) &= \{ \psi\in D(Q)\mid\exists C<\infty\text{ such that } |Q(\phi,\psi)|\leq C\norm{\phi}_{L^2}, \forall\phi\in D(Q) \}.
\end{align*}
We refer to chapter VIII, pages $329$-$334$, \cite{riesz-nagy:fa} and chapter VI, pages 322-326, \cite{kato:plo} for more details.

Our main results are proved in the Banach space $\ol X$ of continuous functions on $R$ satisfying a H\"{o}lder-type condition (see  (\ref{eq: XbarBanachDefn})). From the definition (\ref{defn: tGammaGamma0}) of the  weight $\Gamma_0$ characterizing $\ol X$, it is obvious that
\begin{align*}
\forall\psi\in \ol X,\ Q(\psi, \psi)<\infty, \text{since } \ol K_3(u, v)\big(\Gamma_0(u, v)\big)^2 \text{is absolutely integrable under $\nu\times\nu$.}
\end{align*}
Thus for all $\phi\in D(Q),$ for all $\psi\in \ol X$, the integral defining $Q(\phi, \psi)$ is absolutely convergent. It is easy to check that
\begin{align*}
G\in L^2(\nu),\ \text{where}\ G(u) = \int_{\R}\nu(\rmd v)\ol K_3(u, v)|\psi(u) - \psi(v)|.
\end{align*}
Therefore, we may find a constant $C>0$ such that
\begin{align*}
|Q(\phi, \psi)|&\leq C\norm{\phi}_{L^2}\norm{G}_{L^2}, \text{ so } \psi\in D(\ol L_3).
\end{align*}
Thus $\ol X\subset D(\ol L_3).$

In Section \ref{section: mainresultsBanach}, where our main results are written, we drop the bar and simply denote the extended operator by $L_3$.
Let us point out that, by a similar argument as used above, if $\psi$ is bounded, measurable and satisfies
$\sup_u \int_{\R}\nu(\rmd v)\ol K_3(u, v)|\psi(u) - \psi(v)|<\infty$, then
$\psi\in D(L_3)$ and $L_3\psi$ is given by the absolutely convergent integral.
%
%


The Friedrichs extension $\ol L_3^{\vep}$ of the operator $L_3^{\vep}$ is constructed in an exactly similar manner as above and it is easily seen that $\ol X_{\vep}\subseteq D(\ol L_3^{\vep}),$ where $\ol X_{\vep}$ is the corresponding Banach space.

We now prove the following result about these extensions.

\begin{theorem}\label{theorem: L2SemigroupSolution}
	The linear operators $\ol L_3$ and $\ol L_3^{\vep}$ are non-negative and their zero subspace is spanned by the constant function.  In addition, all have spectral gaps in $L^2(\nu)$ and thus the semigroups generated by these linear operators are strictly contractive in the orthocomplement of the zero subspace.
\end{theorem}

\begin{proof}
	We first prove the claim for $\ol L_3.$ The proof for $\ol L_3^{\vep}$ is similar.
	
	If
	$\psi\in D(\ol L_3)$, we have the following lower bound for the corresponding quadratic form
	\begin{align*}
	(\psi, \ol L_3\psi)& = \frac{1}{2}\int_{\R^2} (\nu\times\nu)(\rmd u, \rmd v) \ol K_3(u, v)|\psi(u) - \psi(v)|^2 \geq (\psi, L'\psi),
	\end{align*}
	where
	\begin{align*}
	&(L'\psi)(u) = V'(u)\psi(u) - \int_{\R}\nu(\rmd v)K'(u, v)\psi(v), \ V'(u)=\int_{\R} \nu(\rmd v) K'(u, v),\\
	&\text{and,}\ K'(u, v) = 2K'_0(u, v) + K'_1(u, v)\Big[ \cf(v>u/2) + \cf(v<2u) \Big] + K'_2(u, v)\Big[ \cf(v>3u/2) + \cf(v<2u/3) \Big].\\
	&\text{Here}\ K'_0(u, v) = \bn\cf(-\ln 2\leq u\leq \ln 2)\cf(-\ln 2\leq v\leq\ln 2)\Big( \ln(1 + \rme^u)\ln(1 + \rme^v) \Big)^{-3/2},\\
	&\ \ \ \ K'_1(u, v) = \bn\cf(u<-\ln 2)\cf(v<-\ln 2)\Big( \ln(1 + \rme^u)\ln(1 + \rme^v) \Big)^{-3/2}\rme^{\min(u, v)}\rme^{-|u-v|},\\
	&\text{and, } K'_2(u, v)=\bn\cf(u>\ln 2)\cf(v>\ln 2)\Big( \ln(1 + \rme^u)\ln(1 + \rme^v) \Big)^{-3/2}\rme^{\max(u, v)}\rme^{-|u-v|}.
	\end{align*}
	
	Then it is easy to see that there exists $C'>0$, such that the potential $V'(u)$ has the following lower bound:
	\begin{align*}
	V'(u)&\geq C'\bn\Big[ \big( 1 - 2\rme^{\frac{1}{2}u} \big)\cf(u<-\ln 4)  + \rme^{\frac{3}{2}u}\cf(u<-\ln 2) + \rme^{-\frac{3}{2}u}\cf(-\ln 2\leq u\leq \ln 2) \Big.\\
	&\ \ \ \ \Big. + u^{-\frac{1}{2}}\rme^{-\frac{1}{2}u}\cf(u>\ln 2) + u^{-\frac{3}{2}}\Big(\left(\frac{2}{3}u\right)^2 - (\ln 2)^2\Big)\cf(u>\frac{3}{2}\ln 2) \Big].
	\end{align*}
	Therefore,
	$V'(u)$ is strictly positive and there exists $a_*>0$ such that $\sigma(V)\subset (a_*, \infty).$ It is also easily checked that the integral operator associated with the kernel $K'$ is Hilbert-Schmidt on $L^2(\nu)$. This means, firstly, that $L'$ is self-adjoint (see Theorem $4.3$,  Chapter IV, \cite{kato:plo}) and secondly, that $L'$ has the same essential spectrum as $V'$ (see Theorem $5.35$, Chapter IV, \cite{kato:plo}). The first implication tells us that $\sigma(L')\subset [0, \infty)$ (the associated quadratic form being non-negative) and the second implication means that $\sigma(L')\cup [0, a_*)$ contains only discrete semi-simple eigenvalues, so that $a_*$ is the only possible accumulation point of the spectrum.
	
	Clearly then, $L'$ and hence $\ol L_3$ have spectral gaps in $L^2(\nu)$, and generate semigroups (see Proposition $3.28$, Chapter II, \cite{EN00}) that are contractive.  If $\psi$ is a constant function,
	it belongs to the domain of $\ol L_3$ and $\ol L_3\psi$ is given by the convergent integral which yields zero.  On the other hand, if $\psi\in D(\ol L_3)$ is not constant almost everywhere, then $(\psi,\ol L_3\psi)>0$ and thus also $\ol L_3\psi\ne 0$.  Therefore, the zero subspace of $\ol L_3$ is given by constant functions.
	
	From the definition (\ref{defn: ApproxL2Operator}) of the kernel function $\ol K_3^{\vep}$ it is clear that, for all $\psi\in D(\ol L_3^{\vep})$, we again have the following lower bound:
	\begin{align*}
	(\psi, \ol L_3^{\vep}\psi)& = \frac{1}{2}\int_{\R^2} (\nu\times\nu)(\rmd u, \rmd v) \ol K_3^{\vep}(u, v)|\psi(u) - \psi(v)|^2 \geq (\psi, L'\psi).
	\end{align*}
	Then the exact same argument as before leads us to the conclusion that $\ol L_3^{\vep}$ has a spectral gap in $L^2(\nu)$,
	its zero subspace is given by constant functions and  it generates a contraction semigroup. 
\end{proof}
$\ol L_3$ and $\ol L_3^{\vep}$ generate strongly continuous semigroups in $L^2(\nu)$. Thus, given initial data in $D(\ol L_3)$ or $D(\ol L_3^{\vep})$, we are guaranteed the existence of unique classical solutions (in $L^2(\nu)$) of (\ref{eq:psievolxvariable}), where the operator on the right-hand side has been replaced by $\ol L_3$ or $\ol L_3^{\vep}$ respectively. Let us use the symbols $S_0$ and $S_{\vep}$ to denote the semigroups generated by $\ol L_3$ and $\ol L_3^{\vep}$ respectively. Then we can write the following:
\begin{align}\label{eq: SemigroupSolnFL}
\forall\psi_0\in D(\ol L_3),\ \ \frac{d}{dt}\psi_t = -  \bn\ol L_3 \psi_t,\ \text{where}\ \psi_t = S_0(t)\psi_0 = \rme^{-t\ol L_3}\psi_0 ,\ t\geq 0,\nonumber\\
\forall\phi_0\in D(\ol L_3^{\vep}),\ \ \frac{d}{dt}\phi_t = -  \bn\ol L_3^{\vep} \phi_t,\ \text{where}\ \phi_t = S_{\vep}(t)\phi_0 = \rme^{-t\ol L_3^{\vep}}\phi_0,\ t\geq 0.
\end{align}
Of course, these classical $L^2$-solutions do not necessarily solve the evolution equations in (\ref{eq: SemigroupSolnFL}) pointwise everywhere. Also note that, since the semigroups $S_0$ and $S_{\vep}$ are not just strongly continuous but also contractive, given initial data $\phi_0\in L^2(\nu)$, $S_0(t)\phi_0$ and $S_{\vep}(t)\phi_0$ solve the corresponding evolution equations in (\ref{eq: SemigroupSolnFL}) for all $t>0$ (see Remark $1.22$ and Example $1.25$, Chapter IX, \cite{kato:plo}). 

Theorem \ref{theorem: L2SemigroupSolution} is obviously independent of our Banach space results. However, the existence of semigroup solutions of the evolution equation associated with $\ol L_3^{\vep}$, guaranteed by this theorem is used later in the proofs of Lemma \ref{lemma: L2solnpsiep} and Theorem \ref{theorem: psiepSoln} in our Banach space analysis. In fact, Theorem \ref{theorem: L2SemigroupSolution} is the only result from our analysis in $L^{2}(\nu)$ that is used in our Banach space analysis in Section \ref{section: mainresultsBanach}.

\subsection{Approximating the Semigroup Solution generated in $L^2(\nu)$ by $\ol L_3$}
For the remainder of our analysis in $L^2(\nu)$, set forth in this subsection, we will concern ourselves only with initial data coming from certain Banach spaces of functions having weighted H\"older regularity; the results presented here allow us to connect the semigroup solutions given by Theorem \ref{theorem: L2SemigroupSolution} with the Banach space solutions obtained in Section \ref{section: mainresultsBanach}. Naturally then, we will need to use in this subsection, certain lemmas and theorems proved in the next section, that show existence and uniqueness of solutions of  Duhamel-integrated evolution equations related to $\ol L_3$ and $\ol L_3^{\vep}$ in these Banach spaces. It is thus advisable to first write down explicit definitions of the relevant weighted spaces used in our analysis in Section \ref{section: mainresultsBanach} and also mention briefly the relevant results.

Let us first write down the definition of the Banach spaces used for the analysis of the original, unregularized operator $\ol L_3$.  We consider initial data $\psi_0\in \ol X$, where $\ol X$ is the Banach space of continuous functions $\psi$ on $\R$ such that
\begin{align}\label{eq: XbarBanachDefn}
\norm{\psi}_{\ol X} = \sup_{ v\in\R }\frac{|\psi(v)|}{\tGamma(v)}\  + \sup_{ (v, r)\in\R\times\R_+ }\frac{|\psi(v) - \psi(v-r)|}{\Gamma_0(v, r)}< \infty,
\end{align}
where
\begin{align}\label{defn: tGammaGamma0}
\tGamma (v) &= f(v)\exp[\mu\max(a, c_0v)],  \text{ where } f(v) = \max\left( (\ln(1 + \rme^v))^{-\alpha}, (\ln 2)^{-\alpha}\right)\text{ and,}\nonumber\\
\Gamma_0(v, r)&=\left( f(v) + f(v-r) \right)\exp[\mu\max(a, c_0v, v-r)] g_0(v, r),\ \text{with}\  g_0(v, r) = \left( 1 - \rme^{-\kappa r} \right)^{\gamma_0}, \nonumber\\
\text{where}\ & 0<\alpha<1/6,\ \mu<\frac{1}{2} - \frac{3}{8}\alpha,\ \mu c_0\in(\alpha, 1/4),\ a\geq 9,\ \gamma_0\in(0, 1/8],\text{ and } \kappa\geq7.
\end{align}
The parameters $\alpha$, $\mu$ etc. appearing above do not have much bearing on the results in this subsection. However, choosing the correct admissible values for them is critical for our results in the Banach spaces in Section \ref{section: mainresultsBanach}. The intervals to which these parameters are restricted are largely determined by computational convenience in the proof of Theorem \ref{theorem: DeltaSolutionsY}. For our computations in this paper, we choose $1/9<\alpha<1/6.5$, $\mu\in[1/3, 7/16]$ and $c_0=0.52$. These choices obviously satisfy the conditions in the last line of (\ref{defn: tGammaGamma0}). In Appendix \ref{appendix: GammaParameters} we describe in some detail how the weight function $\Gamma_0$ and the choices for the above parameters are arrived at.
Let us also mention here that the parameter $\mu'$ appearing in (\ref{defn: ApproxL2Operator}) is such that $\mu'>\mu$, while $a_1=a/c_0.$

A related Banach space is $\ol Y$, the Banach space of continuous functions $\Delta$ on $\R\times\R_+$, bounded with respect to the weight function $\Gamma_0$, so that:
\begin{align}\label{eq: YbarBanachDefn}
\norm{\Delta}_{\ol Y} =  \sup_{ (v, r)\in\R\times\R_+ }\frac{|\Delta(v, r)|}{\Gamma_0(v, r)}< \infty,
\end{align}
where $\Gamma_0(v, r)$ has already been defined above. This is the Banach space  where we consider the evolution equation for the difference functions. Theorem \ref{theorem: DeltaSolutionsY} proved in the next section guarantees the existence and uniqueness of solutions of the Duhamel-integrated evolution equation for differences in this space for at least some short time $T^*>0$.

Note that
\begin{align*}
\forall (u, v)\in\R^2,\ \psi\in L^2(\nu),\ |\psi(u) - \psi(v)| = |\psi(\max(u, v)) - \psi(\max(u, v) - r)|, \ \text{where } r=|u-v|.\\
\end{align*}
In order to keep the notation simple, we will use, whenever convenient, the following equivalent definition for the weight $\Gamma_0$  without changing the symbol (and follow a similar convention for the weights $\ol\Gamma_{\vep}$ defined subsequently) :
\begin{align*}
\Gamma_0(u, v)&=\left( f(u) + f(v) \right)\exp[\mu\max(a, c_0\max(u, v), \min(u, v))] g_0(u, v),\ \\
\text{with}\ &  g_0(u, v) = \left( 1 - \rme^{-\kappa |u-v|} \right)^{\gamma_0},\ \forall(u, v)\in\R^2.
\end{align*}

Let us recall that, because of the line singularity in the kernel of the original linear operator, the solution $\Delta_t$ has to be connected to the difference function associated with $S_0(t)\psi_0$ via the regularized, approximating operator $\ol L_3^{\vep}$. The regularized version of (\ref{eq:psievolxvariable}) is studied in the next section in $X$, which is the Banach space of continuous functions on $\R$ bounded with respect to the weight function $\tGamma$ defined in (\ref{defn: tGammaGamma0}). Clearly $\ol X\subset X\subset D(\ol L_3^{\vep})$. By Theorem \ref{theorem: psiepSoln},  there exists a unique solution to the Duhamel-integrated regularized evolution equation in $X$.

Now $\ol X$ is evidently contained in both $D(\ol L_3)$ and $D(\ol L_3^{\vep})$, and thus for initial datum $\psi_0\in\ol X$, Theorem \ref{theorem: L2SemigroupSolution} guarantees the existence of semigroup solutions $S_0(t)\psi_0$ and $S_{\vep}(t)\psi_0$, unique in $L^2(\nu)$, for the Cauchy problems associated with $\ol L_3$ and $\ol L_3^{\vep}$ respectively.  In this subsection we will prove that $S_0(t)\psi_0$ is actually the limit function that $S_{\vep}(t)\psi_0$ converges to, as $\vep_0\to 0$ and for times $t\in[0, T^*].$ The time $T^*>0$ comes from Theorem \ref{theorem: DeltaSolutionsY} and Proposition  \ref{theorem: rDpsiDeltaApprox}, and it has no dependence on $\vep_0.$

Given the initial datum $\psi_0$ described above, let us define for any $\vep_0>0$, $\vphi_t = S_{\vep}(t)\psi_0$. Then by Proposition \ref{theorem: rDpsiDeltaApprox}, there exists  $T^*>0 \ \text{such that}\ |D\vphi_t|\leq A_1\ol\Gamma_{\vep},\ \nu\text{-almost everywhere},$  for all $t\in[0, T^*] \ \text{and some constant } A_1<\infty$, 
where  $D\vphi_t(v, r) = \vphi_t(v) - \vphi_t(v-r), \ \text{for all }(v, r)\in\R\times\R_+,  \text{ both $A_1$ and $T^*$ being independent of $\vep_0.$ }.$
%
The weight function $\ol\Gamma_{\vep}$ is defined as:
\begin{align}\label{defn: olrGamma}
&\ol\Gamma_{\vep}(v, r)= \left( f(v) + f(v-r) \right)\exp[\mu\max(a, c_0v, v-r)] \ol g(v, r),\ \ol g(v, r) = \left( 1 - \rme^{-\kappa(r+\vep(v))} \right)^{\ol\gamma_0},\ \ol\gamma_0=\gamma_0/2.
\end{align}

Let us now consider a sequence $\vep_n$ for the regularization parameter, such that $\vep_n\to 0$. For any $\vep_n>0$ taking the place of $\vep_0$ appearing in (\ref{defn: ApproxL2Operator}), we write the following for simplicity's sake, by a slight abuse of notation:
\begin{align*}
\ol K_3^{\vep_n}(u, v) &= \ol K_3(u, v) \frac{1 - \rme^{-\min(\vep_n(u, v), |u-v|)}}{1 - \rme^{-\vep_n(u, v)}},\ \vep_n(u, v) = \vep_n\exp\left(-\frac{\mu'}{\gamma_0}\max(a_1, \max(u, v))\right).
\end{align*}
We also write
\begin{align*}
\ol\Gamma_{\vep_n}(v, r)&= \left( f(v) + f(v-r) \right)\exp[\mu\max(a, c_0v, v-r)] \ol g_n(v, r),\ \ol g_n(v, r) = \left( 1 - \rme^{-\kappa(r+\vep_n(v))} \right)^{\ol\gamma_0},
\end{align*}
and use the Banach space $\ol X_{\vep_n}$, defined as the Banach space of continuous functions $\psi$ on $\R$ such that
\begin{align}\label{eq: XvepnbarBanachDefn}
\norm{\psi}_{\ol X_{\vep_n}} = \sup_{ v\in\R }\frac{|\psi(v)|}{\tGamma(v)}\  + \sup_{ (v, r)\in\R\times\R_+ }\frac{|\psi(v) - \psi(v-r)|}{\ol\Gamma_{\vep_n}(v, r)}< \infty.
\end{align}
The corresponding $L^2$-solution is $S_{\vep_n}(t)\psi_0$  and we write $\vphi_t^n= S_{\vep_n}(t)\psi_0.$
Then Theorem \ref{theorem: psiepSoln} and Theorem \ref{theorem: rDpsiSoln}, proved in the next section, imply that for all $\psi_0\in \ol X$, we have a unique solution $\bvphi_t^n$ of the associated Duhamel-integrated Cauchy problem in the Banach space $\ol X_{\vep_n}$, such that
\begin{align*}
\norm{\bvphi_t^n}_{\ol X_{\vep_n}}  < \infty,\ \forall \vep_n>0,\ \text{and}\ \bvphi_t^n = \vphi_t^n,\ \nu\text{-almost everywhere}.
\end{align*}
We denote by $\ol Y_{\vep_n}$ the Banach space of continuous functions on $\R\times\R_+$ bounded with respect to the weight $\ol\Gamma_{\vep_n}.$  This is the space where the evolution equation for differences is considered for the regularized case. Note that $\ol X\subset \ol X_{\vep_n}\subset D(\ol L_3^{\vep_n}),$ for all $\vep_n>0.$


For the rest of this section we will use some results, a couple of which have already been mentioned and all of which are easily obtained via short, straightforward computations. Since these are essential for the main theorem of this section, we collect them in the following lemma for ready reference.

\begin{lemma}\label{lemma: L2bvphi}
	Given initial data $\bvphi_0\in \ol X$ and $\vep_n>0,$ the following are true:
	
	i) There exist $A_1<\infty,\ T^*>0,$ both independent of $\vep_n$, such that  $|D\bvphi_t^n|\leq A_1\norm{D\bvphi_0}_{\ol Y}\ol\Gamma_{\vep_n},$ for all $t\in[0, T^*]$.
	
	ii) $\ol L_3^{\vep_n}\bvphi_t^n\in L^2(\nu)$ for all $t\in \R_+.$ 
	
	iii) There exists a positive, measurable function $\widehat{\Gamma}$ on $\R^2$ such that  \[\int_{\R^2}(\nu\times\nu)(\rmd u, \rmd v) \ol K_3(u, v)(\widehat{\Gamma}(u, v))^2 <\infty.\]
	
	iv) There exists a positive constant $C<\infty,$ such that \[\ol K_3^{\vep_n}(u, v)\big(\ol\Gamma^{\vep_n}(u, v)\big)^2\leq C \ol K_3(u, v)\big(\widehat{\Gamma}(u, v)\big)^2,\ \forall (u, v)\in\R^2.\]
\end{lemma}

\begin{proof}
	
	i) This estimate follows from the upper bound on $|\Delta_t|$ given by Theorem \ref{theorem: DeltaSolutionsY} and the following upper bound obtained in Proposition \ref{theorem: rDpsiDeltaApprox}: 
	\begin{align*}
	|D\bvphi_t^n - \Delta_t|\leq Q_0 \big(M\vep_n\big)^p\ln\big(\min(M, \vep_n^{-1})\big)\norm{\Delta}_Y,\ \forall t\in[0, T^*],
	\end{align*}
	where the constants $Q_0$ and $p>0$ are independent of $\vep_n$, and $\Delta_0 = D\bvphi_0.$
	
	\vspace{0.1in}
	ii) From Theorem \ref{theorem: rDpsiSoln} we know that
	\begin{align*}
	\bvphi^n\in\ol X_{\vep_n}, \forall\vep_n>0,\ \forall t> 0, \text{ i.e., }  
	\sup_{\substack{t>0 \\ (v, r)\in\R\times\R_+ }} \frac{|D\bvphi_t^n(v, r)|}{\ol\Gamma_{\vep_n}(v, r)}<\infty.
	\end{align*}
	It is  straightforward to verify that $G\in L^2(\nu)$, where
	\begin{align*}
	G(u) = \int_{\R}\nu(\rmd v)\ol K_3^{\vep_n}(u, v)\ol\Gamma_{\vep_n}(u, v) .
	\end{align*}
	It then follows naturally that $\ol L_3^{\vep_n}\bvphi_t^n\in L^2(\nu).$		 
	
	\vspace{0.1in}
	iii)
	Define		
	\begin{align*}
	\widehat{\Gamma}(u, v)=\left( f(u) + f(v) \right)\exp[\mu\max(a, c_0\max(u, v), \min(u, v))] \bg_0(u, v),\ \text{with}\ \bg_0(u, v) = \left( 1 - \rme^{-\kappa |u-v|} \right)^{\ol \gamma_0}.
	\end{align*}
	Then it is easy to check that
	\begin{align*}
	\int_{\R^2}(\nu\times\nu)(\rmd u,\rmd v)\ol K_3(u, v)\big(\widehat{\Gamma}(u, v)\big)^2<\infty.
	\end{align*}
	
	\vspace{0.1in}
	iv) For all  $\vep_n>0,$ and $r>0,$ there exist $C_1>0,\ C_2>0,$ depending on $\kappa$, such that the following is true:
	\begin{align*}
	\frac{\big( 1 - \rme^{-\kappa(r+\vep_n)}\big)^{2\ol\gamma_0}}{ 1 - \rme^{-\max(r, \vep_n)}}\leq C_1\big( 1 - \rme^{-\kappa (r+\vep_n)} \big)^{2\ol\gamma_0-1}\leq C_1\big( 1 - \rme^{-\kappa r} \big)^{2\ol\gamma_0-1}\leq C_2\frac{\big( 1 - \rme^{-\kappa r}\big)^{2\ol\gamma_0}}{ 1 - \rme^{-r}}.
	\end{align*}
	It then follows directly from the definitions of the kernel functions and the weight functions that
	\begin{align*}
	\ol K_3^{\vep_n}(u, v)\big(\ol\Gamma^{\vep_n}(u, v)\big)^2\leq C \ol K_3(u, v)\big(\widehat{\Gamma}(u, v)\big)^2,\ \forall (u, v)\in\R^2,\ \text{for some constant $C>0.$}
	\end{align*}
\end{proof}

Given initial data $\psi_0\in \ol X$, let us consider again the sequence $\vep_n$ of regularization parameters, with $\vep_n\to 0.$ Recall that corresponding to every $\vep_n$, the regularized evolution equation (see \ref{subsubsection: psiepsoln}) has a solution $\bvphi_n\in \ol X_{\vep_n}$ for all times by Theorem \ref{theorem: psiepSoln}. Then we can prove the following lemma.

\begin{lemma}\label{lemma: bvphinCauchy}
	Let $\{\vep_n\}$ be a sequence of regularization parameters such that $\vep_n\to 0$ and $\psi_0\in \ol X$ be the initial datum. Then the corresponding sequence $\{\bvphi_n\}$ is Cauchy in $C\big( [0, T^*], L^2(\nu)\big)$.
\end{lemma}	
\begin{proof}
	Consider $m, n\in\mathbbm{N}.$ For the sake of simplicity, let us assume $\min(\vep_m, \vep_n)=\vep_m.$ Then we know that $\bvphi_m\in \ol X_{\vep_m}$ and $\bvphi_n\in\ol X_{\vep_n}$ and from the definitions of $\Gamma_{\vep_m}$ and $\Gamma_{\vep_n}$ it is clear that $\ol X_{\vep_m}\subset\ol X_{\vep_n}$.
	
	Let $\bvphi_t^{n, m}= \bvphi_t^n - \bvphi_t^m,$ then $\bvphi_t^{n, m}\in \ol X_{\vep_n}.$ The function $\bvphi_t^{n, m}$ satisfies
	\begin{align}\label{bvphinmeqn}
	\partial_t\bvphi_t^{n, m}&= -\ol L_3^{\vep_n}\big( \bvphi_t^n - \bvphi_t^m + \bvphi_t^m \big) + \ol L_3^{\vep_m}\bvphi_t^m \nonumber\\
	&= - \ol L_3^{\vep_n}\bvphi_t^{n, m} + g_t^{n, m},
	\end{align}
	\begin{align*}
	\text{where } \ g_t^{n, m}(u) &= \ol L_3^{\vep_m}\bvphi_t^m(u) - \ol L_3^{\vep_n}\bvphi_t^m(u)\\
	&= \int_{\R}\nu(\rmd v)\bigg( \ol K_3^{\vep_m}(u, v) - \ol K_3^{\vep_n} (u, v) \bigg)\bigg( \bvphi_t^m(u) - \bvphi_t^m(v) \bigg)\\
	& = \int_{\R}\nu(\rmd v) \ol S^{\vep_m, \vep_n}(u, v)\bigg( \bvphi_t^m(u) - \bvphi_t^m(v) \bigg),\ \text{with } \ol S^{\vep_m, \vep_n}(u, v)= \ol K_3^{\vep_m}(u, v) - \ol K_3^{\vep_n} (u, v).
	\end{align*}
	Now by part i) of Lemma \ref{lemma: L2bvphi} we can use: 
	\begin{align*}
	\big| \bvphi_t^m(u) - \bvphi_t^m(v) \big|\leq A_1\norm{D\bvphi_0}_{\ol Y}\ol\Gamma_{\vep_m}\leq A_1\norm{\psi_0}_{\ol X}\ol\Gamma_{\vep_m},
	\end{align*}
	to obtain the following estimate:
	\begin{align*}
	|g_t^{n, m}(u)|= \Big| \int_{\R}\nu(\rmd v) \ol S^{\vep_m, \vep_n}(u, v)\bigg( \bvphi_t^m(u) - \bvphi_t^m(v) \bigg) \Big|\leq F_{n, m}(u),
	\end{align*}
	\begin{align*}
	&\text{where}\ F_{n, m}(u) = C_0\norm{\psi_0}_{\ol X}\big( \vep_n - \vep_m\big)^{\ol\gamma_0}\rme^{\mu\max(a, u)}f(u)\big( \ln ( 1 + \rme^u) \big)^{-\frac{1}{2}}\rme^{-\frac{\mu'}{\gamma_0}\ol\gamma_0\max(a_1, u)},
	\end{align*}
	$C_0$ being some positive constant independent of $\vep_m$ and $\vep_n$.
	Evidently then, 
	$\norm{F_{n, m}}_{L^2}\to 0, \text{ as } m, n\to\infty.$
	Looking back at (\ref{bvphinmeqn}) we can now write the following for all $ t\in[0, T^*]$:
	\begin{align*}
	\partial_t\norm{\bvphi_t^{n, m}}_{L^2}^2 &= 2\operatorname{Re}\big( \bvphi_t^{n, m}, \partial_t\bvphi_t^{n, m}\big)\\
	&= 2\operatorname{Re}\big( \bvphi_t^{n, m}, g_t^{n, m} \big) - 2\operatorname{Re}\big( \bvphi_t^{n, m}, \ol L_3^{\vep_n}\bvphi_t^{n, m} \big)\\
	&< 2\operatorname{Re}\big( \bvphi_t^{n, m}, g_t^{n, m} \big) \ \qquad\text{by the positivity of $\ol L_3^{\vep_n}$}\\
	&\leq 2 \norm{\bvphi_t^{n, m}}_{L^2}\norm{g_t^{n, m}}_{L^2}\\
	& \leq 4 \norm{\psi_0}_{L^2}\norm{F_{n, m}}_{L^2},
	\end{align*}
	where in the last line we have used the fact that $S_{\vep_n}(t)\psi_0$ is the $L^2$-representative of $\bvphi_t^n$ (see Lemma \ref{lemma: L2solnpsiep}).
	This means
	\begin{align*}
	\sup_{t\in[0, T^*]}\norm{\bvphi_t^{n, m}}_{L^2}\leq \sqrt{4 T^* \norm{\psi_0}_{L^2}\norm{F_{n, m}}_{L^2}}.
	\end{align*}
	Since the right-hand side of the above equation goes to zero as $m, n\to \infty,$ we conclude that the sequence $\{\bvphi^n\}$ is Cauchy in $C\big( [0, T^*], L^2(\nu)\big)$.
\end{proof}

Since $\{\bvphi^n\}$ is  Cauchy, there exists a limit function in $C\big( [0, T^*], L^2(\nu)\big)$ that this sequence converges to. Let us call it $\tpsi$.
In what follows,
we prove that this limiting function is independent of the choice of the sequence $(\vep_n)$ above, and we indeed then have $\lim_{\vep\to 0} \sup_t\norm{\bvphi^{\vep}_t -\tpsi_t}=0$.  For notational convenience,
we denote already from the start $\bvphi^{\vep}\to\tpsi$ as $\vep\to0$, although strictly speaking these first refer to $\bvphi^{n}\to\tpsi$ for some fixed choice of sequence $\vep_n\to 0$.
We use $\vep$ to parametrize the regularized linear operator, $\bvphi^{\vep}$ denotes the Banach space solution associated with the operator $\ol L_3^{\vep}$, the corresponding Banach space is denoted by $\ol X_{\vep}$ and we are interested in the limit $\vep\to 0.$ 

We are now in a position to prove the main result of this subsection, which states that this limit function $\tpsi_t$ is in fact the $L^2$-solution associated with the original linearized operator $\ol L_3$.

\begin{theorem}\label{theorem: L2limitfunction}
	Given initial datum $\psi_0\in\ol X$, let $\tpsi_t = \lim_{\vep\to 0}\bvphi_t^{\vep}$ in $L^2(\nu)$ for $t\in[0, T^*]$. Then $\tpsi_t = S_0(t)\psi_0\ \ \text{$\nu$-almost everywhere}.$
\end{theorem}

\begin{proof}
	We will first fix $t\in[0, T^*]$ and
	show that $\tpsi_t\in D(\ol L_3)$. Let us begin by demonstrating that $\tpsi_t\in D(Q)$.
	
	There is now a subsequence of the original sequence $(\vep_n)$ such that along the subsequence
	$\bvphi_t^{\vep}\to\tpsi_t$ a.e. Since $\lim_{\vep\to0} \ol K_3^{\vep}(u, v) = \ol K_3(u, v)$ for all $(u, v)\in\R^2,$ we have (along the subsequence)
	\begin{align*}
	&  \ol K_3^{\vep}(u, v)\big|\bvphi_t^{\vep}(u) - \bvphi_t^{\vep}(v)\big|^2\longrightarrow \ol K_3(u, v)\big|\tpsi_t(u) - \tpsi_t(v)\big|^2 \text{ as }\vep\to 0,
	\end{align*}
	except on a set that has measure zero under $\nu\times\nu$.
	Then, by parts i) and iv) of Lemma \ref{lemma: L2bvphi}, there exist constants $A$ and $C<\infty$ such that
	\begin{align*}
	\ol K_3^{\vep}(u, v) \big|\bvphi_t^{\vep}(u) - \bvphi_t^{\vep}(v)\big|^2\leq A(\norm{\bvphi_0}_{\ol X})^2\ol K_3^{\vep}(u, v) \big(\ol\Gamma_{\vep}(u, v)\big)^2\leq C \ol K_3(u, v) \big(\widehat{\Gamma}(u, v)\big)^2,
	\end{align*}
	which implies the following by  part iii) of Lemma \ref{lemma: L2bvphi}, and the dominated convergence theorem:
	\begin{align*}
	\lim_{\vep\to 0}\int_{\R^2}(\nu\times\nu)(\rmd u, \rmd v) \ol K_3^{\vep}(u, v)\big|\bvphi_t^{\vep}(u) - \bvphi_t^{\vep}(v)\big|^2 = \int_{\R^2}(\nu\times\nu)(\rmd u, \rmd v) \ol K_3(u, v)\big|\tpsi_t(u) - \tpsi_t(v)\big|^2<\infty.
	\end{align*}
	Thus, $\tpsi_t\in D(Q).$
	
	We then return to the original sequence $(\vep_n)$. Let us now consider test functions $\phi$ from the space $C_c^{0,\alpha_0}(\R)$ of compactly supported $\alpha_0$-H\"{o}lder continuous functions on $\R$. This is a dense subspace of $L^2(\nu)$ and a subspace of $D(\ol L_3^{\vep})$  for all $\vep>0$, as well as of $D(\ol L_3)$. Then, since $\ol K_3^{\vep}(u, v)\leq \ol K_3(u, v)$ for all $(u, v)\in\R^2$ and $\int_{\R}\nu(\rmd v)\ol K_3(u, v)\big|\phi(u) - \phi(v)\big|$ is integrable, we have, by dominated convergence,
	\begin{align*}
	\lim_{\vep\to 0}\big(\ol L_3^{\vep}\phi\big)(u) = \big(\ol L_3\phi\big)(u),\ \text{pointwise as well as in $L^2(\nu)$.}
	\end{align*}
	Also, by part i) of Lemma \ref{lemma: L2bvphi}, we have, for all $t\in[0, T^*]$, the following:
	\begin{align}\label{eq: Hdef}
	\big| \big(\ol L_3^{\vep}\bvphi_t^{\vep}\big)(u) \big|\leq C_1\int_{\R}\nu(\rmd v)\ol K_3^{\vep}(u, v)\ol\Gamma^{\vep}(u, v)\leq C_1 H(u),
	\end{align}
	where the constant $C_1$ is independent of $\vep$, 
	\begin{align*}
	H(u) = \sup_{\vep>0} \int_{\R}\nu(\rmd v)\ol K_3^{\vep}(u, v)\ol\Gamma^{\vep}(u, v),\ \text{and}\ H\in L^2(\nu).
	\end{align*}
	Then the following limit holds:
	\begin{align*}
	\lim_{\vep\to 0}\big(\phi, \ol L_3^{\vep}\bvphi_t^{\vep}\big) = \lim_{\vep\to 0}  \big( \ol L_3^{\vep}\phi, \bvphi_t^{\vep} \big) = \big( \ol L_3\phi, \tpsi_t\big)= \tQ(\phi, \tpsi_t),
	\end{align*}
	where we have used the self-adjointness of $\ol L_3^{\vep}$, the limits obtained earlier and the fact that $\tpsi_t\in D(Q).$ However, by (\ref{eq: Hdef}), we have the following: 
	\begin{align*}
	&\big|\big(\phi, \ol L_3^{\vep}\bvphi_t^{\vep}\big)\big|\leq C_2\norm{\phi}_{L^2}\norm{H}_{L^2}, \ \text{where $C_2$ is a constant independent of $\vep$.}
	\end{align*}
	This means that
	\begin{align*}
	\big|\tQ(\phi, \tpsi_t)\big|\leq C_2\norm{\phi}_{L^2}\norm{H}_{L^2}, \ \forall\phi\in C_c^{0,\alpha_0}(\R).
	\end{align*}
	Thus the map $\phi\mapsto Q(\phi, \tpsi_t)$ has a unique, bounded extension from the dense subspace containing our test functions to $L^2(\nu)$ and we conclude that $\tpsi_t\in D(\ol L_3)$. Then by self-adjointness of $\ol L_3$,  $\tQ(\phi, \tpsi_t)=\big( \ol L_3\phi, \tpsi_t \big)= \big(\phi, \ol L_3\tpsi_t\big)$
	for any $\phi\in D(\ol L_3)$.   The above bounds, uniform in $\vep$, allow us to  conclude that
	\begin{align*}
	\lim_{\vep\to 0}\big( \phi', \ol L_3^{\vep}\bvphi_t^{\vep}\big) = \big(\phi', \ol L_3\tpsi_t\big),\ \forall \phi'\in L^2(\nu),\ \text{i.e., } \ol L_3^{\vep}\bvphi_t^{\vep}{\longrightarrow}\ol L_3\tpsi_t\ \text{weakly in $L^2(\nu)$,}\ \forall t\in[0, T^*].
	\end{align*}

	Now for any $\widehat{\psi}\in D(\ol L_3)$ and $0\leq s\leq t < T^*,$ we have the following:
	\begin{align*}
	\partial_s\big( S_0(t-s)\widehat{\psi}, \bvphi_s^{\vep} \big) &= \big( \ol L_3 S_0(t-s)\widehat{\psi}, \bvphi_s^{\vep} \big) - \big( S_0(t-s)\widehat{\psi}, \ol L_3^{\vep}\bvphi_s^{\vep} \big)\\
	&= \big( \ol L_3 S_0(t-s)\widehat{\psi}, \bvphi_s^{\vep} - \tpsi_s \big) + \big(   S_0(t-s)\widehat{\psi}, \ol L_3\tpsi_s - \ol L_3^{\vep}\bvphi_s^{\vep}\big)\longrightarrow 0,\ \text{as}\ \vep\to 0.
	\end{align*}
	Continuity in $s$ and uniform boundedness in $\vep$ for the terms on the right hand side above then implies that we can apply the dominated convergence theorem to obtain
	\begin{align*}
	\int_0^t\rmd s\ \partial_s\big( S_0(t-s)\widehat{\psi}, \bvphi_s^{\vep} \big) = \big(\widehat{\psi}, \bvphi_t^{\vep}\big) - \big(S_0(t)\widehat{\psi}, \psi_0\big)\longrightarrow 0\ \text{as}\ \vep\to 0.
	\end{align*}
	Since we know that $\big(\widehat{\psi}, \bvphi_t^{\vep}\big) \longrightarrow \big(\widehat{\psi}, \tpsi_t\big) $ as $\vep\to 0,$ this means $\big(\widehat{\psi}, \tpsi_t\big) = \big(\widehat{\psi}, S_0(t)\psi_0\big)$ for any $\widehat{\psi}\in D(\ol L_3)$ . Therefore, $\tpsi_t = S_0(t)\psi_0$, $\nu$-almost everywhere.
\end{proof}

	\section{Smoothing Solutions in a weighted Banach Space of H\"{o}lder-continuous functions}\label{section: mainresultsBanach}

In this section we study the linearized three waves collision operator described in (\ref{defn: OrigL2Operator}) in a certain Banach space, and prove that the corresponding time evolution has a smoothing property in this space. The proof of this smoothing theorem relies on a few other results, proved in slightly different but related Banach spaces.

This section is organized as follows: first we state and briefly explain the main smoothing result as well as the theorems leading to it; this is followed by two subsections containing the details of the proofs of these theorems. In the first subsection we prove the existence and uniqueness of solutions of an initial value problem derived from our original evolution equation in a weighted Banach space characterized by a  time-dependent H\"{o}lder-type condition. In the second subsection we consider a regularized evolution equation which enables us to show that this solution is in fact identical to the one obtained in a space of differences of functions, with the same initial value, so that the result proved in the first subsection implies the smoothing of solutions of the original evolution equation.
\vspace{0.1in}

We begin by looking at the original evolution equation:
\begin{align*}
\partial_t\psi_t(v) &= -(L_3\psi_t)(v) = \int_{\R}\rmd w\ K_3(v, w)\big(\psi_t(v) - \psi_t(w)\big),
\end{align*}
where the operator $L_3$ 
is now written in terms of the kernel function 
in the flat space without the weight $\nu$. We keep in mind that $L_3$, which we will now analyze in a Banach space introduced in the previous section, is the Friedrichs-extended operator constructed earlier. Then the difference function $\psi_t(v) - \psi_t(v-r)$ evolves in time as
\begin{align}\label{eq: OrigDiffevol}
&\partial_t\left[ \psi_t(v) - \psi_t(v-r)\right] \nonumber \\
=& - \left[(L_3\psi_t)(v) -  (L_3\psi_t)(v-r)\right]  \nonumber\\
=& - \left[ \int_{-\infty}^{\infty} \rmd w\ K_3(v, w) \left( \psi_t(v) - \psi_t(w) \right) - \int_{-\infty}^{\infty} \rmd w\ K_3(v-r, w) \left( \psi_t(v-r) - \psi_t(w) \right)\right].
\end{align}
Let us now write down the formulae defining $K_3$ explicitly.
\begin{align*}
i)\  \text{When }v>w,\ \ 
K_3 (v, w) &= 4\bar{n}\left(\ln ( 1 + \rme^v )\right)^{-\frac{3}{2}} \ln ( 1 + \rme^w ) \frac{ \rme^w + 2 \rme^{-(v-w)} }{ 1 + \rme^w + \rme^{-(v-w)} }\frac{ 1 }{ 1 - \rme^{-(v-w)}  },\\
\text{and, }	ii) \  \text{when } w>v,\ \
K_3 (v, w) &= 4\bar{n}\left(\ln ( 1 + \rme^v )\right)^{-\frac{3}{2}} \ln ( 1 + \rme^w ) \frac{ \rme^v + 2 \rme^{-(w-v)} }{ 1 + \rme^v + \rme^{-(w-v)} }\frac{ \rme^{-(w-v)} }{ 1 - \rme^{-(w-v)}  }.
\end{align*}
For our subsequent computations we will split the kernel function $K_3$ as described below. This splitting is done in order to separate out some terms that do not contain the line-singularity and consequently have  different asymptotic behaviors.
\begin{align*}
&\text{When}\ v>w,\	K_3 (v, w) = K_3^1 (v, w) + K_3^2 (v, w), \text{where} \\
&K_3^1 (v, w) = 4\bar{n}\left(\ln ( 1 + \rme^v )\right)^{-\frac{3}{2}} \ln ( 1 + \rme^w ) \frac{ \rme^w + 2 \rme^{-(v-w)} }{ 1 + \rme^w + \rme^{-(v-w)} }\frac{ \rme^{-(v-w)} }{ 1 - \rme^{-(v-w)}  }, \text{\ and\ } \\
&K_3^2 (v, w) = 4\bar{n}\left(\ln ( 1 + \rme^v )\right)^{-\frac{3}{2}} \ln ( 1 + \rme^w ) \frac{ \rme^w + 2 \rme^{-(v-w)} }{ 1 + \rme^w + \rme^{-(v-w)}}.\\
&\text{For the region $w>v$  we will use the following splitting of $K_3 (v, w)$  in some parts of our computation:}\\
&K_3 (v, w) = \ol{K_3}^1(v, w) + \ol{K_3}^2(v, w),\\
&\text{where }\ol{K_3}^2(v, w) = K_3(v, w) e^{-\frac{1}{2}(w-v)}, \text{and}\ \ol{K_3}^1(v, w) = K_3(v, w)\left( 1 - e^{-\frac{1}{2}(w-v)} \right).
\end{align*}
Note that the part $K_3^2$ does not contain any line-singularity. Also, the integral of $K_3^2(v, w)$ with respect to the variable $w$ yields a square-root-function-like growth for large, positive values of $v$, while $K_3^1(v, w)$ exhibits no such behavior due to the extra exponential decay in it. Let us also note that in the absence of the point-singularity (when $v$ does not assume arbitrarily large negative values), $\ol{K_3}^1(v, w)$ leads to bounded integrals. We will use this splitting of the kernel function for $w>v$, to cut out certain bounded parts later. The factor $1/2$ appearing above has been chosen arbitrarily, but once fixed, this determines the behavior of the weight function $\Gamma(v, r)$ for large, positive values of $v$.
\newline We now write down a crucial but easily derived property of the kernel functions in the form of a lemma.
\begin{lemma}\label{lemma: K3tilt}
	The kernel functions satisfy the following inequalities:
	\begin{align*}
	&\text{when } v-r>w,\ K_3^1 ( v-r, w ) \geq K_3^1 ( v, w ) \text{\ and\ } K_3^2 ( v-r, w) \geq K_3^2 (v, w)  \quad \forall r\geq 0,\\
	&\text{and when } w>v,\ K_3 (v, w) \geq K_3 (v-r, w), \ \text{as well as,} \ \ol{K_3}^2 (v, w) \geq \ol{K_3}^2 (v-r, w)\ \  \forall r\geq 0.
	\end{align*}
\end{lemma}	
\begin{proof}
	When $v-r>w$, the inequalities in the first line are obvious from the formulae for the kernel functions $K_3^1$ and $K_3^2.$
	
	\vspace{0.25in}
	When $w>v$, we can write the following:
	\begin{align*}
	\frac{\partial}{\partial v} K_3 (v, w) &= K_3(v, w) b(v, w), \\
	\text{where}\ b(v, w) &=  - \frac{3}{2}\frac{ \rme^v }{ ( 1 + \rme^v ) \ln ( 1 + \rme^v ) } + 1 +  \frac{ 1 }{ \rme^{w-v} -1 } + \frac{\rme^{w-v}}{1 + \rme^w + \rme^{w-v}} \geq 0.\\
	\text{Therefore}\ 	K_3 (v, w) &\geq K_3 (v-r, w), \  \  \forall r\geq 0.
	\end{align*}	 
	Similarly, $\ol{K_3}^2 (v, w) \geq \ol{K_3}^2 (v-r, w), \ \  \forall r\geq 0.$
\end{proof}
Lemma \ref{lemma: K3tilt} allows us to write a part of the right hand side of equation (\ref{eq: OrigDiffevol}) as the combination of a multiplication operator and a positivity-preserving operator, as we will see shortly. For now, we just write the following:
\begin{align}\label{eq: L3origdiff}
&\left[(L_3\psi_t)(v) -  (L_3\psi_t)(v-r)\right]  \nonumber\\
&= \left[  \int_{-\infty}^{v-r} \rmd w\ K_3^1 (v, w) \left(  \psi_t(v) - \psi_t(w) \right) + \int_{v-r}^v \rmd w\ K_3^1 (v, w) \left(  \psi_t(v) - \psi_t(w) \right) \right.\nonumber\\
&\ \ \ \left. - \int_v^{\infty} \rmd w\ K_3(v, w) \left( \psi_t(w) - \psi_t(v)\right)- \int_{-\infty}^{v-r} \rmd w\ K_3^1(v-r, w) \left( \psi_t(v-r) - \psi_t(w) \right) \right.\nonumber\\
&\ \ \ \left.+ \int_{v-r}^v \rmd w\ K_3(v-r, w)\left( \psi_t(w) - \psi_t(v-r) \right) + \int_v^{\infty} \rmd w\ K_3(v-r, w) \left( \psi_t(w) - \psi_t(v-r)\right) \right]\nonumber\\
&\ \ \ + \left[ \int_{-\infty}^{v-r} \rmd w\ K_3^2 (v, w)\left( \psi_t(v) - \psi_t(w)\right) + \int_{v-r}^v \rmd w\ K_3^2(v, w) \left(\psi_t(v) - \psi_t(w) \right)\right.\nonumber\\
&\ \ \ \left.- \int_{-\infty}^{v-r} \rmd w\ K_3^2 (v-r, w)\left( \psi_t(v-r) - \psi_t(w)\right)\right],
\end{align}
where the last two lines do not contain any line singularity. 

We need to analyze the above operator in order to prove the main result of this paper. Before proceeding to set the stage for that analysis, let us define the time-dependent weight function $\Gamma_t(v, r)$, in terms of which our main theorem is stated:
\begin{align}\label{defn: timeGamma}
\Gamma_t(v,r) = \Big[\big( f(v-r) + f(v) \big) \exp\big(\mu\max\left(a, c_0v, v-r\right)\big)\Big] g_t(v, r),
\end{align}
with the following H\"{o}lder-type time-dependent part $g_t$ :
\begin{align}\label{defn: gt}
g_t(v, r)&=\left( 1 - \rme^{-\kappa r}\right)^{\gamma_t(v-r)},\ \gamma_t(v,r) = \gamma_0 + \ol a(t)\frac{1}{1+\rme^{\beta (v-r)}},\nonumber\\
\ol a(t) &= \frac{1}{8}\min(1,\bn)\frac{t}{1+t},\ \qquad\kappa\geq 7,\  0<\gamma_0\leq 1/8,\ 1\leq\beta\leq\kappa/4.
\end{align}
The choices for the parameters  $\beta,\ \kappa$ and $\gamma_0$ relating to the time-dependent H\"{o}lder-type condition are explained in Appendix \ref{appendix: g}. Note that at $t=0$ this is just the weight $\Gamma_0$ characterizing the Banach spaces $\ol Y$ and $\ol X$ defined in (\ref{eq: YbarBanachDefn}) and (\ref{eq: XbarBanachDefn}) respectively, with exponent $\gamma_0.$
The main result of this paper is then the following theorem:
\begin{theorem}\label{theorem: MainResultPaper}
	Given initial datum $\psi_0\in\ol X$,  let us define $\Delta_0 (v, r)= \psi_0(v) - \psi_0(v-r)$ for all $(v, r)\in\R\times\R_+ $, and $\psi_t = S_0(t)\psi_0.$ Then there exists $T^*>0$ such that for all $t\in[0, T^*]$ the following bound holds:
	\begin{align}
	\big|\psi_t(v) - \psi_t(v-r)\big|\leq C\Gamma_t(v, r)\norm{\Delta_0}_{\ol Y}\ \ \nu\text{-almost everywhere on $\R\times\R_+$},
	\end{align}
	where $C$ is a constant depending on the parameters appearing in the weight function $\Gamma_t.$
\end{theorem}

The theorem above can easily be translated into a statement in terms of the original energy variables $x, y$ in $\R_+$, via the change of variables: $u = \ln (\rme^x -1)$, $v = \ln(\rme^y -1)$. The resulting expressions are somewhat messy, but we can understand the meaning of this theorem by looking at the behavior of the difference function in certain regions of interest. Let us use, by a slight abuse of notation, the same symbols $\psi$, $\Gamma_t$ etc. to denote the relevant quantities in changed variables. Then, it is evident that our initial data must come from a Banach space of functions, whose differences $\Delta_0(x, y)= \psi_0(x) - \psi_0(y)$, are bounded by the following weight function:
\[\max( (\min(x, y))^{-\alpha}, (\ln 2)^{-\alpha}) \exp\left(\mu\max[a, c_0\max(x, y), \min(x, y)]\right) 
\Bigg( 1 - \left(\frac{\rme^{\min(x, y)}-1}{\rme^{\max(x, y)}-1}\right)^{\kappa} \Bigg)^{\gamma_0}.\]
In particular, this means that close to the origin $\Delta_0$ is bounded in terms of \[\max( (\min(x, y))^{-\alpha}, (\ln 2)^{-\alpha}) \Big( 1 - \left(\frac{\min(x, y)}{\max(x, y)}\right)^{\kappa}  \Big)^{\gamma_0},\] while in the region where both the variables are far away from the origin, $\Delta_0$ is bounded with respect to the weight \[\exp\left(\mu\max[a, c_0\max(x, y), \min(x, y)]\right) \left( 1 - \rme^{-\kappa|x-y|}\right)^{\gamma_0}.\]
For initial data such as these, the above theorem states that the evolution of the initial data by the semigroup $S_0$ for time $t$ improves the ``H\"{o}lder" exponent $\gamma_0$ to $\gamma_t$, where \[\gamma_t(x, y) = \gamma_0 + \frac{\ol a(t)}{1 + \left(\rme^{\min(x, y)} -1\right)^{\beta}}.\]
Thus the improvement of the ``H\"{o}lder" exponent is greatest at the origin while it goes to zero as $\min(x, y)$ becomes infinitely large.

Theorem \ref{theorem: MainResultPaper} is the obvious consequence of Theorem \ref{theorem: L2limitfunction} and three other results which are stated below. Before writing down the statements of these three theorems we briefly describe the evolution equations considered in them as well as the function spaces in which these results are obtained. We will just sketch out the schemes here, reserving the details of the constructions for later.

The first theorem deals with an evolution equation derived from the equation (\ref{eq: OrigDiffevol}). The main idea here is to introduce suitable cut-off functions $\delta_1(v, r)$ and $\delta_2(v, r)$, as well as cut-off parameters $m_0$ and $b_0$; and then split the linear operator into an unbounded part $\mcL_u$ which is the sum of a potential function and a positivity-preserving operator, a bounded part $\mcL_b$ and a perturbation $\mcL_{\delta}$ (see Subsection \ref{subsection: DeltaSoln}). For $t\geq 0$ and $(v, r)\in\R\times\R_+$, we define the variable
\begin{align*}
\Delta_t(v,r) = \psi_t(v) - \psi_t(v-r).
\end{align*}
Then equation (\ref{eq: OrigDiffevol}) is recast as the following evolution equation for the $\Delta$-variable:
\begin{align}\label{eq: DeltaEvol}
\partial_t\Delta_t= -\mcL\Delta_t&= -\mcL_u\Delta_t - \mcL_b\Delta_t-\mcL_{\delta}\Delta_t\nonumber\\
&= -\mcV_u\Delta_t + \mcK_u\Delta_t - \mcL_b\Delta_t-\mcL_{\delta}\Delta_t,
\end{align}
where $(\mcL_u\Delta_t)(v, r) = \mcV_u(v, r)\Delta_t(v, r) - (\mcK_u\Delta_t)(v, r)$ and the multiplication operator $\mcV_u$ is defined as
\begin{align}\label{defn: mcVu}
\mcV_u(v, r) &= \int_{-\infty}^{v-r-\delta_1}\rmd w\ K_3^1(v-r, w) + \int_{v+\delta_2}^{\infty}\rmd w\ K_3(v, w) + \int_{v-r+\delta_1}^v\rmd w\ K_3(v-r, w)\nonumber\\
&\ + \int_{v-r}^{v-\delta_2}\rmd w\ K_3^1(v, w) + \int_{v-r}^v\rmd w\ K_3^2(v, w) + \int_{-\infty}^{v-r}\rmd w\ K_3^2(v-r, w).
\end{align}
We consider (\ref{eq: DeltaEvol}) in the following Duhamel-integrated form:
\begin{align} \label{eq: DeltaDuhamel}
\Delta_t &= \rme^{-t\mcV_u}\Delta_0 + \  \int_0^t\ \rmd s\ \rme^{-(t-s)\mcV_u} \mcK_u[\Delta_s] - \int_0^t\ \rmd s\ \rme^{-(t-s)\mcV_u} \mcL_b[\Delta_s] - \int_0^t\ \rmd s\ \rme^{-(t-s)\mcV_u} \mcL_{\delta}[\Delta_s],
\end{align}
in the Banach space $Y$ of functions in $C\left([0, T]\times(\R\times\R_+)\right)$, for some $T>0$, bounded with respect to $\Gamma_t$ as follows: 
\begin{align*}
\norm\Delta_Y := \sup_{t\in[0, T]}\ \sup_{(v,r)\in\R\times\R_+}\frac{|\Delta_t(v, r)|}{\Gamma_t(v, r)}<\infty.
\end{align*}
We then prove the following existence-uniqueness result for solutions of equation (\ref{eq: DeltaDuhamel}).
\begin{theorem}\label{theorem: DeltaSolutionsY}
	There exists $T^*>0$ depending on the cut-off parameters used in equation (\ref{eq: DeltaEvol}), and the parameters $\alpha,\ \mu$ and $c_0$ appearing in the weight $\Gamma_t$, such that the following is true:
	\newline Given initial value $\Delta_0\in \ol Y$, a solution $\Delta$ of (\ref{eq: DeltaDuhamel})
	exists  for all $t\in[0, T^*]$, such that
	\begin{align*}
	\norm\Delta_Y = \sup_{t\in[0, T^*]}\ \sup_{(v,r)\in\R\times\R_+}\frac{|\Delta_t(v, r)|}{\Gamma_t(v, r)}<\infty.
	\end{align*}
	This solution is unique in $Y$ and given by
	\begin{align*}
	\Delta_t=\left( 1 - F \right)^{-1} \rme^{-t\mcV_u}\Delta_0,
	\end{align*}
	where $F$ is the linear operator defined as
	\begin{align}\label{defn: OperatorF}
	F[\Delta_t] = \int_0^t\ \rmd s\ \rme^{-(t-s)\mcV_u} \mcK_u[\Delta_s] - \int_0^t\ \rmd s\ \rme^{-(t-s)\mcV_u} \mcL_b[\Delta_s] - \int_0^t\ \rmd s\ \rme^{-(t-s)\mcV_u} \mcL_{\delta}[\Delta_s].
	\end{align}
\end{theorem}
In the theorem above, the time $T^*$ is chosen so that $\big|\int_0^t\ \rmd s\ \rme^{-(t-s)\mcV_u} \mcL_b[\Delta_s]\big|$ is small enough to guarantee the invertibility of the operator $(1 -F)$ in $Y.$ Observe that, $\psi_0\in\ol X$ (like in Theorem \ref{theorem: MainResultPaper}) implies that $\Delta_0\in\ol Y$, with $\Delta_0(v, r)=\psi_0(v) - \psi_0(v-r),$ $\ol Y$ being the Banach space of functions in $C\left(\R\times\R_+\right)$ bounded with respect to the weight function $\Gamma_0$ defined in  (\ref{defn: tGammaGamma0}). However, the above theorem does not guarantee that, given initial datum $\Delta_0(v, r)=\psi_0(v) - \psi_0(v-r)$ in $\ol Y$, the unique solution $\Delta_t$ of  (\ref{eq: DeltaDuhamel}) obtained via Theorem \ref{theorem: DeltaSolutionsY} is still a difference of the $\psi$-variables. That this is indeed the case, is proved by Theorem \ref{theorem: rDpsiSoln} and Proposition \ref{theorem: rDpsiDeltaApprox}.

Before stating these theorems, let us briefly describe the evolution equations and their solutions considered in them. We start from the regularized evolution equation associated with the operator in (\ref{defn: ApproxL2Operator}),
\begin{align}\label{eq: Regpsievol}
\partial_t\psi_t^{\vep}(u) = -(L_3^{\vep}\psi)(u) &= \int_{\R}\rmd v\ K_3^{\vep}(u, v)\big(\psi_t^{\vep}(u) - \psi_t^{\vep}(v)\big),\ 
\end{align}
where $L_3^{\vep}$ is now written in terms of the kernel function $K_3^{\vep}$ in the flat space without the weight $\nu$. We tag solutions of this equation with $\vep$ in order to avoid confusion later, when we compare the solutions associated with different operators such as $L_3$ and $L_3^{\vep}$.  We consider (\ref{eq: Regpsievol}) for initial data $\psi_0^{\vep}\in X$, where
\begin{align}\label{defn: BanachXregpsi}
X = \big\{ h\in C(\R):\norm{h}_X<\infty \big\},\ \norm{h}_X = \sup_{v\in R} \frac{|h(v)|}{\tGamma(v)},
\end{align}
where $\tGamma$ is defined in (\ref{defn: tGammaGamma0}).

Then Theorem \ref{theorem: psiepSoln} guarantees, for all $t\geq 0$, the existence of a unique solution $\psi_t^{\vep}\in X$, of the following Duhamel-integrated form of  (\ref{eq: Regpsievol}):
\begin{align*}
\psi^{\vep}_t =  \rme^{-tV^{\vep}}\psi^{\vep}_0 + \  \int_0^t\ \rmd s\ \rme^{-(t-s)V^{\vep}} K_u^{\vep}[\psi^{\vep}_s] + \int_0^t\ \rmd s\ \rme^{-(t-s)V^{\vep}} K_b^{\vep}[\psi^{\vep}_s],
\end{align*}
where $V^{\vep}$ is a multiplication operator, $K_b^{\vep}$ is a bounded operator on $L^2(\nu),$ and a splitting akin to (\ref{eq: DeltaEvol}) recasts the operator $L_3^{\vep}$ in the following form:
\begin{align*}
L_3^{\vep}\psi_t^{\vep} = V^{\vep}\psi_t^{\vep} - K_u^{\vep}\psi_t^{\vep} - K_b^{\vep}\psi_t^{\vep}.
\end{align*}
The existence of such a solution $\psi_t^{\vep}$ implies the existence of the difference variable solution $D\psi_t^{\vep},$ with  $D\psi_t^{\vep}(v, r) = \psi_t^{\vep}(v) - \psi_t^{\vep}(v-r)$, of the following evolution equation of differences derived from (\ref{eq: Regpsievol}):
\begin{align}\label{eq: rDpsievol}
\partial_tD\psi^{\vep}_t(v, r) = -(\tmcL^{\vep}D\psi^{\vep}_t)(v, r)= -(\tmcL^{\vep}_uD\psi^{\vep}_t)(v, r) - (\tmcL^{\vep}_sD\psi^{\vep}_t)(v, r) - \mcK^{\vep}_b[\psi^{\vep}_t](v, r),
\end{align}
where again the linear operator $\tmcL^{\vep}$ has been split, \`{a} la (\ref{eq: DeltaEvol}), into an unbounded part $\tmcL_u^{\vep}$, a perturbation $\tmcL_s^{\vep}$, and a part $\mcK_b^{\vep}$ which can be bounded in terms of the $L^2(\nu)$-norm and the $X$-norm of $\psi_t^{\vep}$. Like before, the unbounded part $\tmcL_u^{\vep}$ can be written as the sum $\tmcL_u^{\vep}= \tmcV^{\vep} - \tmcK_u^{\vep}$, of the multiplication operator $\tmcV^{\vep}$ and a positivity-preserving integral operator $\tmcK_u^{\vep},$ leading to the following Duhamel-integrated form of (\ref{eq: rDpsievol}):
\begin{align}\label{eq: rDpsiDuhamel}
D\psi_t^{\vep} = \rme^{-t\tmcV^{\vep}}D\psi^{\vep}_0 + \int_0^t\rmd s\ \rme^{-(t-s)\tmcV^{\vep}} \tmcK^{\vep}_uD\psi^{\vep}_s - \int_0^t\rmd s\ \rme^{-(t-s)\tmcV^{\vep}} \tmcL^{\vep}_sD\psi^{\vep}_s - \int_0^t\rmd s\ \rme^{-(t-s)\tmcV^{\vep}} \tmcK^{\vep}_b[\psi^{\vep}_s].
\end{align}
Then the next theorem shows the existence and uniqueness of solutions of (\ref{eq: rDpsiDuhamel}).
\begin{theorem}\label{theorem: rDpsiSoln}
	Given $\vep>0$, consider initial datum $\psi_0^{\vep}\in X$ such that 
	\begin{align*}
	\norm{D\psi_0^{\vep}}_{\ol Y_{\vep}}=\sup_{(v, r)\in\R\times\R_+}\frac{|\psi_0^{\vep}(v) - \psi_0^{\vep}(v-r)|}{\ol\Gamma_{\vep}(v, r)}<\infty,
	\end{align*}
	$\ol\Gamma_{\vep}$ being the weight function defined in (\ref{defn: olrGamma}). Then there exists for all $t>0$, a unique solution $D\psi_t^{\vep}$ of (\ref{eq: rDpsiDuhamel}) with initial value $\psi_0^{\vep}$, such that $\norm{D\psi_t^{\vep}}_{\ol Y_{\vep}}<\infty.$
\end{theorem}
The next theorem connects the solutions $D\psi_t^{\vep}$ with the solutions $\Delta_t$ obtained in Theorem \ref{theorem: DeltaSolutionsY} fot $t\in[0, T^*].$ Before stating the theorem let us define the variable whose time-evolution is dealt with in this result. This variable is the following difference function, defined for all $t\in[0, T^*]$,
\begin{align}\label{defn: DifferenceFnD}
D_t^{\vep} = D\psi_t^{\vep} - \Delta_t.
\end{align}
From (\ref{eq: DeltaEvol}) and (\ref{eq: rDpsievol}) we see (see Subsection \ref{subsection: regDpsiDiff} for details) that for $t\in[0, T^*]$, the time evolution of $D_t$ is governed by the equation
\begin{align}\label{eq: Devol}
\partial_t D^{\vep}_t = -\tmcL^{\vep}D_t^{\vep} - \tmcL_0^{\vep}[\Delta_t],
\end{align}
where the operator $\tmcL^{\vep}$ is defined exactly like $\tmcL$, but with the kernel function $K_3^{\vep}$ replacing the original kernel function $K_3$; the operator $\tmcL_0^{\vep}$ also has the same structure as $\tmcL$, but with $K_3 - K_3^{\vep}$ replacing $K_3$. Explicit expressions of these operators are given in Subsection \ref{subsection: regDpsiDiff}. Then the next result is
\begin{proposition}\label{theorem: rDpsiDeltaApprox}
	Consider initial datum $\psi_0\in\ol X$. This means $D\psi_0 = \Delta_0\in\ol Y$ and $D_0=0.$  Then there exists $T^*>0$ and a unique Duhamel-integrated solution $D^{\vep}_t$ of (\ref{eq: Devol}), such that $D^{\vep}_t\in\ol Y_{\vep}$ for all $t\in[0, T^*]$. Moreover, there exists some positive constant $Q<\infty$ depending on the parameters $\alpha, \kappa$ and $\gamma_0$, such that $|D_t|$ is bounded above as follows:
	\begin{align*}
	|D_t^{\vep}|\leq Q\left(M\vep_0\right)^p\ln\left(\min(M,\vep_0^{-1})\right)\norm\Delta_Y\ol\Gamma_{\vep},
	\end{align*}
	where $p=p(\gamma_0)>0$ and the positive constant $M<\infty$ appears as a cut-off parameter in the evolution equation (\ref{eq: DeltaEvol}) for $\Delta_t$.
\end{proposition}
The smoothing result of Theorem \ref{theorem: MainResultPaper} can now be obtained in a very straightforward manner, as shown below.
\begin{proof}[Proof of Theorem \ref{theorem: MainResultPaper}]\label{proof: SmoothingResult}
	Given initial datum $\psi_0\in\ol X$, Proposition \ref{theorem: rDpsiDeltaApprox} and Theorem \ref{theorem: L2limitfunction} imply the following $\forall t\in[0, T^*]$:
	\begin{align*}
	\big| \big(\rme^{-t\ol L_3}\psi_0\big)(v) - \big(\rme^{-t\ol L_3}\big)\psi_0(v-r)  \big|&= \lim_{\vep_0\to 0}|D\psi_t^{\vep}(v, r)|\qquad\text{$\nu$-almost everywhere}\\
	&\leq |\Delta_t(v, r)|\\
	&\leq C\norm{\Delta_0}_{\ol Y}\Gamma_t(v, r),
	\end{align*}
	where we have used Theorem \ref{theorem: DeltaSolutionsY} in the last line, $C = \norm{(1 -F)^{-1}}_{Y}<\infty$ and it is easy to see that $\norm{\rme^{-t\mcV_u}\Delta_0}_Y\leq \norm{\Delta_0}_{\ol Y}.$
\end{proof}
Let us make an important remark here about the time $T^*$ mentioned in Theorem \ref{theorem: MainResultPaper}. This upper bound has its origin in Theorem \ref{theorem: DeltaSolutionsY}, where the condition $t\leq T^*$ is used to make $\big|\int_0^t\ \rmd s\ \rme^{-(t-s)\mcV_u} \mcL_b[\Delta_s]\big|< A(b_0, m_0, \mu, c_0)T^*\norm{\Delta}_Y$ small enough to guarantee the invertibility of $(1 - F)$ (see the proof of Theorem \ref{theorem: DeltaSolutionsY} in \ref{subsubsection: DeltaEqn}). Now this result can be extended in time as follows: we first prove the theorem for $t\in[0, T^*]$, then we can choose an initial time $0<t_0<T^*$, and prove the existence of a unique solution of the corresponding Duhamel-integrated equation for $t\in[t_0, t_0+T^*]$. Each such extension in time results in an extra factor of the form $(1 - F_1)^{-1}$ (here $F_1$ represents the operator corresponding to $F$ in the new Duhamel equation) in the formula for $\Delta_t$ obtained in Theorem \ref{theorem: DeltaSolutionsY}, and this process can be continued as long as the norm of the product is finite. This in turn means that Proposition \ref{theorem: rDpsiDeltaApprox}, Theorem \ref{theorem: L2limitfunction} and finally, our main smoothing result Theorem \ref{theorem: MainResultPaper}, all of which inherit the dependence on $T^*$ from Theorem \ref{theorem: DeltaSolutionsY}, can also be extended in time.

The subsections that follow are devoted to the derivations of the evolution equations for $\Delta_t$, $\psi_t^{\vep}$ and $D_t^{\vep}$, and the proofs of Theorems \ref{theorem: DeltaSolutionsY}, \ref{theorem: rDpsiSoln} and Proposition \ref{theorem: rDpsiDeltaApprox}. Of these, the proof of Theorem \ref{theorem: DeltaSolutionsY} is the most delicate and the other proofs rely on estimates that are similar to those employed there, so we use the next subsection to first explain in detail how (\ref{eq: DeltaEvol}) is arrived at and then prove Theorem \ref{theorem: DeltaSolutionsY}.

\subsection{Existence and Uniqueness of Solutions in a Banach space with a  Time-dependent H\"{o}lder-type Condition }\label{subsection: DeltaSoln}
\subsubsection{Derivation of Equation (\ref{eq: DeltaEvol})}\label{subsubsection: DeltaEqn}

As mentioned before, (\ref{eq: DeltaEvol}) is obtained via the introduction of suitable cut-off functions $\delta_1(v, r)$ and $\delta_2(v, r)$ in the original evolution equation (\ref{eq: OrigDiffevol}) for the differences. Before defining these cut-off functions, let us explain the idea behind the recasting of (\ref{eq: OrigDiffevol}). 
Recall that eventually we want to solve a Duhamel-integrated form of the evolution equation, where a potential function is used to control the rest of the terms (cf. the multiplication operator $\mcV_u$ in (\ref{eq: DeltaDuhamel})). When we split our linear operator into a potential $\mcV_u$, a positivity-preserving part $\mcK_u$, a perturbation $\mcL_{\delta}$ and the bounded part $\mcL_b$, we thus want to make $\mcV_u$ as large as possible. We start from the original equation (\ref{eq: OrigDiffevol}) for the evolution of the differences. The right hand side is then rewritten (without the minus sign) as
\begin{align*}
&(L_3\psi_t)(v) - (L_3\psi_t)(v-r)\\
=& \int_{-\infty}^{v-r-\delta_1}\rmd w\ K_3^1(v-r, w)\left( \psi_t(v) - \psi_t(v-r) \right) - \int_{-\infty}^{v-r-\delta_1}\rmd w\ \left( K_3^1(v-r, w) - K_3^1(v, w) \right)\left( \psi_t(v) - \psi_t(w)\right)\\
&\ \ + \int_{v+\delta_1}^{\infty} \ \rmd w\ K_3(v, w)\left( \psi_t(v) - \psi_t(v-r) \right) - \int_{v+\delta_1}^{\infty}\ \rmd w\ \left( K_3(v, w) - K_3(v-r, w) \right)\left( \psi_t(w) - \psi_t(v-r) \right) \\
& \ \ + \int_{v+\delta_2}^{v+\delta_1}\rmd w\ K_3(v, w)\left( \psi_t(v) - \psi_t(v-r) \right) - \int_{v+\delta_2}^{v+\delta_1}\rmd w\ K_3(v, w)\left( \psi_t(w) - \psi_t(v-r) \right)\\
& \ \ + \int_{v-\delta_1}^{v-\delta_2}\rmd w\ K_3^1(v, w)\left( \psi_t(v) -\psi_t(v-r) \right) - \int_{v-\delta_1}^{v-\delta_2}\rmd w\ K_3^1(v, w)\left(  \psi_t(w) - \psi_t(v-r) \right)\\
& \ \ + \int_{v-r}^{v-\delta_1}\ \rmd w\ K_3^1(v, w)\left( \psi_t(v) - \psi_t(v-r) \right) - \int_{v-r}^{v-\delta_1}\rmd w\ K_3^1(v, w)\left( \psi_t(w) - \psi_t(v-r) \right)\\
& \ \ +  \int_{v-r+\delta_1}^v\ \rmd w\ K_3(v-r, w)\left( \psi_t(v) - \psi_t(v-r) \right) - \int_{v-r+\delta_1}^v \ \rmd w\ K_3(v-r, w)\left( \psi_t(v) - \psi_t(w) \right) \\
& \ \ + \left\{  \int_{-\infty}^{v-r}\ \rmd w\ K_3^2(v-r, w)\left( \psi_t(v) - \psi_t(v-r) \right) - \int_{-\infty}^{v-r}\ \rmd w\ \left( K_3^2(v-r, w) - K_3^2(v, w)\right)\left( \psi_t(v) - \psi_t(w)\right)\right. \\
&\ \ \left. + \int_{v-r}^v \ \rmd w\ K_3^2(v, w)\left( \psi_t(v) - \psi_t(v-r) \right) - \int_{v-r}^v\ \rmd w\ K_3^2(v, w)\left( \psi_t(w) - \psi_t(v-r) \right)\right\} \\
&\ \ + \left[ \int_{v-r-\delta_1}^{v-r}\ \rmd w\ K_3^1(v, w)\left( \psi_t(v) - \psi_t(w) \right) - \int_{v-r-\delta_1}^{v-r}\ \rmd w\ K_3^1(v-r, w) \left( \psi_t(v-r) - \psi_t(w) \right) \right.  \\
&\ \ \ +  \int_{v}^{v+\delta_1}\ \rmd w\ K_3(v-r, w)\left( \psi_t(w) - \psi_t(v-r) \right) - \int_{v}^{v+\delta_2}\ \rmd w\  K_3(v, w) \left( \psi_t(w) - \psi_t(v) \right)\\
&\ \ \left.+ \int_{v-\delta_2}^v\ \rmd w\ K_3^1(v, w)\left( \psi_t(v) - \psi_t(w) \right) + \int_{v-r}^{v-r+\delta_1}\ \rmd w\ K_3(v-r, w)\left( \psi_t(w) - \psi_t(v-r) \right)\right].
\end{align*}
The last three lines within square brackets constitute the $\mcL_{\delta}$-part; it is only when these terms with the line-singularity are separated that the rest of the operator can be written in a $(V - K)$-form, with $K$ being positivity-preserving. We now move over to the $\Delta$-variable, with $\Delta_t(v, r) = \psi_t(v) - \psi_t(v-r)$ for all $(v, r)\in\R\times\R_+$, and write the above equation, now in terms of $\Delta_t$, as
\begin{align*}
&\partial_t\Delta_t = -\mcL\Delta_t,\ \ \mcL\Delta_t = \mcL_{pp}\Delta_t + \mcL_{\delta}\Delta_t,\ \text{where}\nonumber\\
\end{align*}
\begin{align}{\label{eq: DeltaEqSplitLdel}}
\mcL_{\delta}\Delta_t(v, r)&= \int_{v-r-\delta_1}^{v-r} \rmd w\ K_3^1(v, w)\Delta_t(v, v-w) + \int_v^{v+\delta_1} \rmd w\ K_3(v-r, w)\Delta_t(w, w-v+r)\nonumber\\
& + \int_{v-r}^{v-r+\delta_1} \rmd w\ K_3(v-r, w)\Delta_t(w, w-v+r) - \int_{v-r-\delta_1}^{v-r} \rmd w\ K_3^1(v-r, w)\Delta_t(v-r, v-r-w) \nonumber\\
& + \int_{v-\delta_2}^v \rmd w\ K_3^1(v, w)\Delta_t(v, v-w) - \int_v^{v+\delta_2}\ dw\ K_3(v, w)\Delta_t(w, w-v),
\end{align}
\begin{align}\label{eq: DeltaEqSplitLpp}
\text{and }\ \mcL_{pp}\Delta_t(v, r)&=\mcV_u(v, r)\Delta_t(v, r) - \int_{-\infty}^{v-r-\delta_1}\rmd w\ \left( K_3^1(v-r, w) - K_3^1(v, w) \right)\Delta_t(v, v-w)\nonumber\\
&\ - \int_{v+\delta_1}^{\infty}\rmd w\ \left( K_3(v, w) - K_3(v-r, w) \right)\Delta_t(w, w-v+r)\nonumber\\
&\ -\int_{-\infty}^{v-r}\rmd w\ \left( K_3^2(v-r, w) - K_3^2(v, w) \right)\Delta_t(v, v-w) - \int_{v-r}^{v-\delta_1}\rmd w\ K_3^1(v, w)\Delta_t(w, w-v+r) \nonumber\\
&\ - \int_{v-r+\delta_1}^v\rmd w\ K_3(v-r, w)\Delta_t(v, v-w) - \int_{v-r}^v\rmd w\ K_3^2(v, w)\Delta_t(w, w-v+r)\nonumber\\
& - \int_{v+\delta_2}^{v+\delta_1}\rmd w\ K_3(v, w)\Delta_t(w, w-v+r) - \int_{v-\delta_1}^{v -\delta_2}\rmd w\ K_3^1(v, w)\Delta_t(w, w-v+r).
\end{align}
The potential $\mcV_u$ has already been defined in (\ref{defn: mcVu}). The next step is then to carve out a bounded piece $\mcL_b$ from $\mcL_{pp}$ as follows:
\begin{align*}
\mcL_{pp}\Delta_t(v, r) = \mcV_u(v, r)\Delta_t(v, r) - \mcK_u\Delta_t(v, r) - \mcL_b\Delta_t(v, r),
\end{align*}
where,
\begin{align}\label{defn: mcLb}
&\mcL_b[\Delta_t](v, r)\nonumber\\
=& - \cf(v<-b_0)\int_{a_1}^{\infty}\ \rmd w\ \left( \ol{K_3}^1(v, w) - \ol{K_{3}}^1(v-r, w) \right)\Delta_t(w, w-v+r) \nonumber\\
& - \cf(v\geq b_0)\int_{v+\delta_1}^{\infty}\ \rmd w\ \left(  \ol{K_{3}}^1(v, w) - \ol{K_{3}}^1(v-r, w) \right)\Delta_t(w, w-v+r) \nonumber\\
& - \cf(v-r<-b_0)\cf(v>0)\int_0^v\ \rmd w\ K_3(v-r, w)\Delta_t(v, v-w)\nonumber\\
&  - \cf(v-r\geq -b_0)\int_{v-r+\delta_1}^v \rmd w\ \ol{K_{3}}^1(v-r, w)\Delta_t(v, v-w) \nonumber\\
& -\cf(v<-m_0)\int_{v-r}^{v-\min(r, -v-b_0)}\rmd w\ K_3^1(v, w)\Delta_t(w, w-v+r)\nonumber \\
& - \cf(-m_0<v<m_0)\int_{v-r}^v \rmd w\ K_3^2(v, w)\Delta_t(w, w-v+r)\nonumber\\
& - \cf(v-r\leq -m_0)\int_{-\infty}^{v-r-\tr}\ \rmd w\ \left( K_3^2(v-r, w) - K_3^2(v, w) \right)\Delta_t(v, v-w) \nonumber\\
& - \cf(-m_0<v-r<m_0)\int_{-\infty}^{v-r}\ \rmd w\ \left( K_3^2(v-r, w) - K_3^2(v, w) \right)\Delta_t(v, v-w)\nonumber \\
& - \cf(v-r\geq m_0)\left[ \cf(v\leq 3r)\int_{-\infty}^{0}\ \rmd w\ \left( K_3^2(v-r, w) - K_3^2(v, w) \right)\Delta_t(v, v-w)\right. \nonumber\\
&\ \left.  + \cf(v>3r)\int_{c_0v}^{v-r}\ \rmd w\ \left( K_3^2(v-r, w) - K_3^2(v, w) \right)\Delta_t(v, v-w)\right]\nonumber \\
& - \cf(v\leq-m_0)\cf(v+r>-b_0)\int_{v-r}^{2v+b_0}\ \rmd w\ K_3^2(v, w)\Delta_t(w, w-v+r) \nonumber\\
& - \cf(v\geq m_0)\left[ \cf(v\leq r)\int_{v-r}^v \rmd w\ K_3^2(v,w)\Delta_t(w, w-v+r)  \right. \nonumber\\
&\left.\ + \cf(0<v-r<c_0v)\int_{c_0^{-1}(v-r)}^v \rmd w\ K_3^2(v, w)\Delta_t(w, w-v+r)\right],
\end{align}
where $a_1=a/c_0$, and $\tr$ appearing in line 7 of the formula for $\mcL_b$, is defined as:
\begin{align}\label{defn: rtilde}
\tr &= \max ( -v-b_0, r ),\qquad\forall v<-b_0,\nonumber\\
&= \max(-b_0 - v+r, 0),\qquad \quad \forall v\geq -b_0.
\end{align}
For our subsequent computations we choose $b_0\geq 10$ and $m_0\geq \max(2a_1, b_0+2)$. Note that whenever $v\geq-b_0,$ $\tr=r-b_0-v,$ since this cut-off is used only when $v-r\leq -m_0.$ We also observe that Lemma \ref{lemma: K3tilt}  guarantees that $\mcK_u$, as well as $\mcL_b$, is positivity-preserving.

The above formulae hold true for all positive functions $\delta_1$ and $\delta_2$. The choice of these functions is then determined by the requirement that $|\mcL_{\delta}\Delta_s|$  be made ``small" compared to  $\norm{\Delta}_Y\mcV_u\Gamma_s$ (see the estimate in Lemma \ref{lemma: mcLdelEstimate}). Let us write down the explicit forms of these functions.
\begin{align}
\begin{aligned}
\delta_1(v, r) &=  \frac{( 1 - \rme^{-\alpha r})^{\frac{4}{\gamma_0}} }{M \exp(\frac{\mu'}{\gamma_0}\max(a_1, (v-r)))}\\
\delta_2(v, r) &=  \frac{( 1 - \rme^{-\alpha r})^{\frac{4}{\gamma_0}} }{M \exp(\frac{\mu'}{\gamma_0}\max(a_1, v))},
\end{aligned}\label{defn: deltaCutOff}
\end{align}
where $\mu'>\mu$ and  $M$ is used as a sort of ``tuning parameter'' to make these cut-off functions ``small''. This ``smallness'' is inherited by $\mathcal{L}_{\delta}[\Delta_t](v, r)$, which consists of singular integrals over intervals of length $\delta_1$ and $\delta_2$. We do not care much about the actual value of the constant $M$. It is quite large and that such a choice $0<M<\infty$ can be made is enough for us. 

Before we conclude this part about the derivation of Equation (\ref{eq: DeltaEvol}), some remarks about the necessity for using two different cut-off  functions are in order. For the moment let us use $\delta$ as a stand-in for the cut-off functions $\delta_1$ and $\delta_2$. Let us take the term $\int_v^{v+\delta}\rmd w\ K_3(v, w)\Delta_t(w, w-v)$ from $\mcL_{\delta}$. Observe that, for large, positive values of the variable $v$ the expression
\begin{align*}
\int_v^{v+\delta}\rmd w\ K_3(v, w)\Gamma_t(w, w-v)
= \int_0^{\delta} \rmd r'\ &K_3(v, v+r') \left(f(v) + f(v+r')\right) \rme^{\mu\max\left(a, v, c_0(v+r')\right) }g_t(v+r', r'),
\end{align*}
contains an extra factor of $\rme^{\mu(1- c_0)v}$ relative to $(\mcV_u\Gamma_t)(v, r)$. Thus the only way this term can be controlled by $\mcV_u\Gamma_t$ is by putting in a countervailing $v$-dependence in the definition of the cut-off function $\delta$ used here. This explains the $v$-dependence in the definition of $\delta_2$ in (\ref{defn: deltaCutOff}). Note that with this definition of the cut-off function the extra factor of $\rme^{\mu(1- c_0)v}$ gets cancelled.

Let us now note that the potential $\mcV_u(v, r)$ defined in (\ref{defn: mcVu}) behaves as
\begin{align*}
\mcV_u(v, r)\simeq \bn\Big( \left(\ln(1 + \rme^{v-r})\right)^{-1/2}\ln\delta_1^{-1} +   \left(\ln(1 + \rme^v)\right)^{-1/2}\ln\delta_2^{-1}  + \sqrt{\max(1,v)} \Big)
\end{align*}
For large, positive values of $v$, $\ln\delta_2^{-1}$  has a part that behaves like $\ln v$.  In order for the main estimate  (see Lemma \ref{lemma: mcLuEstimate}) used in the proof of Theorem \ref{theorem: DeltaSolutionsY} to be true, $\mcL_u[\Gamma_t](v, r)= \mcV_u(v, r)\Gamma_t(v, r) - \mcK_u[\Gamma_t](v, r)$ must have the same asymptotic  behavior as $\mcV_u(v, r)\Gamma_t(v, r)$. It is clear from the proof of Lemma \ref{lemma: mcLuEstimate} that the argument appearing in the point-singularity must be the same as that in the $\ln\delta^{-1}$ multiplying it, otherwise we would end up with terms of the form $\rme^{-\frac{1}{2}(v-r)}\ln v$, which would cause our estimate to fail in the region $v\gg 0$, $v-r\ll 0$. Thus $\delta_1$ and $\delta_2$ have to be two different functions, as defined in (\ref{defn: deltaCutOff}).

\subsubsection{Solutions of (\ref{eq: DeltaEvol}): The proof of Theorem \ref{theorem: DeltaSolutionsY}}\label{subsubsection: mcLuScheme}

This part of subsection \ref{subsection: DeltaSoln} is planned as follows: we first state three lemmas on which the proof of Theorem \ref{theorem: DeltaSolutionsY} rests; then we write down the proof of this theorem; finally, we prove the aforementioned lemmas.
These three lemmas concern the control of the three operators $\mcK_u$, $\mcL_{\delta}$ and $\mcL_b$ appearing in (\ref{eq: DeltaEvol}). Without further ado, we state the lemmas below.
\begin{lemma}\label{lemma: mcLuEstimate}
	There exists $q_0>0$ depending on the parameters $\alpha, c_0, \mu', \gamma_0$, such that the following bound holds:
	\begin{align*}
	\mcL_u\Gamma_s  + \dot{\Gamma}_s&=\mcV_u\Gamma_s - \mcK_u\Gamma_s + \dot{\Gamma}_s\geq \frac{q_0}{\ln M} \mcV_u\Gamma_s.
	\end{align*}
\end{lemma}

\begin{lemma}\label{lemma: mcLdelEstimate}
	There exists $q_1>0$ depending on the parameters $\gamma_0$ and $\kappa$, such that the following bound holds:
	\begin{align*}
	|\mcL_{\delta}\Delta_s|\leq \frac{q_1(\kappa, \gamma_0)}{M^{\gamma_0}\ln M}\norm\Delta_Y\left(\mcV_u\Gamma_s + \dot{\Gamma}_s\right).
	\end{align*}
\end{lemma}


\begin{lemma}\label{lemma: mcLbEstimate}
	The bounded linear operator $\mcL_b$ satisfies
	\begin{align*}
	| \mcL_b\Delta_s |\leq A(b_0, m_0, \mu, c_0)\norm\Delta_Y\Gamma_s,
	\end{align*}
	for some $A(b_0, m_0, \mu, c_0)>0.$
\end{lemma}

Let us now write down the proof of Theorem \ref{theorem: DeltaSolutionsY} on the basis of the above lemmas.
\begin{proof}[Proof of Theorem \ref{theorem: DeltaSolutionsY}] \label{proof: MainDeltaResult}
	Given initial datum $\Delta_0\in \ol X,$ let us recall the Duhamel-integrated form (\ref{eq: DeltaDuhamel}) for the evolution equation for $\Delta_t$:
	\begin{align*}
	\Delta_t =& \rme^{-t\mcV_u}\Delta_0 + F[\Delta_t]\\
	=& \rme^{-t\mcV_u}\Delta_0 + \  \int_0^t\ \rmd s\ \rme^{-(t-s)\mcV_u} \mcK_u[\Delta_s] - \int_0^t\ \rmd s\ \rme^{-(t-s)\mcV_u} \mcL_b[\Delta_s] - \int_0^t\ \rmd s\ \rme^{-(t-s)\mcV_u} \mcL_{\delta}[\Delta_s],
	\end{align*}
	We observe that Lemma \ref{lemma: mcLuEstimate}	implies the following upper bound:
	\begin{align*}
	\mcK_u\Gamma_s\leq \left( 1 - \frac{q_0}{\ln M} \right)\left( \mcV_u\Gamma_s + \dot{\Gamma}_s  \right),
	\end{align*}
	which means we can write
	\begin{align}\label{eq: mcKubound}
	\Bigg| \int_0^t\ \rmd s\ \rme^{-(t-s)\mcV_u} \mcK_u[\Delta_s]\Bigg|&\leq \norm\Delta_Y\int_0^t\rmd s\ \rme^{-(t-s)\mcV_u}\mcK_u\Gamma_s\nonumber\\
	& \leq \left( 1 - \frac{q_0}{\ln M} \right)\norm\Delta_Y\int_0^t\rmd s\ \partial_s\left(\rme^{-(t-s)\mcV_u}\Gamma_s\right)\nonumber\\
	& \leq \left( 1 - \frac{q_0}{\ln M} \right)\norm\Delta_Y\Gamma_t,
	\end{align}
	where in the first line we have used the fact that $\mcK_u$ is positivity-preserving.
	Similarly, Lemma \ref{lemma: mcLdelEstimate} yields the following estimate:
	\begin{align}\label{eq: mcLdelbound}
	\Bigg|\int_0^t\ \rmd s\ \rme^{-(t-s)\mcV_u} \mcL_{\delta}[\Delta_s]\Bigg|\leq \frac{q_1}{M^{\gamma_0}\ln M}\norm\Delta_Y\Gamma_t.
	\end{align}
	Finally let us observe that for all $s'>0$, $\mcV_u\Gamma_{s'} + \dot{\Gamma}_{s'}>0$, which means
	\begin{align*}
	& \qquad\int_s^t\rmd s'\partial_{s'}\left( \rme^{-(t-s')\mcV_u}\Gamma_{s'}\right)>0, \ \forall s\in(0, t),\ \text{implying}\ \Gamma_t>\rme^{-(t-s)\mcV_u}\Gamma_s.
	\end{align*}
	Lemma \ref{lemma: mcLbEstimate} then implies
	\begin{align}\label{eq: mcLbound}
	\Bigg|\int_0^t\ \rmd s\ \rme^{-(t-s)\mcV_u} \mcL_b[\Delta_s]\Bigg|\leq A(b_0, m_0,\mu, c_0)T\norm\Delta_Y,\qquad\quad \forall t\in[0, T],\ T>0.
	\end{align}
	Putting together (\ref{eq: mcKubound}), (\ref{eq: mcLdelbound}) and (\ref{eq: mcLbound}), it is easy to see that
	\begin{align}\label{eq: Fbound}
	|F\Delta_t|\leq \norm\Delta_Y\Gamma_t \left( 1 - \frac{q_0}{\ln M} + AT + \frac{q_1}{M^{\gamma_0}\ln M}  \right),
	\end{align}
	where $q_0$ and $q_1$ have no dependence on $M$ and it is evident that $M<\infty$ can be chosen large enough such that the following is true:
	\begin{align*}
	0< T^* < \frac{1}{A\ln M}\left( q_0 - \frac{q_1}{M^{\gamma_0}} \right), \text{ for some } T^*\in\R_+.
	\end{align*}
	Now that $M$ is chosen in the above manner, clearly we have, for $0<t\leq T^*$, $|F\Delta_t|<\norm\Delta_Y\Gamma_t,$, i.e., $\norm F_Y<1.$
	The corresponding Neumann series then converges and we obtain the following unique solution of our Duhamel-integrated equation:
	\begin{align*}
	\Delta_t=\left( 1 - F \right)^{-1} \rme^{-t\mcV_u}\Delta_0,\ \forall 0<t\leq T^*.
	\end{align*}
\end{proof}
Note that, as we have mentioned already before the beginning of Subsection \ref{subsection: DeltaSoln}, the above estimate can be extended in time.
We now move on to the proofs of the lemmas stated above. We will prove Lemmas \ref{lemma: mcLdelEstimate} and \ref{lemma: mcLbEstimate} first and save the proof of Lemma \ref{lemma: mcLuEstimate}, which is much more involved, for last.
\begin{proof}[Proof of Lemma \ref{lemma: mcLdelEstimate}]
	The time-derivative of $\Gamma_s$ is
	\begin{align}\label{eq: dotGammas}
	\dot{\Gamma}_s(v, r) = - \frac{1}{8}\min(1, \bar{n})\frac{1}{1 + e^{\beta (v-r)}}\cdot\frac{1}{(1 + s)^2}\ln ( 1 - e^{-\kappa r})^{-1}\Gamma_s(v, r),\qquad s>0.
	\end{align}
	From the definition (\ref{defn: mcVu}) of $\mcV_u$, it is easy to see that
	\begin{align}\label{eq: mcVuGammadot}
	\begin{aligned}
	&\mcV_u(v, r)\geq 2\bn\left(\ln(1 + \rme^{v-r})\right)^{-1/2}\left[ \ln M + \ln ( 1 - \rme^{-r})^{-1}\right],\\
	\text{so that,}&\ \mcV_u(v, r)\Gamma_s(v, r) + \dot{\Gamma}_s(v, r)> \bn \left(\ln(1 + \rme^{v-r})\right)^{-1/2}\left[\ln M + \ln ( 1 - \rme^{-r})^{-1}\right]\Gamma_s(v, r).
	\end{aligned}
	\end{align}
	By definition (\ref{eq: DeltaEqSplitLdel}) we have, 
	\begin{align*}
	\mcL_{\delta}\Delta_s(v, r)&= \int_0^{\delta_1}\rmd r'\ K_3^1(v, v-r-r')\Delta_s(v, r+r') + \int_0^{\delta_1}\rmd r'\ K_3(v-r, v+r')\Delta_s(v+r', r+r')\\
	&\ + \int_0^{\delta_1}\rmd r'\ K_3(v-r, v-r+r')\Delta_s(v-r+r', r') - \int_0^{\delta_1}\rmd r'\ K_3^1(v-r, v-r-r')\Delta_s(v-r, r')\\
	&\ + \int_0^{\delta_2}\rmd r'\ K_3^1(v, v-r')\Delta_s(v, r') - \int_0^{\delta_2}\rmd r'\ K_3(v, v+r')\Delta_s(v+r', r').
	\end{align*}
	We will write estimates for two of these terms. The rest can be estimated in a similar manner.
	\begin{align*}
	i)\ \int_0^{\delta_1}\rmd r'\ |K_3^1(v, v-r-r')\Delta_s(v, r+r')|
	\leq&\ \norm\Delta_Y\int_0^{\delta_1}\rmd r'\ K_3^1(v, v-r-r')\Gamma_s(v, r+r')\\
	\leq&\ 4\bar{n}\norm\Delta_Y\rme^{\mu\max(a, c_0v, v-r)}\left( f(v-r) + f(v)\right)\kappa \left(\ln(1 + \rme^v)\right)^{-\frac{1}{2}}\times\\
	&\qquad\times\int_0^{\delta_1}\rmd r'\ \rme^{-r-r'}\left( 1 - \rme^{-r-r'}\right)^{\gamma_s(v-r)-1}\\
	\leq&\ \frac{6\kappa\bar{n}}{\gamma_0}\norm\Delta_Y\Gamma_s(v, r)\left(\ln(1 + e^v)\right)^{-\frac{1}{2}}\delta_1^{\gamma_s(v, r)},\\
	ii)\ \int_0^{\delta_2}\rmd r'\ |K_3(v, v+r')\Delta_s(v+r', r')|
	\leq&\ \norm\Delta_Y\int_0^{\delta_2}\rmd r'\ K_3(v, v+r')\Gamma_s(v+r', r')\qquad\\
	\leq&\ 4\bar{n}\norm\Delta_Y\left(\ln(1 + \rme^v)\right)^{-\frac{1}{2}}\kappa\left(f(v-r) + f(v)\right)\times\\
	&\ \times\int_0^{\delta_2}\rmd r'\ \rme^{-r'}\left(1 - \rme^{-r'}\right)^{\gamma_s(v)-1}\rme^{\mu\max(a, c_0(v+r'), v)}\\
	\leq&\ \frac{6\kappa\bar{n}}{\gamma_0}\norm\Delta_Y\rme^{\mu\max(a, c_0v)}\left(f(v-r) + f(v)\right)\left(\ln(1 + \rme^v)\right)^{-\frac{1}{2}}\delta_2^{\gamma_0}.
	\end{align*}
	Estimating the rest of the terms in a similar manner, we can write the following upper bound:
	\begin{align*}
	|\mathcal{L}_{\delta}[\Delta_s](v, r)|&\leq \norm\Delta_Y \frac{p_1\kappa}{\gamma_0 M^{\gamma_0}}\bar{n}\left(\ln(1 + e^{v-r})\right)^{-\frac{1}{2}} (1 - e^{-\alpha r})^3\ \Gamma_s(v, r),
	\end{align*}
	where $p_1$ is a numerical constant. Obviously then,
	\begin{align*}
	|\mcL_{\delta}[\Delta_s](v, r)|\leq \norm\Delta_Y\frac{q_1(\kappa, \gamma_0)}{M^{\gamma_0}\ln M}\left(\mcV_u\Gamma_s(v, r) + \dot{\Gamma}_s(v, r)\right),\quad \forall(v, r)\in\R\times\R_+,\ \text{where}\  q_1=\frac{2p_1\kappa}{\gamma_0}.
	\end{align*}
	\end{proof}
Lemma \ref{lemma: mcLbEstimate} is quite obvious from the definition (\ref{defn: mcLb}) of the bounded operator $\mcL_b$.

\begin{proof}[Proof of Lemma \ref{lemma: mcLbEstimate}]
	It is clear from definition (\ref{defn: mcLb}) that $\mcL_b\Delta_s(v, r)$ does not contain any point or line singularity. In addition, the exponential decay in the integrands ensures the finiteness of the integrals. The bound in the lemma is then obtained by straightforward computations.
\end{proof}
We now come to the most crucial estimate, i.e., the one contained in Lemma \ref{lemma: mcLuEstimate}. The proof of this lemma relies quite heavily on certain properties of the H\"{o}lder-type condition used in the definition (\ref{defn: timeGamma}) of $\Gamma_t$. These properties are listed (and proved) in Appendix \ref{appendix: g}. Since the computations are quite involved, we will try to give a general idea of the scheme of the proof before writing down the proof formally.

The main idea of Lemma \ref{lemma: mcLuEstimate} is to establish that $(\mcL_u\Gamma_s)(v, r)$ has the same asymptotic behavior as $\mcV_u(v, r)\Gamma_s(v, r),$ for all $(v, r)\in\R\times\R_+$. Let us first observe that the  potential $\mcV_u(v, r)$ grows exponentially at $(-\infty)$, has a line singularity at $r=0 $ and grows like the square-root function at $+\infty$, somewhat like the function
\begin{align*}
V'(v, r)\simeq\bar{n}\left(\rme^{-\frac{1}{2}(v-r)}\ln(2 + 1/r) + \sqrt{\max(1,v)} \right).
\end{align*}
In order to show that $\mcL_u\Gamma_s$ has the same behavior, we will first split it into different parts and then show how each of these parts produces the correct asymptotic behavior in different regions of $\R\times\R_+$.
We write
\begin{align}\label{defn: mcLuSplitting}
\mcL_u[\Gamma_t](v, r) = I_1[\Gamma_t](v, r) + I_2[\Gamma_t](v, r) + I_3[\Gamma_t](v, r) + I_4[\Gamma_t](v, r).
\end{align}
As the explicit formulae for the $I_i$'s ($i\in\{1, 2, 3, 4\}$) written below show, these terms consist of integrals grouped together on the basis of intervals of integration and the type of singularities they contain. For example, the terms in $I_3$ do not contain any line singularity, unlike terms in $I_1$ and $I_2$. On the other hand, the integrals in $I_4$ contain an extra ``smallness'' because they are integrated over intervals of length at most $\delta_1$, while in $I_2$ the integrals are taken over intervals of length at most $r$. As will become apparent shortly, the three groups $I_1$, $I_2$ and $I_3$ produce different kinds of asymptotic behavior in $\mcL_u\Gamma_t$: 
\begin{itemize}
	\item $I_1$ yields a term with point singularity which  behaves like $(\ln (1 + \rme^{v-r}))^{-1/2}\ln (1 - \rme^{-r})^{-1}$, i.e., dominates when $r\ll 1$.
	\item From $I_2$ one gets a term with the point singularity, which becomes dominant for $r\gg 1$, behaving like $(\ln (1 + \rme^{v-r}))^{-1/2} (1 - \rme^{-r})^3.$
	\item Finally, the group $I_3$ contributes to $\mcL_u\Gamma_t$ the necessary  $\sqrt{v}$ growth, when $v\gg 1$ .
\end{itemize}
The $I_i[\Gamma_t]$'s are given by the following formulae:
\begin{align}\label{defn: I1}
&I_1[\Gamma_t](v, r)= \int_{r+\delta_1}^{\infty}\rmd r'\ \Big[ K_3^1(v-r, v-r-r')\Gamma_t(v, r) + K_3^1(v, v-r')\Gamma_t(v, r') - K_3^1(v-r, v-r-r')\Gamma_t(v, r+r')\Big]\nonumber\\
&\   + \cf(v<-b_0)\Bigg[ \int_{r+\delta_1}^{\infty}\rmd r'\ K_3(v, v+r')\Gamma_t(v, r) - \cf(a_1-v\leq r+\delta_1)\left\{\int_{r+\delta_1}^{\infty}\rmd r'\ \ol{K_{3}}^2(v, v+r')\Gamma_t(v+r', r+r')\right.\Bigg.\nonumber\\
&\ \ \ \  \left.-\int_{r+\delta_1}^{r+a_1-v}\rmd r'\ K_3(v-r, v-r+r')\Gamma_t(v-r+r', r')
- \int_{r+a_1-v}^{\infty}\rmd r'\ \ol{K_{3}}^2(v-r, v-r+r')\Gamma_t(v-r+r', r')\right\}\nonumber\\
&\ \ \  +\cf(a_1 -v>r+\delta_1)\left\{ \int_{r+\delta_1}^{a_1-v}\rmd r'\ K_3(v, v+r')\Gamma_t(v+r', r+r')  + \int_{a_1-v}^{\infty}\rmd r'\ \ol{K_{3}}^2(v, v+r')\Gamma_t(v+r', r+r') \right.\nonumber\\
&\ \ \ \ \  \Bigg.\left. -  \int_{r+\delta_1}^{r+a_1-v}\rmd r'\ K_3(v-r, v-r+r')\Gamma_t(v-r+r', r') - \int_{r+a-v}^{\infty}\rmd r'\ \ol{K_{3}}^2(v-r, v-r+r')\Gamma_t(v-r+r', r')\right\}\Bigg]\nonumber\\
&\  +  \cf(v\geq-b_0)\Bigg[ \int_{r+\delta_1}^{\infty}\rmd r'\ K_3(v, v+r')\Gamma_t(v, r) -\int_{r+\delta_1}^{\infty}dr'\ \ol{K_{3}}^2(v, v+r')\Gamma_t(v+r', r+r')\Bigg.\nonumber\\
&\ \ \ \  \Bigg. + \int_{r+\delta_1}^{\infty} dr'\ \ol{K_{3}}^2 (v-r, v-r+r')\Gamma_t(v-r+r', r')\Bigg],
\end{align}
\begin{align}\label{defn: I2}
&I_2[\Gamma_t](v, r)
= \int_{\delta_1}^r\rmd r'\ K_3^1(v, v-r')\left(\Gamma_t(v, r) - \Gamma_t(v-r', r-r')\right) +  \int_{\delta_1}^r \rmd r'\ K_3^1(v-r, v-r-r')\left( \Gamma_t(v, r) - \Gamma_t(v, r+r') \right)\nonumber\\
&\ + \int_{\delta_1}^r \rmd r'\ K_3(v-r, v-r+r')\Gamma_t(v, r) - \left[\cf(v-r<-b_0)\int_{\delta_1}^{\min(r, r-v)} \rmd r'\ K_3(v-r, v-r+r')\Gamma_t(v, r-r')\right.\nonumber\\
&\ \ \ \left.+ \cf(v-r\geq -b_0)\int_{\delta_1}^r\rmd r'\ \ol{K_{3}}^2(v-r, v-r+r')\Gamma_t(v, r-r')\right]\nonumber\\
&\ + \int_{\delta_1}^r \rmd r'\ K_3(v, v+r')\Gamma_t(v, r) - \cf(v<-b_0)\left[ \cf(r + \delta_1\leq a_1-v)\int_{\delta_1}^r \rmd r'\ K_3(v, v+r')\Gamma_t(v+r', r+r')\right.\nonumber\\
&\ \ \ \left.+\cf(a_1-v<r + \delta_1)\left\{ \int_{\delta_1}^{\min(r, a_1-v)} \rmd r'\ K_3(v, v+r')\Gamma_t(v+r', r+r') \right.\right.\nonumber\\
&\ \ \ \  \left.\left.+ \cf(a_1-v<r)\int_{a_1-v}^r \rmd r'\ \ol{K_{3}}^2(v, v+r')\Gamma_t(v+r', r+r') \right\}\right]\nonumber\\
&\ -\cf(v\geq -b_0)\int_{\delta_1}^r\rmd r'\ \ol{K_{3}}^2(v, v+r')\Gamma_t(v+r', r+r'),
\end{align}
\begin{align}\label{defn: I3defn}
I_3[\Gamma_t](v, r)=& \int_0^{\infty} \rmd r'\ K_3^2(v-r, v-r-r')\Gamma_t(v, r)   + \int_0^r\ \rmd r'\ K_3^2(v, v-r')\Gamma_t(v, r)\nonumber\\
&   -  \cf(v-r\leq -m_0)\int_0^{\tr}\left(K_3^2(v-r, v-r-r') - K_3^2(v, v-r-r')\right)\Gamma_t(v, r+r')\nonumber\\
&  -\cf(v-r\geq m_0)\left[\cf(v\leq 3r)\int_0^{v-r} \rmd r'\ \left(K_3^2(v-r, v-r-r') - K_3^2(v, v-r-r')\right)\Gamma_t(v, r+r')\right.\nonumber\\
& \  \ \ \left.+\cf(v>3r)\int_{v-r-c_0v}^{\infty} \rmd r'\ \{K_3^2(v-r, v-r-r') - K_3^2(v, v-r-r')\}\Gamma_t(v, r+r')\right]\nonumber\\
&   - \cf(v\leq -m_0)\left[ \cf(v+r<-b_0)\int_0^r \rmd r'\ K_3^2(v, v-r')\Gamma_t(v-r', r-r')\right.\nonumber\\
& \ \ \  +\left. \cf(v+r\geq-\tilde{b}_0)\int_0^{-v-b_0} \rmd r'\ K_3^2(v, v-r')\Gamma_t(v-r', r-r')\right]\nonumber\\
&  - \cf(v\geq m_0)\left[ \cf(0<v-r<c_0 v)\int_{v-c_0^{-1}(v-r)}^r \rmd r'\ K_3^2(v, v-r')\Gamma_t(v-r', r-r')\right.\nonumber\\
&\qquad \ \  \ \left. +\cf(v-r\geq c_0v)\int_0^r \ \rmd r'\ K_3^2(v, v-r')\Gamma_t(v-r', r-r')\right]\nonumber\\
=&\ I_3^{(1)}[\Gamma_t](v, r) + I_3^{(2)}[\Gamma_t](v, r),
\end{align}
where $I_3^{(1)}$ includes all the terms with $K_3^2(v-r, v-r-r')$ in the integrand, while $I_3^{(2)}$ includes those with $K_3^2(v, v-r')$, and finally,
\begin{align}\label{defn: I4}
&I_4[\Gamma_t](v, r) = \int_r^{r+\delta_1} \rmd r'\ K_3^1(v-r, v-r-r')\left(\Gamma_t(v, r) - \Gamma_t(v, r+r')\right)+\int_r^{r+\delta_1} \rmd r'\ K_3(v, v+r')\Gamma_t(v, r) \nonumber\\
& \ \ -  \cf(v<-b_0)\left[\cf(r\leq a_1-v< r+\delta_1)\left( \int_r^{a_1-v} \rmd r'\ K_3(v, v+r')\Gamma_t(v+r', r+r')\right.\right.\nonumber\\
& \ \ \  \left.+ \int_{a_1-v}^{r+\delta_1} \rmd r'\ \ol{K_{3}}^2(v, v+r')\Gamma_t(v+r', r+r')\right) + +\cf(a_1-v<r)\int_r^{r+\delta_1} \rmd r'\ \ol{K_{3}}^2(v, v+r')\Gamma_t(v+r', r+r') \nonumber\\
& \ \ \ \left.+ \cf(a_1-v\geq r+\delta_1)\int_r^{r+\delta_1}\rmd r'\ K_3(v, v+r')\Gamma_t(v+r', r+r')\right]\nonumber \\
& \ \ - \cf(v\geq -b_0)\int_r^{r+\delta_1}\ \rmd r'\ \ol{K_{3}}^2(v, v+r')\Gamma_t(v+r', r+r')\nonumber\\
&\ \  +  \int_{\delta_2}^{\delta_1}\rmd r'\ K_3(v, v+r')\left( \Gamma_t(v, r) - \Gamma_t(v+r', r+r') \right) + \int_{\delta_2}^{\delta_1}\rmd r'\ K_3^1(v, v-r')\left( \Gamma_t(v, r) - \Gamma_t(v, r') \right).
\end{align}
Let us now proceed to the proof of Lemma \ref{lemma: mcLuEstimate}, keeping in mind that the computational details referred to in this proof are to be found in Appendix \ref{appendix: DeltaCompute}.

\begin{proof}[Proof of \ref{lemma: mcLuEstimate}]\label{proof: mcLuEstimate}
	The proof of this lemma hinges on obtaining a suitable lower bound for $\mcL_u\Gamma_s$. We will outline the general scheme of the estimates leading to this lower bound, writing down explicitly only those terms that lead to the correct asymptotic behavior in different regions and refer to Appendix \ref{appendix: DeltaCompute} for all other details. We write
	\begin{align*}
	\Gamma_t = \Gamma_t^1 + \Gamma_t^2, \text{where }
	\Gamma_t^1(v, r) &= f(v-r) \exp\left(\mu\max(a, c_0v, v-r)\right) g_t(v, r),\\
	\text{and}\ \ \Gamma_t^2(v, r) &= f(v) \exp\left(\mu\max(a, c_0v, v-r)\right) g_t(v, r).
	\end{align*}
	We will now look at each of the $I_i[\Gamma_t]$'s.
	\newline	i) $I_1[\Gamma_t](v, r):$ It is quite easy to obtain the following lower bound for $I_1[\Gamma_t](v, r)$:
	\begin{align*}
	I_1[\Gamma_t](v, r) = I_1[\Gamma_t^1](v, r) + I_1[\Gamma_t^2](v, r)\geq\ \mcJ_0[\Gamma_t^1 + \Gamma_t^2](v, r) + \mcJ_1[\Gamma_t^1](v, r) + \mcJ_2[\Gamma_t^2](v, r) + \mathcal{I}[\Gamma_t^1 + \Gamma_t^2](v, r),
	\end{align*}
	where $\mcJ_0$ denotes the dominant term close to the diagonal and is defined as
	\begin{align}\label{defn: J_0}
	&\mcJ_0[\Gamma_t^1 + \Gamma_t^2](v, r)\nonumber\\
	=&\ f(v-r)\Big[\rme^{\mu\max\left(a, c_0v, v-r\right)}\int_{r+\delta_1}^{\infty} \rmd r'\ K_3^1(v-r, v-r-r')\left( g_t(v, r) + g_t(v, r') - g_t(v, r+r')  \right)\big.\nonumber\\
	+ & \big. \int_{r+\delta_1}^{\infty} \rmd r'\ K_3(v, v+r') \rme^{-\frac{1}{2}r'}\rme^{\mu\max\left(a, c_0(v+r'), v-r\right)}\left( g_t(v, r) + g_t(v-r+r', r') - g_t(v+r', r+r')\right)\Big]\nonumber\\
	+ &\ f(v)\rme^{\mu\max\left(a, c_0v, v-r\right)}\int_{r+\delta_1}^{\infty} \rmd r'\ K_3^1(v-r, v-r-r')\left( g_t(v, r) + g_t(v, r') - g_t(v, r+r')  \right)\nonumber\\
	+ &\ \int_{r+\delta_1}^{\infty} \rmd r'\ K_3(v, v+r') \rme^{-\frac{1}{2}r'}f(v+r')\rme^{\mu\max\left(a, c_0(v+r'), v-r\right)}\left( g_t(v, r) + g_t(v-r+r', r') \right.\nonumber\\
	&\qquad\qquad\left.- g_t(v+r', r+r')\right).
	\end{align}
	$\mcJ_1$, $\mcJ_2$ and $\mathcal{I}$ contain only such integrals which do not contain the factor $\left( 1 - \rme^{-r'}\right)^{-1}$ in the integrand and are thus sub-dominant to $\mcJ_0$ close to the diagonal. Some of these integrals are negative and have to controlled by the other parts. The explicit formulae for these tems as well as the computations showing how they are controlled are in Appendix \ref{appendix: DeltaCompute}.
	Then Lemma \ref{lemma: gtprop3} from Appendix \ref{appendix: g} tells us:
	\begin{align*}
	\int_{r+\delta_1}^{\infty} \rmd r'\ K_3^1(v-r, v-r-r')\left( g_t(v, r) + g_t(v, r') - g_t(v, r+r')  \right)
	\geq 0.4 g_t(v, r)\int_{r+\delta_1}^{\infty}\rmd r' K_3^{1}(v-r, v-r-r'), 
	\end{align*}
	which yields a $(\ln (1 + \rme^{v-r}))^{-1/2}\ln (1 - \rme^{-r})^{-1}$ term typifying the correct asymptotic behavior for $r\ll 1.$\\
	ii) $I_2[\Gamma_t](v, r):$ We write $I_2[\Gamma_t](v, r) = I_2[\Gamma_t^1](v, r) + I_2[\Gamma_t^2] (v, r)$. It is the lower bound for $I_2[\Gamma_t^1](v, r)$ that produces the correct asymptotic behavior in the region away from the diagonal when the point singularity dominates. This bound is:
	\begin{align}\label{eq: I_2_Gamma1LowerBound}
	&I_2[\Gamma_t^1](v, r)
	\geq\ \cf(v-r<-b_0)e^{\mu\max\{a, c_0v\}}\int_{\delta_1}^r \rmd r'\ \Big[ K_3(v-r, v-r+r')\left( f(v-r) - (\ln(1 + e^{v-r+r'}))^{-\alpha}\right)\Big.\nonumber\\
	&\qquad\qquad\qquad - \Big. K_3^1(v-r, v-r-r') \left( f(v-r-r') - f(v-r) \right)\Big] g_t(v, r)\nonumber\\
	& \ \ +  \cf(v< -b_0)\cf(r>a_1-v)f(v-r)\int_{a_1-v}^r \rmd r'\ K_3(v, v+r')\left( \rme^{\mu a} - \rme^{\mu c_0 v}\rme^{-(\frac{1}{2} - \mu c_0)r'}  \right) g_t(v, r)\nonumber\\
	& \ \ +  \cf(v\geq -b_0)\cf(v-r<-b_0)\Big[ f(v-r)\int_{\delta_1}^r \rmd r'\ K_3(v, v+r')\left( \rme^{\mu\max\{a, c_0v\}} - \rme^{-\frac{1}{2}r'} \rme^{\mu\max\{a, c_0(v+r')\}} \right) g_t(v, r)\Big. \nonumber\\
	&\ \ \ \Big. + \int_{\min(r, r-v)}^r \rmd r'\ K_3(v-r, v-r+r') (\ln(1 + e^{v-r+r'}))^{-\alpha}\left( 2 g_t(v, r) - g_t(v, r+r')\right)\Big]+  I_2^{-}[\Gamma_t^1](v, r),
	\end{align}	
	where $I_2^{-}[\Gamma_t^1](v, r)$ includes all the negative parts coming from $I_2[\Gamma_t^1]$; the integrals in $I_2^{-}[\Gamma_t^1]$ contain  differences of the form $\left(g_t(., r+r') - g_t(., r)\right)$, which generate an extra exponential decay of $\rme^{-\kappa r}$, as explained in Appendix \ref{appendix: DeltaCompute}, and hence are ``small'' in the region away from the diagonal.
	
	The relevant dominating behavior with point singularity for $r\gg 1$, comes from the first term of the lower bound (\ref{eq: I_2_Gamma1LowerBound}) for $I_2[\Gamma_t](v, r).$ We call this term $\ol{\mathcal{C}_3}$. Then
	\begin{align}\label{def: mcC3}
	&\ \ol{\mathcal{C}}_3[\Gamma_t^1](v, r)\nonumber\\
	=&\cf(v-r<-b_0)\rme^{\mu\max\{a, c_0v\}}g_t(v, r)\int_{\delta_1}^r \rmd r'\ \Big[ K_3(v-r, v-r+r')\left( f(v-r) - (\ln(1 + \rme^{v-r+r'}))^{-\alpha}\right)\Big.\nonumber\\
	&\qquad\qquad\qquad\qquad\qquad\qquad\qquad\qquad - \Big. K_3^1(v-r, v-r-r') \left( f(v-r-r') - f(v-r) \right)\Big] \\
	\geq&\cf(v-r< -b_0)4\bn\Gamma_t^1(v, r)\left(\ln(1 + \rme^{v-r})\right)^{-\frac{1}{2}}\int_{\delta_1}^r \rmd r'\ \frac{\rme^{v-r} + 2\rme^{-r'}}{1 + \rme^{v-r} + \rme^{-r'}}\rme^{-r'}\left(\frac{\ln (1 + \rme^{v-r-r'})}{\ln(1 + \rme^{v-r})}\right)^{1-\alpha}\times\nonumber\\
	&\ \  \times\left(\frac{\ln(1 + \rme^{v-r+r'})}{\ln(1 + \rme^{v-r-r'})}\right)^{1-2\alpha} \left(\frac{1 - \rme^{-\frac{248}{125}(1 - 2\alpha)r'}}{1 - \rme^{-r'}}\right) \left[ 1 -  \left(\frac{\ln(1 + \rme^{v-r-r'})}{\ln(1 + \rme^{v-r})}\right)^{\alpha}  \right]\nonumber
	\end{align}
	iii) $I_3[\Gamma_t](v, r):$ The square-root-like growth for large, positive values of $v$ comes from $I_3[\Gamma_t](v, r)$ as follows:
	\begin{align*}
	&\cf(v-r\geq-m_0)I_3^{1}[\Gamma_t](v, r) +  \cf(v\geq m_0)I_3^2[\Gamma_t](v, r)\\
	\geq&\ \cf(v\geq m_0)\bn\Gamma_t(v, r)\Bigg[  \cf(v-r\leq 0)(\ln(1 + \rme^v))^{\frac{1}{2}}\Big( 1 - \left(\frac{\ln 2}{\ln (1 + \rme^v)}\right)^2\Big)\Bigg.\\
	&\quad + \cf(0<v-r<c_0v)\frac{3}{2}(\ln(1 + \rme^v))^{\frac{1}{2}}\Big( 1 - \left(\frac{v-r}{c_0v}\right)^2 \Big)\\
	&\quad + \cf(v-r\geq m_0)\Bigg\{ \cf(v-r\leq\frac{2}{3}v)\ \frac{4}{3}\left( \ln(1 + \rme^v)\right)^{\frac{1}{2}} \left(\frac{\ln(1 + \rme^{v-r})}{\ln(1 + \rme^v)}\right)^2\Bigg.\\
	&\qquad\ \ \ \ \Bigg.\Bigg. + \cf(v-r>\frac{2}{3}v)\ 2\left( \ln(1 + \rme^{v-r})\right)^{\frac{1}{2}} \left(1 - \left(\frac{\ln(1 + \rme^{c_0 v})}{\ln(1 + \rme^{v-r})}\right)^2\right)\Bigg\}\Bigg] \\
	&+ \cf(-m_0<v-r<m_0)\rme^{\mu\max(a, c_0v, v-r)}f(v-r)g_t(v, r)\int_0^{\infty}\rmd r'\ K_3^2(v-r, v-r-r')
	\end{align*}	
	The computations in Appendix \ref{appendix: DeltaCompute} lead us to the following lower bound for $\mcL_u\Gamma_s$, for a suitable choice of the constant $M$ :
	\begin{align*}
	\mcL_u\Gamma_t(v, r)& = I_1[\Gamma_t](v, r) + I_2[\Gamma_t](v, r) + I_3^{(1)}[\Gamma_t](v, r) + I_3^{(2)}[\Gamma_t](v, r) + I_4[\Gamma_t](v, r)\geq G[\Gamma_t](v, r),\ \text{where}\\
	G[\Gamma_t](v, r)=& \ \bn \Gamma_t(v, r) \bigg[\cf(v-r\leq -b_0)\bar{b}_1(\alpha)\left(\ln(1 + \rme^{v-r})\right)^{-\frac{1}{2}}\left( 1 - \rme^{-3\alpha r}\right)^3\bigg.\\
	+\ & \frac{1}{2} \left(\ln(1 + \rme^{v-r})\right)^{-\frac{1}{2}}\ln\left( 1 - \rme^{-\frac{7}{2}(r+\delta_1)}\right)^{-1} + \cf(-m_0< v-r <m_0)\frac{1}{4}\left(\ln(1 + \rme^{v-r})\right)^{-\frac{1}{2}}\\  
	+\ & \cf(0<v-r<m_0)\frac{1}{8}\left(\ln(1 + \rme^{v-r})\right)^{\frac{1}{2}}
	+\cf(0<v<m_0)\frac{1}{2}\left(\ln(1 + \rme^v)\right)^{\frac{1}{2}}\left(  1 - \left(\frac{\ln(1 + \rme^{v-r})}{\ln(1 + \rme^v)}\right)^2\right) \\
	\bigg. +\ & \cf(v\geq m_0) b_3(c_0)\left(\ln(1 + \rme^v)\right)^{\frac{1}{2}}\bigg],
	\end{align*}
	where $\bar{b}_1(\alpha)$ and $b_3(c_0)$ are positive numbers bounded away from zero. Then using formula (\ref{eq: dotGammas}) for $\dot{\Gamma}_s(v, r) $ we can conclude that
	\begin{align*}
	\mcL_u\Gamma_s(v, r)  + \dot{\Gamma}_s(v, r)
	\geq G[\Gamma_s](v, r) + \dot{\Gamma}_s(v, r)\geq \ol{G}[\Gamma_s](v, r),
	\end{align*}
	where
	\begin{align*}
	\ol{G}[\Gamma_s](v, r)=&\ \bn\Gamma_s(v, r)\bigg[ \left(\ln (1 + \rme^{v-r})\right)^{-1/2}\left\{  \frac{3}{8} \ln\left( 1 - \rme^{-\frac{7}{2}(r+\delta_1)}\right)^{-1} + \cf(-m_0<v-r<0)\frac{1}{4}  \right.\bigg.\\
	&\qquad \left. + \cf(v-r\leq-b_0)\ol b_1(\alpha)\left( 1 - \rme^{-3\alpha r}\right)^3\right\}\\
	& + \cf(0\leq v-r<m_0)\frac{1}{8}\bigg( \left(\ln (1 + \rme^{v-r})\right)^{-1/2} + \left(\ln (1 + \rme^{v-r})\right)^{1/2} \bigg)\\
	&+ \cf(v>0) \left(\ln(1 + \rme^v)\right)^{\frac{1}{2}}\bigg\{\cf(0<v<m_0)\frac{1}{2}\left(  1 - \left(\frac{\ln(1 + \rme^{v-r})}{\ln(1 + \rme^v)}\right)^2\right)  + \cf(v\geq m_0)b_3(c_0) \bigg\}\bigg].
	\end{align*}
	From the definition  (\ref{defn: mcVu}) of the potential $\mcV_u$, it is easy to see that there exist positive numbers $\ol{C}_1=\ol{C}_1(\mu', \gamma_0, a_1)$ and $C_1=C_1(\mu', \gamma_0, a_1)$, such that the following is true for all $ (v, r)\in\R\times\R_+$:
	\begin{align*}
	&\bn C_1\left[ \left(\ln(1 + e^v)\right)^{\frac{1}{2}} + \left(\ln(1 + e^{v-r})\right)^{-\frac{1}{2}}\ln \left(2M(1 - e^{-\alpha r})^{-1}\right) \right]\\
	\leq&\ \mcV_u(v, r)\leq \bar{n} \ol{C}_1\left[  \left(\ln(1 + e^v)\right)^{\frac{1}{2}} +   \left(\ln(1 + e^{v-r})\right)^{-\frac{1}{2}}\ln \left(2M(1 - e^{-\alpha r})^{-1}\right) \right].
	\end{align*}
	Then it is clear that there exists $q_0>0$, depending on $\alpha, c_0, \mu', \gamma_0, b_0,$ such that:
	\begin{align*}
	\ol{G}[\Gamma_s](v, r)\geq \frac{q_0}{\ln M}\mcV_u(v, r)\Gamma_s(v, r),\ \ \forall (v, r)\in\R\times\R_+,
	\end{align*}
	which implies the lower bound claimed in this lemma.
\end{proof}

\subsection{Solutions of the Regularized Evolution Equation:}\label{subsection: regDpsiDiff}

Let us recall the regularized evolution equation
\begin{align*}
\partial_t\psi_t(v) &= -(L_3^{\vep}\psi_t)(v) = -\int_{\R}\rmd w\ K_3^{\vep}(v, w)\big(\psi_t(v) - \psi_t(w)\big),\ \text{where}\\
K_3^{\vep}(v, w) &= K_3^{\vep(\max(v, w))}(v, w) = K_3(v, w)\frac{1 - \rme^{-\min(\vep(v,w), |v-w|)}}{1 - \rme^{-\vep(v, w)}},\\
\vep(v, w) &=\vep(\max(v, w))= \vep_0\exp\left(-\frac{\mu'}{\gamma_0}\max(a_1, \max(v, w))\right).
\end{align*}
Recall that $\mu'>\mu$ and the admissible values for the parameters $\mu$, $a_1=a/c_0$ and $\gamma_0$ have already been defined in (\ref{defn: tGammaGamma0}). In this subsection our goal is to prove Theorem \ref{theorem: rDpsiSoln} and Theorem \ref{theorem: rDpsiDeltaApprox}. For our subsequent computations it is useful to split the kernel function just like in the beginning of Section 2 (cf. Lemma \ref{lemma: K3tilt}).
\begin{align*}
&\text{When}\ v>w,\	K_3^{\vep} (v, w) = K_3^{1, \vep(v)} (v, w) + K_3^{2, \vep(v)} (v, w), \text{where} \\
&K_3^{1, \vep(v)} (v, w) = K_3^1(v, w) \frac{1 - \rme^{-\min(\vep(v), v-w)}}{1 - \rme^{-\vep(v)}},\ K_3^{2, \vep(v)} (v, w) = K_3^2(v, w) \frac{1 - \rme^{-\min(\vep(v), v-w)}}{1 - \rme^{-\vep(v)}}. \\
&\text{Similarly in the region $w>v$, }\ \
K_3^{\vep} (v, w) = \ol{K_3}^{1,\vep(w)}(v, w) + \ol{K_3}^{2, \vep(w)}(v, w),\ \text{where}\\
&\ol{K_3}^{2,\vep(w)}(v, w) = \ol{K_3}^2(v, w)\frac{1 - \rme^{-\min(\vep(w), w-v)}}{1 - \rme^{-\vep(w)}},\  \text{and}\ \ol{K_3}^{1, \vep(w)}(v, w) = \ol{K_3}^1(v, w)\frac{1 - \rme^{-\min(\vep(w), w-v)}}{1 - \rme^{-\vep(w)}}.
\end{align*}
Just like Lemma \ref{lemma: K3tilt}, we have the following result for the regularized case:

\begin{lemma}\label{lemma: K3eptilt}
	The kernel functions satisfy the following inequalities for all $r>0$:
	\begin{align*}
	&i)\ \text{For all}\ w<v-r,\ K_3^{1,\vep(v-r)} ( v-r, w ) > K_3^{1, \vep(v)}( v, w ),\ \text{whenever either } v\leq a_1\ \text{or} \ v-w\geq\vep(v-r),\\
	&\text{and}\ K_3^{2,\vep(v-r)} ( v-r, w) > K_3^{2, \vep(v)} (v, w), \  \text{whenever}\ v-r-w\geq\vep(v-r).\\
	&ii)\ \text{For all}\ w>v,\ K_3^{\vep(w)} (v, w) > K_3^{\vep(w)} (v-r, w),\ \ol{K_3}^{2,\vep(w)} (v, w) > \ol{K_3}^{2,\vep(w)} (v-r, w).
	\end{align*}
\end{lemma}	

\begin{proof}
	i) When $w<v-r$ let us define $r' = v-r-w.$ Then:
	\begin{align*}
	&K_3^{1,\vep(v-r)}(v-r, v-r-r') - K_3^{1,\vep(v)}(v, v-r-r')\\
	\geq&\ 4\bn(\ln(1 + \rme^{v-r}))^{-3/2}\ln(1 + \rme^{v-r-r'})\frac{\rme^{v-r-r'} + 2\rme^{-r'}}{1 + \rme^{v-r-r"} + \rme^{-r'}}\left[ \frac{\rme^{-r'}}{1 - \rme^{-\max(r',\vep(v-r))}} - \frac{\rme^{-r-r'}}{1 - \rme^{-\max(r+r', \vep(v))}}\right]
	\end{align*}
	Note that each of the two conditions, $v\leq a_1$ (which means $\vep(v)=\vep(v-r)$) and $r+r'\geq\vep(v-r)$, implies that $\max(r+r',\vep(v))\geq \max(r',\vep(v-r)),$ and consequently $K_3^{1,\vep(v-r)}(v-r, v-r-r') > K_3^{1,\vep(v)}(v, v-r-r')$ whenever either of these conditions is met.
	
	Whenever $v-r-w\geq\vep(v-r)$, $K_3^{2,\vep(v-r)} ( v-r, w)=K_3^2(v-r, w)$ and $K_3^{2, \vep(v)} (v, w)=K_3^2(v, w)$, so the corresponding inequality is proved by Lemma \ref{lemma: K3tilt}.
	
	ii) For $w>v$ we have:
	\begin{align*}
	& K_3^{\vep(w)} (v, w) - K_3^{\vep(w)} (v-r, w)\geq \frac{1}{1 - \rme^{\max(\vep(w), w-v)}}\left(\ol f(v, w) - \ol f(v-r, w)\right)\ln(1 + \rme^w),\\
	&\text{where}\ \ol f(v, w) = (\ln(1 +\rme^v))^{-3/2}\frac{\rme^v + 2\rme^{-(w-v)}}{1 + \rme^v + \rme^{-(w-v)}}\rme^{-(w-v)}.\\
	&\text{Then}\ \frac{\partial}{\partial v}\ol f(v, w) \geq \ol f(v, w)\left( -\frac{3}{2}\frac{\rme^v}{(1+\rme^v)\ln(1 + \rme^v)} + \frac{3+\rme^v}{2+\rme^v}\right)>0.
	\end{align*}
	The last inequality then is an obvious consequence of this. 
\end{proof}
Let us remember that Theorem \ref{theorem: rDpsiSoln} deals with the difference variable $D\psi$. In order to show the existence of this difference variable, we will first prove the existence of a unique solution of the $\psi$-equation written above. As mentioned before, we will tag solutions of the regularized equation with $\vep$. 

In what follows, we will first show, given a suitable initial value $\psiep_0$, the existence of a unique Duhamel-integrated solution of (\ref{eq: Regpsievol}). Then we will go on to prove a similar existence-uniqueness result, namely Theorem \ref{theorem: rDpsiSoln}, for the difference variable. The final result of this subsection will be Proposition \ref{theorem: rDpsiDeltaApprox}. The proofs in this subsection follow closely the proofs of Lemmas \ref{lemma: mcLuEstimate}-\ref{lemma: mcLbEstimate} and Theorem \ref{theorem: DeltaSolutionsY} in the previous subsection.

\subsubsection{Regularized Solution $\psiep_t$}\label{subsubsection: psiepsoln}

Let us consider the time-evolution, according to (\ref{eq: Regpsievol}), of initial datum $\psiep_0$ in the Banach space $X$,  defined in (\ref{defn: BanachXregpsi}). Note that, since $X\subset D(\ol L_3^{\vep})$ for all $\vep>0$, Theorem \ref{theorem: L2SemigroupSolution}  guarantees the existence of a unique (in $L^2(\nu)$) solution,  $\vphi_t = S_{\vep}(t)\psiep_0$. The evolution equation is first re-written as follows:
\begin{align}\label{eq: psiepevolsplit}
\partial_t\psiep_t(v)&= - \Big[ \int_0^{\infty}\rmd r'\ K_3^{\vep(v)}(v, v-r') + \int_0^{\infty}\rmd r'\ K_3^{\vep(v+r')}(v, v+r') \Big]\psiep_t(v)\nonumber\\
&\ \ \ + \int_0^{\infty}\rmd r'\ K_3^{\vep(v)}(v, v-r')\psiep_t(v-r') + \int_0^{\infty}\rmd r'\ K_3^{\vep(v+r')}(v, v+r')\psiep_t(v+r')\nonumber\\
&= -V^{\vep}(v)\psiep_t(v) + (K_u^{\vep}\psiep_t)(v) + (K_b^{\vep}\psiep_t)(v),\\
\text{where}\ V^{\vep}(v)&= \int_0^{\infty}\rmd r'\ K_3^{\vep(v)}(v, v-r') + \int_0^{\infty}\rmd r'\ K_3^{\vep(v+r')}(v, v+r'),\nonumber
\end{align}
$K^{\vep}_u$ consists of unbounded parts of the kernel function and $K^{\vep}_b$ is $L^2$-bounded. The explicit formulae for them are given in Appendix \ref{appendix: RegEvolCompute}.

We will prove that there exists a unique Duhamel-integrated solution of the above evolution equation in $X$, for all $t>0$, given by
\begin{align}\label{eq: psiepDuhamel}
\psiep_t =&\ \rme^{-tV^{\vep}}\psiep_0 + \  \int_0^t\ \rmd s\ \rme^{-(t-s)V^{\vep}} K_u^{\vep}[\psiep_s] + \int_0^t\ \rmd s\ \rme^{-(t-s)V^{\vep}} K_b^{\vep}[\psiep_s].
\end{align}
Before proving (\ref{eq: psiepDuhamel}), we will consider the following equation, which differs from (\ref{eq: psiepDuhamel}) only in that the function $\psiep_s$ in the $L^2$-bounded part $K_b$ is replaced by the $L^2$-solution $\vphi_s:$
\begin{align}\label{eq: psiepDuhamelL2}
\psiep_t = &\ \rme^{-tV^{\vep}}\psiep_0 + \  \int_0^t\ \rmd s\ \rme^{-(t-s)V^{\vep}} K_u^{\vep}[\psiep_s] + \int_0^t\ \rmd s\ \rme^{-(t-s)V^{\vep}} K_b^{\vep}[S_{\vep}(s)\psiep_0].
\end{align}
Equation (\ref{eq: psiepDuhamelL2}) is just the Duhamel-integrated form of the evolution equation:
\begin{align}\label{eq: psiepevolsplitL2}
\partial_t\psiep_t(v) = -V^{\vep}(v)\psiep_t(v) + (K_u^{\vep}\psiep_t)(v) + K_b^{\vep}[S_{\vep}(t)\psiep_0](v).
\end{align}
It is important to note here that, since the line singularity has been regularized in the present case, the integrals appearing in the different terms of (\ref{eq: psiepevolsplit}) and (\ref{eq: psiepevolsplitL2}) are all absolutely convergent for $\psiep_t$ in $X$ and yield functions in $L^2(\nu)$ (the integrals constituting $(K_u^{\vep}\psiep_t)(v)$ and $(K_b^{\vep}\psiep_t)(v)$ are explicitly written in Appendix \ref{appendix: RegEvolCompute}). It is then easily checked that solutions of (\ref{eq: psiepDuhamel}) and (\ref{eq: psiepDuhamelL2}) in $X$, when they exist, are strongly differentiable with respect to time, and satisfy the integro-differential equations (\ref{eq: psiepevolsplit}) and (\ref{eq: psiepevolsplitL2})  respectively. These are then also classical solutions of the abstract Cauchy problems corresponding to (\ref{eq: psiepevolsplit}) and (\ref{eq: psiepevolsplitL2}) in $X$ and they satisfy these evolution equations pointwise everywhere. Recall that, by Theorem \ref{theorem: L2SemigroupSolution} we have a classical solution  of the Cauchy problem associated with (\ref{eq: psiepevolsplitL2}) in $L^2(\nu)$ as well. Our next lemma shows that given initial datum $\psiep_0\in X$, the solution of  (\ref{eq: psiepevolsplitL2}) obtained in $X$ must be equal to this $L^2$-solution almost everywhere. 
\begin{lemma}\label{lemma: L2solnpsiep}
	Given initial value $\psiep_0\in X$, suppose $\psiep_t\in X$ solves equation (\ref{eq: psiepDuhamelL2}) for all $t>0$. Then
	\begin{align*}
	\psiep_t = S_{\vep}(t)\psiep_0\ \ \nu\text{-almost evrywhere.}
	\end{align*}
\end{lemma}

\begin{proof} 
	$\psiep_t\in X\subset D( L_3^{\vep}),$ for all $t>0$. This means, given any $t>0$, we have, for all $s\in(0, t)$
	\begin{align*}
	\partial_s\left( S_{\vep}(t-s)\psiep_s\right)
	=&\ S_{\vep}(t-s)\left[ L_3^{\vep}\psiep_s + \partial_s\psiep_s\right]\\
	=&\ S_{\vep}(t-s)\left[ \left(  V^{\vep}\psiep_s - K_u^{\vep}\psiep_s - K_b^{\vep}\psiep_s  \right) - \left(  V^{\vep}\psiep_s - K_u^{\vep}\psiep_s - K_b^{\vep}\vphi_s  \right)\right]\\
	=&\ S_{\vep}(t-s)K_b^{\vep}[\vphi_s - \psiep_s],
	\end{align*}
	where $\vphi_s= S_{\vep}(s)\psiep_0=\rme^{-s L_3^{\vep}}\psiep_0$. Integrating the above, we have
	\begin{align*}
	\psiep_t - \vphi_t &= \int_0^t\rmd s\ S_{\vep}(t-s) K_b^{\vep}[\vphi_s - \psiep_s].\\
	\text{Therefore by Minkowski's integral inequality,}\ &\norm{\psiep_t - \vphi_t}_{L^2}\leq \int_0^t\rmd s\ \norm{K_b^{\vep}}_{L^2}\norm{\psiep_s - \vphi_s}_{L^2}.
	\end{align*}
	Then by Gr\"{o}nwall's inequality $\norm{\psiep_t - \vphi_t}_{L^2}= 0$, i.e., $\psiep_t = \vphi_t$ $\nu$-a.e. $\forall t>0.$
\end{proof}

\begin{lemma}\label{lemma: LuepEstimate}
	There exists a number $s_0>0$, depending on the parameters $\alpha, \mu', \gamma_0$, such that:
	\begin{align*}
	V^{\vep}\tGamma - K_u^{\vep}\tGamma \geq \frac{s_0}{\ln\vep_0^{-1}} V^{\vep}\tGamma.
	\end{align*}
\end{lemma}

\begin{proof}
	From the definition of the potential $V^{\vep}$, it is easy to see that there exist positive numbers $\ol{C}_2=\ol{C}_2(\mu', \gamma_0)>1$ and $C_2=C_2(\mu', \gamma_0)$, such that the following is true:
	\begin{align*}
	\begin{aligned}
	&\bar{n}\left[ \frac{1}{4} \left(\ln(1 + e^v)\right)^{-\frac{1}{2}}\ln \vep_0^{-1} + C_2\left(\ln(1 + e^v)\right)^{\frac{1}{2}}\right]\\
	\leq&\ V^{\vep}(v)\\
	\leq&\  \ol{C}_2\bar{n} \left[  \left(\ln(1 + e^v)\right)^{-\frac{1}{2}}\ln\vep_0^{-1} + \left(\ln(1 + e^v)\right)^{\frac{1}{2}}\right],\ \ \forall v\in\R.
	\end{aligned}
	\end{align*}
	From \ref{subappendix: CompLuep} in Appendix \ref{appendix: RegEvolCompute}  we get the following lower bound:
	\begin{align*}
	V^{\vep}(v)\tGamma(v) - (K_u^{\vep}\tGamma)(v)
	\geq \tGamma(v)\left[ \cf(v\leq 0)p_2(\alpha)\left(\ln(1 + \rme^v)\right)^{-\frac{1}{2}} + \cf(v>0)p_3 \left(\ln(1 + \rme^v)\right)^{\frac{1}{2}} \right],
	\end{align*}
	where $p_2$ and $p_3$ are positive constants bounded away from zero.
	
	It is then obvious that for some $s_0=s_0(\alpha, \mu', \gamma_0)>0$ we can write
	\begin{align*}
	V^{\vep}(v)\tGamma(v) - (K_u^{\vep}\tGamma)(v) \geq \frac{s_0}{\ln\vep_0^{-1}} V^{\vep}(v)\tGamma(v),\quad \forall v\in\R.
	\end{align*}
\end{proof}

\begin{theorem}\label{theorem: psiepSoln}
	
	Given initial datum $\psiep_0\in X,$ there exists, for all $t>0$, a unique solution $\psiep_t$ of (\ref{eq: psiepDuhamel}), such that:
	\begin{align*}
	\norm{\psiep}'=\sup_{t\in\R_{+}}\norm{\psiep_t}_X= \sup_{t\in\R_{+}}\sup_{v\in\R}\frac{|\psiep_t(v)|}{\tGamma(v)}<\infty.
	\end{align*}
\end{theorem}

\begin{proof}
	
	In order to prove the statement of this theorem, we will first show that given initial datum $\psiep_0\in X$, there exists a unique solution $\psiep_t$ of (\ref{eq: psiepDuhamelL2}), such that $\sup_{t\in\R_{+}}\norm{\psiep_t}_X<\infty$. Given such a solution, the statement of this theorem is implied by Lemma \ref{lemma: L2solnpsiep}.
	\begin{align*}
	\text{Let}\ \	F_1[\ol\psi_t] =  \int_0^t\ \rmd s\ \rme^{-(t-s)V^{\vep}} K_u^{\vep}[\psiep_s],\ \ \text{and,}\ \ 	B_t[\psiep_0] = \int_0^t\ \rmd s\ \rme^{-(t-s)V^{\vep}} K_b^{\vep}[S_{\vep}(s)\psiep_0].
	\end{align*}
	Then we can write (\ref{eq: psiepDuhamelL2}) as
	\begin{align*}
	\psiep_t=& \rme^{-tV^{\vep}}\psiep_0 + F_1[\psiep_t] + B_t[\psiep_0].
	\end{align*}
	By Lemma \ref{lemma: LuepEstimate}  and the fact that $K_u^{\vep}$ is positivity-preserving, we have:
	\begin{align*}
	|F_1[\psiep_t](v)| &\leq \norm{\psiep}'\left( 1 - \frac{s_0}{\ln\vep_0^{-1}} \right)\int_0^t\rmd s\ \partial_s\left(\rme^{-(t-s)V^{\vep}}\tGamma\right)(v)\\
	&\leq \norm{\psiep}'\left( 1 - \frac{s_0}{\ln\vep_0^{-1}} \right)\tGamma(v),\qquad\forall v\in\R.
	\end{align*}
	Thus $\norm{F_1}'<1$, and the corresponding Neumann series converges.
	
	The invertibility of $F_1$ means we can write
	\begin{align*}
	\psiep_t = ( 1 - F_1 )^{-1}\left[ \rme^{-tV^{\vep}}\psiep_0 + B_t[\psiep_0] \right],
	\end{align*}
	which is clearly the unique solution of (\ref{eq: psiepDuhamelL2})  corresponding to the initial value $\psiep_0.$
	
	Then, by Lemma \ref{lemma: L2solnpsiep}, $\psiep_t = S_{\vep}(t)\psiep_0\ $ $\nu$-a.e. for all $t\geq0$,  so that we can replace $S_{\vep}(t)\psiep_0$ by $\psiep_t$ in the $L^2$-bounded integrals and write
	\begin{align*}
	\int_0^t\ \rmd s\ \rme^{-(t-s)V^{\vep}} K_b^{\vep}[\rme^{-sL_3^{\vep}}\psiep_0]= \int_0^t\ \rmd s\ \rme^{-(t-s)V^{\vep}} K_b^{\vep}[\psiep_s].
	\end{align*}
	Thus $\psiep_t$ solves (\ref{eq: psiepDuhamel}). The uniqueness of this solution is straightforward since any such solution is also a solution of (\ref{eq: psiepDuhamelL2}).
\end{proof}

\subsubsection{Regularized Evolution Equation for the Difference Variable $D\psiep_t$ and the Existence of a Unique Solution}\label{subsubsection: Dpsiepsoln}

Our results above guarantee the existence of the difference variable $D\psiep$, where $D\psiep_t(v, r) = \psiep_t(v) - \psiep_t(v-r)$,  in the Banach space $Z$ of continuous functions defined on $\R_+\times(\R\times\R_+)$, such that
\begin{align*}
\norm {D\psiep}_{ Z} = \sup_{\substack{t>0 \\ (v, r)\in\R\times\R_+ }}\frac{|D\psiep_t(v, r)|}{\ol\Gamma(v, r)}<\infty,
\end{align*}
where
\begin{align*}
\ol\Gamma(v, r) = \left( f(v) + f(v-r) \right)\rme^{\mu\max(a, c_0v, v-r)}.
\end{align*}
Given $\psiep_0\in X$, this difference variable $D\psiep_t$ provides us with a solution of the following Cauchy problem (derived from the corresponding problem for $\psiep_t$):
\begin{align}\label{eq: DpsiepEvol}
\partial_tD\psiep_t(v, r) &= -(\tmcL^{\vep}D\psiep_t)(v, r)\nonumber\\
&=  -\int_{w<v}\rmd w\ K_3^{\vep}(v, w)D\psiep_t(v, v-w) + \int_{w>v}\rmd w\ K_3^{\vep}(v, w)D\psiep_t(w, w-v)\nonumber\\
&\ \  +\int_{w<v-r}\rmd w\ K_3^{\vep}(v-r, w)D\psiep_t(v-r, v-r-w) - \int_{w>v-r}\rmd w\ K_3^{\vep}(v-r, w)D\psiep_t(w, w-v+r).
\end{align}
Given $\psiep_0$, $\psiep_t$ is uniquely determined in $X$ by Theorem \ref{theorem: psiepSoln}, and we have the following trivial bound for $D\psiep_t$:
\begin{align}\label{eq: DPsiepBound}
|D\psiep_t(v, r)|&= \bigg| ( 1 - F_1 )^{-1}\left[ \rme^{-tV^{\vep}}\ol\psi_0 + B_t[\ol\psi_0] \right](v) - ( 1 - F_1 )^{-1}\left[ \rme^{-tV^{\vep}}\psiep_0 + B_t[\psiep_0] \right](v-r)\bigg|\nonumber\\
&\leq\ A_0(\alpha, \mu', \gamma_0, b_0, m_0)\left(\ln \vep_0^{-1}\right)\left( \norm{\psiep_0}_{X} + \norm{\psiep_0}_{L^2} \right)\ol\Gamma(v, r).
\end{align}
Observe that although our results in \ref{subsubsection: psiepsoln} guarantee the existence of $D\psiep_t$ evolving according to (\ref{eq: DpsiepEvol}), this solution may not be unique in $Z.$

We will now show that equation (\ref{eq: DpsiepEvol}) has a unique Duhamel-integrated solution in a certain subspace $Y_{\vep}$ of $ Z,$
containing functions $h$ bounded in the following norm:
\begin{align*}
\norm{h}_{Y_{\vep}} = \sup_{\substack{t>0 \\ (v, r)\in\R\times\R_+ }}\frac{|h_t(v, r)|}{\ol\Gamma_{\vep}(v, r)}<\infty,\ \ \text{where}
\end{align*}
\begin{align*}
\ol\Gamma_{\vep}(v, r) = \left( f(v) + f(v-r) \right)\rme^{\mu\max(a, c_0v, v-r)}\ol g(v, r),\quad \ol g(v, r) = \left( 1 - \rme^{-\kappa(r+\vep(v))} \right)^{\ol\gamma_0},\ \ \ol\gamma_0 = \gamma_0/2.
\end{align*}

If we can prove that the Duhamel-integrated form of (\ref{eq: DpsiepEvol}) has a unique solution in $ Z$ and also a solution in $ Y_{\vep}\subset Z,$ then we will have proved that this unique solution corresponding to the given initial datum, satisfies a certain H\"{o}lder-like condition (see the definition of $\ol g$). This result will be crucial in the final step of the proof of the main smoothing result. In our computations to prove the existence-uniqueness result in both $Z$ and $ Y_{\vep}$, we will use the following generic weight function which covers both cases:
\begin{align*}
\ol\Gamma'_{\vep} (v, r) = \left( f(v) + f(v-r) \right)\rme^{\mu\max(a, c_0v, v-r)}\thg(v, r),\quad \thg(v, r) = \left( 1 - \rme^{-\kappa(r+\vep(v))} \right)^{\tg_0},
\end{align*}
where $\tg_0\in\{0,\ol\gamma_0\}.$

The reason for using this weight is that almost all of our computations work for both $\ol\Gamma$ and $\ol\Gamma_{\vep}$. So our results are proved for the generic $\ol\Gamma'_{\vep}$. Whenever there is a difference, we point it out in the relevant computations. We denote by $\tilde{Y}_{\vep}$ the Banach space corresponding to the weight $\ol\Gamma'_{\vep}$, which means $\tilde{Y}_{\vep}$ is either $ Z$ or $ Y_{\vep}$.

The analysis of the evolution equation for $D\psiep_t$ is done along the same lines as  the analysis of the $\Delta_t$-equation in Subsection \ref{subsubsection: DeltaEqn}. Thus, here too the linear operator governing the time evolution is split into a potential, a positivity-preserving part, a bounded part and a perturbation. Let us recall (\ref{eq: rDpsievol}), where this splitting has already been effected and the equation has been cast in a form amenable to our subsequent computations:
\begin{align*}
\partial_tD\psiep_t(v, r) = -(\tmcL^{\vep}D\psiep_t)(v, r)= -(\tmcL^{\vep}_uD\psiep_t)(v, r) - (\tmcL^{\vep}_sD\psiep_t)(v, r) - \tmcK^{\vep}_b[\psiep_t](v, r).
\end{align*}
Here
\begin{align}\label{defn: tmcLs}
&(\tmcL^{\vep}_sD\psiep_t)(v, r)\nonumber \\
=& -\cf(v\geq a_1)\int_{-\infty}^{v-r}\rmd w\ \left( K_3^{1, \vep(v-r)}(v-r, w) - K_3^{1,\vep(v)}(v, w) \right) D\psiep_t(v, v-w)\cf(v-w\leq\vep(v-r))\nonumber\\
&\qquad - \int_{v-r-\vep(v-r)}^{v-r}\rmd w\ \left( K_3^{2,\vep(v-r)}(v-r, w) - K_3^{2, \vep(v)}(v, w)  \right)D\psiep_t(v, v-w),
\end{align}
where $\tr$ has already been defined in (\ref{defn: rtilde}). Also, as before, the unbounded operator $\tmcL^{\vep}_u$ can be written as
\begin{align*}
&(\tmcL^{\vep}_u D\psiep_t)(v, r)= \tmcV^{\vep}(v, r)D\psiep_t(v, r) - (\tmcK_u^{\vep}D\psiep_t)(v, r),\ \text{with}\\
&\tmcV^{\vep}(v, r) = \int_{-\infty}^{v-r}\rmd w\ K_3^{\vep(v-r)}(v-r, w) + \int_{v-r}^v\rmd w\ K_3^{\vep(v)}(v, w) + \int_{v-r}^v\rmd w\ K_3^{\vep(w)}(v-r, w) + \int_v^{\infty} K_3^{\vep(w)}(v, w).
\end{align*}
A useful bound for the $\mcK_b^{\vep}$-part (see Appendix \ref{subappendix: tmcLuep}) is obtained in a simple, straightforward way and is hence stated without proof in the following lemma.

\begin{lemma}\label{lemma: mcKbEstimate}
	Given initial datum $\psiep_0\in X$ the following bound holds:
	\begin{align*}
	\big|\tmcK_b^{\vep}\psiep_t\big|\leq \mathcal{M}_0(\ln\vep_0^{-1})\left(\norm{\psiep_0}_X +\norm{\psiep_0}_{L^2}\right)\ol\Gamma'_{\vep},
	\end{align*}
	where $\mathcal{M}_0$ depends on the parameters $\alpha, \mu, \mu', c_0, a_1, b_0, \gamma_0$ and $m_0$.
\end{lemma}

\begin{lemma}\label{lemma: tmcLsEstimate}
	There exists $q_2>0$, depending on the parameters $\gamma_0, \mu, \mu'$ and $\kappa$, such that the following bound holds:
	\begin{align*}
	|\tmcL^{\vep}_sD\psiep_t|\leq\ \frac{q_2(\kappa,\mu, \mu', \gamma_0)\vep_0}{\ln \vep_0^{-1}}\norm{D\psiep}_{\tilde{Y_{\vep}}}\tmcV^{\vep}\ol\Gamma'_{\vep}.
	\end{align*}
\end{lemma}

\begin{proof}
	There are two possible cases:
	\vspace{0.05in}
	
	\textbf{case $1$:}
	When we have $\tg_0=0$ and $\ol\Gamma'_{\vep}=\ol\Gamma$. It is quite easy to see there exist $C_3=C_3(\mu',\gamma_0)>0$ and $C_4=C_4(\mu',\gamma_0)>0$ such that:
	\begin{align*}
	&|\tmcL^{\vep}_sD\psiep(v, r)|\\
	\leq&\ \norm{D\psiep}_{Z}\ol\Gamma(v, r)\left[ C_3\vep(v-r)\left(\ln(1 + \rme^{v-r})\right)^{-\frac{1}{2}} + C_4\cf(v\geq a_1)\cf(r\leq\vep(v-r))(\vep(v-r))^2\left(\ln(1 + \rme^v)\right)^{-\frac{1}{2}}  \right].
	\end{align*}
	\textbf{case $2$:}
	When $\tilde{Y}=\ol Y_{\vep}$, we have $\tg_0=\ol\gamma_0$ and $\ol\Gamma'_{\vep}=\ol\Gamma_{\vep}$. It is quite easy to see there exist $\ol C_3=\ol C_3(\mu',\kappa, \mu, \gamma_0)>0$ and $\ol C_4=\ol C_4(\mu',\kappa, \mu, \gamma_0)>0$ such that:
	\begin{align*}
	&|\tmcL^{\vep}_sD\psiep(v, r)|\\
	\leq&\ \norm{D\psiep}_{Y_{\vep}}\ol\Gamma_{\vep}(v, r)\left[ C_3\vep(v-r)\left(\ln(1 + \rme^{v-r})\right)^{-\frac{1}{2}} + C_4\cf(v\geq a_1)\cf(r\leq\vep(v-r))(\vep(v-r))^2\left(\ln(1 + \rme^v)\right)^{-\frac{1}{2}}  \right].
	\end{align*}
	Putting these together and using the lower bound of $\tmcV^{\vep}$ ( see Lemma \ref{lemma: tmcLuEstimate}),  we obtain the required bound.
\end{proof}

\begin{lemma}\label{lemma: tmcLuEstimate}
	There exists a number $\sigma>0$, depending on the parameters $\alpha, c_0, \mu', \gamma_0$, such that:
	\begin{align*}
	\tmcL^{\vep}_u\ol\Gamma'_{\vep} &=\tmcV^{\vep}\ol\Gamma'_{\vep} - \tmcK^{\vep}_u\ol\Gamma'_{\vep} \geq \frac{\sigma}{\ln \vep_0^{-1}} \tmcV^{\vep}\ol\Gamma'_{\vep}.
	\end{align*}
\end{lemma}

\begin{proof} 
	It is easy to see that there exist positive numbers $\ol C_3=\ol C_3(\mu',\gamma_0, a_1)>1$ and $C_3=C_3(\mu',\gamma_0, a_1)$ such that the following bound holds:
	\begin{align*}
	&\bar{n}C_3(\mu',\gamma_0, a_1)\Bigg[ \left(\ln(1 + \rme^{v-r})\right)^{-\frac{1}{2}}\ln\vep_0^{-1} + \left(\ln(1 + \rme^v)\right)^{\frac{1}{2}}  \Bigg]\\
	\leq&\ \tmcV^{\vep}(v, r)\\
	\leq&\ \bar{n}\ol C_3(\mu',\gamma_0, a_1)\Bigg[ \left(\ln(1 + \rme^{v-r})\right)^{-\frac{1}{2}}\ln\vep_0^{-1} + \left(\ln(1 + \rme^v)\right)^{\frac{1}{2}}  \Bigg],\qquad \forall(v, r)\in\R\times\R_+.
	\end{align*}
	By computations outlined in Appendix \ref{appendix: RegEvolCompute} we have
	\begin{align*}
	&	\tmcL^{\vep}_u\ol\Gamma'_{\vep}(v, r)\\
	\geq&\ \bar{n}\ol\Gamma'_{\vep}(v, r)\left(\ln(1 + \rme^{v-r})\right)^{-\frac{1}{2}}\Big[ \cf(v-r<-b_0)\bar{b}_1(\alpha)\left( 1 - \rme^{-3\alpha r} \right)^2\left( 1 - \rme^{-3\alpha\max(0, r-\vep(v-r))} \right) \Big.\\
	&\qquad\Big.+ \frac{1}{2}\ln\left( 1 - \rme^{-\frac{7}{2}\max(r,\vep(v-r))}\right)^{-1}\Big]\\
	& + \cf(-m_0<v-r<m_0)\bn\ol\Gamma'_{\vep}(v, r)\Bigg[ \frac{1}{6}\left(\ln(1 + \rme^{v-r})\right)^{-\frac{1}{2}} + \cf(v-r>\ln 2)\frac{1}{9}\left(\ln(1 + \rme^{v-r})\right)^{\frac{1}{2}}\Bigg]\\
	& + \cf(\ln 2<v<m_0) \frac{\bar{n}}{2}\ol\Gamma'_{\vep}(v, r)\left(\ln(1 + \rme^v)\right)^{\frac{1}{2}}\Bigg[ 1 - \left(\frac{\ln(1 + \rme^{v-r})}{\ln(1 + \rme^{v-\min(r,\vep(v))})}\right)^2\Bigg]\\
	& + \cf(v\geq m_0) b_4(c_0)\bn\ol\Gamma'_{\vep}(v, r)\left(\ln(1 + \rme^v)\right)^{\frac{1}{2}}\\
	=&\ G_0[\ol\Gamma'_{\vep}](v, r).
	\end{align*}
	From the upper bound on $\tmcV^{\vep}_u$ it is clearly seen that there exists $\sigma>0$, depending on the parameters $\alpha, \mu', \gamma_0, a_1$ and $c_0$, such that
	\begin{align*}
	G_0[\ol\Gamma'_{\vep}](v, r)\geq \frac{\sigma}{\ln \vep_0^{-1}} \tmcV^{\vep}(v, r)\ol\Gamma'_{\vep}(v, r),\qquad\forall(v, r)\in\R\times\R_+,
	\end{align*}

	as claimed.
\end{proof}
We are now in a position to prove the main result about the solution $D\psiep_t$, namely, Theorem \ref{theorem: rDpsiSoln}.

\begin{proof}[Proof of Theorem \ref{theorem: rDpsiSoln}]
	Let us begin by recalling that, for initial datum $\psiep_0\in X$, the Duhamel-integrated form of the difference variable $D\psiep_t$ is given by equation (\ref{eq: rDpsiDuhamel}). We can write
	\begin{align*}
	D\psi_t^{\vep} &= \rme^{-t\tmcV^{\vep}}D\psi^{\vep}_0 + \int_0^t\rmd s\ \rme^{-(t-s)\tmcV^{\vep}} \tmcK^{\vep}_uD\psi^{\vep}_s - \int_0^t\rmd s\ \rme^{-(t-s)\tmcV^{\vep}} \tmcL^{\vep}_sD\psi^{\vep}_s - \int_0^t\rmd s\ \rme^{-(t-s)\tmcV^{\vep}} \tmcK^{\vep}_b[\psi^{\vep}_s]\\
	&=\rme^{-t\tmcV^{\vep}}D\psiep_0  + \ol F[D\psiep_t] - \ol B_t[\psiep_0] ,
	\end{align*}
	where	
	\begin{align*}
	\ol F[D\psiep_t] = \int_0^t\rmd s\ \rme^{-(t-s)\tmcV^{\vep}} \tmcK^{\vep}_uD\psiep_s - \int_0^t\rmd s\ \rme^{-(t-s)\tmcV^{\vep}} \tmcL^{\vep}_sD\psiep_s,\ \
	\text{and,}\ \ \ol B_t[\psiep_0]= \int_0^t\rmd s\ \rme^{-(t-s)\tmcV^{\vep}} \tmcK^{\vep}_b[\psiep_s].
	\end{align*}
	Then by Lemmas \ref{lemma: tmcLuEstimate} and \ref{lemma: tmcLsEstimate}, we arrive at the following estimate:
	\begin{align*}
	|\ol F[D\psiep_t]|\leq \norm{D\psiep}_{\tilde{Y}_{\vep}}\ol\Gamma'_{\vep}\left( 1 - \frac{\sigma}{\ln\vep_0^{-1}} + \frac{q_2\vep_0}{\ln\vep_0^{-1}}\right).
	\end{align*}
	Since $\sigma$ and $q_2$ have no dependence on $\vep_0$, we can choose the regularization parameter $\vep_0>0$ small enough so that the following is true:
	\begin{align*}
	\frac{\sigma}{\ln\vep_0^{-1}} - \frac{q_2\vep_0}{\ln\vep_0^{-1}}>0 ,
	\end{align*}
	which implies $\norm{\ol F}_{\tilde{Y}_{\vep}}<1.$ The corresponding Neumann series then converges and we have the following unique solution for the Duhamel-integrated equation:
	\begin{align*}
	D\psiep_t = (1 - \ol F)^{-1}\left( \rme^{-t\tmcV^{\vep}}D\psiep_0 - \ol B_t[\psiep_0] \right).
	\end{align*}
	Finally, it is easily seen that there exists  $A_1=A_1(\mu,\mu', c_0, a_1, b_0, m_0)>0$ such that
	\begin{align*}
	|\ol B_t[\psiep_0](v, r)|
	\leq A_1(\ln\vep_0^{-1})\ol\Gamma'_{\vep}(v, r)\left(\norm{\psiep_0}_{\ol X_0} + \norm{\psiep_0}_{L^2}\right),\qquad\forall(v, r)\in\R\times\R_+,
	\end{align*}
	which completes the proof of our assertion.
\end{proof}

Thus, given $\psiep_0\in X$, we obtain a unique solution $D\psiep\in Y_{\vep}$ of (\ref{eq: rDpsiDuhamel}), $\forall t>0,$ satisfying
\begin{align*}
|D\psiep_t(v, r)|\leq C\left( f(v) + f(v-r) \right)\rme^{\mu\max(a, c_0v, v-r)}\left( 1 - \rme^{-\kappa(r+\vep(v))} \right)^{\ol\gamma_0},\ \forall(v, r)\in\R\times\R_+  \ \text{and some } C<\infty.
\end{align*}

\subsubsection{Connecting the Variables $\Delta_t$ and $D\psiep_t$: Evolution Equation for $D_t=D\psiep_t-\Delta_t$}\label{subsubsection: Dtsoln}

Recalling the original evolution equation (\ref{eq: DeltaEvol}) for $\Delta_t$, it is easily seen that we can write
\begin{align}\label{eq: Deltaepeq}
\partial_t\Delta_t = -\mcL^{\vep}\Delta_t - \mcL_0^{\vep}\Delta_t,
\end{align}
where the part $\mcL^{\vep}$ contains $K^{\vep}$ and the part $\mcL_0^{\vep}$ has exactly the same structure as $\mcL^{\vep}$, but with $K - K^{\vep}.$ The explicit expressions are given in Appendix \ref{appendix: DvariableCompute}. Let us now look back at (\ref{eq: rDpsievol}). It is easily checked that
\begin{align*}
\tmcL^{\vep}D\psiep_t(v, r) = \mcL^{\vep}D\psiep_t(v, r),\quad\forall (v, r)\in\R\times\R_+.
\end{align*}
Let us define the following difference function, for all $t\in[0, T^*]$:
\[D_t= D\psiep_t - \Delta_t.\]
Given $\psiep_0\in X$ and $\Delta_0\in \ol Y$, such that $D\psiep_0=\Delta_0$, we now look at the Cauchy problem for the variable $D_t$. By (\ref{eq: Deltaepeq}) and (\ref{eq: rDpsievol}), we can write
\begin{align}\label{eq: Dteqn}
\partial_tD_t=-\mcL^{\vep}D_t - \mcL^{\vep}_0[\Delta_t],\qquad D_0=0.
\end{align}
Let us define the Banach space $Y_{\vep}^*$, analogous to $Y_{\vep}$, but defined only for $t\in[0, T^*]$, via the following norm:
\begin{align*}
\norm{h}_{Y_{\vep}^*} = \sup_{\substack{t\in[0, T^*] \\ (v, r)\in\R\times\R_+ }}\frac{|h_t(v, r)|}{\ol\Gamma_{\vep}(v, r)}<\infty,
\end{align*}
Since $Y\subseteq Y_{\vep}^*,$ and $D\psiep\in Y_{\vep}^*,$ it follows that $D\in Y_{\vep}^*.$

Our goal here is to show that $D_t$ can be made arbitrarily small by choosing $\vep_0$ small enough. This will be done by proving that a unique solution of the Duhamel-integrated form of (\ref{eq: Dteqn}) exists and can be made arbitrarily small as claimed.

The analysis of (\ref{eq: Dteqn}) is done in the same manner as already seen for the evolution equations of the $\Delta_t$-variable and the $D\psiep_t$-variable. Thus, we split the operator into three parts, namely i) $\mcL_u^{\vep}$, which can be written in the form of  $(\mcV^{\vep}_u - \mcK_u^{\vep})$, $\mcK_u^{\vep}$ being positivity-preserving, ii) $\mcL_{\delta}^{\vep}$, which is to be controlled by the potential as a ``perturbation'', and, iii) a bounded part $\mcL_b^{\vep}$. These operators and the expressions that appear subsequently are only slightly different from the ones we have seen already.  So we will dispense with writing out most of the terms. We will mention only those terms for which the differences from analogous expressions encountered in \ref{subsubsection: DeltaEqn} and \ref{subsubsection: Dpsiepsoln} become evident in the final estimates. We write
\begin{align*}
\partial_tD_t = -\mcL^{\vep}_uD_t - \mcL^{\vep}_{\delta}D_t - \mcL^{\vep}_bD_t - \mcL^{\vep}_0[\Delta_t]= -\mcV^{\vep}_uD_t + \mcK^{\vep}_uD_t - \mcL^{\vep}_{\delta}D_t - \mcL^{\vep}_bD_t - \mcL^{\vep}_0[\Delta_t].
\end{align*}
The Duhamel-integrated form that we will be analyzing is:
\begin{align}\label{eq: DtepDuhamel}
D_t =&\  \rme^{-t\mcV^{\vep}_u}D_0 + \  \int_0^t\ \rmd s\ \rme^{-(t-s)\mcV^{\vep}_u} \mcK^{\vep}_u[D_s] - \int_0^t\ \rmd s\ \rme^{-(t-s)\mcV^{\vep}_u} \mcL^{\vep}_b[D_s] - \int_0^t\ \rmd s\ \rme^{-(t-s)\mcV^{\vep}_u} \mcL^{\vep}_{\delta}[D_s] \nonumber\\
&\qquad -  \int_0^t\ \rmd s\ \rme^{-(t-s)\mcV^{\vep}_u} \mcL_0^{\vep}[\Delta_s].
\end{align}
The operator $\mcL_{\delta}^{\vep}$ is somewhat different from $\mcL_{\delta}$, in the sense that it contains extra terms coming from the perturbation $\tmcL^{\vep}_s$, so we write it down below.
\begin{align}\label{eq: mcLdelep}
\mcL_{\delta}^{\vep}D_t(v, r)\nonumber=&\  \Bigg[ \int_{v-r-\delta_1}^{v-r}\rmd w\ K_3^{1,\vep(v)}(v, w)D_t(v, v-w) - \int_{v-r-\delta_1}^{v-r}\rmd w\ K_3^{1,\vep(v-r)}(v, w)D_t(v-r, v-r-w)    \Bigg.\nonumber\\
& + \int_v^{v+\delta_1}\rmd w\ K_3^{\vep(w)}(v-r, w)D_t(w, w-v+r) - \int_v^{v+\delta_2}\rmd w\ K_3^{\vep(w)}(v, w)D_t(w, w-v)\nonumber\\
&\Bigg. + \int_{v-\delta_2}^v\rmd w\ K_3^{1,\vep(v)}(v, w)D_t(v, v-w) + \int_{v-r}^{v-r+\delta_1}\rmd w\ K_3^{\vep(w)}(v-r, w)D_t(w, w-v+r)\Bigg]\nonumber\\
&-\cf(v\geq a_1)\int_{-\infty}^{v-r}\rmd w\ \left( K_3^{1, \vep(v-r)}(v-r, w) - K_3^{1,\vep(v)}(v, w) \right) D_t(v, v-w)\cf(v-w\leq\vep(v-r))\nonumber\\
& - \int_{v-r-\vep(v-r)}^{v-r}\rmd w\ \left( K_3^{2,\vep(v-r)}(v-r, w) - K_3^{2, \vep(v)}(v, w)  \right)D_t(v, v-w)
\end{align}
The following lemmas are easy to obtain via simple and straightforward computations, so we state them without proof.

\begin{lemma}\label{lemma: mcLdelepEstimate}
	There exists $\ol q_1>0$, depending on the parameters $a_1$, $\gamma_0$ and $\kappa$, such that the following bound holds:
	\begin{align*}
	|\mcL^{\vep}_{\delta}D_s|\leq \frac{\ol q_1(\kappa, \gamma_0)}{M^{\ol\gamma_0}\ln \left(\min\left(M, \vep_0^{-1}\right)\right)}\norm D_{Y_{\vep}^*}\mcV^{\vep}_u\ol\Gamma_{\vep}.
	\end{align*}
\end{lemma}
\begin{lemma}\label{lemma: mcLbepEstimate}
	The bounded linear operator $\mcL_b$ satisfies:
	\begin{align*}
	| \mcL^{\vep}_b D_s |\leq A(b_0, m_0, \mu, c_0)\norm D_{Y_{\vep}^*}\ol\Gamma_{\vep},
	\end{align*}
	for some $\ol A(b_0, m_0, \mu, c_0)>0.$
\end{lemma}

We now state and prove the lemma which is the key to proving the existence of a unique solution $D_t$:

\begin{lemma}\label{lemma: mcLuepEstimate}
	There exists a number $\ol\sigma>0$, depending on the parameters $\alpha, c_0, \mu', \gamma_0$, such that:
	\begin{align*}
	\mcL^{\vep}_u\ol\Gamma_{\vep} &=\mcV^{\vep}_u\ol\Gamma_{\vep} - \mcK^{\vep}_u\ol\Gamma_{\vep} \geq \frac{\ol\sigma}{\ln\left(\min\left(M, \vep_0^{-1}\right)\right)}\mcV^{\vep}_u\ol\Gamma_{\vep}.
	\end{align*}
\end{lemma}

\begin{proof} 
	As before, it is readily seen that  there exist $\ol C_4=\ol C_4(\mu',\gamma_0, a_1)>0$ and $C_4=C_4(\mu',\gamma_0, a_1)>0$ such that the following bound holds:
	\begin{align*}
	&\bar{n}C_4(\mu',\gamma_0, a_1)\Bigg[ \left(\ln(1 + \rme^{v-r})\right)^{-\frac{1}{2}}\Big\{ \ln \left( 1 - \rme^{-\alpha r}\right)^{-\frac{4}{\gamma_0}} + \ln \min\left(M, \vep_0^{-1}\right) \Big\} + \left(\ln(1 + \rme^v)\right)^{\frac{1}{2}}  \Bigg]\\
	\leq&\ \tmcV^{\vep}_u(v, r)\\
	\leq&\ \bar{n}\ol C_4(\mu',\gamma_0, a_1)\Bigg[ \left(\ln(1 + \rme^{v-r})\right)^{-\frac{1}{2}}\Big\{ \ln \left( 1 - \rme^{-\alpha r}\right)^{-\frac{4}{\gamma_0}} + \ln \min\left(M, \vep_0^{-1}\right) \Big\} + \left(\ln(1 + \rme^v)\right)^{\frac{1}{2}}  \Bigg].
	\end{align*}
	Computations, with the like of which we have become quite familiar by now, also reveal the following:
	\begin{align*}
	&	\mcL^{\vep}_u\ol\Gamma_{\vep}(v, r)\\
	\geq&\ \bar{n}\ol\Gamma_{\vep}(v, r)\left(\ln(1 + \rme^{v-r})\right)^{-\frac{1}{2}}\Big[ \cf(v-r<-b_0)\ol b_2(\alpha)\left( 1 - \rme^{-3\alpha r} \right)^2\left( 1 - \rme^{-3\alpha\max(0, r-\vep(v-r))} \right) \Big.\\
	&\qquad\Big.+ \frac{1}{2}\ln\left( 1 - \rme^{-\frac{7}{2}\max(r+\delta_1,\vep(v-r))}\right)^{-1}\Big]\\
	& + \cf(-m_0<v-r<m_0)\bar{n}\ol\Gamma_{\vep}(v, r)\Bigg[ \frac{1}{6}\left(\ln(1 + \rme^v)\right)^{-\frac{1}{2}} + \cf(v-r>\ln 2)\frac{1}{9}\left(\ln(1 + \rme^{v-r})\right)^{\frac{1}{2}}\Bigg]\\
	& + \cf(\ln 2<v<m_0) \frac{\bar{n}}{2}\ol\Gamma_{\vep}(v, r)\left(\ln(1 + \rme^v)\right)^{\frac{1}{2}}\Bigg[ 1 - \left(\frac{\ln(1 + \rme^{v-r})}{\ln(1 + \rme^{v-\min(r,\vep(v))})}\right)^2\Bigg]\\
	& + \cf(v\geq m_0) \ol b_4(c_0)\bar{n}\ol\Gamma_{\vep}(v, r)\left(\ln(1 + \rme^v)\right)^{\frac{1}{2}}\\
	=&\ \ol G_0[\ol\Gamma_{\vep}](v, r).
	\end{align*}
	This makes it clear that there exists some $\ol\sigma>0$, depending on $\alpha, \mu', c_0,\gamma_0$, such that 
	\begin{align*}
	\ol G_0[\ol\Gamma_{\vep}](v, r)\geq \frac{\ol\sigma}{\ln\left(\min\left(M, \vep_0^{-1}\right)\right)}\tmcV^{\vep}_u(v, r)\ol\Gamma_{\vep}(v, r),\qquad\forall (v, r)\in\R\times\R_+.
	\end{align*}
\end{proof}

Before the next step let us note that henceforth we will assume $M\vep_0<1$ since $M$ will be  fixed while $\vep_0$ will eventually be taken to zero. We note the following estimate for $\mcL_0^{\vep}[\Delta_t]$.

\begin{lemma}\label{lemma: mcL0epEstimate}
	There exist some positive constant $C'<\infty$, depending $\kappa,\gamma_0$ and $\alpha$, and, $p=p(\gamma_0)>0$ such that the following bound is true:
	\begin{align*}
	\Big|\mcL_0^{\vep}[\Delta_t](v, r)\Big|\leq C'(\kappa,\gamma_0,\alpha)(M\vep_0)^p\norm\Delta_Y\ol\Gamma_{\vep}(v, r) \mcV_u^{\vep}(v, r),\qquad\forall(v, r)\in\R\times\R_+.
	\end{align*}
\end{lemma}

\begin{proof}
	We note that in many of the terms in $\mcL_0^{\vep}[\Delta]$, an extra ``smallness'' comes from the fact that the integrand contains a term like $\Delta(. , r')$, where $r'$ is the variable of integration. However, there are some terms for which this is not true. We will take one such term below and show what makes such a term still ``small'' (compared to the potential). We keep in mind that now we can assume $M\vep_0<1$. 
	
	Consider the term \[\cf(\delta_1<\vep(v-r))\int_{\delta_1}^{\vep(v-r)}\rmd r'\ K_3^1(v-r, v-r-r')\frac{\rme^{-r'} - \rme^{-\vep(v-r)}}{1 - \rme^{\vep(v-r)}}\Delta_t(v, r).\]
	\begin{align*}
	\delta_1<\vep(v-r)\implies \frac{1}{M}\left( 1 - \rme^{-\alpha r}\right)^{4/\gamma_0}<\vep_0&\implies r<\frac{1}{\alpha}\ln\left( 1 - (M\vep_0)^{\gamma_0/4} \right)^{-1}\\
	&\implies \left( 1 - \rme^{-\kappa r}\right)< 1 - \left( 1 - (M\vep_0)^{\gamma_0/4} \right)^{\kappa/\alpha}\\
	&\implies \left( 1 - \rme^{-\kappa r}\right)^{\gamma_0-\ol\gamma_0}< \left(\frac{\kappa}{\alpha}\right)^{\gamma_0 - \ol \gamma_0}(M\vep_0)^{\frac{\gamma_0}{4}(\gamma_0 - \ol\gamma_0)}.
	\end{align*}
	Then we can write the following:
	\begin{align*}
	&\Bigg|\cf(\delta_1<\vep(v-r))\int_{\delta_1}^{\vep(v-r)}\rmd r'\ K_3^1(v-r, v-r-r')\frac{\rme^{-r'} - \rme^{-\vep(v-r)}}{1 - \rme^{\vep(v-r)}}\Delta_t(v, r)\Bigg|\\
	\leq\ & \cf(\delta_1<\vep(v-r))\norm\Delta_Y\left( \ln( 1 + \rme^{v-r})\right)^{-\frac{1}{2}}\rme^{\mu\max(a, c_0v, v-r)}\left( f(v-r) + f(v) \right)\left( 1 - \rme^{\kappa r}\right)^{\gamma_0}\int_{\delta_1}^{\vep(v-r)}\rmd r'\ \frac{\rme^{-r'}}{1 - \rme^{-r'}}\\
	\leq\ & \cf(\delta_1<\vep(v-r))\norm\Delta_Y\ol\Gamma_{\vep}(v, r)\left( \ln( 1 + \rme^{v-r})\right)^{-\frac{1}{2}}\left( 1 - \rme^{\kappa r}\right)^{\gamma_0-\ol\gamma_0}\ln\left(\frac{\vep_0}{\frac{1}{M}(1 - \rme^{-\alpha r})^{4/\gamma_0}}\right)\\
	\leq\ & \cf(\delta_1<\vep(v-r))\norm\Delta_Y\ol\Gamma_{\vep}(v, r)\left( \ln( 1 + \rme^{v-r})\right)^{-\frac{1}{2}}\left( 1 - \rme^{\kappa r}\right)^{\gamma_0-\ol\gamma_0}\frac{4}{\gamma_0}\ln\left(1 - \rme^{-\alpha r}\right)^{-1}\\
	\leq\ & \cf(\delta_1<\vep(v-r))\norm\Delta_Y\ol\Gamma_{\vep}(v, r)\left( \ln( 1 + \rme^{v-r})\right)^{-\frac{1}{2}}\frac{4}{\gamma_0}\left(\frac{\kappa}{\alpha}\right)^{\frac{1}{4}\gamma_0}\left(M\vep_0\right)^{\frac{1}{16}\gamma_0}\Big[ \left( 1 - \rme^{\kappa r}\right)^{\frac{1}{4}\gamma_0} \ln\left(1 - \rme^{-\alpha r}\right)^{-1} \Big],
	\end{align*}
	where we have used the fact that $\ol\gamma_0 = \gamma_0/2.$
	Following similar computations it is easy to see that there exist $p>0$ depending on $\gamma_0$ and $C'=C'(\kappa, \gamma_0,\alpha)>0$, such that
	\begin{align*}
	\Big|\mcL_0^{\vep}[\Delta_t](v, r)\Big|\leq C'(\kappa,\gamma_0,\alpha)(M\vep_0)^p\norm\Delta_Y\ol\Gamma_{\vep}(v, r) \mcV_u^{\vep}(v, r),\qquad(v, r)\in\R\times\R_+.
	\end{align*}
\end{proof}

Finally, we are in a position to prove Proposition \ref{theorem: rDpsiDeltaApprox}.

\begin{proof}[Proof of Proposition \ref{theorem: rDpsiDeltaApprox}]
	We first write the Duhamel-integrated equation (\ref{eq: DtepDuhamel}) as 
	\begin{align*}
	D_t &= \rme^{-t\mcV^{\vep}_u}D_0 + \tF[D_t] -  \int_0^t\ \rmd s\ \rme^{-(t-s)\mcV^{\vep}_u} \mcL_0^{\vep}[\Delta_s],\quad\text{where}\\
	\ \  &\tF[D_t] = \int_0^t\ \rmd s\ \rme^{-(t-s)\mcV^{\vep}_u} \mcK^{\vep}_u[D_s] - \int_0^t\ \rmd s\ \rme^{-(t-s)\mcV^{\vep}_u} \mcL^{\vep}_b[D_s] - \int_0^t\ \rmd s\ \rme^{-(t-s)\mcV^{\vep}_u} \mcL^{\vep}_{\delta}[D_s].
	\end{align*}
	Then, as before, Lemmas \ref{lemma: mcLuepEstimate}, \ref{lemma: mcLdelepEstimate} and \ref{lemma: mcLbepEstimate} mean that, for all $ t\leq T^*$:
	\begin{align}\label{defn: tFbound}
	|\tF D_t|\leq \norm D_{Y_{\vep}^*}\ol\Gamma_{\vep} \left( 1 - \frac{\ol\sigma}{\ln \left(\min\left(M, \vep_0^{-1}\right)\right)} + \ol AT^* + \frac{\ol q_1}{M^{\ol \gamma_0}\ln \left(\min\left(M, \vep_0^{-1}\right)\right)}  \right).
	\end{align}
	Obviously then, we can choose $M<\infty$ large enough and $\vep_0^{-1}>M$, so that  we have, for some $T^*$,
	\begin{align*}
	0< T^* < \frac{1}{\ol A\ln M}\left( \ol\sigma - \frac{\ol q_1}{M^{\gamma_0}} \right).
	\end{align*}
	This guarantees that $\norm\tF_{Y_{\vep}^*}<1.$
	The convergence of the relevant Neumann series then  gives us the following unique solution of (\ref{eq: DtepDuhamel}):
	\begin{align*}
	D_t= - \left( 1 - \tF \right)^{-1} \int_0^t\ \rmd s\ \rme^{-(t-s)\mcV^{\vep}_u} \mcL_0^{\vep}[\Delta_s],\ \forall 0<t\leq T^*,
	\end{align*}
	since the initial conditions are such that $D_0=0.$
	
	Finally, from lemma \ref{lemma: mcL0epEstimate} it is clear that for some $Q>0,$ $p>0$ depending on the parameters of the weight functions, we have
	
	\[|D_t|\leq Q\ln\left(\min(M,\vep_0^{-1})\right)\left(M\vep_0\right)^p\norm\Delta_Y\ol\Gamma_{\vep},\qquad\quad\forall t\in[0, T^*].\]
\end{proof}

Note that the time $T^*$ above may be different from the time that appears in Theorem \ref{theorem: DeltaSolutionsY}, although the same symbol has been used. This does not cause any problem because it is enough to choose the minimum of these two times, name it our new $T^*$, and understand that this minimum is the $T^*$ that appears in our main result Theorem \ref{theorem: MainResultPaper}  as well as Theorem \ref{theorem: L2limitfunction}.



\appendix

\section*{Appendices:}
\renewcommand{\thesection}{\Alph{section}}

\section{Linearization of the Three-wave Collision Operator $\mathcal{C}_3$}\label{appendix: Linearization}

In what follows, we describe briefly how the expressions for the kernel functions appearing in (\ref{defn: OrigL2Operator}) and (\ref{eq: OrigDiffevol}) are obtained from the linearization of the three waves collision operator $\mathcal{C}_3.$ The linearized operator is obtained first in terms of energy variables. We reserve the letters $x$ and $y$ ($x, y$ are in $\R_+$) for these variables, so that the kernel function is written, by a slight abuse of notation, as $K_3(x, y)$ in the flat metric (in the weighted $L^2$ space we use the notation $\ol K_3(x, y)$). Then we change variables to $u$ and $v$, which take values in $\R$, and obtain the functions $\ol K_3(u, v)$ and $K_3(u, v)$ appearing in (\ref{defn: OrigL2Operator}) and (\ref{eq: OrigDiffevol}) respectively. We also show that our linearized operator is identical, upto a numerical factor, to the operator considered in \cite{escobedo:TM} and \cite{escobedo:FM}.

As mentioned in the introduction, the operator $\mathcal{C}_3$ is linearized around the equilibrium distribution $\fbec$, by considering perturbations $f_{per}$ of the form $f_{per}(x, t)= x\fbec(x)\ftbec(x)\psi_t(x).$ Considering the action of $\mathcal{C}_3$ on $\ftot=\fbec + f_{per}$ and keeping in mind that $\mathcal{C}_3[\fbec]=0,$ the following evolution equation for the variable $\psi$ in the linearized model is easily obtained:
\begin{align}\label{eq: LinearizeH}
&\partial_t\psi_t(x) = -L_3\psi_t(x)\nonumber\\
&\ = -\frac{2\bn}{\sqrt{x}\gbec(x)\gtbec(x)}\left[ \int_0^x \rmd y\ x\ftbec(x)\fbec(y)\fbec(x-y)\Big( x\psi_t(x) - y\psi_t(y) - (x-y)\psi_t(x-y)\Big)\right.\nonumber\\
&\ \quad\left. + 2\int_x^{\infty}\rmd y\ x\ftbec(x)\fbec(y)\ftbec(y-x)\Big( x\psi_t(x) - y\psi_t(y) + (y-x)\psi_t(y-x) \Big) \right]\nonumber\\
&\ = -4\bn(\hw(x))^2\left[ \int_0^{\infty}\rmd y\ H(\min(x, y), |x-y|)\frac{(\psi_t(x) - \psi_t(y))}{|x-y|} - \int_0^{\infty}\rmd y\ H(x, y)\frac{(\psi_t(x) - \psi_t(y))}{x+y}\right],
\end{align}
where \[ H(x, y) = x\fbec(x)y\fbec(y)(x+y)\ftbec(x+y), \text{ and } \hw(x)= \left(\sqrt{x}\gbec(x)\gtbec(x)\right)^{-\frac{1}{2}}.\]
It is also easy to see that one can then write
\begin{align}\label{defn: L3xvariable}
L_3\psi_t(x) = \int_0^{\infty}\rmd y\ \tilde{H}(x, y)\left(\psi_t(x) - \psi_t(y)\right),
\end{align}
where
\[ \tilde{H}(x, y) = 4\bn\hw(x)^2 x y \rme^{-\min(x, y)}\fbec(|x-y|)\ftbec\left(\max(x, y)\right)\ftbec(x+y)\left( 1 + \rme^{-\max(x, y)} \right), \]
and we have used the following identity:
\[ 1 - \frac{\rme^{\min(x, y)}\fbec(x+y)}{\fbec(|x-y|)} = \frac{\fbec(x+y)}{\fbec(\min(x, y))}\left( 1 + \rme^{\max(x, y)}\right).\]

From formula (\ref{defn: L3xvariable}) it is evident that the operator $L_3$ has a non-negative, symmetric sesquilinear form on a dense domain of a weighted $L^2$-space characterized by the weight $\nu(\rmd x)=\hw(x)^{-2}\rmd x,$ and that $L_3$ is conveniently defined as an unbounded operator in $L^2(\nu)$ as follows:
\begin{align}\label{defn: L2xvariable}
L_3\psi_t(x) &= \int_0^{\infty}\nu(\rmd y)\ol K_3(x, y)\left(\psi_t(x) - \psi_t(y)\right),\\
\text{where } \ol K_3(x, y)&=  4\bn\hw(x)^2\hw(y)^2 x y \rme^{-\min(x, y)}\fbec(|x-y|)\ftbec\left(\max(x, y)\right)\ftbec(x+y)\left( 1 + \rme^{-\max(x, y)} \right)\nonumber.
\end{align}
This yields (\ref{eq:L3fulldef}).

We now change variables from $x, y$ to $u, v$, where $u=\ln(\rme^x -1)$ and $v=\ln(\rme^y -1).$ For the convenience of writing, we indulge in a slight abuse of notation and continue to use the same symbols $\nu$ and $\psi$ as before for the $L^2$-weight and in the perturbed density respectively. Then it is obvious that in terms of the new variable the weight in our $L^2$-space can be written as $\nu(\rmd v) = \rme^{-v}\left(\ln ( 1 + \rme^v)\right)^{\frac{5}{2}}\rmd v$ and that (\ref{defn: OrigL2Operator}) is just (\ref{defn: L2xvariable}) written in terms of the new variables in $\R.$

In \cite{escobedo:TM} and \cite{escobedo:FM} the same linearized model has been considered. Note that in these papers the model is described in terms of the variable $|p|/\sqrt{2}$, $p$ being the momentum, while we have used the energy variable here. To establish the equivalence, let us agree to use letters $x, y$ with primes for these new variables, so that  $x$ and $y$ in Equations (1.9) and (1.12) in \cite{escobedo:FM} are now replaced by $x'$ and $y'$. Then Equations (1.9) and (1.12) in \cite{escobedo:FM}  read:
\begin{align*}
\frac{\partial u}{\partial t}&= p_c(t)\int_0^{\infty}\rmd y'\ \left(u(t, y') - u(t, x')\right) M(x', y'),\\
\text{and }M(x', y') &= \left(\frac{1}{\sinh|x'^2 - y'^2|} - \frac{1}{\sinh(x'^2 + y'^2)}\right) \frac{y'^3\sinh x'^2}{x'^3\sinh y'^2},
\end{align*}
respectively. The variable $x'$ is related to our energy variable $x$ via $x = 2x'^2.$ Changing variables $x'\to x$, retaining the same symbol for the perturbation $u$, we get
\begin{align*}
\frac{\partial u}{\partial t}= \frac{1}{4}p_c(t)\int_0^{\infty}\rmd y\ \left(u(t, y) - u(t, x)\right) \ol M(x, y),\ 
\text{where }\ol M(x, y) = \frac{M(x', y')}{y'}\Bigg|_{x=2x'^2, y=2y'^2}.
\end{align*}
Using the fact that $\sinh x = \frac{1}{2}\frac{\exp(-x)}{\fbec(2x)}$ , we easily see that
\begin{align*}
\ol M(x, y) &= \frac{2\sqrt{2}}{\sqrt{x}}\frac{y}{x}\frac{\fbec(y)}{\fbec(x)}\rme^{-\frac{x-y}{2}}\left[\rme^{\frac{1}{2}|x-y|}\fbec(|x-y|) - \rme^{\frac{1}{2}(x+y)}\fbec(x+y)\right]\\
&= 2\sqrt{2}\hw(x)^2 xy\fbec(x)\fbec(y)\rme^{x+y}\rme^{-\min(x, y)}\left( \fbec(|x-y|) - \rme^{\min(x, y)}\fbec(x+y)\right)\\
&=\frac{\sqrt{2}}{2\bn}\tilde{H}(x, y).
\end{align*}

\section{Parameters $\alpha, c_0$ and $\mu$ appearing in the Weight Functions for the Banach Spaces}\label{appendix: GammaParameters}

As is evident from the arguments used in the proof of Theorem \ref{theorem: DeltaSolutionsY}, the choice of the weight function $\Gamma$ is motivated by the requirement that $\mcL_u[\Gamma_t](v, r)= \mcV_u(v, r)\Gamma_t(v, r) - \mcK_u[\Gamma_t](v, r)$ mimics the asymptotic behavior of $\mcV_u(v, r)\Gamma_t(v, r)$, so that Lemma \ref{lemma: mcLuEstimate} is true.  Thus we want to cut out as much as possible of the bounded part of $\mcL_{pp}$ (see (\ref{eq: DeltaEqSplit})), and try to ensure the asymptotic dominance of $\mcK_u\Gamma$ by $\mcV_u\Gamma$.

We know that the potential $\mcV_u(v, r)$ grows exponentially at $(-\infty)$, has a line singularity at $r=0 $ and grows like the square-root function at $+\infty$, somewhat like 
\begin{align*}
V(v, r) =  \bar{n}\left(\rme^{-\frac{1}{2}(v-r)}\ln(2 + 1/r) + \sqrt{\max(1,v)} \right).
\end{align*}
The line singularity appearing in the kernel functions contributes the factor $g_t(v, r)= ( 1 - \rme^{-\kappa r})^{\gamma_t(v, r)}$ to the norm $\Gamma_t(v, r).$ The nature of the dependence of the smoothing exponent $\gamma$ on $(v-r)$ is determined by the way $\dot{\Gamma}_t$ has to behave. The other factors in the norm are determined by the ``desired'' behavior of $\mcL_u[\Gamma_t](v, r)$ for asymptotically large positive and negative values of the argument $v$. These considerations are simple to understand when we look at the behavior of $K_3(v, v-r')$ and $K_3(v, v+r')$ in the regions $v\ll 0$ and $v\gg0$ respectively, as we explain below.

\vspace{0.1in}
\textbf{a) Behavior when $v\to -\infty$:}
\vspace{0.1in}

When $v\ll 0$, we need to consider the case when the point singularity becomes dominant. To see how it affects the behavior of our linearized operator (note that, since we are not looking at the line singularity now, it makes sense to think about this in the space of $\psi$-variables instead of the differences), it is enough to consider the following toy model:
\begin{align*}
T(v, v-r') &= \rme^{-\frac{1}{2}v}\rme^{-2r'}\\
T(v, v+r') &= \rme^{-\frac{1}{2}v}\rme^{-r'}.
\end{align*}
The important thing to note here is the relative tilt (extra factor of $\rme^{-r'}$)  on the left side of the diagonal with respect to the right side. A quick back-of-the-envelope calculation then tells us the following: the relative tilt in the kernel function means that the weight function for the $\psi$-variable should behave as $\rme^{-\alpha v}$ asymptotically for some $\alpha>0$,  which in turn implies that, the asymptotic behavior of our $\Gamma_t(v, r)$ should be like $\rme^{-\alpha(v-r)}$. A cheap upper bound is already imposed on $\alpha$ by the conservation of energy in the BEC problem,  so $0<\alpha<1/2$.

Coming back to our original linear operator $\mcL_u$ in the space of differences now, we observe that the choice of a negative $\alpha$ generates some extra negative terms coming from the left side of the diagonal and these need to be controlled. An easy way out is to add a sub-dominant factor $\rme^{-\alpha v}$ to $\Gamma_t(v, r).$ This generates some extra positive terms. Finally, these singular factors are needed only when we are dealing with large, negative values of the argument, so we include a cut-off in the relevant terms, e.g. we use $\max(e^{-\alpha(v-r)}, (\ln 2)^{\alpha})$ instead of the bare $\rme^{-\alpha(v-r)}.$ The inclusion of this cut-off makes our computations simpler for large positive values of the arguments.

\vspace{0.1in}
\textbf{b) Behavior when $v\to +\infty$:}
\vspace{0.1in}

At $+\infty$ the behavior is dominated by the part $K_3^2$ of the kernel function (this is in fact the motivation behind the splitting of $K_3(v, v-r')$ into $K_3^1(v, v-r')$ and $K_3^2(v, v-r')$). Note that in this case there is no relative tilt unlike for the region of large, negative values of the arguments. 
We have already mentioned that for asymptotically large values of $v$, the potential grows like $\sqrt{v}$. Actually, the relevant behavior in this case comes from the last two terms appearing in equation (\ref{defn: mcVu}) defining $\mcV_u$, i.e., the terms $\int_{v-r}^v\rmd w\ K_3^2(v, w)$ and $\int_{-\infty}^{v-r}\rmd w\ K_3^2(v-r, w)$. The latter term becomes important when $v-r\gg0$, especially when $v$ and $v-r$ are comparable.

Let us then begin by comparing the contributions from $K_3^2(v, v-r')$ and $K_3^2(v-r, v-r-r')$ to the potential term for $v-r\gg 0$.
\begin{align*}
&V_1(v, r) = \int_0^r dr'\ K_3^2(v, v-r')\sim 4\bar{n}\sqrt{v}\left(1 - \left(\frac{v-r}{v}\right)^2\right),\\
& V_2(v, r) = \int_0^{\infty} dr'\ K_3^2(v-r, v-r-r')\sim 4\bar{n}\sqrt{v-r}.\\
\end{align*}

Now define $x=\frac{v-r}{v},$ and consider the quantity $V_2 - V_1.$
Let  $g(x) = x^2 + \sqrt{x} -1,$ so that $V_2 - V_1 =\sqrt{v} g(x).$ Note that $0<x<1,\ $for all $v-r>0$.
Evidently, $g'(x)>0.$ Thus, for all $x\in(0, 1),\ g(x)$ has one and  only one zero. Let us call it $\ol c_0$. Then $g(\ol c_0)=0,$ and it is easy to ascertain that $\ol c_0\approx0.525.$
Thus $x<\ol c_0$ implies $V_1> V_2,$ and $x>\ol c_0$ implies $V_2>V_1$.

%
A similar growth for $\mcL_u$ can only come from the part $I_3[\Gamma_t](v, r)$, defined in (\ref{defn: I3defn}). We obtain this behavior is by making sure $\mcL_u[\Gamma_t](v, r)$ inherits as much as possible of the potential term. This  motivates us to look for a norm which renders the following terms bounded in certain regions for large values of $(v-r)$ and $v$:
\begin{align*}
\int_0^{\infty} dr'\ K_3^2(v-r, v-r-r')\Gamma_t(v, r+r')\qquad\text{and}\qquad\int_0^r dr'\ K_3^2(v, v-r')\Gamma_t(v-r', r-r').
\end{align*}
Note that
\begin{align*}
&\forall \ \ v-r-r'\gg0,\  K_3^2(v-r, v-r-r')\sim 4\bar{n}(v-r)^{-\frac{3}{2}}(v-r-r'),\\
&\text{and, }\forall \ v-r'\gg0,\  K_3^2(v, v-r')\sim 4\bar{n} v^{-\frac{3}{2}}(v-r').
\end{align*}
The corresponding integrals can be rendered bounded if we choose the weight function $\Gamma_t$ in such a way that $\Gamma_t(v, r+r')$ grows exponentially as  $ v-r-r'\to(+\infty)$, and $\Gamma_t(v-r', r-r')$ has a similar exponential growth as $v-r'\to(+\infty).$
Clearly, we must include a factor $h(v, r)$ in the weight function that behaves as follows:
\begin{align*}
h(v, r) = \exp\left[\mu\max (v-r, c_0v)\right],\ \text{for some suitably chosen } c_0\in(0, 1).
\end{align*}
Then our little computation above, comparing $V_1$ and $V_2$, gives us an idea about how the relative weight $c_0$ can be chosen. We choose $c_0=0.52$.

We obtain upper bounds on permissible values of $\alpha$ and $\mu$ in the course of our computations described in Appendix \ref{appendix: DeltaCompute}. A brief description of how these bounds are arrived at is given below.
\begin{itemize}
	\item Look at the combination of terms denoted by Comb.$3$ in Appendix \ref{appendix: DeltaCompute}. The positivity of this combination rests on the fact that $\alpha$ has a certain upper bound less than $\frac{1}{6}$. Essentially we do the following:
	\begin{align*}
	&\text{The term } \left[-(K_3^2(v-r, v-r-r') - K_3^2(v, v-r-r'))(f(v-r-r') - f(v-r))\right]\ \text{is controlled by}\\
	&\text{the term}\ \ol{K}_3^1(v, v+r')f(v+r').\\
	&\text{We are in the region $v-r<-m_0$, $v<-b_0$}. \text{ Thus it makes sense to look at the toy model without}\\
	&\text{the line singularity. Then the lower bound}\\
	&\ol{K}_3^1(v, v+r')f(v+r')>(K_3^2(v-r, v-r-r') - K_3^2(v, v-r-r'))(f(v-r-r') - f(v-r))\\ &\text{holds true for all}\  -v-b_0>r'>r,\ \text{for all}\ r>0,\\
	&\text{if}\ \ \rme^{-(\frac{1}{2}+\alpha)v} e^{-(1 + \alpha)r} > \rme^{-(\frac{1}{2}+\alpha)(v-r)} e^{-(2-\alpha)r},\ \text{for all}\ r>0,\
	\text{or equivalently, if}\ \alpha<\frac{1}{6}.
	\end{align*}
	\item Our computations for the lower bound for $I_2[\Gamma_t](v, r)$ leads to upper bounds on $\mu$ and $\mu c_0$. We choose $\mu c_0\leq\frac{1}{4}$, so that the following is true:
	\[\rme^{\mu\max(\tilde{a}, v-r, c_0(v-r'))} - \rme^{-\frac{1}{2}r'}\rme^{\mu\max(\tilde{a}, v-r, c_0(v+r'))}>0.\]
	
	The other upper bound chosen for computational convenience while controlling $I_2[\Gamma_t](v, r)$ is the following:
	\[\mu < \frac{1}{2} - \frac{3}{8}\alpha.\]
	Finally, our computation for an estimate on $I_4[\Gamma_t](v, r)$ relies on the choice $\mu c_0>\alpha.$
\end{itemize}

\section{Properties of the H\"{o}lder-type Conditions}\label{appendix: g}

\subsection {The Time-dependent H\"{o}lder-type Condition}\label{subsection: g_t}

Let us recall the definition (\ref{defn: gt}) of the time-dependent H\"{o}lder-type condition $g_t(v, r)$:
\begin{align*}
g_t(v, r)&=\left( 1 - \rme^{-\kappa r}\right)^{\gamma_t(v,r)},\ \gamma_t(v,r) = \gamma_t(v-r)= \gamma_0 + \ol a(t)\frac{1}{1+\rme^{\beta (v-r)}},\nonumber\\
\ol a(t) &= \frac{1}{8}\min(1,\bn)\frac{t}{1+t},\ \qquad\kappa\geq 7,\  0<\gamma_0\leq 1/8,\ 1\leq\beta\leq\kappa/4.
\end{align*}
We will now prove a few bounds for certain combinations of  $g_t$, which are used in the computations detailed in Appendix \ref{appendix: DeltaCompute} and which will also elucidate the above bounds on $\beta$, $\gamma_0$ and $\kappa$.

\begin{lemma}\label{lemma: gtprop1}
	For all $ 0<r'\leq r$, the following are true:
	\begin{align*}
	&i)\ g_t(v, r) - g_t(v, r-r')\geq 0,\ \ \ g_t(v, r+r') - g_t(v, r)\geq 0,\\
	&ii)\ g_t(v, r) - g_t(v-r', r-r')\geq 0,\ \ \ g_t(v+r', r+r') - g_t(v, r)\geq 0.
	\end{align*}
\end{lemma}

\begin{proof}
	For part i) it is enough to notice that, for all $r''\geq 0,$
	\begin{align*}
	\frac{\partial g_t(v, r'')}{\partial r''} &\geq (\kappa - \beta)g_t(v, r'')\frac{\ol a(t)}{1 + \rme^{\beta(v-r'')}}\frac{\rme^{-\kappa r''}}{1 - \rme^{-\kappa r''}}\\
	&> 0,\ \ \ \forall \kappa>\beta,
	\end{align*}
	so that $g_t(v, r'')$ is an increasing function in the second variable $r''$, which implies the inequalities in part i). 
	
	The inequalities in part ii) are obvious from the definition of the H\"{o}lder condition $g_t(v, r)$, since in this case the H\"{o}lder exponent is $\gamma_t(v, r)$ for all the terms involved.
\end{proof}

\begin{lemma}\label{lemma: gtprop2}
	For all $ r'>r$, for all $(v, r)\in\R\times\R_+,$ the following inequality holds:
	
	\begin{align*}
	g_t(v, r) + g_t(v-r+r', r') - g_t(v+r', r+r')> \left( 2 - (1 + \rme^{-\kappa r})^{\gamma_t(v,r)} \right) g_t(v, r).
	\end{align*}
\end{lemma}

\begin{proof} Let us define a function $b_0$ as follows: $b_0(v, r, r')= g_t(v, r) + g_t(v-r+r', r') - g_t(v+r', r+r')$.
	\begin{align*}
	&\text{Then}\ \frac{\partial}{\partial r'}b_0 = \kappa \gamma_t(v, r)\bigg[\rme^{-\kappa r'}\left( 1 - \rme^{-\kappa r'}\right)^{\gamma_t(v, r) -1} - \rme^{-\kappa (r+r')}\left( 1 - \rme^{-\kappa (r+r')}\right)^{\gamma_t(v, r) -1}\bigg]>0.
	\end{align*}
	Then we can write the following:
	\begin{align*}
	b_0(v, r, r')> b_0(v, r, r) &=  2 g_t(v, r) - g_t(v+r, 2r)\\
	& =2 \left( 1 - \rme^{-\kappa r}\right)^{\gamma_t(v, r)} - \left( 1 - \rme^{-2\kappa r} \right)^{\gamma_t(v, r)}\\
	& = \bigg( 2 - \left( 1 + \rme^{-\kappa r}\right)^{\gamma_t(v, r)} \bigg)g_t(v, r).
	\end{align*}
\end{proof}

\begin{lemma}\label{lemma: gtprop3}
	For all $ r'>r$, for all $(v, r)\in\R\times\R_+,$ the following inequality holds:
	\begin{align*}
	g_t(v, r) + g_t(v, r') - g_t(v, r+r')\geq \left( 1 - \frac{\rme^{-\kappa r'}}{ (1 + \rme^{-\kappa r})^{1 - \gamma_t(v, r)}} \right) g_t(v, r).
	\end{align*}
\end{lemma}

\begin{proof}
	Note that:
	\begin{align*}
	g_t(v, r+r') - g_t(v, r') &= \left( 1 - \rme^{-\kappa(r+r')}\right)^{\gamma_t(v, r+r')} - \left( 1 - \rme^{-\kappa r'}\right)^{\gamma_t(v, r')}\\
	&\leq \left( 1 - \rme^{-\kappa(r+r')}\right)^{\gamma_t(v, r')} - \left( 1 - \rme^{-\kappa r'}\right)^{\gamma_t(v, r')}\\
	&\leq \left( 1 - \rme^{-\kappa(r+r')}\right)^{\gamma_t(v, r')} \left[ 1 - \left( 1 - \rme^{-\kappa r'}\frac{1 - \rme^{-\kappa r}}{1 - \rme^{-\kappa (r+r')}}\right)^{\gamma_t(v, r')}\right]\\
	&\leq \left( 1 - \rme^{-\kappa(r+r')}\right)^{\gamma_t(v, r)-1} \rme^{-\kappa r'}(1 - \rme^{-\kappa r})\\
	& \leq \left( 1 - \rme^{-2\kappa r}\right)^{\gamma_t(v, r)-1} \rme^{-\kappa r'}(1 - \rme^{-\kappa r})\\
	& \leq \left( 1 - \rme^{-\kappa r} \right)^{\gamma_t(v, r)}\frac{\rme^{-\kappa r'}}{\left( 1 + \rme^{-\kappa r} \right)^{1 - \gamma_t(v, r)}},
	\end{align*}
	which leads to the bound stated in the lemma.
\end{proof}

As an aside, let us observe that, for all $ r'>r$: 
\begin{align*}
\frac{\rme^{-\kappa r'}}{\left( 1 + \rme^{-\kappa r} \right)^{1 - \gamma_t(v, r)}}\leq f_1(r) = \frac{\rme^{-\kappa r}}{\left( 1 + \rme^{-\kappa r} \right)^{3/4}}.
\end{align*}
Then it is easy to check that $f_1$ is a decreasing function, so that the following estimate holds:
\begin{align*}
&\frac{\rme^{-\kappa r'}}{\left( 1 + \rme^{-\kappa r} \right)^{1 - \gamma_t(v, r)}}\leq f_1(r) < 0.6,\\
\text{and consequently,}\ \ & g_t(v, r) + g_t(v, r') - g_t(v, r+r')> 0.4 g_t(v, r).
\end{align*}
\begin{lemma}\label{lemma: gtprop4}
	For all $0\leq r'<r $, for all $(v, r)\in\R\times\R_+,$ the following inequality holds:
	\begin{align*}
	2g_t(v, r) - g_t(v, r-r') - g_t(v, r+r')\geq 0.
	\end{align*}
\end{lemma}

\begin{proof} Let $b_2(v, r, r') = -g_t(v, r-r') - g_t(v, r+r').$
	\begin{align*}
	&\text{Then}\ \frac{\partial b_2}{\partial r'} = \gamma_0\kappa\left[ \rme^{-\kappa(r-r')}\left(1 - \rme^{-\kappa(r-r')}\right)^{\gamma_t(v-r+r')-1} - \rme^{-\kappa(r+r')}\left(1 - \rme^{-\kappa(r+r')}\right)^{\gamma_t(v-r-r')-1}\right]\\
	&\qquad\qquad\qquad + \ol a(t)\left(H(v, r-r') - H(v, r+r')\right),\\
	&\text{where}\ H(v, r') = \frac{1}{1 + \rme^{\beta(v-r')}}\left( 1 - \rme^{-\kappa r'}\right)^{\gamma_t(v-r')}\left[ \frac{\kappa \rme^{-\kappa r'}}{1 - \rme^{-\kappa r'}} + \frac{\beta\rme^{\beta (v-r')}}{1 + \rme^{\beta(v-r')}}\ln (1 - \rme^{-\kappa r'})   \right]\\
	&\qquad\qquad\quad  =  \frac{1}{1 + \rme^{\beta(v-r')}}\left( 1 - \rme^{-\kappa r'}\right)^{\gamma_t(v-r')} F_2(v, r'),\\
	&\quad\text{with}\ F_2(v, r') = \frac{\kappa \rme^{-\kappa r'}}{1 - \rme^{-\kappa r'}} + \frac{\beta\rme^{\beta (v-r')}}{1 + \rme^{\beta(v-r')}}\ln (1 - \rme^{-\kappa r'}) .
	\end{align*}
	Then it is not difficult to see that:
	\begin{align*}
	&\frac{\partial}{\partial r'} H(v, r')\\
	= &\ \frac{1}{1 + \rme^{\beta(v-r')}}\left( 1 - \rme^{-\kappa r'}\right)^{\gamma_t(v-r')}\bigg[ - \left\{ \kappa^2\frac{\rme^{-\kappa r'}}{(1 - \rme^{-\kappa r'})^2} - \beta \frac{\rme^{\beta(v-r')}}{1 + \rme^{\beta (v-r')}}\left( \frac{\kappa\rme^{-\kappa r'}}{1 - \rme^{-\kappa r'}} - \beta\frac{\ln(1 - \rme^{-\kappa r'})}{1 + \rme^{\beta(v-r')}} \right)   \right\}\bigg.\\
	&\bigg.\ \ \  + \left( \kappa\gamma_0\frac{\rme^{-\kappa r'}}{1 - \rme^{-\kappa r'}} + \beta\frac{\rme^{\beta(v-r')}}{1 + \rme^{\beta(v-r')}} \right) F_2(v, r') + \frac{\ol a(t)}{1 + \rme^{\beta(v-r')}}(F_2(v, r'))^2\bigg]\\
	\leq&\ \frac{g_t(v, r')}{1 + \rme^{\beta (v-r')}}\bigg[ -\kappa^2\frac{\rme^{-\kappa r'}}{\left(1 - \rme^{-\kappa r'}\right)^2}\left( 1 - \frac{\beta}{\kappa} - \left(\frac{\beta}{\kappa}\right)^2\right)  + \frac{\ol a(t)}{1 + \rme^{\beta (v-r')}}\frac{\kappa^2\rme^{-2\kappa r'}}{\left( 1 - \rme^{-\kappa r'}\right)^2} \bigg.\\
	&\quad \bigg. + \kappa\frac{\rme^{-\kappa r'}}{1 - \rme^{-\kappa r'}}\left(\beta\frac{\rme^{\beta(v-r')}}{1 + \rme^{\beta(v-r')}} + \kappa\gamma_0\frac{\rme^{-\kappa r'}}{1 - \rme^{-\kappa r'}}\right)\bigg]\\
	<&\ 0.
	\end{align*}
	The upper bound in the last line holds because our parameters have been chosen so as to guarantee the following inequality:  
	\begin{align*}
	1 - \gamma_0 -\ol a(t)> \frac{2\beta}{\kappa} + \left(\frac{\beta}{\kappa}\right)^2.
	\end{align*}
	The above inequality means that $H$ is a strictly decreasing function of $r'$ and so, $H(v, r-r')\geq H(v, r+r')$. On the other hand, note that $\gamma_t(v-r+r')\leq\gamma_t(v-r-r'),$ and it is easily seen that
	\begin{align*}
	&\rme^{-\kappa(r-r')}\left(1 - \rme^{-\kappa(r-r')}\right)^{\gamma_t(v-r+r')-1} \geq \rme^{-\kappa(r+r')}\left(1 - \rme^{-\kappa(r+r')}\right)^{\gamma_t(v-r-r')-1}.\\
	&\text{Thus,}\qquad \frac{\partial}{\partial r'}b_2\geq 0.
	\end{align*}
	This means we have the following inequality:
	\begin{align*}
	2g_t(v, r) - g_t(v, r-r') - g_t(v, r+r')&= 2g_t(v, r) + b_2(v, r, r')\geq 2g_t(v, r) + b_2(v, r, r')\big|_{r'=0}=0.
	\end{align*}
\end{proof}

\begin{lemma}\label{lemma: gtprop5}
	For all $0\leq r'<r $, for all $(v, r)\in\R\times\R_+,$ the following inequality holds:
	\begin{align*}
	2g_t(v, r) - g_t(v-r', r-r') - g_t(v+r', r+r')\geq 0.
	\end{align*}
\end{lemma}

\begin{proof}
	Taking the partial derivative with respect to $r'$, we see that:
	\begin{align*}
	\frac{\partial}{\partial r'}\left[2g_t(v, r) - g_t(v-r', r-r') - g_t(v+r', r+r')\right]&= \kappa\gamma_t(v, r)\left[ \rme^{-\kappa(r-r')}\left(1 - \rme^{-\kappa(r-r')}\right)^{\gamma_t(v-r)-1} \right.\\
	&\quad\quad\left. -\ \rme^{-\kappa(r+r')}\left(1 - \rme^{-\kappa(r+r')}\right)^{\gamma_t(v-r)-1}\right]\\
	&\geq 0,
	\end{align*}
	so that,
	\begin{align*}
	2g_t(v, r) - g_t(v-r', r-r') - g_t(v+r', r+r')\geq \left[2g_t(v, r) - g_t(v-r', r-r') - g_t(v+r', r+r')\right]\bigg|_{r'=0}=0.
	\end{align*}
\end{proof}

\begin{lemma}\label{lemma: gtprop6}
	For all $r'>r$, for all $(v, r)\in\R\times\R_+$, the following inequality holds:
	\begin{align*}
	2\left(g_t(v, r) + g_t(v-r+r', r') - g_t(v+r', r+r')\right) - g_t(v, r)> \left( 1 - \rme^{-r} \right)^{\frac{1}{2}}g_t(v, r).
	\end{align*}
\end{lemma}

\begin{proof}
	From Lemma \ref{lemma: gtprop2}, we have, for all $r'>r$:
	\begin{align*}
	2\left(g_t(v, r) + g_t(v-r+r', r') - g_t(v+r', r+r')\right) - g_t(v, r)> \bigg( 2\left(2 - \left(1 + \rme^{-\kappa r}\right)^{\gamma_t(v, r)}\right) - 1 \bigg)g_t(v, r).
	\end{align*}
	For all $r<(\ln 2)/3$, it is obvious that $ 2\left(2 - \left(1 + \rme^{-\kappa r}\right)^{\gamma_t(v, r)}\right) - 1>\left( 1 - \rme^{-r}\right)^{\frac{1}{2}}$, since $\gamma_t(v, r)<1/4$.
	\newline For all $r\geq (\ln 2)/3$, let us define $f_2(r) =  2\left(2 - \left(1 + \rme^{-\kappa r}\right)^{1/4}\right) - 1 - \left( 1 - \rme^{-r}\right)^{1/2}.$
	Then it is easy to see that, \[\frac{d}{dr}f_2<0,\] so that, $f_2(r)>0,$
	which implies $\bigg( 2\left(2 - \left(1 + \rme^{-\kappa r}\right)^{\gamma_t(v, r)}\right) - 1 \bigg)g_t(v, r)> \left( 1 - \rme^{-r}\right)^{1/2} g_t(v, r).$
\end{proof}

\begin{lemma}\label{lemma: gtprop7}
	For all $ r'\geq r,$ for all $ p>\gamma_t(v, r),$ and for all $\nu\leq1$, the following is true:
	\begin{align*}
	\left(1 - \rme^{-\nu r}\right)^p g_t(v, r')\leq \left(1 - \rme^{-\nu r}\right)^p g_t(v-r+r', r')\leq \left(1 - \rme^{-\nu r'}\right)^p g_t(v, r).
	\end{align*}
\end{lemma}

\begin{proof}
	Let $h_1(r) = ( 1 - \rme^{-\nu r}) ( 1 - \rme^{-\kappa r})^{-1}.$ Then it is easily computed that \[\frac{d}{dr'}h_1\geq 0,\ \text{so that,}\ h_1(r')\geq h_1(r),\ \forall r'\geq r.\]
	This means
	\[\frac{1 - \rme^{-\nu r}}{1 - \rme^{-\nu r'}}\leq \frac{1 - \rme^{-\kappa r}}{1 - \rme^{-\kappa r'}} < 1,\]
	which implies, \[\left(\frac{1 - \rme^{-\nu r}}{1 - \rme^{-\nu r'}}\right)^p\leq \left(\frac{1 - \rme^{-\kappa r}}{1 - \rme^{-\kappa r'}} \right)^p\leq \left(\frac{1 - \rme^{-\kappa r}}{1 - \rme^{-\kappa r'}} \right)^{\gamma_t(v, r)},\ \forall p>\gamma_t(v, r),\] that is, 
	\[\left( 1 - \rme^{-\nu r} \right)^p \left(1 - \rme^{-\kappa r'}\right)^{\gamma_t(v, r)}\leq \left( 1 - \rme^{-\nu r'} \right)^p \left(1 - \rme^{-\kappa r}\right)^{\gamma_t(v, r)}.\]
	Since $\gamma_t(v, r)$ is, by definition, an increasing function of the radial variable $r$, $(1 - \rme^{-\kappa r'})^{\gamma_t(v, r')}\leq( 1 - \rme^{-\kappa r'})^{\gamma_t(v, r)}$, and this, together with the inequality obtained above, proves the statement of this lemma.
\end{proof}

\subsection {The Regularized, Time-independent H\"{o}lder-type Condition}\label{subsection: g_tep}

In this case, the H\"{o}lder exponent, denoted here by $\tg_0$, is independent of time, and the results below hold for each of the two values $\tg_0$ is allowed to take in our computations for the regularized linear operator, namely $0$ and $\gamma_0/2$. The H\"{o}lder-type condition $\thg(v, r)$ depends on the first variable $v$ only through the inclusion of the regularization parameter $\vep(v)$ in the argument as follows:
\begin{align*}
\thg(v, r) = \left( 1 - \rme^{-\kappa(r+\vep(v))} \right)^{\tg_0},\ \ \tg_0\in\{0, \gamma_0/2\}, \ \text{and,}\ \ \vep(v) = \vep_0\rme^{-\frac{\mu}{\gamma_0}\max(a_1, v)}.
\end{align*}
Then the function $\thg$ has properties similar to those proved in lemmas \ref{lemma: gtprop1} -\ref{lemma: gtprop7} above. Most of the proofs are also very similar to those recorded above, therefore we will write down the statements of the lemmas while omitting some of the proofs.

\begin{lemma}\label{lemma: thgtprop1}
	For all $0<r'\leq r$, the following are true:
	\begin{align*}
	&i)\ \thg_t(v, r) - \thg_t(v, r-r')\geq 0,\ \ \ \thg_t(v, r+r') - \thg_t(v, r)\geq 0,\\
	&ii)\ \thg_t(v, r) - \thg_t(v-r', r-r')\geq 0,\ \ \ \thg_t(v+r', r+r') - \thg_t(v, r)\geq 0.
	\end{align*}
\end{lemma}

\begin{proof}
	The inequalities contained in part i) are obvious.

	For part ii), let us observe that $0\leq \vep(v-r') - \vep(v)<r',\ \text{and}\ 0\leq \vep(v) - \vep(v+r')<r'.$ So:
	\begin{align*}
	\thg(v-r', r-r') &= \left( 1- \rme^{-\kappa(r+(\vep(v-r')-r')) }\right)^{\tg_0}\leq \left(  1 - \rme^{-\kappa (r+\vep(v))}\right)^{\tg_0}=\thg(v, r),\\
	\thg(v+r', r+r') &= \left( 1- \rme^{-\kappa(r+(\vep(v+r')+r')) }\right)^{\tg_0}\geq \left(  1 - \rme^{-\kappa (r+\vep(v))}\right)^{\tg_0}=\thg(v, r).
	\end{align*}
\end{proof}

\begin{lemma}\label{lemma: thgtprop2}
	For all $r'>r$, for all $(v, r)\in\R\times\R_+,$ the following inequality holds:
	\begin{align*}
	\thg(v, r) + \thg(v, r') - \thg(v, r+r')\geq \left( 2 - (1 + \rme^{-\kappa (r+\vep(v))})^{\tg_0} \right) \thg(v, r).
	\end{align*}
\end{lemma}
\begin{proof}
	A straightforward differentiation reveals
	\begin{align*}
	&\frac{\partial}{\partial r'} \left[ \thg(v, r') - \thg(v, r+r') \right]>0.\\
	&\text{So, }\thg(v, r) + \thg(v, r') - \thg(v, r+r') > \thg(v, r) + \left[\thg(v, r') - \thg(v, r+r')\right]\bigg|_{r'=r}\geq \left( 2 - (1 + \rme^{-\kappa (r+\vep(v))})^{\tg_0} \right) \thg(v, r).
	\end{align*}
\end{proof}

\begin{lemma}\label{lemma: thgtprop3}
	For all $r'>r$, for all $(v, r)\in\R\times\R_+,$ the following inequality holds:
	\begin{align*}
	\thg(v, r) + \thg(v-r+r', r') - \thg(v+r', r+r')\geq \left( 2 - (1 + \rme^{-\kappa (r+\vep(v))})^{\tg_0} \right) \thg(v, r).
	\end{align*}
\end{lemma}
\begin{proof}
	
	In the region $v+r'\leq a_1,$ the statement of this lemma is the same as the previous lemma. 
	
	When $v+r'>a_1,$
	\begin{align*}
	&\thg(v, r) + \thg(v-r+r', r') - \thg(v+r', r+r')\\
	&>\thg(v, r) + \thg(v+r', r') - \thg(v+r', r+r')\\
	&>\thg(v, r) + \left[ \thg(v+r', r') - \thg(v+r', r+r')  \right]\bigg|_{r'=r} \ \text{, by a simple differentiation,}\\
	&\geq\left( 2 - (1 + \rme^{-\kappa (r+\vep(v))})^{\tg_0} \right) \thg(v, r).
	\end{align*}
\end{proof}

\begin{lemma}\label{lemma: thgtprop4}
	For all $0\leq r'<r $, for all $(v, r)\in\R\times\R_+,$ the following inequality holds:
	\begin{align*}
	2\thg(v, r) - \thg(v, r-r') - \thg(v, r+r')\geq 0.
	\end{align*}
\end{lemma}

This lemma is obtained by following the proof of Lemma \ref{lemma: gtprop4} exactly and keeping in mind that the exponent $\tg_0$ is time-independent (so $\ol a(t)=0$) in this case.

\begin{lemma}\label{lemma: thgtprop5}
	For all $ 0\leq r'<r $, for all $(v, r)\in\R\times\R_+,$ the following inequality holds:
	\begin{align*}
	2\thg(v, r) - \thg(v-r', r-r') - \thg(v+r', r+r')\geq -\bigg[ \left( 1- \rme^{-\kappa(r-r'+\vep(v-r')) }\right)^{\tg_0}   - \left( 1- \rme^{-\kappa(r-r'+\vep(v)) }\right)^{\tg_0}    \bigg] .
	\end{align*}
\end{lemma}

\begin{proof}
	The inequality of this lemma is obtained by the following simple observation:
	\begin{align*}
	&2\thg(v, r) - \thg(v-r', r-r') - \thg(v+r', r+r')\\
	\geq&\ 2\thg(v, r) - \thg(v, r-r') - \thg(v, r+r') - \left[  \left( 1- \rme^{-\kappa(r-r'+\vep(v-r')) }\right)^{\tg_0}   - \left( 1- \rme^{-\kappa(r-r'+\vep(v)) }\right)^{\tg_0} \right]\\
	\geq&\  - \left[  \left( 1- \rme^{-\kappa(r-r'+\vep(v-r')) }\right)^{\tg_0}   - \left( 1- \rme^{-\kappa(r-r'+\vep(v)) }\right)^{\tg_0} \right],\ \ \text{by Lemma \ref{lemma: thgtprop4}}.
	\end{align*}
\end{proof}

\begin{lemma}\label{lemma: thgtprop6}
	For all $r'>r$, for all $(v, r)\in\R\times\R_+$, the following inequality holds:
	\begin{align*}
	2\left(\thg(v, r) + \thg(v-r+r', r') - \thg(v+r', r+r')\right) - \thg(v, r)> \left( 1 - \rme^{-r} \right)^{\frac{1}{2}}\thg(v, r).
	\end{align*}
\end{lemma}

The proof of Lemma \ref{lemma: thgtprop6} follows the corresponding proof for \ref{lemma: gtprop6}.

\begin{lemma}\label{lemma: thgtprop7}
	For all $r'\geq r,$ for all $p>\tg_0,$ for all $\nu\leq1$, the following is true:
	\begin{align*}
	\left(1 - \rme^{-\nu r}\right)^p \thg(v, r')\leq \left(1 - \rme^{-\nu r'}\right)^p \thg(v, r).
	\end{align*}
\end{lemma}

The proof is the same as that of Lemma \ref{lemma: gtprop7}.

	\section{Computations for Lemma \ref{lemma: mcLuEstimate}}\label{appendix: DeltaCompute}

\subsection{Some Useful Inequalities involving the Kernel Function}\label{subsection: KernelProp}
In the computations for Lemma \ref{lemma: mcLuEstimate} we use certain properties of the kernel function $K_3$ and the weight function $\Gamma_t$, collected in the lemmas below.

\begin{lemma}\label{lemma: mcC3bound}
	For all $ v-r<-b_0, $ for all $0<r'\leq r$ the following inequality holds:
	\begin{align*}
	& K_3(v-r, v-r+r')\left( f(v-r) - (\ln(1 + \rme^{v-r+r'}))^{-\alpha}\right) - K_3^1(v-r, v-r-r')\left( f(v-r-r') - f(v-r)\right)\\
	\geq&\ K_3(v-r, v-r+r') \left(\ln(1 + \rme^{v-r+r'})\right)^{-\alpha}\left(\frac{\ln( 1 + \rme^{v-r-r'})}{\ln(1 + \rme^{v-r+r'})}\right)^{\alpha}\left\{ 1 -  \left(\frac{\ln( 1 + \rme^{v-r-r'})}{\ln(1 + \rme^{v-r+r'})}\right)^{1-2\alpha} \right\}\times\\
	&\qquad\times \left\{ 1 -  \left(\frac{\ln( 1 + \rme^{v-r-r'})}{\ln(1 + \rme^{v-r})}\right)^{\alpha} \right\}
	\end{align*}
\end{lemma}

\begin{proof}
	Note that:
	\begin{align*}
	&K_3(v-r, v-r+r')\left( f(v-r) - (\ln(1 + \rme^{v-r+r'}))^{-\alpha}\right) - K_3^1(v-r, v-r-r')\left( f(v-r-r') - f(v-r)\right)\\
	\geq&\ 4\bn \left(\ln(1 + \rme^{v-r})\right)^{-3/2}\frac{\rme^{-r'}}{1 - \rme^{-r'}}\frac{\rme^{v-r} + 2\rme^{-r'}}{1 + \rme^{v-r} + \rme^{-r'}}\left[ \ln(1 + \rme^{v-r+r'})(\ln(1 + \rme^{v-r}))^{-\alpha}\left\{ 1 - \left(\frac{\ln(1 + \rme^{v-r})}{\ln (1 + \rme^{v-r+r'})}\right)^{\alpha} \right\}\right.\\
	&\qquad\left. - (\ln(1 + \rme^{v-r-r'}))^{1-\alpha} \left\{ 1 - \left(\frac{\ln(1 + \rme^{v-r-r'})}{\ln (1 + \rme^{v-r})}\right)^{\alpha} \right\}\right].
	\end{align*}
	Now,		
	\begin{align}\label{eq: C3deriv1}
	&\ln(1 + \rme^{v-r+r'})(\ln(1 + \rme^{v-r}))^{-\alpha}\left\{ 1 - \left(\frac{\ln(1 + \rme^{v-r})}{\ln (1 + \rme^{v-r+r'})}\right)^{\alpha} \right\} \nonumber\\
	&\quad - (\ln(1 + \rme^{v-r-r'}))^{1-\alpha} \left\{ 1 - \left(\frac{\ln(1 + \rme^{v-r-r'})}{\ln (1 + \rme^{v-r})}\right)^{\alpha} \right\}\nonumber\\
	=&\ (\ln(1 + \rme^{v-r+r'}))^{1-2\alpha}\left[ (\ln(1 + \rme^{v-r+r'}))^{\alpha}\left\{  \left(\frac{\ln(1 + \rme^{v-r+r'})}{\ln(1 + \rme^{v-r})}\right)^{\alpha} -1\right\}  \right]\nonumber\\
	&\qquad -  (\ln(1 + \rme^{v-r-r'}))^{1-2\alpha}\left[ (\ln(1 + \rme^{v-r-r'}))^{\alpha}\left\{ 1 - \left(\frac{\ln(1 + \rme^{v-r-r'})}{\ln(1 + \rme^{v-r})}\right)^{\alpha}\right\}  \right].
	\end{align}
	Now for all $r'\geq 0,$ let us define the function $h$ as follows:
	\begin{align*}
	&\ (\ln(1 + \rme^{v-r}))^{\alpha} h(v-r, r') =  (\ln(1 + \rme^{v-r+r'}))^{\alpha}\left\{  \left(\frac{\ln(1 + \rme^{v-r+r'})}{\ln(1 + \rme^{v-r})}\right)^{\alpha} -1\right\}\\
	&\qquad\qquad\qquad\qquad\qquad\qquad\qquad - (\ln(1 + \rme^{v-r-r'}))^{\alpha}\left\{ 1 - \left(\frac{\ln(1 + \rme^{v-r-r'})}{\ln(1 + \rme^{v-r})}\right)^{\alpha}\right\}.\\
	\text{Then}\ \ \frac{\partial}{\partial r'}h(v-r, r') &= \alpha (\ln(1 + \rme^{v-r+r'}))^{\alpha-1}\frac{\rme^{v-r+r'}}{1 + \rme^{v-r+r'}} \left[ 2(\ln(1 + \rme^{v-r+r'}))^{\alpha} -  (\ln(1 + \rme^{v-r}))^{\alpha} \right] \\
	&\quad - \alpha (\ln(1 + \rme^{v-r-r'}))^{\alpha-1}\frac{\rme^{v-r-r'}}{1 + \rme^{v-r-r'}} \left[ 2(\ln(1 + \rme^{v-r-r'}))^{\alpha} -  (\ln(1 + \rme^{v-r}))^{\alpha} \right] .
	\end{align*}
	Observe that for all $v-r<-b_0,\ \rme^{-r'}<\frac{\ln(1 + \rme^{v-r-r'})}{\ln(1 + \rme^{v-r})}<\rme^{-\frac{24}{25}r'},$ since $b_0\geq 10.$
	Then for all $r'>0$ such thtat \[\left(\frac{\ln(1 + \rme^{v-r-r'})}{\ln(1 + \rme^{v-r})}\right)^{\alpha}<\frac{1}{2},\] it is obvious from the definition  that
	\[\frac{\partial}{\partial r'} h(v-r, r')>0.\]
	On the other hand, for all $r'>0$ such that
	\[\left(\frac{\ln(1 + \rme^{v-r-r'})}{\ln(1 + \rme^{v-r})}\right)^{\alpha}\geq\frac{1}{2},\]  we have $\rme^{-\frac{24}{25}\alpha r'}>\frac{1}{2},$ which, as one can easily check, implies $ v-r+r'<-3.5,$ for all $v-r<-b_0,$ and this in turn means,
	\[ (\ln(1 + \rme^{v-r+r'}))^{\alpha-1}\frac{\rme^{v-r+r'}}{1 + \rme^{v-r+r'}} > (\ln(1 + \rme^{v-r-r'}))^{\alpha-1}\frac{\rme^{v-r-r'}}{1 + \rme^{v-r-r'}} , \]
	because $\frac{\rme^w}{1 + \rme^w}(\ln(1 + \rme^w))^{\alpha-1}$ is a strictly increasing function for all $w<-3$. Thus we have, again, \[\frac{\partial}{\partial r'} h(v-r, r')>0 \text{ in this region.}\]
	Since we have proved $h(v-r, r')$ to be a strictly increasing function of $r'$, the following inequality holds:
	\begin{align*}
	(\ln(1 + \rme^{v-r+r'}))^{\alpha}\left\{  \left(\frac{\ln(1 + \rme^{v-r+r'})}{\ln(1 + \rme^{v-r})}\right)^{\alpha} -1\right\}> (\ln(1 + \rme^{v-r-r'}))^{\alpha}\left\{ 1 - \left(\frac{\ln(1 + \rme^{v-r-r'})}{\ln(1 + \rme^{v-r})}\right)^{\alpha}\right\}.
	\end{align*}
	Using the above we can refer to equation (\ref{eq: C3deriv1}) and write:
	\begin{align*}
	&(\ln(1 + \rme^{v-r+r'}))^{1-2\alpha}\left[ (\ln(1 + \rme^{v-r+r'}))^{\alpha}\left\{  \left(\frac{\ln(1 + \rme^{v-r+r'})}{\ln(1 + \rme^{v-r})}\right)^{\alpha} -1\right\}  \right]\\
	&\qquad -  (\ln(1 + \rme^{v-r-r'}))^{1-2\alpha}\left[ (\ln(1 + \rme^{v-r-r'}))^{\alpha}\left\{ 1 - \left(\frac{\ln(1 + \rme^{v-r-r'})}{\ln(1 + \rme^{v-r})}\right)^{\alpha}\right\}  \right]\\
	>&\ \left(\ln(1 + \rme^{v-r+r'})\right)^{1-\alpha}\left(\frac{\ln( 1 + \rme^{v-r-r'})}{\ln(1 + \rme^{v-r+r'})}\right)^{\alpha}\left\{ 1 -  \left(\frac{\ln( 1 + \rme^{v-r-r'})}{\ln(1 + \rme^{v-r+r'})}\right)^{1-2\alpha} \right\}\times\\
	&\qquad\times \left\{ 1 -  \left(\frac{\ln( 1 + \rme^{v-r-r'})}{\ln(1 + \rme^{v-r})}\right)^{\alpha} \right\},
	\end{align*}
	which obviously implies the inequality stated in this lemma.
\end{proof}

\begin{lemma}\label{lemma: K3Bound1}
	For all $ v<-b_0,$ for all $0<r'\leq r$ and $ v+r<0$ the following inequality holds:
	\begin{align*}
	\int_{\delta_1}^r\rmd r' K_3(v, v+r')\left(  f(v) - f(v+r') \right) g_t(v, r)\geq \int_{\delta_1}^r\rmd r' K_3^1(v, v-r')\left(  f(v-r') - f(v) \right) g_t(v-r', r-r').
	\end{align*}
\end{lemma}

\begin{proof} Let us define
	\[h_2(v, r') = \ln(1 + \rme^{v+r'})\left[ (\ln(1 + \rme^v))^{-\alpha} - (\ln(1 + \rme^{v+r'}))^{-\alpha} \right] - \ln(1 + \rme^{v-r'})\left[ (\ln(1 + \rme^{v-r'}))^{-\alpha} - (\ln(1 + \rme^{v}))^{-\alpha} \right].\]
	Then 
	\begin{align*}
	\frac{\partial}{\partial r'}h_2 &= \frac{\rme^{v+r'}}{1 + \rme^{v+r'}}\left[ (\ln(1 + \rme^v))^{-\alpha} - (1 - \alpha)(\ln(1 + \rme^{v+r'}))^{-\alpha}  \right]\\
	&\qquad\qquad - \frac{\rme^{v-r'}}{1 + \rme^{v-r'}}\left[ (\ln(1 + \rme^v))^{-\alpha} - (1 - \alpha)(\ln(1 + \rme^{v-r'}))^{-\alpha}  \right]\\
	&> 0, \quad \text{since}\ (\ln(1 + \rme^{v-r'}))^{-\alpha} >(\ln(1 + \rme^{v+r'}))^{-\alpha} .
	\end{align*}
	Thus $h_2(v, r')>h_2(v, 0)=0,$  for all $r'>0.$ Then the inequality stated in the lemma is obvious from the definition of the kernel functions and the fact that $g_t(v, r)>g_t(v-r', r-r'),$ for all $r'>0.$
\end{proof}

\begin{lemma}\label{lemma: K3Bound2}
	For all $v+r'<0,$ for all $r'>r$ the following is true:
	\begin{align*}
	K_3(v, v+r') f(v+r') - K_3^1(v-r, v-r-r') f(v-r-r')>0.
	\end{align*}
\end{lemma}
\begin{proof}
	\begin{align*}
	&K_3(v, v+r') f(v+r') - K_3^1(v-r, v-r-r') f(v-r-r')\\
	\geq\ & 4\bn\frac{\rme^{-r'}}{1 - \rme^{-r'}}\frac{\rme^v + 2\rme^{-r'}}{1 + \rme^v +\rme^{-r'}}(\ln(1 + \rme^v))^{-\frac{3}{2}}(\ln(1 + \rme^{v+r'}))^{1-\alpha}\left[ 1- \left(\frac{\ln(1 + \rme^v)}{\ln(1 + \rme^{v-r})}\right)^{\frac{3}{2}}\left(\frac{\ln(1 + \rme^{v-r-r'})}{\ln(1 + \rme^{v+r'})}\right)^{1-\alpha}\right]\\
	\geq\ &4\bn \frac{\rme^{-r'}}{1 - \rme^{-r'}}\frac{\rme^v + 2\rme^{-r'}}{1 + \rme^v +\rme^{-r'}}(\ln(1 + \rme^v))^{-\frac{3}{2}}(\ln(1 + \rme^{v+r'}))^{1-\alpha}\left[ 1- \left(\frac{\ln(1 + \rme^{v-r-r'})}{\ln(1 + \rme^{v-r})}\right)^{1-\alpha}\left(\frac{\ln(1 + \rme^v)}{\ln(1 + \rme^{v-r})}\right)^{\frac{1}{2}+\alpha}\right]\\
	>\ & 0,
	\end{align*}
	where we have used the fact that $1-\alpha>\frac{1}{2} + \alpha,$ for all $\alpha<1/4,$ and that \[ \frac{\ln(1 + \rme^{v-r})}{\ln(1 + \rme^{v-r-r'})}> \frac{\ln(1 + \rme^v)}{\ln(1 + \rme^{v-r'})}.\]
\end{proof}

	\subsection{Lower Bounds for $I_i[\Gamma_t](v, r),\ i\in\{1, 2, 3, 4\}$}
\subsubsection{Estimating $I_1[\Gamma_t](v, r)$:}

We start with the following lower bound, which can be derived quite easily by using  Lemma \ref{lemma: gtprop7} proved above:
\begin{align}
I_1[\Gamma_t](v, r) &= I_1[\Gamma_t^1](v, r) + I_1[\Gamma_t^2](v, r)\nonumber\\
&\geq\ \mathcal{J}_0[\Gamma_t^1 + \Gamma_t^2](v, r) + \mathcal{J}_1[\Gamma_t^1](v, r) + \mathcal{J}_2[\Gamma_t^2](v, r) + \mathcal{I}[\Gamma_t^1 + \Gamma_t^2](v, r).
\end{align}
$\mathcal{J}_0$ has already been defined in (\ref{defn: J_0}). We now write down explicitly the other terms. As mentioned before, these terms are sub-dominant to $\mathcal{J}_0$ close to the diagonal.
\begin{align}\label{defn: J_1}
&\mathcal{J}_1[\Gamma_t^1] (v, r)\nonumber\\
=&\ f(v-r)\Bigg[ \cf(v<-b_0)\left\{ \rme^{\mu a} \int_{r+\delta_1}^{r+a_1-v} \rmd r'\ \Big( (K_3(v-r, v-r+r') - K_3(v, v+r')) g_t(v-r+r', r')\right.\Bigg.\nonumber\\
&\ \ \ \ \qquad \Big. - (K_3^1(v-r, v-r-r') - K_3^1(v, v-r'))g_t(v, r')\Big)\nonumber\\
& +   \int_{r+a_1-v}^{\infty} \rmd r'\ \Big( (K_3(v-r, v-r+r') - K_3(v, v+r'))\rme^{-\frac{1}{2}r'} e^{\mu c_0 (v-r+r')} g_t(v-r+r', r')\nonumber\\
& \qquad \left.\Big. - (K_3^1(v-r, v-r-r') - K_3^1(v, v-r')) e^{\mu a}g_t(v, r')\Big)\right\}\nonumber\\
+ & \cf(v\geq-b_0) \int_{r+\delta_1}^{\infty} \rmd r'\ \Big[\left( K_3(v-r, v-r+r') - K_3(v, v+r') \right)\rme^{-\frac{1}{2}r'}\rme^{\mu\max\left(a, c_0(v-r+r'), v-r\right)}g_t(v-r+r', r')\Big.\nonumber\\
&\ \qquad\Bigg.\Big. - \rme^{\mu\max\left(a, c_0v, v-r\right)}\left( K_3^1(v-r, v-r-r') - K_3^1(v, v-r') \right) g_t(v, r')\Big]\Bigg],
\end{align}
\begin{align}\label{defn: J_2}
&\mathcal{J}_2[\Gamma_t^2] (v, r)\nonumber\\
= &\cf(v<-b_0)\Bigg[ \cf(r\leq a_1-v\leq r+\delta_1)\Bigg\{ \int_{r+\delta_1}^{\infty}\rmd r'\ \Big( \left(K_3(v-r, v-r+r') - K_3(v, v+r')\right)\rme^{-\frac{1}{2}r'}\times \Big.\Bigg.\Bigg.\nonumber\\
&\ \ \ \ \ \qquad\qquad\qquad\ \ \ \times\exp\left(\mu\max\{a, c_0(v-r+r'), v-r\}\right)f(v-r+r')g_t(v-r+r', r')\nonumber\\
&\ \ \ \quad \Bigg.\Big. - \left( K_3^1(v-r, v-r-r') - K_3^1(v, v-r') \right)\exp\left(\mu\max\{a, c_0v, v-r'\}\right)f(v) g_t(v, r')\Big)\Bigg\}\nonumber\\
&\ \ + \cf(r+\delta_1< a_1-v)\Bigg\{ \int_{r+\delta_1}^{a_1-v} \rmd r'\ \Big(  \left(K_3(v-r, v-r+r') - K_3(v, v+r')\right)\times\Bigg.\nonumber\\
& \qquad\qquad\qquad\qquad\qquad\qquad\times\exp\left(\mu\max\{a, c_0(v-r+r'), v-r\}\right) f(v+r')g_t(v-r+r', r')\nonumber\\
& \ \ \ \ \ \qquad\ \ \ \Big.- \left( K_3^1(v-r, v-r-r') - K_3^1(v, v-r') \right) \exp\left(\mu\max\{a, c_0v, v-r'\}\right)f(v)g_t(v, r')\Big)\nonumber\\
&\ \ \ \qquad + \int_{a_1-v}^{\infty} \rmd r'\ \Big( \left(K_3(v-r, v-r+r') - K_3(v, v+r')\right)\rme^{-\frac{1}{2}r'}\times\Big.\nonumber\\
& \qquad\qquad\qquad\qquad\qquad\qquad\qquad\times \exp\left(\mu\max\{a, c_0(v-r+r'), v-r\}\right) f(v+r') g_t(v-r+r', r')\nonumber\\
& \ \ \ \ \ \ \ \ \ \qquad\ \ \Bigg.\Big.- \left( K_3^1(v-r, v-r-r') - K_3^1(v, v-r') \right)\exp\left(\mu\max\{a, c_0v, v-r'\}\right)f(v) g_t(v, r')\Big)\Bigg\}\nonumber\\
&\ \ + \cf(r > a_1-v)\Bigg\{ \int_{r+\delta_1}^{r+ a_1-v} \rmd r'\ \Big( \left(K_3(v-r, v-r+r') - \ol{K_3}^2(v, v+r')\right)\times\Bigg.\nonumber\\
& \qquad\qquad\qquad\qquad\qquad\qquad\qquad\times\exp\left(\mu\max\{a, c_0(v-r+r'), v-r\}\right) f(v+r') g_t(v-r+r', r')\nonumber\\
& \ \ \ \ \ \ \ \quad\Big.- \left( K_3^1(v-r, v-r-r') - K_3^1(v, v-r') \right)\exp\left(\mu\max\{a, c_0v, v-r'\}\right)f(v) g_t(v, r')\Big)\nonumber\\
&\ \ \ \ + \int_{r+ a_1-v}^{\infty} \rmd r'\ \Big( \left(K_3(v-r, v-r+r') - K_3(v, v+r')\right)\rme^{-\frac{1}{2}r'} \times\Big.\nonumber\\
& \qquad\qquad\qquad\qquad\qquad\qquad\qquad\times\exp\left(\mu\max\{a, c_0(v-r+r'), v-r\}\right) f(v+r') g_t(v-r+r', r')\nonumber\\
& \ \ \ \ \ \ \ \ \ \ \ \Bigg.\Bigg.\Big.- \left( K_3^1(v-r, v-r-r') - K_3^1(v, v-r') \right)\exp\left(\mu\max\{a, c_0v, v-r'\}\right)f(v) g_t(v, r')\Big)\Bigg\}\Bigg]\nonumber\\
&+ \ \cf(v\geq-b_0)\Bigg[\int_{r+\delta_1}^{\infty}\rmd r'\ \Big(\left(K_3(v-r, v-r+r') - K_3(v, v+r')\right)\rme^{-\frac{1}{2}r'}\times \Big.\Bigg.\nonumber\\
&\ \ \ \ \ \qquad\qquad\qquad\ \ \ \times\exp\left(\mu\max\{a, c_0(v-r+r'), v-r\}\right)f(v+r') g_t(v-r+r', r')\nonumber\\
&\ \ \ \ \ \ \ \Bigg.\Big. - \left( K_3^1(v-r, v-r-r') - K_3^1(v, v-r') \right)\exp\left(\mu\max\{a, c_0v, v-r'\}\right)f(v) g_t(v, r')\Big)\Bigg],
\end{align}
and finally,
\begin{align}\label{defn: mcI}
&\mathcal{I}[\Gamma_t^1 + \Gamma_t^2](v, r)\nonumber\\
=&\ - \rme^{\mu\max\left(a, c_0v\right)}\int_{r+\delta_1}^{\infty} \rmd r'\ K_3^1(v, v-r')\left(  f(v-r-r') - f(v-r')\right) g_t(v, r') \nonumber\\
& -  \rme^{\mu\max\left(a, c_0v\right)}\int_{r+\delta_1}^{\infty} \rmd r'\left(K_3^1(v-r, v-r-r') - K_3^1(v, v-r')\right)\left(f(v-r-r') - f(v-r)\right)g_t(v, r')\nonumber\\
& -  \rme^{\mu\max\left(a, c_0v\right)}\int_{r+\delta_1}^{\infty} \rmd r'\ K_3^1(v-r, v-r-r') \left(f(v-r-r') - f(v-r)\right)\left( g_t(v, r+r') - g_t(v, r')\right)\nonumber\\
& + \mathcal{I}^{(1)}[\Gamma_t^1](v, r) \ + \ \mathcal{I}^{(2)}[\Gamma_t^2](v, r) \ + \ \mathcal{I}^{(3)}[\Gamma_t^1 + \Gamma_t^2](v, r),
\end{align}
with
\begin{align}\label{defn: mcI^1}
&\mathcal{I}^{(1)}[\Gamma_t^1](v, r)\nonumber\\
=& \cf(v<-b_0) f(v-r)\Bigg[ \int_{r+\delta_1}^{r+a_1-v} \rmd r'\ K_3(v, v+r') \left( \rme^{\mu a} - \rme^{-\frac{1}{2}r'} \rme^{\mu\max\{a, v-r, c_0(v+r')\}} \right)\times\Bigg.\nonumber\\
&\qquad\qquad\qquad\qquad\qquad\qquad\times\left( g_t(v, r) + g_t(v-r+r', r') - g_t(v+r', r+r')\right)\nonumber\\
&\ \ + \cf(a_1-v\geq r+\delta_1)\int_{a_1-v}^{r+a_1-v} \rmd r'\ K_3(v, v+r') \left( \rme^{\mu a} - \rme^{\mu c_0 v}\rme^{-(\frac{1}{2}-\mu c_0)r'} \right)g_t(v+r', r+r')\nonumber\\
& \ \ \Bigg.+\ \cf(a_1-v<r + \delta_1)\int_{r+\delta_1}^{r+a_1-v} \rmd r'\ K_3(v, v+r') \left( \rme^{\mu a} - \rme^{\mu c_0 v}\rme^{-(\frac{1}{2}-\mu c_0)r'} \right)g_t(v+r', r+r')\Bigg],
\end{align}
\begin{align}\label{defn: mcI^2}
&\mathcal{I}^{(2)}[\Gamma_t^2](v, r)\nonumber\\
=& \ \cf(v<-b_0)\Bigg[ \int_{\max(r+\delta_1, a_1-v)}^{\infty}  \rmd r'\ K_3(v, v+r') f(v+r')\left(\rme^{\mu a} - \rme^{\mu c_0 v}\right) g_t(v, r)\ \ \ \Bigg.\nonumber\\
&\ \   + \int_{r+\delta_1}^{\infty}  \rmd r'\ K_3(v, v+r')\left(f(v) - f(v+r')\right) \rme^{\mu a} g_t(v, r)\nonumber\\
&\ \ + \cf(r\leq a_1-v\leq r+\delta_1)\Bigg\{ \int_{r+\delta_1}^{r+a_1-v}  \rmd r'\ \ol{K_{3}}^2(v, v+r') f(v+r') \left( \rme^{\mu a} - \rme^{\mu c_0 (v-r+r')}\right) g_t(v-r+r', r')\Big.\nonumber\\
&\ \ \ \ + \int_{r+\delta_1}^{\infty}  \rmd r'\ \ol{K_{3}}^2(v, v+r') \left(f(v-r+r') - f(v+r')\right) \rme^{\mu\max\{ a, c_0(v-r+r')\}} g_t(v-r+r', r')\nonumber\\
&\ \ \ \ + \int_{r+\delta_1}^{\infty}  \rmd r'\ K_3(v, v+r')f(v+r')\rme^{\mu c_0 v}\left( 1 - 2\rme^{-(\frac{1}{2} -\mu c_0)r'} + \rme^{-2(\frac{1}{2} - \mu c_0)r'}\right) g_t(v, r)\nonumber\\
&\ \ \ \ \Big. + \int_{r+\delta_1}^{r+a_1-v}  \rmd r'\ \ol{K_{3}}^1(v-r, v-r+r') f(v-r+r') \rme^{\mu a} g_t(v-r+r', r')\Bigg\} \nonumber\\
&\ \ + \cf(a_1-v>r+\delta_1)\Bigg\{ \int_{r+\delta_1}^{\infty}  \rmd r'\ \ol{K_{3}}^2(v-r, v-r+r')\left( f(v-r+r') - f(v+r')  \right)\times\Big.\nonumber\\
&\qquad\qquad\qquad\qquad\qquad\qquad\qquad\times \rme^{\mu\max\{ a, c_0(v-r+r')\}}g_t(v-r+r', r')\nonumber\\
&\ \ \ \ \ + \int_{r+\delta_1}^{r+a_1-v}  \rmd r'\ \left( \ol{K_{3}}^1(v-r, v-r+r') - \ol{K_{3}}^1(v, v+r')\right)\left( f(v-r+r') - f(v+r') \right)\times\nonumber\\
&\qquad\qquad\qquad\qquad\qquad\qquad\times \rme^{\mu\max\{ a, c_0(v-r+r')\}} g_t(v-r+r', r')\nonumber\\
&\ \ \ \ \ +  \int_{a_1-v}^{r+a_1-v}  \rmd r'\ \left( \ol{K_{3}}^1(v-r, v-r+r') - \ol{K_{3}}^1(v, v+r')\right)f(v+r')\rme^{\mu\max\{ a, c_0(v-r+r')\}} g_t(v-r+r', r')\nonumber\\
&\ \ \ \ \ + \rme^{\mu a}\int_{r+\delta_1}^{a_1-v}  \rmd r'\ \ol{K_{3}}^1(v, v+r') f(v+r')\left( g_t(v, r) + g_t(v-r+r', r') - g_t(v+r', r+r') \right)\nonumber\\
&\ \ \ \ \ + \int_{a_1-v}^{r+a_1-v}  \rmd r'\ K_3(v, v+r')f(v+r')\left[ \rme^{\mu c_0 v}\left( 1 - \rme^{-(\frac{1}{2}-\mu c_0)r'}\right) g_t(v, r) +\right.\nonumber\\
&\qquad\qquad\qquad\qquad\qquad\qquad\left.+ \left( \rme^{\mu a} - \rme^{\mu c_0 v}\rme^{-(\frac{1}{2}-\mu c_0)r'} \right)g_t(v-r+r', r')\right]\nonumber\\
&\ \ \ \ \ \Big. + \rme^{\mu a} \int_{r+\delta_1}^{r+a_1-v}  \rmd r'\ \ol{K_{3}}^1(v, v+r')\left( f(v-r+r') - f(v+r')\right)g_t(v-r+r', r')\Bigg\}\nonumber\\
&\ \ + \cf(a_1-v<r)\Bigg\{ \int_{r+\delta_1}^{r+a_1-v}  \rmd r'\ \ol{K_{3}}^2(v, v+r') f(v+r')\left( \rme^{\mu a} - \rme^{\mu c_0 (v-r+r')}\right)g_t(v-r+r', r')\Bigg.\nonumber\\
&\ \ \ \ \ + \rme^{\mu a}\int_{r+\delta_1}^{r+a_1-v}  \rmd r'\ \ol{K_{3}}^2(v, v+r')\left( f(v-r+r') - f(v+r') \right) g_t(v-r+r', r')\nonumber\\
&\ \ \ \ +  \int_{r+\delta_1}^{r+a_1-v}  \rmd r'\ \ol{K_{3}}^1(v-r, v-r+r')\rme^{\mu\max\{ a, c_0(v-r+r')\}}\left( f(v-r+r') - f(v+r') \right) g_t(v-r+r', r')\nonumber\\
&\ \ \ \ + \int_{r+\delta_1}^{\infty}  \rmd r'\ \left( \ol{K_{3}}^2(v-r, v-r+r') - \ol{K_{3}}^2(v, v+r')\right)\left(f(v-r+r') - f(v+r')\right)\times\nonumber\\
&\qquad\qquad\qquad\qquad\qquad\Bigg.\Bigg.\times \rme^{\mu\max\{ a, c_0(v-r+r')\}}g_t(v-r+r', r')\Bigg\}\Bigg],
\end{align}
and finally, 
\begin{align}\label{defn: mcI^3}
&\mathcal{I}^{(3)}[ \Gamma_t^2](v, r)\nonumber\\
=& \cf(v\geq -b_0) \int_{r+\delta_1}^{\infty}  \rmd r'\ \ol{K_{3}}^2(v-r, v-r+r') \left( f(v-r+r') - f(v+r') \right) \rme^{\mu\max( a, v-r, c_0(v-r+r'))} g_t(v-r+r', r')
\end{align}
The parts $\mathcal{J}_1$ and $\mathcal{J}_2$ yield some negative terms that will need to be controlled, as we will see shortly.

\vspace{0.1in}

\subsubsection{Estimating $I_2[\Gamma_t](v, r)$}

\begin{align*}
I_2[\Gamma_t](v, r) &= I_2[\Gamma_t^1](v, r) + I_2[\Gamma_t^2](v, r).
\end{align*}

\vspace{0.1in}

\paragraph{1. A Lower Bound for $I_2{[}\Gamma_t^1{]}(v, r)$:}

Recall from (\ref{defn: I2})
\begin{align*}
&I_2[\Gamma_t^1](v, r)\nonumber \\
\geq&\ \cf(v-r<-b_0)e^{\mu\max\{a, c_0v\}}\int_{\delta_1}^r \rmd r'\ \Big[ K_3(v-r, v-r+r')\left( f(v-r) - (\ln(1 + e^{v-r+r'}))^{-\alpha}\right)\Big.\nonumber\\
&\qquad\qquad\qquad\qquad\qquad - \Big. K_3^1(v-r, v-r-r') \left( f(v-r-r') - f(v-r) \right)\Big] g_t(v, r)\nonumber\\
&+ \cf(v< -b_0)\cf(r>a-v)f(v-r)\int_{a_1-v}^r \rmd r'\ K_3(v, v+r')\left( \rme^{\mu a} - \rme^{\mu c_0 v}\rme^{-(\frac{1}{2} - \mu c_0)r'}  \right) g_t(v, r)\nonumber\\
&+ \cf(v\geq -b_0)\cf(v-r<-b_0)\Big[ f(v-r)\int_{\delta_1}^r \rmd r'\ K_3(v, v+r')\left( \rme^{\mu\max\{a, c_0v\}} - \rme^{-\frac{1}{2}r'} \rme^{\mu\max\{a, c_0(v+r')\}} \right) g_t(v, r)\Big. \nonumber\\
&\ \ \ \ \Big. + \int_{\min(r, r-v)}^r \rmd r'\ K_3(v-r, v-r+r') (\ln(1 + e^{v-r+r'}))^{-\alpha}\left( 2 g_t(v, r) - g_t(v, r+r')\right)\Big]+  I_2^{-}[\Gamma_t^1](v, r),
\end{align*}	
where
\begin{align}\label{defn1: I_2^-}
&I_2^-[\Gamma_t^1](v, r)\nonumber \\
=& - f(v-r)\rme^{\mu\max(a, v-r, c_0v)}\int_{\delta_1}^r \rmd r'\ \left( K_3(v, v+r') - K_3^1(v, v-r') \right)\left(g_t(v+r', r+r') - g_t(v, r+r')\right)\nonumber\\
- &\cf(v<-b_0)\rme^{\mu a}\Big[ 4\bar{n}\left(\ln (1 + \rme^{v-r})\right)^{-\frac{3}{2}}\int_{\delta_1}^r \rmd r'\ \left(\ln(1 + \rme^{v-r+r'})\right)^{1 - \alpha}\frac{\rme^{-r'}}{1 - \rme^{-r'}}\times\Big.\nonumber\\
&\qquad\qquad\qquad\qquad\qquad\qquad\times\left( 1 - \left(\frac{\ln (1 + \rme^{v-r-r'})}{\ln(1 + \rme^{v-r+r'})}\right)^{\alpha}\right)\left( g_t(v, r+r') - g_t(v, r) \right)\nonumber\\
&\ \ \ \ \ +\  4\bar{n}f(v-r)\left(\ln (1 + \rme^v)\right)^{-\frac{3}{2}}\int_{\delta_1}^r \rmd r'\ \ln(1 + \rme^{v+r'})\frac{\rme^{v-r'}(1 - \rme^{-r'})(1 - \rme^{-\frac{5}{4}r'})}{(1 + \rme^{v-r} + \rme^{-r'})(1 + \rme^v + \rme^{-r'})}\times\nonumber\\
&\qquad\qquad\qquad\qquad\qquad\Big.\times\left( g_t(v, r+r') - g_t(v, r) \right)\Big]\nonumber\\
- &\cf(v\geq-b_0)\Big[  \cf(v-r<-b_0) e^{\mu\max(a, c_0v)}\Big\{ 4\bar{n}\left(\ln (1 + \rme^{v-r})\right)^{-\frac{3}{2}}\int_{\delta_1}^r \rmd r'\ \left(\ln(1 + \rme^{v-r-r'})\right)^{1 - \alpha}\times\Big.\nonumber\\
&\qquad\qquad\qquad\qquad\qquad\times\frac{\rme^{-r'}}{1 - \rme^{-r'}}\left( 1 - \left(\frac{\ln (1 + \rme^{v-r-r'})}{\ln(1 + \rme^{v-r+r'})}\right)^{\alpha}\right)\left( g_t(v, r+r') - g_t(v, r) \right)\nonumber\\
&\ \ \Big.+\  4\bar{n}f(v-r)\left(\ln (1 + \rme^v)\right)^{-\frac{3}{2}}\int_{\delta_1}^r \rmd r'\ \ln(1 + \rme^{v+r'})\frac{\rme^{-r'}(1 - \rme^{-r-2r'})}{(1 + \rme^{v-r} + \rme^{-r'})}\times\left( g_t(v, r+r') - g_t(v, r) \right)\Big\}\nonumber\\
&\ \ + \cf(v-r\geq-b_0) \rme^{\mu\max(a, c_0v, v-r)}\Big\{ 4\bar{n}f(v-r)\left(\ln (1 + \rme^v)\right)^{-\frac{3}{2}}\int_{\delta_1}^r \rmd r'\ \ln(1 + \rme^{v+r'}) \frac{\rme^{-r'}(1 - \rme^{-2r'})}{(1 + \rme^{v-r} + \rme^{-r'})}\times\Big.\nonumber\\
&\qquad\qquad\qquad\qquad\qquad\qquad\qquad\qquad\qquad\times\left( g_t(v, r+r') - g_t(v, r) \right)\nonumber\\
&\ \ \ + 4\bar{n}f(v-r)\left(\ln (1 + \rme^v)\right)^{-\frac{3}{2}}\int_{\delta_1}^r \rmd r'\ \ln(1 + \rme^{v-r'}) \frac{\rme^{v-r'}(1 - \rme^{-r'})}{(1 + \rme^{v-r'} + \rme^{-r'})(1 + \rme^v + \rme^{-r'})}\times\nonumber\\
&\qquad\qquad\qquad\qquad\qquad\qquad\qquad\qquad\qquad\Big.\Big.\times\left( g_t(v, r+r') - g_t(v, r) \right)\Big\}\Big]\nonumber\\
\end{align}

\paragraph{2. A Lower Bound for $I_2{[}\Gamma_t^2{]}(v, r)$:}

\begin{align}\label{eq: I_2_Gamma2LowerBound}
I_2[\Gamma_t^2](v, r)&\geq \cf(v<-b_0)\Big[ \rme^{\mu a}\int_{\delta_1}^r \rmd r'\ K_3(v, v+r') \left( f(v) - f(v+r')  \right) g_t(v, r)\Big.\nonumber\\
&\ \ \ + \Big.\cf(r>a_1-v)\int_{\delta_1}^r \rmd r'\ K_3(v, v+r')f(v+r')\left( \rme^{\mu a} - \rme^{\mu c_0 v} \rme^{-(\frac{1}{2} - \mu c_0)r'}\right) g_t(v, r)\Big]\nonumber\\
&+ \cf(v\geq -b_0)\int_{\delta_1}^r \rmd r'\ K_3(v, v+r')f(v)\left(\rme^{\mu\max(a, v-r, c_0v)} - \rme^{-\frac{1}{2}r'}\rme^{\mu\max(a, v-r, c_0(v+r')}\right) g_t(v, r)\nonumber\\
& + I_2^- [\Gamma_t^2](v, r),
\end{align}
where
\begin{align}\label{defn2: I_2^-}
&I_2^-[\Gamma_t^2](v, r)\nonumber\\
=& - \rme^{\mu\max(a, v-r, c_0 v)} f(v) \int_{\delta_1}^r \rmd r'\ \left( K_3(v, v+r') - K_3^1(v, v-r')\right)\left( g_t(v+r', r+r') - g_t(v, r+r') \right)\nonumber\\
& - 4\bar{n}\ \rme^{\mu\max(a, v-r, c_0 v)} f(v)\left(\ln (1 + \rme^v)\right)^{-\frac{3}{2}}\int_{\delta_1}^r \rmd r'\ \ln(1 + \rme^{v+r'}) \frac{\rme^{-r'}(1 - \rme^{-2r'})}{1 + \rme^{v-r} + \rme^{-r'}}\left( g_t(v, r+r') - g_t(v, r) \right)\nonumber\\
& -  4\bar{n}\ \rme^{\mu\max(a, v-r, c_0 v)} f(v)\left(\ln (1 + \rme^v)\right)^{-\frac{3}{2}}\int_{\delta_1}^r \rmd r'\ \ln(1 + \rme^{v-r'}) \frac{\rme^{v-r'}(1 - \rme^{-r'})}{(1 + \rme^{v-r'} + \rme^{-r'})(1 + \rme^v + \rme^{-r'})}\times\nonumber\\
&\ \qquad\qquad\qquad\qquad\qquad\qquad\qquad\qquad\times\left( g_t(v, r+r') - g_t(v, r) \right)\nonumber
\end{align}

\subsubsection{Lower Bound for $I_3[\Gamma_t](v, r)$:}

\[I_3[\Gamma_t](v, r) = I_3^{(1)}[\Gamma_t](v, r) + I_3^{(2)}[\Gamma_t](v, r).\]
For $I_3$ we will just write out the relevant definitions explicitly.

\vspace{0.1in}

\textbf{Expression for $I_3^{(1)}[\Gamma_t](v, r):$}
\begin{align}
\textbf{a)}\  &I_3^{(1)}[\Gamma_t^1](v, r)\nonumber\\
=&\ \cf(v-r\leq -m_0) \rme^{\mu\max(a, c_0v)}f(v-r)\Big[ \int_{\tr}^{\infty} \rmd r'\ K_3^2(v-r, v-r-r') g_t(v, r)\Big.\nonumber\\
&\qquad\Big. + \int_0^{\tilde{r}} \rmd r'\ \Big( K_3^2(v-r, v-r-r') g_t(v, r) - (K_3^2(v-r, v-r-r') - K_3^2(v, v-r-r')) g_t(v, r+r') \Big)\Big]\nonumber\\
&\ + \cf(-m_0<v-r<m_0) \rme^{\mu\max(a, v-r, c_0 v)} f(v-r) \int_0^{\infty} \rmd r'\ K_3^2(v-r, v-r-r') g_t(v, r)\nonumber\\
&\ + \cf(v-r\geq m_0)f(v-r)\Big[ \cf(v\leq 3r)\Big\{\rme^{\mu\max(a, v-r, c_0 v)} \int_{v-r}^{\infty} dr' K_3^2(v-r, v-r-r') g_t(v, r) \Big.\Big.\nonumber\\
&\ \ \  \ \ + \int_0^{v-r} \rmd r'\ \Big( \rme^{\mu\max(a, v-r, c_0 v)} K_3^2(v-r, v-r-r') g_t(v, r)\Big.\nonumber\\
&\ \ \ \qquad \qquad \Big.\Big.- \rme^{\mu\max(a, v-r-r', c_0 v)}\left(K_3^2(v-r, v-r-r') - K_3^2(v, v-r-r')\right) g_t(v, r+r')\Big)\Big\}\nonumber\\
& \ \ \ \ \ + \cf(v>3r)\Big\{ \rme^{\mu\max(a, v-r, c_0 v)}\int_0^{v-r-c_0v} \rmd r'\ K_3^2(v-r, v-r-r') g_t(v, r)  \Big.\nonumber\\
&\ \ \ \ \ \ \ + \int_{v-r-c_0v}^{\infty} \rmd r'\ \Big( \rme^{\mu\max(a, v-r, c_0 v)} K_3^2(v-r, v-r-r') g_t(v, r)\Big.\nonumber\\
&\ \ \ \qquad\qquad \Big.\Big.\Big.- \rme^{\mu\max(a, v-r-r', c_0 v)}\left(K_3^2(v-r, v-r-r') - K_3^2(v, v-r-r')\right) g_t(v, r+r')\Big)\Big\}\Big]\nonumber\\
&\ + A_1^-[\Gamma_t^1](v, r) + A_2^-[\Gamma_t^1](v, r)  + A_3^-[\Gamma_t^1](v, r)  + A_4^-[\Gamma_t^1](v, r),
\end{align}
where
\begin{align}\label{defn: A_1^-}
&A_1^-[\Gamma_t^1](v, r) \nonumber\\
=& -\cf(v-r\leq -m_0)\cf(v<-b_0)\cf(\tilde{r}>r+\delta_1)  \int_{r+\delta_1}^{\tilde{r}} \rmd r'\ \left( K_3^2(v-r, v-r-r') - K_3^2(v, v-r-r') \right)\times\nonumber\\
&\qquad\qquad\qquad\qquad\qquad\qquad\qquad\qquad\qquad\times \rme^{\mu\max(a, c_0v)}\left( f(v-r-r') - f(v-r)  \right) g_t(v, r+r')
\end{align}
\begin{align}\label{defn: A_2^-}
&A_2^-[\Gamma_t^1](v, r) \nonumber\\
=& -\cf(v-r\leq -m_0) \int_{\min(\delta_1, \tilde{r})}^{\min(r, \tilde{r})} \rmd r'\ \left( K_3^2(v-r, v-r-r') - K_3^2(v, v-r-r') \right)\rme^{\mu\max(a, c_0v)}\times\nonumber\qquad\qquad\qquad\\
&\qquad\qquad\qquad\qquad\qquad\qquad\times\left( f(v-r-r') - f(v-r)  \right) g_t(v, r+r')
\end{align}
\begin{align}\label{defn: A_3^-}
&A_3^-[\Gamma_t^1](v, r)\nonumber\\
=& -\cf(v-r\geq m_0)\cf(v>3r) \int_{v-r}^{\infty} \rmd r'\ \left( K_3^2(v-r, v-r-r') - K_3^2(v, v-r-r') \right)\times\nonumber\\
&\qquad\qquad\qquad\qquad\qquad\qquad\qquad\times \rme^{\mu\max(a, c_0v)}\left( f(v-r-r') - f(v-r)  \right) g_t(v, r+r')\qquad\qquad
\end{align}
\begin{align}\label{defn: A_4^-}
&A_4^-[\Gamma_t^1](v, r) \nonumber\\
=& - \cf(v-r\leq -m_0) \rme^{\mu\max(a, c_0v)}\Bigg[ \int_0^{\min(\delta_1, \tilde{r})} \rmd r'\ \left( K_3^2(v-r, v-r-r') - K_3^2(v, v-r-r') \right)\times\Bigg.\nonumber\qquad\qquad\qquad\\
&\qquad\qquad\qquad\qquad\qquad\qquad\qquad\qquad\qquad\times \left( f(v-r-r') - f(v-r)  \right) g_t(v, r+r')\nonumber\\
& \ \ + \cf(v<-b_0)\int_r^{\min(\tr, r+\delta_1)} \rmd r'\  \left( K_3^2(v-r, v-r-r') - K_3^2(v, v-r-r') \right)\times\nonumber\\
&\qquad\qquad\qquad\qquad\qquad\qquad\Bigg.\times  \left( f(v-r-r') - f(v-r)  \right) g_t(v, r+r')\Bigg].
\end{align}
\begin{align}\label{defn: I_3^1Gamma2}
\textbf{b)}\  &I_3^{(1)}[\Gamma_t^2](v, r)\nonumber\\
=& \cf(v-r\leq -m_0)f(v)\rme^{\mu\max(a, c_0v)}\Bigg[  \int_{\tr}^{\infty} \rmd r'\ K_3^2(v-r, v-r-r')g_t(v, r)    \Bigg.\nonumber\\
&\qquad\Bigg. + \int_0^{\tr} \rmd r'\ \Big\{ K_3^2(v-r, v-r-r') g_t(v, r)- \left( K_3^2(v-r, v-r-r') - K_3^2(v, v-r-r') \right)g_t(v, r+r')\Big\}\Bigg]\nonumber\\
+& \cf(-m_0< v-r < m_0) f(v) \rme^{\mu\max(a, v-r, c_0v)}\int_0^{\infty}\rmd r' K_3^2(v-r, v-r-r') g_t(v, r)\nonumber\\
+& \cf(v-r\geq m_0) f(v)\Bigg[ \cf(v\leq 3r)\Bigg\{ \int_0^{v-r} \rmd r'\ \left( \rme^{\mu\max(a, v-r, c_0v)} - \rme^{\mu\max(a, v-r-r', c_0v)} \right)\times\nonumber\\
&\qquad\qquad\qquad\qquad\qquad \times K_3^2(v-r, v-r-r')g_t(v, r)\Bigg.\Bigg.\nonumber\\
&\ \ \ \ \ + \int_0^{v-r} \rmd r'\ \rme^{\mu\max(a, v-r-r', c_0v)}\Big(K_3^2(v-r, v-r-r') g_t(v, r)- \Big.\nonumber\\
&\qquad\qquad\qquad\qquad\qquad\Big. - \left( K_3^2(v-r, v-r-r') - K_3^2(v, v-r-r') \right)g_t(v, r+r')\Big)\nonumber\\
&\ \ \ \ \ \Bigg.+ \rme^{\mu\max(a,c_0v, v-r)}\int_{v-r}^{\infty} \rmd r'\ K_3^2(v-r, v-r-r') g_t(v, r)\Bigg\}\nonumber\\
&\ \ \ + \cf(v> 3r)\Bigg\{ \rme^{\mu\max(a, c_0v, v-r)}\int_0^{v-r-c_0 v} \rmd r'\ K_3^2(v-r, v-r-r') g_t(v, r)\Bigg.\nonumber\\
&\ \ \ \ \ + \int_{v-r-c_0v}^{\infty} \rmd r'\ K_3^2(v-r, v-r-r')\left( \rme^{\mu\max(a, c_0v, v-r)} - \rme^{\mu\max(a, c_0v, v-r-r')} \right)g_t(v, r)\nonumber \\
&\ \ \ \ \ +   \int_{v-r-c_0v}^{\infty} \rmd r'\ \rme^{\mu\max(a, v-r-r', c_0v)}\Big(K_3^2(v-r, v-r-r') g_t(v, r)\Big.\nonumber\\
&\qquad\qquad\qquad\qquad\qquad\Bigg.\Bigg.\Big. - \left( K_3^2(v-r, v-r-r') - K_3^2(v, v-r-r') \right)g_t(v, r+r')\Big)\Bigg\}\Bigg].\nonumber\\
\end{align}
\textbf{Expression for $I_3^{(2)}[\Gamma_t](v, r):$}
\begin{align}\label{defn: I_3^2Gamma1}
\textbf{a)}\ 	& I_3^{(2)}[\Gamma_t^1](v, r)\nonumber\\
=&\ \cf(v\leq -m_0)\rme^{\mu a}f(v-r)\Big[ \cf(v+r<-b_0)\int_0^r \rmd r'\ K_3^2(v, v-r')\left( g_t(v, r) - g_t(v-r', r-r') \right)\Big.\nonumber\\
& \ \ \qquad +  \cf(v+r\geq-b_0)\int_0^{-v-b_0} \rmd r'\ K_3^2(v, v-r')\left( g_t(v, r) - g_t(v-r', r-r') \right)\nonumber\\
&\Big. + \cf(v+r\geq -b_0)\int_{-v-b_0}^r \rmd r'\ K_3^2(v, v-r') g_t(v, r)\Big]\nonumber\\
& + \cf(-m_0<v<m_0)\rme^{\mu\max(a, v-r, c_0v)}f(v-r)\int_0^r \rmd r'\ K_3^2(v, v-r')g_t(v, r)\nonumber\\
& + \cf(v\geq m_0)f(v-r)\Big[ \cf(v\leq r) \rme^{\mu\max(a, v-r, c_0v)}\int_0^r \rmd r'\ K_3^2(v, v-r') g_t(v, r)\Big.\nonumber\\
&\ \qquad + \cf(0< v-r < c_0v) \Big\{ \rme^{\mu\max(a, v-r, c_0v)}\int_0^r \rmd r'\ K_3^2(v, v-r') g_t(v, r)\Big.\nonumber\\
&\ \ \ \ \ \qquad\Big. - \int_{v - c_0^{-1}(v-r)}^r \rmd r'\ K_3^2(v, v-r') \rme^{\mu\max(a, v-r, c_0(v-r'))} g_t(v-r', r-r')\Big\}\nonumber\\
&\ \  \Big.+ \cf(v-r\geq c_0 v)\int_0^r \rmd r'\ K_3^2(v, v-r')\Big( \rme^{\mu\max(a, v-r, c_0v)} g_t(v, r) - \rme^{\mu\max(a, v-r, c_0(v-r'))} g_t(v-r', r-r') \Big)\Big]
\end{align}
\begin{align}\label{defn: I_3^2Gamma2}
\textbf{b)}\  &I_3^{(2)}[\Gamma_t^2](v, r)\nonumber\\
=&\ \cf(v\leq -m_0)\rme^{\mu a}\Big[ -\cf(v+r<-b_0)\int_0^r \rmd r'\ K_3^2(v, v-r')\left( f(v-r') - f(v) \right)g_t(v-r', r-r')\Big.\nonumber\\
&\ \ \ -\cf(v+r\geq-b_0)\int_0^{-v-b_0} \rmd r'\ K_3^2(v, v-r')\left( f(v-r') - f(v) \right)g_t(v-r', r-r')\nonumber\\
&\ \ \ + \cf(v+r<-b_0)f(v)\int_0^r \rmd r'\ K_3^2(v, v-r')\left( g_t(v, r) - g_t(v-r', r-r')\right)\nonumber\\
&\ \ \ + \cf(v+r\geq -b_0)f(v)\Big\{ \int_{-v-b_0}^r \rmd r'\ K_3^2(v, v-r') g_t(v, r)\Big.\nonumber\\
&\ \qquad \qquad\Big.\Big. + \int_0^{-v-b_0} \rmd r'\ K_3^2(v, v-r')\left( g_t(v, r) - g_t(v-r', r-r') \right) \Big\}\Big]\nonumber\\
& + \cf(-m_0 < v < m_0) \rme^{\mu\max(a, v-r, c_0v)} f(v)\int_0^r \rmd r'\ K_3^2(v, v-r') g_t(v, r)\nonumber\\
& + \cf(v\geq m_0)\Big[ \cf(v\leq r)\rme^{\mu\max(a, v-r, c_0v)} f(v) \int_0^r \rmd r'\ K_3^2(v, v-r') g_t(v, r)\Big.\nonumber\\
& \ \ + \cf(0<v-r<c_0v)\Big\{ \rme^{\mu\max(a, v-r, c_0v)} f(v) \int_0^{v - c_0^{-1}(v-r)} \rmd r'\ K_3^2(v, v-r') g_t(v, r)\Big.\nonumber\\
&\ \ \  \Big.+ \int_{v- c_0^{-1}(v-r)}^r \rmd r'\ K_3^2(v, v-r') f(v)\Big( \rme^{\mu\max(a, v-r, c_0v)} g_t(v, r) - \rme^{\mu\max(a, v-r, c_0(v-r'))} g_t(v-r', r-r')  \Big)\Big\}\nonumber\\
&\ \ \ \ \Big.+ \cf(v-r\geq c_0v)\rme^{\mu\max(a, v-r)}f(v)\int_0^r \rmd r'\ K_3^2(v, v-r')\left(  g_t(v, r) - g_t(v-r', r-r')\right)\Big].
\end{align}
Let us recall that the correct asymptotic behavior of $\mcL_u\Gamma_t$ for large, positive values of $v$ comes from $\cf(v-r\geq m_0)I_3^{(1)}[\Gamma_t](v, r)$ (when $v$ and $(v-r)$ are comparable) and $\cf(v\geq m_0)I_3^{(2)}[\Gamma_t](v, r)$ (when $v$ is much bigger than $(v-r)$).

We will look at $I_4$ later because we already know it inherits a degree of ``smallness'' from the length of the interval of integration.

	\subsection{Combining the Estimates to arrive at the Main Lower Bound:}

In this section we combine the terms written out in the previous subsection, in a way which will enable us to find a useful lower bound on $\mcL_u\Gamma_t(v, r).$ As mentioned already in Lemma \ref{lemma: mcLuEstimate}, $\mcJ_0[\Gamma_t](v, r)$ from the lower bound for $I_1[\Gamma_t](v, r)$ and $\mathcal{C}_3[\Gamma_t^1](v, r)$, defined in (\ref{def: mcC3}), from the lower bound for $I_2[\Gamma_t^1](v, r)$, contribute to $\mcL_u\Gamma_t(v, r)$ the correct asymptotic behavior with point singularity.

Our main objective now is to tackle the negative terms obtained in the estimates above by using some suitable combinations of integrals. We enumerate these combinations as Comb.$1$, Comb.$2$ etc. These combinations yield useful lower bounds for sums of the $I_i$'s, which eventually lead to our main estimate. In what follows these lower bounds are denoted by LB$1$, LB$2$ etc. In other words, for ease of referencing, we will denote the most important estimates as LB$1$, LB$2$ etc., the particular combinations of terms leading to these estimates by Comb.$1$, Comb.$2$ etc. and reserve the usual numbering by section for everything else.

Note that terms of the form $(g_t(., r+r') - g_t(., r))$ generate an extra exponential decay of $e^{-\kappa r}$.
This additional ``smallness'' makes it easy for us to deal with such negative terms. So we will focus first on those negative terms which do not contain differences of the form $(g_t(., r+r') - g_t(., r))$. 

We will first look at the first two negative terms in (\ref{defn: mcI}). We combine the first negative term  with  terms from
$\mathcal{I}^{(2)}[\Gamma_t^2](v, r)$ and $I_3^1[\Gamma_t^2](v, r)$ as follows:
\begin{align*}
\tag{\textbf{Comb.1}}
&- \rme^{\mu\max(a, c_0v)}\int_{r+\delta_1}^{\infty} \rmd r'\ K_3^1(v, v-r')\left( f(v-r-r') - f(v-r')\right) g_t(v, r') \\ 
& + \cf(v-r\leq -m_0)\rme^{\mu\max(a, c_0v)} f(v) \int_{\tr}^{\infty} \rmd r'\ K_3^2(v-r, v-r-r') g_t(v, r)\\
& + \cf(-m_0<v-r<m_0)\rme^{\mu\max(a, v-r, c_0v)} f(v) \int_{r+\delta_1}^{\infty} \rmd r'\ K_3^2(v-r, v-r-r') g_t(v, r)\\
& + \cf(v<-b_0)\cf(r+\delta_1< a_1-v)\int_{r+\delta_1}^{\infty} \rmd r'\ \ol{K_3}^2(v-r, v-r+r')\left( f(v-r+r') - f(v+r') \right)\times\\
&\qquad\qquad\qquad\qquad\qquad\qquad\qquad\times \rme^{\mu\max(a, c_0(v-r+r'))} g_t(v-r+r', r')\\
\geq & -\cf(v\geq 0)\rme^{\mu\max(a, c_0v)}\int_{r+\delta_1}^{2r + \delta_1} \rmd r'\ K_3^1(v, v-r')\left( f(v-r-r') - f(v-r')\right) g_t(v, r')\\
&- \cf(v-r\geq m_0)\rme^{\mu\max(a, c_0v)}\int_{2r+\delta_1}^{\infty} \rmd r'\ K_3^1(v, v-r')\left( f(v-r-r') - f(v-r')\right) g_t(v, r') + B_1^+[\Gamma_t^2](v, r),
\end{align*}
where
\begin{align}\label{defn: B1+}
&B_1^+[\Gamma_t^2](v, r)\nonumber\\
=&\  \cf(v<-b_0)\rme^{\mu a}f(v) g_t(v, r)\Bigg[ \cf(v-r<-m_0) \int_{\max(-v-b_0, r)}^{\infty} \rmd r'\ K_3^2(v-r, v-r-r')\times\Bigg.\nonumber\\
&\qquad\qquad\qquad\qquad\qquad\qquad\times\Big( 1 - \rme^{-(\frac{122}{250}- 2\alpha)r'} + \rme^{-(\frac{122}{250}-\alpha)r}\rme^{-r'} \frac{1 - \rme^{-(1 - \alpha)r'}}{1 - \rme^{-r'}}\Big) \nonumber\\
&\ \ \ \ \ \quad + \cf(v-r\geq-m_0)\int_{r+\delta_1}^{\infty} \rmd r'\ K_3^2(v-r, v-r-r')\times\nonumber\\
&\qquad\qquad\qquad\qquad\qquad\qquad\qquad\times\Bigg.\Big( 1 - \rme^{-(\frac{122}{250} - 2\alpha)r'} + \rme^{-(\frac{122}{250}-\alpha)r} \rme^{-r'} \frac{1 - \rme^{-(1 - \alpha)r'}}{1 - \rme^{-r'}}\Big) \Bigg]\nonumber\\
+ & \cf(v\geq -b_0)\Bigg[ \rme^{\mu a}f(v)g_t(v, r)\cf(-b_0\leq v<0)\int_{r+\delta_1}^{\infty}\rmd r'\ K_3^2(v-r, v-r-r')\Big( 1 - \rme^{-2(\frac{1}{3} - \alpha)r'} \rme^{-(\frac{1}{3}+\alpha)(r'-r)}\Big.\Bigg.\nonumber\\
&\qquad\qquad\qquad\qquad\qquad\qquad\qquad \qquad\qquad \Big. + \rme^{-(r'-r)}\rme^{-(\frac{2}{3}-\alpha)r}\frac{1 - \rme^{-(1-\alpha)r'}}{1 - \rme^{-r'}}\Big)\nonumber\\
& + \cf(v\geq 0)\cf(v-r<m_0)f(v)g_t(v, r)\Bigg\{ \rme^{\mu\max(a, c_0v)}\int_{r+\delta_1}^{\infty} \rmd r'\ K_3^2(v-r, v-r-r')\times\Bigg.\nonumber\\
&\qquad\qquad\qquad\qquad\qquad\qquad\qquad\qquad\qquad\times\Big( 1 - \left(\frac{\ln(1 + \rme^{v-r})}{\ln(1 + \rme^v)}\right)^{\frac{3}{2}} \rme^{-(1 - 2\alpha)r}\rme^{-(1-\alpha)r'}\Big) \nonumber\\
& \ + 4\bar{n}\ \rme^{\mu\max(a, c_0v)}\left(\ln(1 + \rme^v)\right)^{-\frac{3}{2}}\int_{r+\delta_1}^{\infty} \rmd r'\ \ln(1 + \rme^{v-r-r'})\frac{\rme^{v-r-r'} + 2\rme^{-r'}}{1 + \rme^{v-r-r'} + \rme^{-r'}}\frac{1 - \rme^{-(1 - \alpha)(r+r')}}{1 - \rme^{-r-r'}} \rme^{-r'} \rme^{-(1-\alpha)r}\nonumber\\
&\ \ \ \  \Bigg.\Bigg. + \int_{r+\delta_1}^{\infty} \rmd r'\ K_3^2(v-r, v-r-r')\Big( \rme^{\mu\max(a, v-r, c_0v)} - \rme^{\mu\max(a, c_0v)}  \Big)\Bigg\}\Bigg]
\end{align}
Now we look at the second term in (\ref{defn: mcI}). We combine this with a term from
$\mathcal{I}^{(2)}[\Gamma_t^2](v, r)$ and $B_1^+[\Gamma_t^2](v, r)$ as follows. This is the second combination Comb.$2$.
\begin{align*}
\tag{\textbf{Comb.2}}
&- \rme^{\mu\max(a, c_0v)}\int_{r+\delta_1}^{\infty} \rmd r'\ \left( K_3^1(v-r, v-r-r') - K_3^1(v, v-r')  \right)\left( f(v-r-r') - f(v-r) \right) g_t(v, r')\\
& + \rme^{\mu a}\cf(v<-b_0)\cf(a_1-v> r+\delta_1)\int_{r+\delta_1}^{a_1-v} \rmd r'\ \ol{K_3}^1(v, v+r')f(v+r')\times\\
&\qquad\qquad\qquad\qquad\qquad\qquad\qquad\qquad\times\left( g_t(v, r) + g_t(v-r+r', r') - g_t(v+r', r+r') \right)\\
& + B_1^+[\Gamma_t^2](v, r)\\
\geq& -\cf(v-r\geq m_0)\rme^{\mu\max(a, c_0v)}\int_{r+\delta_1}^{\infty} \rmd r'\ \left( K_3^1(v-r, v-r-r') - K_3^1(v, v-r')  \right)\times\\
&\qquad\qquad\qquad\qquad\qquad\qquad\qquad\times\left( f(v-r-r') - f(v-r) \right) g_t(v, r')\\
& \qquad +E_1[\Gamma_t^2](v, r) +  E_2[ \Gamma_t^2](v, r),
\end{align*}
where
\begin{align}\label{defn: E_1Gamma2}
&E_1[\Gamma_t^2](v, r)\nonumber\\
=& \cf(v<-b_0)\rme^{\mu a}\Bigg[ \cf(-v-b_0> r+\delta_1)\Bigg\{ \int_{-v-b_0}^{a_1-v} \rmd r'\ \ol{K_3}^1(v, v+r') f(v+r')\left( 2 - (1 + \rme^{-\kappa r})^{\gamma_t(v, r)}\right) g_t(v, r)  \Bigg.\Bigg.\nonumber\\
&\  + \int_{r+\delta_1}^{-v-b_0} \rmd r'\ \ol{K_3}^1(v, v+r')f(v+r')\left( 1 - \rme^{-\frac{124}{125}(1 - \alpha)r'} \right) \left( 1 - \rme^{-\kappa r} \right)^{2\gamma_1}\nonumber\\
&\ + f(v)\int_{-v-b_0}^{\infty} \rmd r'\ K_3^2(v-r, v-r-r')\Big( \left( 1 - \rme^{-(\frac{122}{250}- 2\alpha)r'} \right)\left( 1 - \rme^{-(1- 2\alpha)r'}\right)\Big.\nonumber\\
&\qquad\qquad\qquad\qquad\qquad\Bigg.\Big. + \rme^{-\frac{116}{750}r'}\rme^{-(1 - 2\alpha)r'}( 1 - \rme^{-\alpha r'})\frac{1 - \rme^{-\frac{634}{750}r'}}{1 - \rme^{-r'}}\Big) g_t(v, r)\Bigg\}\nonumber\\
&\  + \cf(-v-b_0\leq r+\delta_1)\rme^{\mu a}\Bigg\{ \cf(r+\delta_1\leq a_1-v) \int_{r+\delta_1}^{a_1-v} \rmd r'\ \ol{K_3}^1(v, v+r')f(v+r')\times\Bigg.\nonumber\\
&\qquad\qquad\qquad\qquad\qquad\times\left(2 - (1 + \rme^{-\kappa r})^{\gamma_t(v, r)} \right) g_t(v, r)\nonumber\\
&\ \ \Bigg. + f(v) \int_{r+\delta_1}^{\infty} \rmd r'\ K_3^2(v-r, v-r-r') \left( 1 - \rme^{-(\frac{122}{250}-2\alpha)r'} \right)\left( 1 - \rme^{-(1 - 2\alpha)r'} \right) g_t(v, r')\Bigg\}\nonumber\\
&\ \ \ \Bigg.+ \rme^{\mu a}f(v)\int_{\max(-v-b_0, r+\delta_1)}^{\infty} \rmd r'\ K_3^2(v-r, v-r-r')\rme^{-r'}\rme^{-(\frac{122}{250}-\alpha)r}\frac{1 - \rme^{-(1-\alpha)r'}}{1 - \rme^{-r'}} g_t(v, r)\Bigg],
\end{align}
and,
\begin{align}\label{defn: E_2Gamma2}
&E_2[\Gamma_t^2](v, r)\nonumber\\
=& \cf(v\geq -b_0)\Bigg[ \cf(-b_0\leq v<0)\rme^{\mu a}f(v)\int_{r+\delta_1}^{\infty}\rmd r'\ K_3^2(v-r, v-r-r')\Bigg( \left(1 - \rme^{-2(\frac{1}{3} - \alpha)r'}\right)\left(  1 - \rme^{-(1 - 2\alpha)r'}\right)\Bigg.\Bigg.\nonumber\\
&\qquad\qquad\qquad\qquad\qquad\Bigg. + \rme^{-(1 - \alpha)r'}\left(1 - \rme^{-(\frac{2}{3} - 3\alpha)r'}\right) + \rme^{-(r'-r)}\rme^{-(\frac{2}{3}-\alpha)r}\frac{1 - \rme^{-(1-\alpha)r'}}{1 - \rme^{-r'}}\Bigg)g_t(v, r)\nonumber\\
&  + \cf(v\geq 0)\cf(v-r< m_0)f(v) g_t(v, r)\Bigg\{ \rme^{\mu\max(a, c_0v)}\left(\frac{\ln (1 + \rme^{v-r})}{\ln(1 + \rme^v)}\right)^{\frac{3}{2}}\int_{r+\delta_1}^{\infty} \rmd r'\ K_3^2(v-r, v-r-r') \Bigg. \times\nonumber\\
&\qquad\qquad\qquad\qquad\qquad\qquad\qquad\qquad\qquad\qquad\qquad\times\rme^{-r'} \rme^{-(1 - \alpha)r}\frac{1 - \rme^{-(1 - \alpha)(r+r')}}{1 - \rme^{-(r+r')}}\nonumber\\
&\ \  + \cf(v-r\leq \max(a, c_0v))\int_{r+\delta_1}^{\infty} \rmd r'\ K_3^2(v-r, v-r-r')\Big( \left(1 - \rme^{-(1 - 2\alpha)r'}\right)\left( 1 - \rme^{-(1 - \alpha)r'}\rme^{-(1 - 2\alpha)r} \right)    \Big.\Big.\nonumber\\
&\qquad\qquad\qquad\qquad\qquad \qquad\qquad\qquad\Big. + \rme^{-(1-\alpha)r'}\left( 1 - \rme^{-(1-2\alpha)(r+r')}  \right)\Big)\rme^{\mu\max(a, c_0v)}\nonumber\\
& + \cf(v-r>\max(a, c_0v))\left( \rme^{\mu (v-r)} - \rme^{\mu\max(a, c_0v)}\right)\int_{r+\delta_1}^{\infty} \rmd r'\ K_3^2(v-r, v-r-r')\left( 1 - \rme^{-(1-2\alpha)r} \rme^{-(1 - \alpha)r'}  \right)\nonumber\\
& + \cf(v-2r-\delta_1<\max(a, c_0v)<v-r)  \int_{r+\delta_1}^{\infty} \rmd r'\ K_3^2(v-r, v-r-r') \rme^{-\alpha r'}\left( 1 - \rme^{-\alpha r} \right.\nonumber\\
&\qquad\qquad\qquad\qquad\qquad\qquad\qquad\qquad\qquad\qquad\qquad\left.+ \rme^{-\alpha r}(1 - \rme^{-(1 - 3\alpha)(r+r')})  \right)\rme^{\mu\max(a, c_0v)}\nonumber\\
& + \cf(v-2r-\delta_1\geq\max(a, c_0v)) \int_{r+\delta_1}^{\infty} \rmd r'\ K_3^2(v-r, v-r-r')\Big( \left(1 - \rme^{-(1 - 2\alpha)r'}\right)\left( 1 - \rme^{-(1 - \alpha)r'}\rme^{-(1 - 2\alpha)r} \right)    \Big.\Big.\nonumber\\
&\qquad\qquad\qquad\qquad\qquad \qquad\qquad\qquad\Bigg.\Bigg.\Big. + \rme^{-(1-\alpha)r'}\left( 1 - \rme^{-(1-2\alpha)(r+r')}  \right)\Big)\rme^{\mu\max(a, c_0v)}\Bigg\}\Bigg]\nonumber\\
\end{align}
Clearly, what the combinations Comb.$1$ and Comb.$2$ do to the two negative terms of $\mathcal{I}$, is push them away from the point singularities into regions where $v-r>0$ and $v>0$.

\vspace{0.1in}

We now combine a term from $E_1[\Gamma_t^2](v, r)$ with $A_1^{-1}[\Gamma_t^1](v, r)$ as follows, in Comb.$3$:

\begin{align*}
\tag{\textbf{Comb.3}}
&A_1^{-}[\Gamma_t^1](v, r) \\
& +\cf(v<-b_0)\cf(-v-b_0>r+\delta_1)\rme^{\mu a}\int_{r+\delta_1}^{-v-b_0}\rmd r'\ \ol{K_3}^1(v, v+r')f(v+r')\left( 1 - \rme^{-\frac{124}{125}(1 - \alpha)r'}\right)\left( 1 - \rme^{-\kappa r} \right)^{2\gamma_1}\\
\geq& \cf(r+\delta_1<-v-b_0) \Bigg[\cf(v-r<-m_0)4\bar{n}\left(\ln(1 + \rme^v)\right)^{-\frac{3}{2}}\int_{r+\delta_1}^{-v-b_0} \rmd r'\ \frac{\rme^{v-r'}(\ln(1 + \rme^{v+r'}))^{1-\alpha}}{(1 + \rme^v + \rme^{-r'})(1 + \rme^{v-r-r'} + \rme^{-r'})}\times   \Bigg.\\
&\qquad\qquad\qquad\qquad\qquad\qquad\qquad\qquad\times ( 1 - \rme^{-\frac{1}{2}r'})(1 - \rme^{-(r'+r)})\left( 1 - \rme^{-\frac{124}{125}(1 - \alpha)r'}\right)(1 - \rme^{-\kappa r})^{2\gamma_1}\\
&\ \ \ \ \ \Bigg. + \cf(v-r\geq-m_0)\int_{r+\delta_1}^{-v-b_0} \rmd r'\ \ol{K_3}^1(v, v+r') f(v+r')\left( 1 - \rme^{-\frac{124}{125}(1 - \alpha)r'} \right)(1 - \rme^{-\kappa r})^{2\gamma_1}\Bigg]\rme^{\mu a}.
\end{align*}

\vspace{0.1in}
Let us define:
\begin{align}\label{defn: E'_1Gamma2}
&E'_1[\Gamma_t^2](v, r)\nonumber\\
=&\ E_1[\Gamma_t^2](v, r) -\cf(-v-b_0>r+\delta_1)\rme^{\mu a}\int_{r+\delta_1}^{-v-b_0}\rmd r'\ \ol{K_3}^1(v, v+r')f(v+r')\left( 1 - \rme^{-\frac{124}{125}(1 - \alpha)r'}\right)\left( 1 - \rme^{-\kappa r} \right)^{2\gamma_1}\nonumber\\
+ & \cf(-v-b_0>r+\delta_1) \Bigg[\cf(v-r<-m_0)4\bar{n}\left(\ln(1 + \rme^v)\right)^{-\frac{3}{2}}\int_{r+\delta_1}^{-v-b_0} \rmd r'\ \frac{\rme^{v-r'}(\ln(1 + \rme^{v+r'}))^{1-\alpha}}{(1 + \rme^v + \rme^{-r'})(1 + \rme^{v-r-r'} + \rme^{-r'})}\times   \Bigg.\nonumber\\
&\qquad\qquad\qquad\qquad\qquad\qquad\qquad\qquad\times ( 1 - \rme^{-\frac{1}{2}r'})(1 - \rme^{-(r'+r)})\left( 1 - \rme^{-\frac{124}{125}(1 - \alpha)r'}\right)(1 - \rme^{-\kappa r})^{2\gamma_1}\nonumber\\
&\ \ \ \ \ \Bigg. + \cf(v-r\geq-m_0)\int_{r+\delta_1}^{-v-b_0} \rmd r'\ \ol{K_3}^1(v, v+r') f(v+r')\left( 1 - \rme^{-\frac{124}{125}(1 - \alpha)r'} \right)(1 - \rme^{-\kappa r})^{2\gamma_1}\Bigg]\rme^{\mu a}.\nonumber\\
\end{align}
This means we can put together the combinations (Comb.$1$), (Comb.$2$) and (Comb.$3$), and obtain the following lower bound:
\begin{align*}
\tag{\textbf{LB1}}
&\mathcal{I}[\Gamma_t^1 + \Gamma_t^2](v, r) + I_3^{(1)}[\Gamma_t^2](v, r) + A_1^{-}[\Gamma_t^1](v, r) \\
\geq&\ B_1^-[\Gamma_t^1](v, r) + B_2[\Gamma_t^2](v, r) + \mathcal{I}^{(1)}[\Gamma_t^1](v, r) + B_3[\Gamma_t^2](v, r) + E_1'[\Gamma_t^2](v, r) + E_2[\Gamma_t^2](v, r) + \mathcal{I}^{(3)}[\Gamma_t^2](v, r),
\end{align*}
where
\begin{align}\label{eq: B1-defn}
& B_1^-[\Gamma_t^1](v, r)\nonumber\\
=& - \cf(v-r\geq m_0)e^{\mu\max(a, c_0v)}\Big[\int_{2r+\delta_1}^{\infty} \rmd r'\ K_3^1(v, v-r')\left( f(v-r-r') - f(v-r')\right) g_t(v, r')\Big.\nonumber\\
&\qquad \Big.+ \int_{r+\delta_1}^{\infty} \rmd r'\ \left( K_3^1(v-r, v-r-r') - K_3^1(v, v-r')  \right)\left( f(v-r-r') - f(v-r) \right) g_t(v, r')\Big]\nonumber\\
&\ - \cf(v\geq 0)e^{\mu\max(a, c_0v)}\int_{r+\delta_1}^{2r+\delta_1} \rmd r'\ K_3^1(v, v-r')\left( f(v-r-r') - f(v-r')\right) g_t(v, r')\nonumber\\
&\ - e^{\mu\max(a, c_0v)}\int_{r+\delta_1}^{\infty} \rmd r'\ K_3^1(v-r, v-r-r')\left( f(v-r-r') - f(v-r) \right)(g_t(v, r+r') - g_t(v, r')),
\end{align}
\begin{align}\label{defn: B2}
&B_2[\Gamma_t^2](v, r)\nonumber\\
=&\ I_3^{(1)}[\Gamma_t^2](v, r) - \cf(v-r\leq -m_0)\rme^{\mu\max(a, c_0v)} f(v) \int_{\tr}^{\infty} \rmd r'\ K_3^2(v-r, v-r-r') g_t(v, r)\nonumber\qquad\qquad\qquad\\
&\ \ - \cf(-m_0<v-r<m_0)e^{\mu\max(a, v-r, c_0v)} f(v) \int_{r+\delta_1}^{\infty} \rmd r'\ K_3^2(v-r, v-r-r') g_t(v, r),
\end{align}
and,
\begin{align}\label{defn: B3}
&B_3[\Gamma_t^2](v, r)\nonumber\\
=&\ \mathcal{I}^{(2)}[\Gamma_t^2](v, r) - \cf(v<-b_0)\cf(r+\delta_1< a_1-v)\Big[\int_{r+\delta_1}^{\infty} \rmd r'\ \ol{K_3}^2(v-r, v-r+r')\left( f(v-r+r') - f(v+r') \right)\Big.\times\nonumber\\
&\qquad\qquad\qquad\qquad\qquad\qquad\qquad\qquad\qquad\times \rme^{\mu\max(a, c_0(v-r+r'))} g_t(v-r+r', r')\nonumber\\
&\qquad \Big.+ \rme^{\mu a}\int_{r+\delta_1}^{a_1-v} \rmd r'\ \ol{K_3}^1(v, v+r')f(v+r')\left( g_t(v, r) + g_t(v-r+r', r') - g_t(v+r', r+r') \right)\Big]\nonumber\\
\end{align}

We will now turn to $\mathcal{J}_1[\Gamma_t^1] (v, r)$ and $\mathcal{J}_2[\Gamma_t^2] (v, r)$, defined
in (\ref{defn: J_1}) and (\ref{defn: J_2}). It is not difficult to obtain the following lower bound on the sum of these terms:
\begin{align*}
&\mathcal{J}_1[]\Gamma_t^1](v, r) + \mathcal{J}_2[\Gamma_t^2](v, r)\\
\geq&\ \cf(v\leq-b_0)\ 4\bar{n}\Big[ -\rme^{\mu a}f(v-r)\int_{r+\delta_1}^{r+a_1-v} \rmd r'\ \left(\ln(1 + \rme^v)\right)^{-\frac{3}{2}}\frac{\rme^{v-r'}(1 - \rme^{-r})\ln(1 + \rme^{v+r'})}{(1 + \rme^v + \rme^{-r'})(1 + \rme^{v-r} + \rme^{-r'})} g_t(v-r+r', r')\\
&\ \  - f(v-r) \int_{r+ a_1-v}^{\infty} \rmd r'\ \left(\ln(1 + \rme^v)\right)^{-\frac{3}{2}}\rme^{\mu c_0 (v-r+r')}\frac{\rme^{v-\frac{3}{2}r'}(1 - \rme^{-r})\ln(1 + \rme^{v+r'})}{(1 + \rme^v + \rme^{-r'})(1 + \rme^{v-r} + \rme^{-r'})} g_t(v-r+r', r')\\
&\ \ - \cf(r\leq a_1-v\leq r+\delta_1)\int_{r+\delta_1}^{\infty} \rmd r'\ \left(\ln(1 + \rme^v)\right)^{-\frac{3}{2}}\frac{\rme^{v-\frac{3}{2}r'}(1 - \rme^{-r})\ln(1 + \rme^{v+r'})}{(1 + \rme^v + \rme^{-r'})(1 + \rme^{v-r} + \rme^{-r'})} \times\\
&\qquad\qquad\qquad\qquad\qquad\times \left( f(v-r+r') - f(v+r') \right)\rme^{\mu\max(a, c_0(v-r+r'))}g_t(v-r+r', r')\\
&\ \ - \int_{r+\delta_1}^{\infty} \rmd r'\ \left(\ln(1 + \rme^v)\right)^{-\frac{3}{2}}\frac{\rme^{v-\frac{3}{2}r'}(1 - \rme^{-r})\ln(1 + \rme^{v+r'})}{(1 + \rme^v + \rme^{-r'})(1 + \rme^{v-r} + \rme^{-r'})} \rme^{\mu\max(a, c_0(v-r+r'))} f(v+r') g_t(v-r+r', r')\\
&\ \ - \cf(r+\delta_1\leq a_1-v) \int_{r+\delta_1}^{a_1-v} \rmd r'\ \left(\ln(1 + \rme^v)\right)^{-\frac{3}{2}}\frac{\rme^{v-r'}(1 - \rme^{-\frac{1}{2}r'})(1 - \rme^{-r})\ln(1 + \rme^{v+r'})}{(1 + \rme^v + \rme^{-r'})(1 + \rme^{v-r} + \rme^{-r'})} \times\\
&\qquad\qquad\qquad\qquad\qquad\qquad\qquad\times \rme^{\mu\max(a, c_0(v-r+r'))} f(v+r') g_t(v-r+r', r')\\
& \ \ + 2 \rme^{\mu a}\int_{r+\delta_1}^{r+a_1-v} \rmd r'\ \left(\ln(1 + \rme^{v-r})\right)^{-\frac{3}{2}}\frac{\rme^{-2r'}\ln(1 + \rme^{v-r+r'})}{(1 - \rme^{-r'})(1 + \rme^v + \rme^{-r'})}\left( 1 - \frac{\ln(1 + \rme^{v-r-r'})}{\ln(1 + \rme^{v-r+r'})}  \right)\times\\
&\qquad\qquad\qquad\qquad\qquad\qquad\times\left( 1 - \left(\frac{\ln( 1 + \rme^{v-r})}{\ln(1 + \rme^v)}\right)^{\frac{3}{2}}\frac{\ln(1 + \rme^{v+r'})}{\ln(1 + \rme^{v-r+r'})}  \right)f(v-r)g_t(v-r+r', r')\\
& \ \ + 2 \int_{r+a_1-v}^{\infty} \rmd r'\ \left(\ln(1 + \rme^{v-r})\right)^{-\frac{3}{2}}\frac{\rme^{-\frac{5}{2}r'}\ln(1 + \rme^{v-r+r'})}{(1 - \rme^{-r'})(1 + \rme^v + \rme^{-r'})}\left( 1 - \rme^{\frac{1}{2}r'}\frac{\ln(1 + \rme^{v-r-r'})}{\ln(1 + \rme^{v-r+r'})}  \right)\times\\
&\qquad\qquad\qquad\qquad\qquad\times\left( 1 - \left(\frac{\ln( 1 + \rme^{v-r})}{\ln(1 + \rme^v)}\right)^{\frac{3}{2}}\frac{\ln(1 + \rme^{v+r'})}{\ln(1 + \rme^{v-r+r'})}  \right)\rme^{\mu c_0 (v-r+r')}f(v-r)g_t(v-r+r', r')\\
&\ \ +\Big. \cf(r>a_1-v)\int_{r+\delta_1}^{r+a_1-v} \rmd r'\ \ol{K_3}^1(v-r, v-r+r') f(v+r') \rme^{\mu\max(a, c_0(v-r+r'))} g_t(v-r+r', r')\Big]\\
& - \cf(v>-b_0)\Big[ 4\bar{n}\left(\ln(1 + \rme^v)\right)^{-\frac{3}{2}}\int_{r+\delta_1}^{\infty} \rmd r'\ \left( f(v-r) + f(v+r')\right)\frac{\rme^{v-\frac{3}{2}r'}(1 - \rme^{-r})\ln(1 + \rme^{v+r'})}{(1 + \rme^v + \rme^{-r'})(1 + \rme^{v-r} + \rme^{-r'})}\Big.\times\\
&\qquad\qquad\qquad\qquad\qquad\qquad\qquad\qquad\qquad\times \rme^{\mu\max(a, v-r, c_0(v-r+r'))} g_t(v-r+r', r')\\
&\qquad \Big. + \cf(v-r\geq 0)\ \rme^{\mu\max(a, v-r, c_0v)}\int_{r+\delta_1}^{\infty} \rmd r'\ K_3^1(v-r, v-r-r') \left( f(v-r) + f(v) \right)\Big.\times\\
&\qquad\qquad\qquad\qquad\qquad\qquad\qquad\qquad\qquad\qquad\Big.\times ( 1 - \rme^{-\frac{1}{2}r'}) ( 1 - \rme^{-\frac{1}{2}r}) g_t(v-r+r', r')\Big]
\end{align*}
We now combine $\mathcal{J}_0[\Gamma_t^1 + \Gamma_t^2](v, r)$, $\mathcal{I}^{(1)}[\Gamma_t^1](v, r)$ and  $E'_1[\Gamma_t^2](v, r)$ with the lower bound for $\mathcal{J}_1[\Gamma_t^1] (v, r)+ \mathcal{J}_2[\Gamma_t^2] (v, r)$ written above, to arrive at the following estimate denoted by Comb.$4$.
\begin{align*}
\tag{\textbf{Comb.4}}
&\mathcal{J}_1[\Gamma_t^1] (v, r)+ \mathcal{J}_2[\Gamma_t^2] (v, r) + \mathcal{J}_0[\Gamma_t^1 + \Gamma_t^2] (v, r) + \mathcal{I}^{(1)}[\Gamma_t^1](v, r) + E'_1[\Gamma_t^2](v, r)\\
\geq&\ \ol{\mathcal{J}}_0[\Gamma_t^1 + \Gamma_t^2] (v, r) + \ol{E}[\Gamma_t^1 + \Gamma_t^2](v, r),
\end{align*}
where
\begin{align}\label{defn: barJ0}
\begin{aligned}
&\ol{\mathcal{J}}_0[\Gamma_t^1 + \Gamma_t^2] (v, r)\\
=&\ -\cf(v<-b_0)\cf(r\leq a_1-v\leq r+\delta_1) 4\bar{n}\left(\ln(1 + \rme^v)\right)^{-\frac{3}{2}}\int_{r+\delta_1}^{r-v} \rmd r'\ \frac{\rme^{v-\frac{3}{2}r'}( 1 - \rme^{-\frac{1}{2}r'})(1 - \rme^{-r})}{(1 + \rme^v + \rme^{-r'})(1 + \rme^{v-r} + \rme^{-r'})}\times\\
&\qquad\qquad\qquad\qquad\qquad\times \rme^{\mu a}\ln(1 + \rme^{v+r'})\left( f(v-r+r') - f(v+r') \right) g_t(v-r+r', r')\\
& + \cf(v<-b_0)\rme^{\mu a}\left( f(v-r) + f(v)\right)\int_{r+\delta_1}^{\infty} \rmd r'\ K_3^1(v-r, v-r-r')\left( g_t(v, r) + g_t(v, r') - g_t(v, r+r')   \right)\\
& + \cf(v\geq-b_0) \rme^{\mu\max(a, c_0v, v-r)}\left( f(v-r) + f(v) \right)\Big[\cf(v-r<0)\int_{r+\delta_1}^{\infty} \rmd r'\ K_3^1(v-r, v-r-r')\times\Big.\\
&\qquad\qquad\qquad\qquad\qquad\qquad\qquad\times\left( g_t(v, r) + g_t(v, r') - g_t(v, r+r')   \right)\\
&\ \ \ \ \Big. + \cf(v-r\geq 0)\int_{r+\delta_1}^{\infty} \rmd r'\ K_3^1(v-r, v-r-r') \rme^{-\frac{1}{2}r'} \left( g_t(v, r) + g_t(v, r') - g_t(v, r+r')   \right)\Big]\\
&\ + 4\bar{n}\left(\ln(1 + \rme^v)\right)^{-\frac{3}{2}}\int_{r+\delta_1}^{\infty} \rmd r'\ \ln(1 + \rme^{v+r'})\frac{\rme^v + 2}{ 1 +\rme^v + \rme^{-r'}}\frac{\rme^{-\frac{5}{2}r'}}{1 - \rme^{-r'}} \rme^{\mu\max(a, v-r, c_0(v+r'))}\times\\
&\qquad\qquad\qquad\qquad\qquad\qquad\times \left(f(v-r) + f(v+r')\right)\left( g_t(v, r) + g_t(v-r+r', r') - g_t(v+r', r+r')   \right)\\
&\ + \cf(v\geq -b_0)4\bar{n}\left(\ln(1 + \rme^v)\right)^{-\frac{3}{2}}\int_{r+\delta_1}^{\infty} \rmd r'\ \ln(1 + \rme^{v+r'})\frac{\rme^{v-r} + \rme^{-r'}}{ 1 +\rme^{v-r} + \rme^{-r'}}\frac{\rme^{v-\frac{3}{2}r'}}{1 + \rme^v + \rme^{-r'}}\times\\
&\ \qquad\qquad\qquad\qquad\qquad\qquad\times \left( f(v-r) + f(v+r') \right)\left( 2 - (1 + \rme^{-\kappa r})^{\gamma_t(v, r)}\right)\rme^{\mu\max(a, v-r, c_0(v+r'))} g_t(v, r),
\end{aligned}
\end{align}
and,
\begin{align}\label{defn: barE}
&\overline{E}[\Gamma_t^1 + \Gamma_t^2](v, r)\nonumber\\
=&\ \cf(v<-b_0)\Big[ 4\bar{n}f(v-r)\left(\ln(1 + \rme^v)\right)^{-\frac{3}{2}}\int_{r+\delta_1}^{r+a_1-v} \rmd r'\ \frac{\rme^v + 2}{1 + \rme^v + \rme^{-r'}}\frac{\rme^{-2r'}}{1 - \rme^{-r'}}\ln(1 + \rme^{v+r'})\times\Big.\nonumber\\
&\qquad\qquad\qquad\qquad\times \left( \rme^{\mu a} - \rme^{-\frac{1}{2}r'} \rme^{\mu\max(a, v-r, c_0(v+r'))} \right)\left( g_t(v, r) + g_t(v-r+r', r') - g_t(v+r', r+r')  \right)\nonumber\\
&\ \ \ +  \rme^{\mu a} f(v) \int_{\max(-v-b_0, r+\delta_1)}^{\infty} \rmd r'\ K_3^2(v-r, v-r-r') \rme^{-r'}\rme^{-(\frac{122}{250}-\alpha)r}\left(\frac{1 - \rme^{-(1-\alpha)r'}}{1 - \rme^{-r'}}\right) g_t(v, r)\nonumber\\
& \ \ + \cf(r+\delta_1<-v-b_0)\Bigg\{ \int_{r+\delta_1}^{\infty} \rmd r'\ \ol{K_3}^2(v-r, v-r+r')\left( f(v-r+r') - f(v+r')\right)\times\Bigg.\nonumber\\
&\qquad\qquad\qquad\qquad\qquad\qquad\qquad\qquad\times\left( \rme^{\mu\max(a, c_0(v-r+r'))} - \rme^{\mu a} \right) g_t(v-r+r', r') \nonumber\\
&\ \ \ +  \rme^{\mu a}f(v)\int_{-v-b_0}^{\infty} \rmd r'\ K_3^2(v-r, v-r-r')\Big(\left( 1 - \rme^{-(\frac{122}{250}-2\alpha)r'} \right)\left( 1 - \rme^{-(1 - 2\alpha)r'} \right)\Big.\nonumber\\
&\qquad\qquad\qquad\qquad\qquad \Bigg. \Big.+ \rme^{-\frac{116}{750}r'} \rme^{-(1 - 2\alpha)r'} ( 1 - \rme^{-\alpha r'})\left( \frac{1 - \rme^{-\frac{634}{750}r'}}{1 - \rme^{-r'}} \right)\Big) g_t(v, r)\Bigg\}\nonumber\\
&\ \ + \cf(r+\delta_1>-v-b_0) \rme^{\mu a} f(v)\int_{\max(-v-b_0, r)}^{\infty} \rmd r'\ K_3^2(v-r, v-r-r') \left( 1 - \rme^{-(\frac{122}{250}-2\alpha)r'}  \right)\times\nonumber\\
&\qquad\qquad\qquad\qquad\qquad\qquad\qquad\times\left(1 - \rme^{-(1-2\alpha)r'}\right) g_t(v, r)\nonumber\\
&\ \ + \cf(r>a_1-v)\int_{r+\delta_1}^{r+a_1-v} \rmd r'\ \ol{K_3}^1(v-r, v-r+r')f(v+r') \rme^{\mu\max(a, c_0(v-r+r'))} g_t(v-r+r', r')\nonumber\\
&\ \ + \cf(r+\delta_1>a_1-v) f(v-r)\int_{r+\delta_1}^{r+a_1-v} \rmd r'\ K_3(v, v+r')\left( \rme^{\mu a} - \rme^{\mu c_0 v}\rme^{-(\frac{1}{2}-\mu c_0)r'}\right)g_t(v+r', r+r')\nonumber\\
&\ \ \Big.+ \cf(a_1-v\geq r+\delta_1) f(v-r)\int_{a_1-v}^{r+a_1-v} \rmd r'\ K_3(v, v+r')\left( \rme^{\mu a} - \rme^{\mu c_0 v} \rme^{-(\frac{1}{2} -\mu c_0)r'}  \right) g_t(v+r', r+r')\Big].
\end{align}
This leads us to consolidate (Comb.$4$) and (LB$1$) into the second lower bound (LB2)
\begin{align*}
\tag{\textbf{LB2}}
&I_1[\Gamma_t^1 + \Gamma_t^2](v, r) + I_3^{(1)}[\Gamma_t^2](v, r) + A_1^{-}[\Gamma_t^1](v, r) \\
\geq&\ B_1^-[\Gamma_t^1](v, r) + B_2[\Gamma_t^2](v, r) +  B_3[\Gamma_t^2](v, r) + E_2[\Gamma_t^2](v, r) + \mathcal{I}^{(3)}[\Gamma_t^2](v, r) \\
&\ \ + \overline{\mathcal{J}}_0[\Gamma_t^1 + \Gamma_t^2] (v, r) + \overline{E}[\Gamma_t^1 + \Gamma_t^2](v, r),
\end{align*}
Recall now the lower bound on $I_2[\Gamma_t^1](v, r)$ given by (\ref{eq: I_2_Gamma1LowerBound}) and the following definition (cf. (\ref{def: mcC3})):
\begin{align*}
&\ol{\mathcal{C}}_3[\Gamma_t^1](v, r)\\
=&\ \cf(v-r<-b_0)\rme^{\mu\max\{a, c_0v\}}g_t(v, r)\int_{\delta_1}^r \rmd r'\ \Big[ K_3(v-r, v-r+r')\left( f(v-r) - (\ln(1 + \rme^{v-r+r'}))^{-\alpha}\right)\Big.\nonumber\\
&\qquad\qquad\qquad\qquad\qquad\qquad\qquad\qquad - \Big. K_3^1(v-r, v-r-r') \left( f(v-r-r') - f(v-r) \right)\Big].
\end{align*}
This term now has to be used to offset the negative term $A_2^-[\Gamma_t^1](v, r)$. By virtue of Lemma \ref{lemma: mcC3bound}, this combination yields the following inequality:
\begin{align*}
\tag{\textbf{Comb.5}}
& \overline{\mathcal{C}}_3[\Gamma_t^1](v, r) + A_2^- [\Gamma_t^1](v, r)\nonumber\\
\geq&\ \cf(-m_0<v-r<-b_0) \overline{\mathcal{C}}_3[\Gamma_t^1](v, r) \\
+&\ \cf(v-r\leq -m_0)\rme^{\mu\max(a, c_0v)} 4\bar{n}\left(\ln(1 + \rme^{v-r})\right)^{-\frac{3}{2}}\int_{\delta_1}^{\min(r, \tilde{r})} \rmd r'\ \frac{\rme^{v-r-r'} + 2\rme^{-2r'}}{1 + \rme^{v-r} + \rme^{-r'}}\left( \ln(1 + \rme^{v-r-r'}) \right)^{1-\alpha}\times\nonumber\\
&\ \ \ \times\left(\frac{\ln(1 + \rme^{v-r+r'})}{\ln(1 + \rme^{v-r-r'})}\right)^{1-2\alpha}\left[ \left(\frac{1 - \rme^{-\frac{248}{125}(1 - 2\alpha)r'}}{1 - \rme^{-r'}}\right) - \rme^{r'}\rme^{-\frac{248}{125}(1 - 2\alpha)r'}  \right]\left[ 1 -  \left(\frac{\ln(1 + \rme^{v-r-r'})}{\ln(1 + \rme^{v-r})}\right)^{\alpha}  \right]g_t(v, r)\nonumber\\
\geq&\ \cf(-m_0<v-r<-b_0) \overline{\mathcal{C}}_3[\Gamma_t^1](v, r)+ \cf(v-r\leq -m_0) \bar{n}\left(\ln(1 + \rme^{v-r})\right)^{-\frac{1}{2}}\Gamma_t^1(v,r) b_1(\alpha)\left( 1 - \rme^{-3\alpha r} \right)^3.
\end{align*}
Let us define
\begin{align}\label{defn: tildeB}
\tilde{B}[\Gamma_t^1](v, r)&= \cf(v-r\leq-m_0) \rme^{\mu\max(a, c_0v)}f(v-r) b_1( \alpha)\bar{n}\left(\ln(1 + \rme^{v-r})\right)^{-\frac{1}{2}}\left( 1 - \rme^{-3\alpha r} \right)^3 g_t(v, r),
\end{align}
where $b_1$ is positive and bounded away from $0$, for all $\alpha>0.$
Note that $\tilde{B}$ has the requisite singular behavior like $(\ln (1 + \rme^{v-r}))^{-1/2}$ for $r\gg 1.$

Then we combine the rest of $I_2[\Gamma_t](v, r)$ with $\overline{E}[\Gamma_t^1 + \Gamma_t^2](v, r)$ and $\overline{\mathcal{J}}_0[\Gamma_t^1 + \Gamma_t^2] (v, r)$, and obtain the following lower bound:
\begin{align*}
\tag{\textbf{LB3}}
&I_1[\Gamma_t](v, r) + I_2[\Gamma_t](v, r) + I_3^{(1)}[\Gamma_t^2](v, r) + A_1^-[\Gamma_t^1](v, r) + A_2^-[\Gamma_t^1](v, r)\\
\geq&\ B_1^-[\Gamma_t^1](v, r)+ \cf(-m_0<v-r<-b_0) \ol{\mathcal{C}}_3[\Gamma_t^1](v, r) + \tilde{B}[\Gamma_t^1](v, r) + \ol{B}_2[\Gamma_t^2](v, r) +  B_3^f[\Gamma_t^2](v, r)\\
&\qquad+ E_2[\Gamma_t^2](v, r) + \mathcal{I}^{(3)}[\Gamma_t^2](v, r) + \overline{\mathcal{J}}_0^f[\Gamma_t^1 + \Gamma_t^2] (v, r) + \overline{E}^f[\Gamma_t^1 + \Gamma_t^2](v, r),
\end{align*}
where
\begin{align}\label{defn: barB2}
&  \overline{B}_2[\Gamma_t^2](v, r) \nonumber\\
= & \ B_2[\Gamma_t^2](v, r)  + \cf(v<-b_0) \rme^{\mu a}\int_{\delta_1}^r \rmd r'\ K_3(v, v+r')(f(v) - f(v+r')) g_t(v, r)\nonumber\\
&\qquad\qquad+ \cf(v\geq -b_0)\int_{\delta_1}^r \rmd r'\ K_3(v, v+r')f(v)\left( \rme^{\mu\max(a, v-r, c_0v)} - \rme^{-\frac{1}{2}r'} \rme^{\mu\max(a, v-r, c_0(v+r'))}  \right)g_t(v, r),
\end{align}
\begin{align}\label{defn: barJ0f}
&\overline{\mathcal{J}}_0^f[\Gamma_t^1 + \Gamma_t^2] (v, r)\nonumber\\
=&\ \cf(v<-b_0)\rme^{\mu a}\left( f(v-r) + f(v)\right)\int_{r+\delta_1}^{\infty} \rmd r'\ K_3^1(v-r, v-r-r')\left( g_t(v, r) + g_t(v, r') - g_t(v, r+r')   \right)\nonumber\\
& + \cf(v\geq-b_0) \rme^{\mu\max(a, c_0v, v-r)}\left( f(v-r) + f(v) \right)\Big[\cf(v-r<0)\int_{r+\delta_1}^{\infty} \rmd r'\ K_3^1(v-r, v-r-r')\times\Big.\nonumber\\
&\qquad\qquad\qquad\qquad\qquad\qquad\qquad\times\left( g_t(v, r) + g_t(v, r') - g_t(v, r+r')   \right)\nonumber\\
&\ \ \ \ \Big. + \cf(v-r\geq 0)\int_{r+\delta_1}^{\infty} \rmd r'\ K_3^1(v-r, v-r-r') \rme^{-\frac{1}{2}r'} \left( g_t(v, r) + g_t(v, r') - g_t(v, r+r')   \right)\Big]\nonumber\\
&\ + 4\bar{n}\left(\ln(1 + \rme^v)\right)^{-\frac{3}{2}}\int_{2r}^{\infty} \rmd r'\ \ln(1 + \rme^{v+r'})\frac{\rme^{-\frac{5}{2}r'}}{1 - \rme^{-r'}} \times\nonumber\\
&\qquad\qquad\qquad\qquad\qquad\qquad\times \rme^{\mu\max(a, v-r, c_0v)} f(v+r')\left( 2 - (1 + \rme^{-\kappa r})^{\gamma_t(v, r)}\right) g_t(v, r)\nonumber\\
&\ + \cf(v\geq -b_0)4\bar{n}\left(\ln(1 + \rme^v)\right)^{-\frac{3}{2}}\int_{2r}^{\infty} \rmd r'\ \ln(1 + \rme^{v+r'})\frac{\rme^{-\frac{5}{2}r'}}{1 - \rme^{-r'}} \times\nonumber\\
&\qquad\qquad\qquad\qquad\qquad\times \rme^{\mu\max(a, v-r, c_0v)}f(v-r)\left( 2 - (1 + \rme^{-\kappa r})^{\gamma_t(v, r)}\right) g_t(v, r),
\end{align}
\begin{align}\label{defn: olEf}
&\ol{E}^f[\Gamma_t^1 + \Gamma_t^2](v, r)\nonumber\\
=&\ \cf(v<-b_0)\Bigg[\rme^{\mu a} f(v) \int_{\max(-v-b_0, r+\delta_1)}^{\infty} \rmd r'\ K_3^2(v-r, v-r-r') \rme^{-r'}\rme^{-(\frac{122}{250}-\alpha)r}\left(\frac{1 - \rme^{-(1-\alpha)r'}}{1 - \rme^{-r'}}\right) g_t(v, r)\Bigg.\nonumber\\
& \ \ + \cf(r+\delta_1<-v-b_0)\Bigg\{ \int_{r+\delta_1}^{\infty} \rmd r'\ \ol{K_3}^2(v-r, v-r+r')\left( f(v-r+r') - f(v+r')\right)\times\Bigg.\nonumber\\
&\qquad\qquad\qquad\qquad\qquad\qquad\qquad\qquad\times\left( \rme^{\mu\max(a, c_0(v-r+r'))} - \rme^{\mu a} \right) g_t(v-r+r', r') \nonumber\\
&\ \ \ +  \rme^{\mu a}f(v)\int_{-v-b_0}^{\infty} \rmd r'\ K_3^2(v-r, v-r-r')\Big(\left( 1 - \rme^{-(\frac{122}{250}-2\alpha)r'} \right)\left( 1 - \rme^{-(1 - 2\alpha)r'} \right)\Big.\nonumber\\
&\qquad\qquad\qquad\qquad\qquad \Bigg. \Big.+ \rme^{-\frac{116}{750}r'} \rme^{-(1 - 2\alpha)r'} ( 1 - \rme^{-\alpha r'})\left( \frac{1 - \rme^{-\frac{634}{750}r'}}{1 - \rme^{-r'}} \right)\Big) g_t(v, r')\Bigg\}\nonumber\\
&\ \ + \cf(r+\delta_1\geq-v-b_0) \rme^{\mu a} f(v)\int_{\max(-v-b_0, r)}^{\infty} \rmd r'\ K_3^2(v-r, v-r-r') \left( 1 - \rme^{-(\frac{122}{250}-2\alpha)r'}  \right)\times\nonumber\\
&\qquad\qquad\qquad\qquad\qquad\qquad\qquad\times\left(1 - \rme^{-(1-2\alpha)r'}\right) g_t(v, r')\nonumber\\
&\ \ + \Bigg.\cf(r>a_1-v)\int_{r+\delta_1}^{r+a_1-v} \rmd r'\ \ol{K_3}^1(v-r, v-r+r')f(v+r') \rme^{\mu\max(a, c_0(v-r+r'))} g_t(v-r+r', r')\Bigg],
\end{align}
and
\begin{align*}
& B_3^f[\Gamma_t^2](v, r)\\
=\ & \cf(v<-b_0)\Bigg[ \int_{\max(r+\delta_1, a_1-v)}^{\infty} \rmd r'\ K_3(v, v+r') f(v+r')\left(\rme^{\mu a} - \rme^{\mu c_0 v}\right) g_t(v, r)\ \ \ \Big.\nonumber\\
&\ \ \ \ \ \ \ + \int_{r+\delta_1}^{\infty} \rmd r'\ K_3(v, v+r')\left(f(v) - f(v+r')\right) \rme^{\mu a} g_t(v, r)\Bigg.\nonumber\\
&\ \ + \cf(r\leq a_1-v\leq r+\delta_1)\Bigg\{ \int_{r+\delta_1}^{r+a_1-v} \rmd r'\ \ol{K_3}^2(v, v+r') f(v+r') \left( \rme^{\mu a} - \rme^{\mu c_0 (v-r+r')}\right) g_t(v-r+r', r')\Bigg.\nonumber\\
&\ \ \ + \int_{r+\delta_1}^{r-v} \rmd r'\ \ol{K_3}^2(v, v+r') \rme^{-\frac{1}{2}r'} \left(f(v-r+r') - f(v+r')\right) \rme^{\mu a} g_t(v-r+r', r')\nonumber\\
&\ \ \  + \int_{r+\delta_1}^{\infty} \rmd r'\ K_3(v, v+r')f(v+r')\rme^{\mu c_0 v}\left( 1 - 2\rme^{-(\frac{1}{2} -\mu c_0)r'} + \rme^{-2(\frac{1}{2} - \mu c_0)r'}\right) g_t(v, r)\nonumber\\
&\ \ \  \Bigg. + \int_{r+\delta_1}^{r+a-v} \rmd r'\ \ol{K_3}^1(v-r, v-r+r') f(v-r+r') \rme^{\mu a} g_t(v-r+r', r')\Bigg\} \nonumber\\
&\ + \cf(a_1-v>r+\delta_1)\Bigg\{ \int_{r+\delta_1}^{a_1-v} \rmd r'\  \ol{K_3}^1(v-r, v-r+r') \left( f(v-r+r') - f(v+r') \right) \rme^{\mu a} g_t(v-r+r', r')\Bigg.\nonumber\\
&\ \  +  \int_{a_1-v}^{r+a_1-v} \rmd r'\ \ol{K_3}^2(v, v+r') f(v+r')\left(\rme^{\mu a} - \rme^{\mu c_0v} \rme^{-(\frac{1}{2}-\mu c_0)r'} \right) g_t(v-r+r', r')\nonumber\\
& \ \ + 8\bar{n}\ \rme^{\mu c_0v} \int_{a_1-v}^{r+a_1-v} \rmd r'\ \left(\ln(1 + \rme^v)\right)^{-\frac{3}{2}}\left(\ln(1 + \rme^{v+r'})\right)^{1-\alpha}\frac{\rme^{-2r'}(1 - \rme^{-\frac{1}{2}r'})(1 - \rme^{-(\frac{1}{2}-\mu c_0)r'})}{(1 - \rme^{-r'})(1 + \rme^v + \rme^{-r'})}g_t(v, r)\\
& \ \  \Bigg.+ 8\bar{n}\ \rme^{\mu a} \int_{a_1-v}^{r+a_1-v} \rmd r'\ \left(\ln(1 + \rme^{v-r})\right)^{-\frac{3}{2}}\left(\ln(1 + \rme^{v-r+r'})\right)^{1-\alpha}\frac{\rme^{-2r'}(1 - \rme^{-\frac{1}{2}r'})^2}{(1 - \rme^{-r'})(1 + \rme^v + \rme^{-r'})}g_t(v-r+r', r')\Bigg\}\\
&\ \ + \cf(a_1-v<r)\Bigg\{ \int_{r+\delta_1}^{r+a_1-v} \rmd r'\ \ol{K_3}^2(v-r, v-r+r') \left( f(v-r+r') - f(v+r')\right) \rme^{\mu a}g_t(v-r+r', r')\Bigg.\nonumber\\
&\ \ \ \ \ + \int_{r+\delta_1}^{r+a_1-v} \rmd r'\ \ol{K_3}^2(v, v+r')\left( \rme^{\mu a} - \rme^{\mu c_0 (v-r+r')} \right)f(v+r') g_t(v-r+r', r')\nonumber\\
&\ \ \ \ \ \Bigg.\Bigg. +  \int_{r+\delta_1}^{r+a_1-v} \rmd r'\ \ol{K_3}^1(v-r, v-r+r')\rme^{\mu a}\left( f(v-r+r') - f(v+r') \right) g_t(v-r+r', r')\Bigg\}\Bigg].
\end{align*}
Let us now turn to the negative terms in $\cf(v<-m_0)I_3^{(2)}[\Gamma_t^2](v, r)$ and observe that we can use a term from $\overline{B}_2[\Gamma_t^2](v, r)$ to control these negative terms as follows:
\begin{align*}
&\cf(v\leq -m_0)\rme^{\mu a}\Big[ -\cf(v+r<b_0)\int_0^r \rmd r'\ K_3^2(v, v-r')\left( f(v-r') - f(v) \right)g_t(v-r', r-r')\Big.\\
&\qquad \Big.-\cf(v+r\geq b_0)\int_0^{-v-b_0} \rmd r'\ K_3^2(v, v-r')\left( f(v-r') - f(v) \right)g_t(v-r', r-r')\Big]\\
&+ \cf(v<-b_0) \rme^{\mu a}\int_{\delta_1}^r \rmd r'\ K_3(v, v+r')(f(v) - f(v+r')) g_t(v, r)\\
\geq&\ \cf(-m_0<v<-b_0)\rme^{\mu a}\int_{\delta_1}^r \rmd r'\ K_3(v, v+r')(f(v) - f(v+r')) g_t(v, r),
\end{align*}
so that we can write
\begin{align*}
\tag{\textbf{Comb.6}}
&\overline{B}_2[\Gamma_t^2](v, r) + I_3^{(2)}[\Gamma_t^2](v, r)\\
\geq&\ \overline{B}_2^+[\Gamma_t^2](v, r) + \cf(v-r\geq m_0) I_3^{(1)}[\Gamma_t^2](v, r) + \cf(v\geq m_0)I_3^{(2)}[\Gamma_t^2](v, r),\\
&\qquad\qquad + \cf(-m_0<v<m_0)\rme^{\mu\max(a, c_0v, v-r)}f(v)\int_0^r \rmd r'\ K_3^2(v, v-r') g_t(v, r),
\end{align*}
where
\begin{align}\label{defn: barB+}
&\overline{B}_2^+[\Gamma_t^2](v, r)\nonumber\\
=\ &\cf(v-r\leq -m_0)\rme^{\mu\max(a, c_0v)} f(v)\int_0^{\tr} \rmd r'\ \Big\{ K_3^2(v-r, v-r-r') g_t(v, r) \Big.\nonumber\\
&\qquad\qquad\qquad\qquad\qquad\qquad\Big.- \left( K_3^2(v-r, v-r-r') - K_3^2(v, v-r-r') \right)g_t(v, r+r')\Big\}\nonumber\\
& + \cf(v\leq -m_0) \rme^{\mu a} f(v)\Bigg[\cf(v+r<b_0)\int_0^r \rmd r'\ K_3^2(v, v-r')\left( g_t(v, r) - g_t(v-r', r-r')\right)\Bigg.\nonumber\\
&\qquad\ \ \ + \cf(v+r\geq -b_0)f(v)\Bigg\{ \int_{-v-b_0}^r \rmd r'\ K_3^2(v, v-r') g_t(v, r)  
\Bigg.\nonumber\\
&\ \qquad\qquad\qquad \qquad\Bigg.\Bigg. + \int_0^{-v-b_0} \rmd r'\ K_3^2(v, v-r')\left( g_t(v, r) - g_t(v-r', r-r') \right) \Bigg\}\Bigg]\nonumber\\
& + \cf(-m_0<v<-b_0)\rme^{\mu a}\int_{\delta_1}^r \rmd r'\ K_3(v, v+r')(f(v) - f(v+r')) g_t(v, r)\nonumber\\
& + \cf(v\geq -b_0)\int_{\delta_1}^r \rmd r'\ K_3(v, v+r')f(v)\left( \rme^{\mu\max(a, v-r, c_0v)} - \rme^{-\frac{1}{2}r'} \rme^{\mu\max(a, v-r, c_0(v+r'))}  \right)g_t(v, r).
\end{align}
We control the negative term $A_3^{-}[\Gamma_t^1]$ by a term from $I_3^{(1)}[\Gamma_t^2]$ as follows:
\begin{align*}
& \cf(v-r\geq m_0)\cf(v>3r)f(v)\left( \rme^{\mu\max(a, v-r, c_0v)} -  \rme^{\mu\max(a, c_0v)} \right)g_t(v, r)\int_{v-r-c_0v}^{\infty} \rmd r'\ K_3^2(v-r, v-r-r')  \\
& + \ A_3^-[\Gamma_t^1](v, r)\\
> & \ \cf(v-r\geq m_0)\cf(v>3r) 4\bar{n}\left(\ln(1 + \rme^{v-r})\right)^{-\frac{3}{2}}\int_{v-r}^{\infty} \rmd r'\ \frac{\rme^{v-r-r'} + 2\rme^{-r'}}{1 + \rme^{v-r-r'} + \rme^{-r'}}\times\\
&\qquad\qquad\qquad\qquad\qquad\times\left(\ln(1 + \rme^{v-r-r'})\right)^{1-\alpha} \rme^{-\frac{2}{3}r}\left(1 - \rme^{-(\kappa -\frac{2}{3})r}\right) \rme^{\mu c_0 v} g_t(v, r+r'),
\end{align*}	
which allows us to arrive at the following estimate:
\begin{align*}
& \cf(v-r\geq m_0)\left(I_3^{(1)}[\Gamma_t^1](v, r) + I_3^{(1)}[\Gamma_t^2](v, r)\right)\\
\geq& \ B_4[\Gamma_t](v, r) + \cf(v-r\geq m_0)\overline{I}_3^{1, +}[\Gamma_t](v, r),
\end{align*}
where
\begin{align}\label{defn: B4}
&B_4[\Gamma_t](v, r)\nonumber\\
=&\ \cf(v-r\geq m_0)\cf(v>3r) 4\bar{n}\left(\ln(1 + \rme^{v-r})\right)^{-\frac{3}{2}}\int_{v-r}^{\infty} \rmd r'\ \frac{\rme^{v-r-r'} + 2\rme^{-r'}}{1 + \rme^{v-r-r'} + \rme^{-r'}}\times\nonumber\\
&\qquad\qquad\qquad\qquad\qquad\times\left(\ln(1 + \rme^{v-r-r'})\right)^{1-\alpha} \rme^{-\frac{2}{3}r}\left(1 - \rme^{-(\kappa -\frac{2}{3})r}\right) \rme^{\mu c_0 v} g_t(v, r+r')\nonumber\\
& +\cf(v-r\geq m_0)\Bigg[ \cf(v\leq 3r)\left( f(v-r) + f(v)\right)\Bigg\{\bar{n}(\ln 2)^2\left(\ln(1 + \rme^{v-r})\right)^{-\frac{3}{2}} \rme^{\mu\max(a, v-r, c_0v)} g_t(v, r)\Bigg.\Bigg.\nonumber\\
&\quad\Bigg. + \int_0^{v-r} \rmd r'\ \left( K_3^2(v-r, v-r-r') - K_3^2(v, v-r-r')  \right)\left( \rme^{\mu\max(a, v-r, c_0v)} - \rme^{\mu\max(a, v-r-r', c_0v)}  \right)g_t(v, r)\Bigg\}\nonumber\\
&\qquad+ \cf(v>3r)\Bigg\{ f(v-r) \int_{v-r-c_0v}^{\infty} \rmd r'\ K_3^2(v-r, v-r-r') \left( \rme^{\mu\max(a, v-r, c_0v)} - \rme^{\mu\max(a, c_0v)}  \right) g_t(v, r) \Bigg.\nonumber\\
&\qquad\quad + \left(f(v-r) + f(v)\right) \rme^{\mu\max(a, c_0v)}\int_{v-r-c_0v}^{\infty} \rmd r'\ \Big( K_3^2(v-r, v-r-r') g_t(v, r)\Big.\nonumber\\
&\qquad\qquad\qquad\qquad\Bigg.\Bigg. - \Big.\left( K_3^2(v-r, v-r-r') - K_3^2(v, v-r-r') \right)g_t(v, r+r')\Big)\Bigg\}\Bigg],
\end{align}
and
\begin{align}\label{defn: I_3^1+}
&\ol{I}_3^{1,+}[\Gamma_t](v, r)\nonumber\\
= &\ \cf(v-r\geq m_0)\Gamma_t(v, r)\Bigg[ \cf(v\leq 3r)\ \frac{4}{3}\bar{n}\left( \ln(1 + \rme^v)\right)^{\frac{1}{2}} \left(\frac{\ln(1 + \rme^{v-r})}{\ln(1 + \rme^v)}\right)^2\nonumber\\
&\qquad\Bigg. + \cf(v>3r)\ 2\bar{n}\left( \ln(1 + \rme^{v-r})\right)^{\frac{1}{2}} \left(1 - \left(\frac{\ln(1 + \rme^{c_0 v})}{\ln(1 + \rme^{v-r})}\right)^2\right)\Bigg]\nonumber\\
\end{align}
$\ol{I}_3^+$ contributes the correct asymptotic behavior for large, positive values of $v-r,$ when $v$ and $v-r$ are comparable.

Putting the above estimates together, we can write the following lower bound:
\begin{align*}
\tag{\textbf{LB4}}
&I_1[\Gamma_t](v, r) + I_2[\Gamma_t](v, r) + I_3^{(1)}[\Gamma_t](v, r) + I_3^{(2)}[\Gamma_t](v, r)\\
\geq&\ B_1^-[\Gamma_t^1](v, r)+A_4^-[\Gamma_t^1](v, r) +  \cf(-m_0<v-r<-b_0) \ol{\mathcal{C}}_3[\Gamma_t^1](v, r) + \tilde{B}[\Gamma_t^1](v, r) \\
& + \overline{B}_2^+[\Gamma_t^2](v, r) + B_3^f[\Gamma_t^2](v, r)  + E_2[\Gamma_t^2](v, r) +  \overline{\mathcal{J}}_0^f[\Gamma_t] (v, r) + \overline{E}^f[\Gamma_t](v, r) + B_4[\Gamma_t](v, r) \\
& + \cf(-m_0<v-r<m_0) I_3^{(1)}[\Gamma_t^1](v, r)  + \cf(-m_0<v<m_0) I_3^{(2)}[\Gamma_t](v, r) \\
& + \ol{I}_3^{1,+}[\Gamma_t](v, r) + \cf(v\geq m_0)I_3^2[\Gamma_t](v, r) + \mathcal{I}^{(3)}[\Gamma_t^2](v, r).
\end{align*}
Let us reflect for a moment on the asymptotic behavior for large, positive values of $v$. Note that $v\leq 3r\iff v-r\leq \frac{2}{3}v$ and $v>3r\iff v-r> \frac{2}{3}v$. The desired behavior mimicking that of the potential $\mcV_u(v, r)$ comes from $\ol{I}_3^{1,+}[\Gamma_t](v, r)$, $\cf(-m_0<v-r<m_0)I_3^{(1)}[\Gamma_t](v, r)$ and $\cf(v\geq m_0)I_3^2[\Gamma_t](v, r)$, as is evident from the following lower bound:
\begin{align}
&\ol{I}_3^{1,+}[\Gamma_t](v, r) + \cf(-m_0<v-r<m_0)I_3^{(1)}[\Gamma_t](v, r) + \cf(v\geq m_0)I_3^2[\Gamma_t](v, r)\nonumber\\
\geq&\ \cf(v\geq m_0)\Bigg[  \cf(v-r\leq 0)\bar{n}(\ln(1 + \rme^v))^{\frac{1}{2}}\Big( 1 - \left(\frac{\ln 2}{\ln (1 + \rme^v)}\right)^2\Big)\Gamma_t(v, r)\Bigg.\nonumber\\
&\quad + \cf(-m_0<v-r<m_0)\rme^{\mu\max(a, c_0v, v-r)}f(v-r)g_t(v, r)\int_0^{\infty}\rmd r'\ K_3^2(v-r, v-r-r')\nonumber\\
&\quad + \cf(0<v-r<c_0v)\frac{3}{2}\bar{n}(\ln(1 + \rme^v))^{\frac{1}{2}}\Big( 1 - \left(\frac{v-r}{c_0v}\right)^2 \Big)\Gamma_t(v, r)\nonumber\\
&\quad + \cf(v-r\geq m_0)\Gamma_t(v, r)\Bigg\{ \cf(v-r\leq\frac{2}{3}v)\ \frac{4}{3}\bar{n}\left( \ln(1 + \rme^v)\right)^{\frac{1}{2}} \left(\frac{\ln(1 + \rme^{v-r})}{\ln(1 + \rme^v)}\right)^2\Bigg.\nonumber\\
&\qquad\ \ \ \ \Bigg.\Bigg. + \cf(v-r>\frac{2}{3}v)\ 2\bar{n}\left( \ln(1 + \rme^{v-r})\right)^{\frac{1}{2}} \left(1 - \left(\frac{\ln(1 + \rme^{c_0 v})}{\ln(1 + \rme^{v-r})}\right)^2\right)\Bigg\}\Bigg]
\end{align}	\label{eq: +infbhvr}
In the lower bound (LB$4$) the first two terms, i.e., $B_1^{-}[\Gamma_t^1]$ and $A_4^-[\Gamma_t^1]$, are the only non-positive ones. Between these two, $A_4$ has a $\delta_1$-smallness (i.e., while the term may have a point singularity, it also contains a factor of $\delta_1$, which we can choose to be arbitrarily small), as is evident from (\ref{defn: A_4^-}). Our next step is to combine some suitable positive terms with $B_1^{-}$, so that we are left with a negative term which has a $\delta_1$-smallness. Thus, at the end of the next step all the negative terms will have this kind of smallness.

We combine some terms from $\ol{B}_2^+ [\Gamma_t^2]$, $\cf(-m_0<v<m_0)I_3^{(2)}[\Gamma_t^2]$ and $\overline{\mathcal{J}}_0^f[\Gamma_t]$ to control the negative term $B_1^-[\Gamma_t^1]$ as follows:
\begin{align*}
\tag{\textbf{Comb.$7$}}
&B_1^-[\Gamma_t^1](v, r) + \cf(v\geq -b_0) f(v)\int_{\delta_1}^r \rmd r'\ K_3(v, v+r')\left( \rme^{\mu\max(a, c_0v, v-r)} - \rme^{-\frac{1}{2}r'}\rme^{\mu\max(a, c_0(v+r'), v-r)}\right)g_t(v, r)\\
&\ \ + \rme^{\mu\max(a, c_0v, v-r)}f(v)\int_{r+\delta_1}^{\infty} \rmd r'\ K_3^1(v-r, v-r-r') \rme^{-\frac{1}{2}r'}\left( g_t(v, r) + g_t(v, r') - g_t(v, r+r')\right)\\
&\ \ + \cf(0\leq v <m_0)\rme^{\mu\max(a, c_0v, v-r)}f(v)\int_0^r \rmd r'\ K_3^2(v, v-r') g_t(v, r)\\
&\geq \mathcal{B}_1^{-}[\Gamma_t^1](v, r) + \tilde{B}_1^{+}[\Gamma_t^2](v, r),
\end{align*}
where
\begin{align*}
&\mathcal{B}_1^{-}[\Gamma_t^1](v, r)= -\cf(v\geq 0) \int_{2r}^{2r+\delta_1} \rmd r'\ K_3^1(v, v-r') \rme^{\mu\max(a, c_0v)}\left( f(v-r-r') - f(v-r')\right) g_t(v, r'),
\end{align*}
and
\begin{align*}
&\tilde{B}_1^{+}[\Gamma_t^2](v, r)=  \cf(0\leq v<m_0) \rme^{\mu\max(a, c_0v, v-r)} f(v) g_t(v, r)\frac{1}{1 + \rme^{-\frac{1}{2}r}}\int_0^r \rmd r'\ K_3^2(v, v-r')\\
&\quad + \frac{1}{2}\rme^{\mu\max(a, c_0v, v-r)}f(v)\int_{r+\delta_1}^{\infty} \rmd r'\ K_3^1(v-r, v-r-r') \rme^{-\frac{1}{2}r'}\left( g_t(v, r) + g_t(v, r') - g_t(v, r+r')\right)\\
&\quad + \frac{1}{2}\left(2 - 2^{\gamma_0} -\frac{1}{2}\right)\rme^{\mu\max(a, c_0v, v-r)}f(v)g_t(v,r)\int_{r+\delta_1}^{\infty} \rmd r'\ K_3^1(v-r, v-r-r') \rme^{-\frac{1}{2}r'}(1 - \rme^{-(\kappa - 1)r'})\\
&\quad + \cf(v\geq 0) \rme^{\mu\max(a, c_0v, v-r)} f(v) g_t(v, r)\int_{\delta_1}^r \rmd r'\ K_3(v, v+r')\left( 1 - \rme^{-(\frac{1}{2} - \mu c_0)r'}\right).
\end{align*}

At this stage, all the negative terms we are left with, are integrals over intervals of lengths $\delta_1$ and $(\delta_1 - \delta_2).$ Let us denote the sum of these negative terms by $S_{\delta_1, \delta_2}^-[\Gamma_t](v, r),$ as follows: 
\begin{align*}
&S_{\delta_1, \delta_2}^-[\Gamma_t](v, r)= \mathcal{B}_1^{-}[\Gamma_t^1](v, r) + I_4[\Gamma_t](v, r) + A_4^-[\Gamma_t^1](v, r).
\end{align*}
Since $S_{\delta_1, \delta_2}^-$ has a $\delta_1$-smallness, the constant $M$ in the definition of the cut-off functions (see (\ref{defn: deltaCutOff})) may be chosen large enough, so that the positive terms appearing in the estimate (LB$4$) can be used to offset this negative term, leading to the following lower bound:
\begin{align*}
\tag{\textbf{LB5}}
I_1[\Gamma_t](v, r) + I_2[\Gamma_t](v, r) + I_3^{(1)}[\Gamma_t](v, r) + I_3^{(2)}[\Gamma_t](v, r) + I_4[\Gamma_t](v, r) 
\geq \ G[\Gamma_t](v, r),
\end{align*}
where
\begin{align}\label{eq: Gdefn}
&G[\Gamma_t](v, r)\nonumber\\
=&\ \cf(v-r\leq -b_0)\bar{b}_1(\alpha)\bar{n}\Gamma_t(v, r)\left(\ln(1 + \rme^{v-r})\right)^{-\frac{1}{2}}\left( 1 - \rme^{-3\alpha r}\right)^3\nonumber\\
+ &\ \frac{\bar{n}}{2} \left(\ln(1 + \rme^{v-r})\right)^{-\frac{1}{2}}\ln\left( 1 - \rme^{-\frac{7}{2}(r+\delta_1)}\right)^{-1}\Gamma_t(v, r) + \cf(-m_0< v-r <m_0)\frac{\bar{n}}{4}\left(\ln(1 + \rme^{v-r})\right)^{-\frac{1}{2}}\Gamma_t(v, r)\nonumber\\  
+ &\ \cf(0<v<m_0)\frac{\bar{n}}{2}\left(\ln(1 + \rme^v)\right)^{\frac{1}{2}}\left(  1 - \left(\frac{\ln(1 + \rme^{v-r})}{\ln(1 + \rme^v)}\right)^2\right)\Gamma_t(v, r) \nonumber\\
+ &\ \cf(0<v-r<m_0)\frac{\bar{n}}{8}\left(\ln(1 + \rme^{v-r})\right)^{\frac{1}{2}}\Gamma_t(v, r)\nonumber\\
+ &\ \cf(v\geq m_0) b_3(c_0)\bar{n}\left(\ln(1 + \rme^v)\right)^{\frac{1}{2}}\Gamma_t(v, r),
\end{align}
where $\bar{b}_1(\alpha)$ and $b_3(c_0)$ are positive numbers bounded away from zero.

It is important to note here that, for controlling $S_{\delta_1, \delta_2}^-$, we need to put a suitable upper bound on the constant $M$. Let us describe briefly how that comes about. The bound comes from the fact that $\tilde{B}[\Gamma_t^1](v, r)$ is used to control $A_4^-[\Gamma_t^1](v, r)$ as follows:
\begin{align*}
A_4^-[\Gamma_t^1](v, r) \geq&-\rme^{\mu\max(a, c_0v)} f(v-r)g_t(v, r)\bar{n}\left(\ln( 1 + \rme^{v-r})\right)^{-\frac{1}{2}}\cf(v-r\leq -m_0)\Bigg[\frac{3}{4}\left(\min(\delta_1, \tr)\right)^2\Bigg.\\
&\qquad\Bigg. \Bigg. + \cf(v<-b_0)\cf(r<-b_0-v) 2 \rme^{-r}\delta_1\Bigg],
\end{align*}
and we choose $M$ such that
\begin{align*}
\frac{1}{M^2} + \frac{1}{M}\leq \frac{b_1(\alpha)}{10},
\end{align*}
so that $\tilde{B}[\Gamma_t^1](v, r)$ (see (\ref{defn: tildeB})) 
dominates $A_4^-[\Gamma_t^1](v, r).$

\section{Computations pertaining to the Regularized Problem} 
\label{appendix: RegEvolCompute}

\subsection{Formulae and Computations Relating to Theorem \ref{theorem: psiepSoln}}\label{subappendix: CompLuep}

The the unbounded operator $K_u$ and the $L^2$-bounded part $K_b$ are defined as follows, via cut-off parameters $b'_0$ and $m'_0$:
\begin{align*}
&(K_u^{\vep}\psiep_t)(v)\\
=& \cf(v<-b'_0)\Big[\int_0^{-v-b'_0}\rmd r'\ K_3^{1, \vep(v)}(v, v-r')\psiep_t(v-r') + \int_{-v-b'_0}^{\infty}\rmd r'\ K_3^{1,\vep(v)}(v, v-r')\rme^{-r'}\psiep_t(v-r')\Big.\\
&\ \ \Big. + \int_0^{-v-b'_0}\rmd r'\ K_3^{\vep(v+r')}(v, v+r')\psiep_t(v+r') + \int_{-v-b'_0}^{\infty}\rmd r'\ K_3^{\vep(v+r')}(v, v+r')\rme^{-r'}\psiep_t(v+r') \Big]\\
+&\ \cf(v\geq -b'_0)\Big[ \cf(-b'_0\leq v \leq m'_0)\int_0^{\infty}\rmd r'\ K_3^{1, \vep(v)}(v, v-r')\rme^{-r'}\psiep_t(v-r') \Big.\\
&\  + \cf(v>m'_0)\int_0^v\rmd r'\ K_3^{1, \vep(v)}(v, v-r')\psiep_t(v-r') + \cf(v\leq m'_0)\int_0^{\infty}\rmd r'\ K_3^{\vep(v+r')}(v, v+r')\rme^{-r'}\psiep_t(v+r')\\
&\quad\Big. + \cf(v>m'_0)\int_0^{\infty}\rmd r'\ K_3^{\vep(v+r')}(v, v+r')\psiep_t(v+r')\Big]\\
&+ \cf(v<-m'_0)\int_0^{r_0}\rmd r'\ K_3^{2, \vep(v)}(v, v-r')\psiep_t(v-r') + \cf(v>m'_0)\int_0^{v-a_1}\rmd r'\ K_3^{2,\vep(v)}(v, v-r')\psiep_t(v-r'),
\end{align*}
and 
\begin{align*}
&(K_b^{\vep}\psiep_t)(v)\\
=& \cf(v<-b'_0)\Big[ \int_{-v-b'_0}^{\infty}\rmd r'\ K_3^{1,\vep(v)}(v, v-r')(1 - \rme^{-r'})\psiep_t(v-r') \Big.\\
&\ \Big.+ \int_{-v-b'_0}^{\infty}\rmd r'\ K_3^{\vep(v+r')}(v, v+r') (1 - e^{-r'}) \psiep_t(v+r') \Big]\\
+&\ \cf(v\geq -b'_0)\Big[ \cf(v\leq m'_0)\int_0^{\infty}\rmd r'\ K_3^{1, \vep(v)}(v, v-r')(1 - \rme^{-r'})\psiep_t(v-r')\Big.\\
&\ \Big. + \ \cf(v>m'_0)\int_v^{\infty}\rmd r'\ K_3^{1,\vep(v)}(v, v-r')\psiep_t(v-r')\Big.\\
&\ \Big. + \cf(-b'_0\leq v\leq m_0)\int_0^{\infty}\rmd r'\ K_3^{\vep(v+r')}(v, v+r')(1 - \rme^{-r'})\psiep_t(v+r')   \Big]\\
+&\ \cf(v<-m'_0)\int_{r_0}^{\infty}\rmd r'\ K_3^{2,\vep(v)}(v, v-r')\psiep_t(v-r') + \cf(-m'_0\leq v\leq m'_0)\int_0^{\infty}\rmd r'\ K_3^{2,\vep(v)}(v, v-r')\psiep_t(v-r')\\
+&\ \cf(v>m'_0)\int_{v-a_1}^{\infty}\rmd r'\ K_3^{2, \vep(v)}(v, v-r')\psiep_t(v-r').
\end{align*}
In the formulae above $r_0=-v-b'_0$ plays the same role as $\tr$ in the previous section.
$m'_0>\max(b'_0, 2a_1)$  has to be chosen large enough.

The computations for Lemma \ref{lemma: LuepEstimate} follow the same scheme as employed in Appendix \ref{appendix: DeltaCompute} and are quite straightforward, so we skip them and write the resulting estimate, for some positive constants $p_2$ and $p_3$ bounded away from zero:
\begin{align*}
V^{\vep}(v)\tGamma(v) - (K_u^{\vep}\tGamma)(v)
\geq \tGamma(v)\left[ \cf(v\leq 0)p_2(\alpha)\left(\ln(1 + \rme^v)\right)^{-\frac{1}{2}} + \cf(v>0)p_3 \left(\ln(1 + \rme^v)\right)^{\frac{1}{2}} \right].
\end{align*}

\subsection{Formulae and Computations Relating to Theorem \ref{theorem: rDpsiSoln} }\label{subappendix: tmcLuep}

The computations for Theorem \ref{theorem: rDpsiSoln} are almost the same as those described in Appendix \ref{appendix: DeltaCompute} (which lead to a similar result for the $\Delta$-variable, namely Theorem \ref{theorem: DeltaSolutionsY}),
so there is nothing to be gained by repeating those arguments and estimates here. We will only write down the explicit forms of the operators, since the limits of the integrals are slightly different now (owing to the difference between the cut-off functions $\delta$ and $\vep$ in these two cases). 

Recall that the solution we are seeking proves the equation (\ref{eq: rDpsiDuhamel}). The operator $\tmcL^{\vep}_s$ has already been defined in (\ref{defn: tmcLs}). We now write down the expressions for the other operators.		
\begin{align*}
&(\tmcL^{\vep}_uD\psiep_t)(v, r)\\
=&\ \cf(v<a_1)\Bigg[ \int_{-\infty}^{v-r}\rmd w\ K_3^{1, \vep(v-r)}(v-r, w)D\psiep_t(v, r)\Bigg.\\
&\Bigg. - \int_{-\infty}^{v-r}\rmd w \left( K_3^{1, \vep(v-r)}(v-r, w) - K_3^{1, \vep(v-r)}(v, w) \right)D\psiep_t(v, v-w)\Bigg]\ \\
+&\ \cf(v\geq a_1)\Bigg[ \int_{-\infty}^{v-r}\rmd w\ K_3^{1, \vep(v-r)}(v-r, w)D\psiep_t(v, r)\Bigg.\\
&\qquad\Bigg. - \int_{-\infty}^{v-r}\rmd w\ \left( K_3^{1,\vep(v-r)}(v-r, w) - K_3^{1,\vep(v)}(v, w)  \right)D\psiep_t(v, v-w)\cf(v-w>\vep(v-r))\Bigg]\\
+ &  \int_v^{\infty}\rmd w\ K_3^{\vep(w)}(v, w)D\psiep_t(v, r) - \Bigg[\cf(v<-b_0)\Big[ \int_v^{a_1}\rmd w \left( K_3^{\vep(w)}(v, w) - K_3^{\vep(w)}(v-r, w) \right) D\psiep_t(w, w-v+r)\Big.\Bigg.\\
&\ \ \Big. + \int_{a_1}^{\infty}\rmd w\ \left( \ol{K_3}^{2,\vep(w)}(v, w) - \ol{K_3}^{2,\vep(w)}(v-r, w)  \right)D\psiep_t(w, w-v+r)\Big]\\
&\ \ \Bigg. + \cf(v\geq-b_0)\int_v^{\infty}\rmd w\  \left( \ol{K_3}^{2,\vep(w)}(v, w) - \ol{K_3}^{2,\vep(w)}(v-r, w)  \right)D\psiep_t(w, w-v+r)\Bigg]\\
+ & \int_{v-r}^v\rmd w\ K_3^{\vep(w)}(v-r, w)D\psiep_t(v, r) - \Bigg[ \cf(v-r<-b_0)\Big[\cf(v>0)\int_{v-r}^0\rmd w\ K_3^{\vep(w)}(v-r, w)D\psiep_t(v, v-w)\Big.\Bigg.\\
\Bigg.\Big. + & \cf(v\leq 0)\int_{v-r}^v\rmd w\ K_3^{\vep(w)}(v-r, w)D\psiep_t(v, v-w)\Big] + \cf(v-r\geq -b_0)\int_{v-r}^v\rmd w \ol{K_3}^{2, \vep(w)}(v-r, w)D\psiep_t(v, v-w)\Bigg]\\
+ & \int_{v-r}^v\rmd w\ K_3^{1,\vep(v)}(v, w)D\psiep_t(v, r)\\
- & \cf(v<-m_0)\int_{v-r}^v\rmd w\cf(v-w<\min(r, -v-b_0))K_3^{1,\vep(v)}(v, w)D\psiep_t(w, w-v+r)\\
- & \cf(v\geq -m_0)\int_{v-r}^v\rmd w\ K_3^{1,\vep(v)}(v, w)\rme^{-(v-w)}D\psiep_t(w, w-v+r)+  \int_{-\infty}^{v-r}\rmd w\ K_3^{2,\vep(v-r)}(v-r, w)D\psiep_t(v, r)\\
- & \cf(v-r\leq-m_0)\int_{v-r-\max(\tr, \vep(v-r))}^{v-r-\vep(v-r)}\rmd w\ \left( K_3^{2,\vep(v-r)}(v-r, w) - K_3^{2,\vep(v-r)}(v, w) \right) D\psiep_t(v, v-w) \\ 
- & \cf(v-r\geq m_0)\Big[ \cf(v\leq 3r)\int_0^{v-r-\vep(v-r)}\rmd w\left( K_3^{2,\vep(v-r)}(v-r, w) - K_3^{2,\vep(v-r)}(v, w) \right) D\psiep_t(v, v-w) \Big.\\
&\qquad \Big. +\cf(v>3r)\int_{-\infty}^{c_0v}\rmd w\left( K_3^{2,\vep(v-r)}(v-r, w) - K_3^{2,\vep(v-r)}(v, w) \right) D\psiep_t(v, v-w)\Big]\\
+ & \int_{v-r}^v\rmd w\ K_3^{2,\vep(v)}(v, w)D\psiep_t(v, r) - \cf(v\leq-m_0)\Big[\cf(v+r<-b_0)\int_{v-r}^v\rmd w\ K_3^{2,\vep(v)}(v, w)D\psiep_t(w, w-v+r)\Big.\\
&\qquad\Big. + \cf(v+r\geq-b_0)\int_{2v+b_0}^v\rmd w\ K_3^{2,\vep(v)}(v, w)D\psiep_t(w, w-v+r)\Big]\\
& - \cf(v\geq m_0)\Big[ \cf(0<v-r<c_0v)\int_{v-r}^{c_0^{-1}(v-r)}\rmd w\ K_3^{2,\vep(v)}(v, w)D\psiep_t(w, w-v+r)   \Big.\\
&\Big.\qquad + \cf(v-r\geq c_0v)\int_{v-r}^v\rmd w\ K_3^{2,\vep(v)}(v, w)D\psiep_t(w, w-v+r)\Big],
\end{align*}
which means we can write
\begin{align*}
(\tmcL^{\vep}_uD\psiep_t)(v, r) = \tmcV^{\vep}(v, r)D\psiep_t(v, r) - (\tmcK^{\vep}_uD\psiep_t)(v, r),
\end{align*}
the definitions of $\tmcV^{\vep}$ and $\tmcK^{\vep}_u$ being obvious from the formula for $\tmcL^{\vep}_u$ written above. Finally,
\begin{align*}
&\mcK^{\vep}_b[\psiep_t](v, r)\\
=& -\cf(v<-b_0)\int_{a_1}^{\infty}\rmd w\left( \ol{K_3}^{1,\vep(w)}(v, w) - \ol{K_3}^{1,\vep(w)}(v-r, w) \right)\left( \psiep_t(w) - \psiep_t(v-r) \right)\\
& -\cf(v\geq-b_0)\int_v^{\infty}\rmd w\left( \ol{K_3}^{1,\vep(w)}(v, w) - \ol{K_3}^{1,\vep(w)}(v-r, w) \right)\left( \psiep_t(w) - \psiep_t(v-r) \right)\\
&-\cf(v-r<-b_0)\cf(v>0)\int_0^v\rmd w\ K_3^{\vep(w)}(v-r, w)\left( \psiep_t(v) - \psiep_t(w) \right)\\
&-\cf(v-r\geq -b_0)\int_{v-r}^v\rmd w\ \ol{K_3}^{1,\vep(w)}(v-r, w)\left( \psiep_t(v) - \psiep_t(w) \right)\\
&-\cf(v<-m_0)\int_{v-r}^{v-\min(r,-v-b_0)}\rmd w\ K_3^{1,\vep(v)}(v, w)\left( \psiep_t(w) - \psiep_t(v-r) \right)\\
&-\cf(v-r\leq-m_0)\int_{-\infty}^{v-r-\max(\tr, \vep(v-r))}\rmd w\left( \ol{K_3}^{2,\vep(v-r)}(v-r, w) - \ol{K_3}^{2,\vep(v-r)}(v, w) \right)\left( \psiep_t(v) - \psiep_t(w) \right) \\
&-\cf(-m_0<v-r<m_0)\int_{-\infty}^{v-r-\vep(v-r)}\rmd w\left( \ol{K_3}^{2,\vep(v-r)}(v-r, w) - \ol{K_3}^{2,\vep(v-r)}(v, w) \right)\left( \psiep_t(v) - \psiep_t(w) \right)\\
&-\cf(v-r\geq m_0)\Big\{ \cf(v\leq 3r)\int_{-\infty}^0\rmd w\left( \ol{K_3}^{2,\vep(v-r)}(v-r, w) - \ol{K_3}^{2,\vep(v-r)}(v, w) \right)\left( \psiep_t(v) - \psiep_t(w) \right) \Big.\\
&\quad \Big.+ \cf(v>3r)\int_{c_0v}^{v-r-\vep(v-r)}\rmd w\left( \ol{K_3}^{2,\vep(v-r)}(v-r, w) - \ol{K_3}^{2,\vep(v-r)}(v, w) \right)\left( \psiep_t(v) - \psiep_t(w) \right)\Big\}\\
&-\cf(v\leq -m_0)\cf(v+r\geq-b_0)\int_{v-r}^{2v+b_0}\rmd w\ K_3^{2,\vep(v)}(v, w)\left( \psiep_t(w) - \psiep_t(v-r) \right)\\
&-\cf(-m_0<v<m_0)\int_{v-r}^v\rmd w\ K_3^{2,\vep(v)}(v, w)\left( \psiep_t(w) - \psiep_t(v-r) \right)\\
&-\cf(v\geq m_0)\Bigg[ \cf(0<v-r<c_0v)\int_{c_0^{-1}(v-r)}^v\rmd w\ K_3^{2,\vep(v)}(v, w)\left( \psiep_t(w) - \psiep_t(v-r) \right)\Bigg.\\
&\qquad \Bigg.+ \cf(v-r\leq 0)\int_{v-r}^v\rmd w\ K_3^{2,\vep(v)}(v, w)\left( \psiep_t(w) - \psiep_t(v-r) \right)\Bigg].
\end{align*}
The results for $D\psiep$ are proved via methods which are completely analogous to those employed in section 1 for the variable $\Delta$.
Thus it is critical to establish that $(\tmcL^{\vep}_u\ol\Gamma)(v, r)$ has the same asymptotic behavior as $\tmcV^{\vep}(v, r)\ol\Gamma(v, r),$ $\forall(v, r)\in\R\times\R_+$. Like in \ref{subsubsection: mcLuScheme}  (cf. equation (\ref{defn: mcLuSplitting})), we will now split $\tmcL^{\vep}_u\ol\Gamma$ into several terms as shown below. We use the symbol $\ol I_i$ (recall $I_i$ from \ref{subsubsection: mcLuScheme}) for the main integral terms. Then it is easy to see:
\begin{align*}
&(\tmcL^{\vep}_u\ol\Gamma'_{\vep})(v, r)\\
=& \ol I_1[\ol\Gamma](v, r) + \ol I_2[\ol\Gamma'_{\vep}](v, r) + \ol I_3[\ol\Gamma'_{\vep}](v, r) + \sum_{i=1}^8\ol e_i[\ol\Gamma'_{\vep}](v, r) +\ol I_4^{\vep}[\ol\Gamma'_{\vep}](v, r).
\end{align*}
Let us first define $\ol I_1, \ol I_2$ and $\ol I_3$.
\begin{align*}
1) & \ \ \ol I_1[\ol\Gamma'_{\vep}](v, r)\\
=& \int\limits_{\max(r,\vep(v-r))}^{\infty}\rmd r' \Big( K_3^1(v-r, v-r-r')\ol\Gamma'_{\vep}(v, r) + K_3^1(v, v-r')\ol\Gamma'_{\vep}(v, r')- K_3^1(v-r, v-r-r')\ol\Gamma'_{\vep}(v, r+r')\Big) \\
+ &\ \cf(v<-b_0)\Bigg[ \int\limits_{\max(r,\vep(v-r))}^{\infty}\rmd r'\ K_3(v, v+r')\ol\Gamma'_{\vep}(v, r)\Bigg.\Bigg.\\
-&\ \cf(r\leq a_1-v)\Bigg\{ \int\limits_{\max(r, \vep(v-r))}^{a_1-v}\rmd r'\ K_3(v, v+r')\ol\Gamma'_{\vep}(v+r', r+r') + \int_{a_1-v}^{\infty}\rmd r'\ \ol{K_3}^2(v, v+r')\ol\Gamma'_{\vep}(v+r', r+r')\\
\Bigg.- &\int\limits_{\max(r,\vep(v-r))}^{r+a_1-v}\rmd r'\ K_3(v-r, v-r+r')\ol\Gamma'_{\vep}(v-r+r', r') - \int\limits_{r+a_1-v}^{\infty}\rmd r'\ \ol{K_3}^2(v-r, v-r+r')\ol\Gamma'_{\vep}(v-r+r', r')\Bigg\}\\
- & \cf(r>a_1-v)\Bigg\{ \int\limits_{\max(r,\vep(v-r))}^{\infty}\rmd r'\ \ol{K_3}^2(v, v+r')\ol\Gamma'_{\vep}(v+r', r+r')\Bigg.\\
\Bigg.\Bigg. - & \int\limits_{\max(r,\vep(v-r))}^{r+a_1-v}\rmd r'\ K_3(v-r, v-r+r')\ol\Gamma'_{\vep}(v-r+r', r') - \int\limits_{r+a_1-v}^{\infty}\rmd r'\ \ol{K_3}^2(v-r, v-r+r')\ol\Gamma'_{\vep}(v-r+r', r')\Bigg\}\Bigg]\\
+ &\ \cf(v\geq-b_0)\int\limits_{\max(r,\vep(v-r))}^{\infty}\rmd r'\ \Big( K_3(v, v+r')\ol\Gamma'_{\vep}(v, r) + \ol{K_3}^2(v-r, v-r+r')\ol\Gamma'_{\vep}(v-r+r', r') \Big.\\
&\qquad\qquad\qquad\qquad\qquad\Big.- \ol{K_3}^2(v, v+r')\ol\Gamma'_{\vep}(v+r', r+r')\Big),
\end{align*}
\begin{align*}
2)&\ \ \ol I_2[\ol\Gamma'_{\vep}](v, r)\\
=&\ \cf(v-r<-b_0)\Bigg[ \int\limits_{\min(r,\vep(v-r))}^{\min(r, r-v)}\rmd r'\ K_3(v-r, v-r+r')\left( \ol\Gamma'_{\vep}(v, r) - \ol\Gamma'_{\vep}(v, r-r')  \right)\Bigg.\\
\Bigg. + & \int\limits_{\min(r-v, r)}^r\rmd r'\ K_3(v-r, v-r+r')\ol\Gamma'_{\vep}(v, r) - \int\limits_{\min(r, \vep(v-r))}^{r}\rmd r'\ K_3^1(v-r, v-r-r')\left( \ol\Gamma'_{\vep}(v, r+r') - \ol\Gamma'_{\vep}(v, r)  \right)\Bigg]\\
& + \cf(v-r\geq -b_0)\Bigg[ \int\limits_{\min(r, \vep(v-r))}^r\rmd r'\ K_3(v-r, v-r+r')\left( \ol\Gamma'_{\vep}(v, r) - \rme^{-\frac{1}{2}r'}\ol\Gamma'_{\vep}(v, r-r') \right)\Bigg.\\
&\qquad \Bigg. -\int\limits_{\min(r,\vep(v-r))}^r\rmd r'\ K_3^1(v-r, v-r-r')\left( \ol\Gamma'_{\vep}(v, r+r') - \ol\Gamma'_{\vep}(v, r) \right)\Bigg]\\
&+ \cf(v<-b_0)\Bigg[ \int\limits_{\min(r, \vep(v-r))}^r\rmd r'\ K_3^1(v, v-r')\left( \ol\Gamma'_{\vep}(v, r) - \ol\Gamma'_{\vep}(v-r', r-r') \right) \Bigg.\\
&\  -\int\limits_{\min(r,\vep(v-r))}^{\min(r, a_1-v)}\rmd r' K_3(v, v+r')\left( \ol\Gamma'_{\vep}(v+r', r+r') - \ol\Gamma'_{\vep}(v, r) \right) \\
&\qquad\ \ \Bigg. + \int\limits_{\min(r, a_1-v)}^r\rmd r'\ K_3(v, v+r')\left( \rme^{-\frac{1}{2}r'}\ol\Gamma'_{\vep}(v+r', r+r') - \ol\Gamma'_{\vep}(v, r) \right)\Bigg]\\
& + \cf(v\geq -b_0)\Big[ \int\limits_{\min(r, \vep(v-r))}^r\rmd r'\ K_3^1(v, v-r')\left( \ol\Gamma'_{\vep}(v, r) - \ol\Gamma'_{\vep}(v-r', r-r') \right)\Big.\\
&\ \ \Big. - \int\limits_{\min(r,\vep(v-r))}^r\rmd r'\ K_3(v, v+r')\left( \rme^{-\frac{1}{2}r'}\ol\Gamma'_{\vep}(v+r', r+r') - \ol\Gamma'_{\vep}(v, r)\right)\Big] \\
& + \cf(v<-m_0)\int\limits_{\min(r, -v-b_0)}^r\rmd r' K_3^1(v, v-r')\ol\Gamma'_{\vep}(v-r', r-r')\\
& + \cf(v\geq-m_0)\int\limits_{\min(r,\vep(v-r))}^r\rmd r' K_3^1(v, v-r')(1 - \rme^{-r'})\ol\Gamma'_{\vep}(v-r', r-r'),
\end{align*}
and
\begin{align*}
3)\ \ \ol I_3[\ol\Gamma'_{\vep}](v, r) = \ol I_3^{(1)}[\ol\Gamma'_{\vep}](v, r) + \ol I_3^{(2)}[\ol\Gamma'_{\vep}](v, r),\qquad\qquad\qquad\qquad\qquad\qquad\hspace{2in}
\end{align*}
where
\begin{align*}
a)\ \ &\ol I_3^{(1)}[\ol\Gamma'_{\vep}](v, r)\\
=& \cf(v-r\leq-m_0)\Bigg[ \cf(v<-b_0)\Bigg\{ \cf(\tr<\max(r,\vep(v-r)))\int\limits_{\vep(v-r)}^{\infty}\rmd r'\ K_3^2(v-r, v-r-r')\ol\Gamma'_{\vep}(v, r)\Bigg.\Bigg.\\
& + \cf(\tr\geq\max(r,\vep(v-r)))\Bigg( \cf(r>\vep(v-r))\int\limits_{\vep(v-r)}^r\rmd r'\rme^{\mu a}(f(v-r)+f(v))\Big( K_3^2(v-r, v-r-r')\thg(v, r)\Big.\Bigg.\\
&\qquad\qquad\qquad\qquad \Big. - \left(K_3^2(v-r, v-r-r') - K_3^2(v, v-r-r')\right)\thg(v, r+r')\Big)\\
&\ \ + \int\limits_{\max(r,\vep(v-r))}^{\tr}\rmd r'\rme^{\mu a}(f(v-r)+f(v))\Big( K_3^2(v-r, v-r-r')\thg(v, r)\Big.\\
&\qquad\qquad\qquad\qquad\Big.- \left(K_3^2(v-r, v-r-r') - K_3^2(v, v-r-r')\right)\thg(v, r+r')\Big)\\
&\quad\Bigg.\Bigg. + \int_{\tr}^{\infty}\rmd r'\ K_3^2(v-r, v-r-r')\ol\Gamma'_{\vep}(v, r)\Bigg)\Bigg\}\\
&\ \ +\cf(v\geq-b_0)\Bigg\{ \int\limits_{\max(\tr, \vep(v-r))}^{\infty}\rmd r'\ K_3^2(v-r, v-r-r')\ol\Gamma'_{\vep}(v, r) \Bigg.\\
&\ \ + \int\limits_{\min(\tr, \vep(v-r))}^{\tr}\rmd r' \Bigg( K_3^2(v-r, v-r-r')\ol\Gamma'_{\vep}(v, r)\Bigg.\\
&\quad \Bigg.\Bigg.\Bigg. - \Big(K_3^2(v-r, v-r-r') - K_3^2(v, v-r-r')\Big)\rme^{\mu\max(a, c_0v, v-r-r')}(f(v) + f(v-r))\thg(v, r+r')\Bigg)\Bigg\}\Bigg]\\
& + \cf(-m_0<v-r<m_0)\int_{\vep(v-r)}^{\infty}\rmd r'\ K_3^2(v-r, v-r-r')\ol\Gamma'_{\vep}(v, r)\\
& + \cf(v-r\geq m_0)\Bigg[ \cf(v\leq 3r)\Bigg\{ \int_{v-r}^{\infty}\rmd r'\ K_3^2(v-r, v-r-r')\ol\Gamma'_{\vep}(v, r) \Bigg.\Bigg.\\
&\qquad+ \int_{\vep(v-r)}^{v-r}\left(f(v)+ f(v-r)\right)\Big( K_3^2(v-r, v-r-r')\rme^{\mu\max(a, c_0v, v-r)}\thg(v, r) \Big.\Bigg.\Bigg.\\
&\qquad\Big.\Bigg. - \rme^{\mu\max(a, c_0v, v-r-r')}\left(K_3^2(v-r, v-r-r') - K_3^2(v, v-r-r')\right)\thg(v, r+r')\Big)\Bigg\}\\
&\ \ \ + \cf(v>3r)\Bigg\{  \int\limits_{\vep(v-r)}^{v-r-c_0v}\rmd r'\ K_3^2(v-r, v-r-r')\ol\Gamma'_{\vep}(v, r) \Bigg.\Bigg.\\
&\qquad+ \int\limits_{v-r-c_0v}^{\infty}(f(v)+ f(v-r))\Big( K_3^2(v-r, v-r-r')\rme^{\mu\max(a, c_0v, v-r)}\thg(v, r) \Big.\Bigg.\Bigg.\\
&\qquad\Big.\Bigg. - \rme^{\mu\max(a, c_0v, v-r-r')}\left(K_3^2(v-r, v-r-r') - K_3^2(v, v-r-r')\right)\thg(v, r+r')\Big)\Bigg\}\\
& + \ol A_1^-[\ol\Gamma^{', 1}_{\vep}](v, r) + \ol A_2^-[\ol\Gamma^{', 1}_{\vep}](v, r) + \ol A_3^-[\ol\Gamma^{',1}_{\vep}](v, r),
\end{align*}
with
\begin{align*}
&\ol A_1^-[\ol\Gamma^{',1}_{\vep}](v, r) = -\cf(v-r\leq-m_0)\cf(v<-b_0)\cf(\tr>\max(r,\vep(v-r)))\rme^{\mu a}\times\\
&\qquad\quad\times\int\limits_{\max(r,\vep(v-r))}^{\tr}\rmd r' \Big( K_3^2(v-r, v-r-r') - K_3^2(v, v-r-r') \Big)\left( f(v-r-r') - f(v-r) \right)\thg(v, r+r'),
\end{align*}
\begin{align*}
\ol A_2^-[\ol\Gamma^{',1}_{\vep}](v, r) =&-\cf(v-r\leq-m_0)\Bigg[ \cf(v<-b_0)\cf(r>\vep(v-r))\rme^{\mu a}\times\Bigg.\\
&\ \ \times\int\limits_{\vep(v-r)}^r\rmd r'\Big( K_3^2(v-r, v-r-r') - K_3^2(v, v-r-r') \Big)\left( f(v-r-r') - f(v-r) \right)\thg(v, r+r')\\
&\ \ \Bigg. + \cf(v\geq-b_0)\cf(\tr>\vep(v-r))\int\limits_{\vep(v-r)}^{\tr}\rmd r'\Big( K_3^2(v-r, v-r-r') - K_3^2(v, v-r-r') \Big)\times\Bigg.\\
&\qquad\qquad\Bigg.\times\left( f(v-r-r') - f(v-r) \right)\thg(v, r+r')\Bigg],
\end{align*}
\begin{align*}
\ol A_3^-[\ol\Gamma^{',1}_{\vep}](v, r) =&-\cf(v-r>m_0)\cf(v>3r)\int\limits_{v-r}^{\infty}\rmd r'\Big( K_3^2(v-r, v-r-r') - K_3^2(v, v-r-r') \Big)\times\\
&\qquad\qquad\qquad\qquad\qquad\qquad\qquad\times\rme^{\mu\max(a, c_0v, v-r-r')}\left( f(v-r-r') - f(v-r) \right)\thg(v, r+r'), 
\end{align*}
\begin{align*}
\text{and}\ \	b) \ \ \ol I_3^{(2)}[\ol\Gamma'_{\vep}](v, r)
&= \int\limits_{\min(r, \vep(v-r))}^r\rmd r'\ K_3^2(v, v-r')\ol\Gamma'_{\vep}(v, r)\\
& - \cf(v\leq-m_0)\Bigg\{   \cf(v+r<-b_0)\int\limits_{\min(r, \vep(v-r))}^r\rmd r'\ K_3^2(v, v-r')\ol\Gamma'_{\vep}(v-r', r-r')\Bigg.\\
&\quad\qquad \Bigg.+ \cf(v+r\geq-b_0)\int\limits_{\vep(v-r)}^{-v-b_0}\rmd r'\ K_3^2(v, v-r')\ol\Gamma'_{\vep}(v-r', r-r')\Bigg\}\\
&\ - \cf(v\geq m_0)\Bigg\{ \cf(0<v-r<c_0v)\int\limits_{\max(v-c_0^{-1}(v-r),\vep(v-r))}^r\rmd r'\ K_3^2(v, v-r')\ol\Gamma'_{\vep}(v-r', r-r')  \Bigg.\\
&\qquad\Bigg. + \cf(v-r\geq c_0v)\int\limits_{\min(r,\vep(v-r))}^r\rmd r'\ K_3^2(v, v-r')\ol\Gamma'_{\vep}(v-r', r-r')\Bigg\}.
\end{align*}
The other terms are ``small'', as seen from the definitions below:
\begin{align*}
4)\ \ &\ol I_4^{\vep}[\ol\Gamma'_{\vep}](v, r)
=  \cf(v\geq a_1)\cf(r\leq\vep(v-r))\Big[ \int\limits_{r}^{\vep(v-r)}\rmd r'\ K_3^{1,\vep(v-r)}(v-r, v-r-r') \ol\Gamma'_{\vep}(v, r) \Big.\qquad\qquad\\
&\qquad \Big.- \int\limits_{\max(r,\vep(v-r)-r)}^{\vep(v-r)}\rmd r'\ K_3^{1,\vep(v-r)}(v-r, v-r-r')\ol\Gamma'_{\vep}(v, r+r')\Big]\\
& + \cf(v\geq-b_0)\int_0^{\infty}\rmd r'\left( \ol{K_3}^{2,\vep(v-r+r')}(v-r, v-r+r') -   \ol{K_3}^{2,\vep(v+r')}(v-r, v-r+r')  \right)\ol\Gamma'_{\vep}(v-r+r', r'),
\end{align*}
\begin{align*}
5)\ \ &\ol e_1[\ol\Gamma'_{\vep}](v, r)\qquad\\
= &\cf(v<a_1)\int\limits_{r}^{\max(r, \vep(v-r))}\rmd r'\ \left[K_3^{1,\vep(v-r)}(v-r, v-r-r')\ol\Gamma'_{\vep}(v, r) - K_3^{1,\vep(v-r)}(v-r, v-r-r')\ol\Gamma'_{\vep}(v, r+r')\right.\\
&\qquad\left. K_3^{1,\vep(v-r)}(v, v-r')\ol\Gamma'_{\vep}(v, r')\right],
\end{align*}
\begin{align*}
6)\ \ \ol e_2[\ol\Gamma'_{\vep}](v, r)
=& \cf(v<-b_0)\int\limits_r^{\max(r,\vep(v-r))}\rmd r'\left[ K_3^{\vep(v+r')}(v, v+r')\ol\Gamma'_{\vep}(v, r) - K_3^{\vep(v+r')}(v, v+r')\ol\Gamma'_{\vep}(v+r', r+r')\right.\\
&\qquad\left. + K_3^{\vep(v+r')}(v-r, v-r+r')\ol\Gamma'_{\vep}(v-r+r', r')\right],
\end{align*}
\begin{align*}
7)\ \ \ol e_3[\ol\Gamma'_{\vep}](v, r)
=& \cf(v\geq-b_0)\int\limits_r^{\max(r,\vep(v-r))}\rmd r' \left[K_3^{\vep(v+r')}(v, v+r')\ol\Gamma'_{\vep}(v, r) - \ol{K_3}^{2,\vep(v+r')}(v, v+r')\ol\Gamma'_{\vep}(v+r', r+r')\right.\\
&\qquad\left. + \ol{K_3}^{2,\vep(v+r')}(v-r, v-r+r')\ol\Gamma'_{\vep}(v-r+r', r')\right],
\end{align*}
\begin{align*}
8)\ \ &\ol e_4[\ol\Gamma'_{\vep}](v, r)
= \cf(v-r<-b_0)\int\limits_0^{\min(r,\vep(v-r))}\rmd r'\ K_3^{\vep(v-r+r')}(v-r, v-r+r')\left(\ol\Gamma'_{\vep}(v, r) - \ol\Gamma'_{\vep}(v, r-r') \right)\qquad\qquad\\
& + \cf(v-r\geq-b_0)\int\limits_0^{\min(r,\vep(v-r))}\rmd r'\ K_3^{\vep(v-r+r')}(v-r, v-r+r')\left(\ol\Gamma'_{\vep}(v, r) - \rme^{-\frac{1}{2}r'}\ol\Gamma'_{\vep}(v, r-r') \right)\\
&- \cf(v<a_1)\int\limits_0^{\min(r,\vep(v-r))}\rmd r'\ K_3^{1,\vep(v-r)}(v-r, v-r-r')\left(\ol\Gamma'_{\vep}(v, r+r') - \ol\Gamma'_{\vep}(v, r) \right)\\
&- \cf(v\geq a_1)\int\limits_0^{\min(r,\vep(v-r))}\rmd r'\ \cf(r+r'>\vep(v-r))K_3^{1,\vep(v-r)}(v-r, v-r-r')\left(\ol\Gamma'_{\vep}(v, r+r') - \ol\Gamma'_{\vep}(v, r) \right)\\
& + \cf(v\geq a_1)\int\limits_0^{\min(r,\vep(v-r))}\rmd r'\ \cf(r+r'\leq\vep(v-r))K_3^{1,\vep(v-r)}(v-r, v-r-r')\ol\Gamma'_{\vep}(v, r),
\end{align*}
\begin{align*}
9)\ \ \ol e_5[\ol\Gamma'_{\vep}](v, r)
&=\int\limits_0^{\min(r,\vep(v-r))}\rmd r'\ K_3^{1,\vep(v)}(v, v-r')\left( \ol\Gamma'_{\vep}(v, r) - \ol\Gamma'_{\vep}(v-r', r-r') \right)\qquad\qquad\\ 
& -\int\limits_0^{\min(r,\vep(v-r))}\rmd r'\ K_3^{\vep(v+r')}(v, v+r')\left( \ol\Gamma'_{\vep}(v+r', r+r') - \ol\Gamma'_{\vep}(v, r)\right)\\
&+\cf(v\geq-b_0)\int\limits_0^{\min(r,\vep(v-r))}\rmd r'\ K_3^{\vep(v+r')}(v, v+r')(1 - \rme^{-\frac{1}{2}r'})\ol\Gamma'_{\vep}(v+r', r+r'),
\end{align*}
\begin{align*}
10)\ \ &\ol e_6[\ol\Gamma'_{\vep}](v, r)
= \cf(v\geq-m_0)\int\limits_0^{\min(r,\vep(v-r))}\rmd r'\ K_3^{1, \vep(v)}(v, v-r')(1 - \rme^{-r'})\ol\Gamma'_{\vep}(v+r', r+r'),\hspace{2in}\\
11)\ \ &\ol e_7[\ol\Gamma'_{\vep}](v, r) = \int_0^{\vep(v-r)}\rmd r'\ K_3^{v-r, v-r-r'}\ol\Gamma'_{\vep}(v, r),
\end{align*} 
and
\begin{align*}
12) \ \ \ol e_8[\ol\Gamma'_{\vep}](v, r)
&= \int\limits_0^{\min(r,\vep(v-r))}\rmd r' K_3^{2, \vep(v)}(v, v-r')\ol\Gamma'_{\vep}(v, r)\qquad\qquad\\
&\ - \cf(v\leq-m_0)\Bigg\{ \cf(v+r<-b_0)\int\limits_0^{\min(r,\vep(v-r))}\rmd r'K_3^{2,\vep(v)}(v, v-r')\ol\Gamma'_{\vep}(v-r', r-r')  \Bigg.\\
&\ \ \Bigg. + \cf(v+r\geq-b_0)\int_0^{\vep(v-r)}\rmd r'\ K_3^{2, \vep(v)} (v, v-r')\ol\Gamma'_{\vep}(v-r', r-r')\Bigg\}\\
&\ -\cf(v\geq m_0)\Bigg\{ \cf(0<v-r<c_0v)\int\limits_{v-c_0^{-1}(v-r)}^{\max(\vep(v-r, v-c_0^{-1}(v-r)))}\rmd r'\ K_3^{2,\vep(v)}(v, v-r')\ol\Gamma'_{\vep}(v-r', r-r') \Bigg.\\
&\Bigg.\qquad + \cf(v-r\geq c_0v)\int\limits_0^{\min(r,\vep(v-r))}\rmd r'\ K_3^{2,\vep(v)}(v, v-r')\ol\Gamma'_{\vep}(v-r', r-r')\Bigg\}.
\end{align*}

\section{Computations Relating to the Evolution of $D_t= D\psiep_t-\Delta_t$} 
\label{appendix: DvariableCompute}

In the evolution equation (\ref{eq: Deltaepeq}) the operator $\mcL^{\vep}$ is given by: 
\begin{align*}
&\mcL^{\vep}\Delta_t(v, r)\\
=& \int_{-\infty}^{v-r-\delta_1}\rmd w\ K_3^{1,\vep(v-r)}(v-r, w)\Delta_t(v, r) - \int_{-\infty}^{v-r-\delta_1}\rmd w\ \left( K_3^{1,\vep(v-r)}(v-r, w) - K_3^{1,\vep(v)}(v, w) \right)\Delta_t(v, v-w)\\
+ & \int_{v+\delta_1}^{\infty}\rmd w\ K_3^{\vep(w)}(v, w)\Delta_t(v, r) - \int_{v+\delta_1}^{\infty}\rmd w\ \left( K_3^{\vep(w)}(v, w) - K_3^{\vep(w)}(v-r, w) \right)\Delta_t(w, w-v+r)\\
+ & \int_{v-r}^{v-\delta_1}\rmd w\ K_3^{1,\vep(v)}(v, w)\left(\Delta_t(v, r) - \Delta_t(w, w-v+r)\right) + \int_{v-r+\delta_1}^v\rmd w\ K_3^{\vep(w)}(v-r, w)\left( \Delta_t(v, r) - \Delta_t(v, v-w) \right)\\
+ & \int_{v+\delta_2}^{v+\delta_1}\rmd w\ K_3^{\vep(w)}(v, w)\left( \Delta_t(v, r) - \Delta_t(w, w-v+r)  \right) + \int_{v-\delta_1}^{v-\delta_2}\rmd w\ K_3^{1,\vep(v)} (v, w)\left( \Delta_t(v, r) - \Delta_t(w, w-v+r)  \right)\\
+ & \int_{v-r}^v\rmd w\ K_3^{2,\vep(v)}(v, w)\left( \Delta_t(v, r) - \Delta_t(w, w-v+r) \right)\\
& + \int_{-\infty}^{v-r}\rmd w\ K_3^{2,\vep(v-r)}(v-r, w)\Delta_t(v, r) - \int_{-\infty}^{v-r}\rmd w\ \left( K_3^{2,\vep(v-r)}(v-r, w) - K_3^{2,\vep(v)}(v, w) \right)\Delta_t(v, w)\\
& + \Bigg[ \int_{v-r-\delta_1}^{v-r}\rmd w\ K_3^{1,\vep(v)}(v, w)\Delta_t(v, v-w) - \int_{v-r-\delta_1}^{v-r}\rmd w\ K_3^{1,\vep(v-r)}(v, w)\Delta_t(v-r, v-r-w)    \Bigg.\\
& + \int_v^{v+\delta_1}\rmd w\ K_3^{\vep(w)}(v-r, w)\Delta_t(w, w-v+r) - \int_v^{v+\delta_2}\rmd w\ K_3^{\vep(w)}(v, w)\Delta_t(w, w-v)\\
&\Bigg. + \int_{v-\delta_2}^v\rmd w\ K_3^{1,\vep(v)}(v, w)\Delta_t(v, v-w) + \int_{v-r}^{v-r+\delta_1}\rmd w\ K_3^{\vep(w)}(v-r, w)\Delta_t(w, w-v+r)\Bigg].
\end{align*}
Note that the last six terms in square brackets contribute towards our new ``$\mcL_{\delta}\Delta_t(v, r)$''. The other part $\mcL_0^{\vep}$ is defined exactly as above, but with $K_3^{\vep(\ ,\ )}(\  ,\ )$ substituted by $K_3(\  ,\ ) - K_3^{\vep(\  ,\ )}(\ ,\ )$.

Like before, $\mcL^{\vep}$ is separated into three parts as $\mcL^{\vep} = \mcL^{\vep}_u + \mcL^{\vep}_{\delta} + \mcL^{\vep}_b$  and then $\mcL_u^{\vep}\ol\Gamma_{\vep}$ is written as:
\begin{align*}
\mcL_u^{\vep}\ol\Gamma_{\vep}(v, r) =& \tI_1[\ol\Gamma_{\vep}](v, r) + \tI_2[\ol\Gamma_{\vep}](v, r) + \tI_3[\ol\Gamma_{\vep}](v, r) + \sum_{i=1}^8\te_i[\ol\Gamma_{\vep}](v, r) + \tI_4^{\vep}[\ol\Gamma_{\vep}](v, r).
\end{align*}
We will write the definitions for $\tI_1$, $\tI_2$ and $\tI_3$ below. These terms determine the asymptotic behavior of $\mcL_u^{\vep}\ol\Gamma_{\vep}$. The lower limits of the relevant integrals are different from similar terms seen before, owing to the new interplay between the two kinds of ``smallness''-parameters, namely the $\delta$-functions and the $\vep$-functions, so we write below their complete expressions. $\tI_4$ and all the $\te_i$'s have ``smallness'' coming from either $\vep$ or $\delta$, and we have already seen how such terms are controlled, so we will skip writing down the explicit formulae for them.
\begin{align*}
1) & \ \ \tI_1[\ol\Gamma_{\vep}](v, r)\\
=& \int\limits_{\max(r+\delta_1,\vep(v-r))}^{\infty}\rmd r' \Big( K_3^1(v-r, v-r-r')\ol\Gamma_{\vep}(v, r) + K_3^1(v, v-r')\ol\Gamma_{\vep}(v, r')- K_3^1(v-r, v-r-r')\ol\Gamma_{\vep}(v, r+r')\Big) \\
+ & \cf(v<-b_0)\Bigg[ \int\limits_{\max(r+\delta_1,\vep(v-r))}^{\infty}\rmd r'\ K_3(v, v+r')\ol\Gamma_{\vep}(v, r)\Bigg.\Bigg.\\
- & \cf(r+\delta_1\leq a_1-v)\Bigg\{ \int\limits_{\max(r+\delta_1, \vep(v-r))}^{a_1-v}\rmd r'\ K_3(v, v+r')\ol\Gamma_{\vep}(v+r', r+r')\Bigg.\\
+ &  \int_{a_1-v}^{\infty}\rmd r'\ \ol{K_3}^2(v, v+r')\ol\Gamma_{\vep}(v+r', r+r')\\
\Bigg.- &\int\limits_{\max(r+\delta_1,\vep(v-r))}^{r+a_1-v}\rmd r'\ K_3(v-r, v-r+r')\ol\Gamma_{\vep}(v-r+r', r') - \int\limits_{r+a_1-v}^{\infty}\rmd r'\ \ol{K_3}^2(v-r, v-r+r')\ol\Gamma_{\vep}(v-r+r', r')\Bigg\}\\
- & \cf(r+\delta_1>a_1-v)\Bigg\{ \int\limits_{\max(r+\delta_1,\vep(v-r))}^{\infty}\rmd r'\ \ol{K_3}^2(v, v+r')\ol\Gamma_{\vep}(v+r', r+r')\Bigg.\\
\Bigg.\Bigg. -& \int\limits_{\max(r,\vep(v-r))}^{r+a_1-v}\rmd r'\ K_3(v-r, v-r+r')\ol\Gamma_{\vep}(v-r+r', r') - \int\limits_{r+a_1-v}^{\infty}\rmd r'\ \ol{K_3}^2(v-r, v-r+r')\ol\Gamma_{\vep}(v-r+r', r')\Bigg\}\Bigg]\\
+ & \cf(v\geq-b_0)\int\limits_{\max(r+\delta_1,\vep(v-r))}^{\infty}\rmd r'\ \Big( K_3(v, v+r')\ol\Gamma_{\vep}(v, r) + \ol{K_3}^2(v-r, v-r+r')\ol\Gamma_{\vep}(v-r+r', r') \Big.\\
&\qquad\qquad\qquad\qquad\qquad\Big.- \ol{K_3}^2(v, v+r')\ol\Gamma_{\vep}(v+r', r+r')\Big),
\end{align*}
\begin{align*}
2)&\ \ \tI_2[\ol\Gamma_{\vep}](v, r)\\
=\ &\cf(v-r<-b_0)\Bigg[ \int\limits_{\min(r,\max(\delta_1,\vep(v-r)))}^{\min(r, r-v)}\rmd r'\ K_3(v-r, v-r+r')\left( \ol\Gamma_{\vep}(v, r) - \ol\Gamma_{\vep}(v, r-r')  \right)\Bigg.\\
+ & \int\limits_{\min(r-v, r)}^r\rmd r'\ K_3(v-r, v-r+r')\ol\Gamma_{\vep}(v, r) \\
- &\Bigg. \int\limits_{\min(r, \max(\delta_1,\vep(v-r)))}^{r}\rmd r'\ K_3^1(v-r, v-r-r')\left( \ol\Gamma_{\vep}(v, r+r') - \ol\Gamma_{\vep}(v, r)  \right)\Bigg]\\
+ & \cf(v-r\geq -b_0)\Bigg[ \int\limits_{\min(r,\max(\delta_1,\vep(v-r)))}^r\rmd r'\ K_3(v-r, v-r+r')\left( \ol\Gamma_{\vep}(v, r) - \rme^{-\frac{1}{2}r'}\ol\Gamma_{\vep}(v, r-r') \right)\Bigg.\\
&\qquad \Bigg. -\int\limits_{\min(r,\max(\delta_1,\vep(v-r)))}^r\rmd r'\ K_3^1(v-r, v-r-r')\left( \ol\Gamma_{\vep}(v, r+r') - \ol\Gamma_{\vep}(v, r) \right)\Bigg]\\
+ & \cf(v<-b_0)\Bigg[ \int\limits_{\min(r,\max(\delta_1,\vep(v-r)))}^r\rmd r'\ K_3^1(v, v-r')\left( \ol\Gamma_{\vep}(v, r) - \ol\Gamma_{\vep}(v-r', r-r') \right) \Bigg.\\
&\  -\int\limits_{\min(r,\max(\delta_1,\vep(v-r)))}^{\min(r, a_1-v)}\rmd r' K_3(v, v+r')\left( \ol\Gamma_{\vep}(v+r', r+r') - \ol\Gamma_{\vep}(v, r) \right) \\
& \ \Big. + \int\limits_{\min(r, a_1-v)}^r\rmd r'\ K_3(v, v+r')\left( \rme^{-\frac{1}{2}r'}\ol\Gamma_{\vep}(v+r', r+r') - \ol\Gamma_{\vep}(v, r) \right)\Big]\\
+ & \cf(v\geq -b_0)\Bigg[ \int\limits_{\min(r,\max(\delta_1,\vep(v-r)))}^r\rmd r'\ K_3^1(v, v-r')\left( \ol\Gamma_{\vep}(v, r) - \ol\Gamma_{\vep}(v-r', r-r') \right)\Bigg.\\
&\Bigg. - \int\limits_{\min(r,\max(\delta_1,\vep(v-r)))}^r\rmd r'\ K_3(v, v+r')\left( \rme^{-\frac{1}{2}r'}\ol\Gamma_{\vep}(v+r', r+r') - \ol\Gamma_{\vep}(v, r)\right)\Big] \\
& + \cf(v<-m_0)\int\limits_{\min(r, -v-b_0)}^r\rmd r' K_3^1(v, v-r')\ol\Gamma_{\vep}(v-r', r-r')\\
& + \cf(v\geq-m_0)\int\limits_{\min(r,\max(\delta_1,\vep(v-r)))}^r\rmd r' K_3^1(v, v-r')(1 - \rme^{-r'})\ol\Gamma_{\vep}(v-r', r-r'),
\end{align*}
\begin{align*}
3)\ \ \tI_3[\ol\Gamma_{\vep}](v, r) = \tI_3^{(1)}[\ol\Gamma_{\vep}](v, r) + \tI_3^{(2)}[\ol\Gamma_{\vep}](v, r),\qquad\qquad\qquad\qquad\qquad\qquad\hspace{2in}
\end{align*}
where
\begin{align*}
a)\ \ &\tI_3^{(1)}[\ol\Gamma_{\vep}](v, r)\\
=& \cf(v-r\leq-m_0)\Bigg[ \cf(v<-b_0)\Bigg\{ \int\limits_{\max(\tr, \vep(v-r))}^{\infty}\rmd r'\ K_3^2(v-r, v-r-r')\ol\Gamma_{\vep}(v, r)\Bigg.\Bigg.\\
& +  \cf(r>\vep(v-r))\int\limits_{\max(\delta_1,\vep(v-r))}^r\rmd r'\ \rme^{\mu a}(f(v-r)+f(v))\Big( K_3^2(v-r, v-r-r')\ol g(v, r)\Big.\\
&\qquad\qquad\qquad\qquad \Big. - (K_3^2(v-r, v-r-r') - K_3^2(v, v-r-r'))\ol g(v, r+r')\Big)\\
&\ \ + \cf(\tr>\max(r+\delta_1,\vep(v-r)))\int\limits_{\max(r+\delta_1,\vep(v-r))}^{\tr}\rmd r'\rme^{\mu a}(f(v-r)+f(v))\Big( K_3^2(v-r, v-r-r')\ol g(v, r)\Big.\\
&\qquad\qquad\qquad\qquad\Bigg.\Big. - (K_3^2(v-r, v-r-r') - K_3^2(v, v-r-r'))\ol g(v, r+r')\Big)\Bigg\}\\
&\ \ +\cf(v\geq-b_0)\Bigg\{ \int\limits_{\tr}^{\infty}\rmd r'\ K_3^2(v-r, v-r-r')\ol\Gamma_{\vep}(v, r) \Bigg.\\
&\ \ + \cf(\tr>\max(\delta_1, \vep(v-r))) \int\limits_{\max(\delta_1, \vep(v-r))}^{\tr}\rmd r' \Bigg( K_3^2(v-r, v-r-r')\ol\Gamma_{\vep}(v, r)\Bigg.\\
&\quad \Bigg.\Bigg.\Bigg. - (K_3^2(v-r, v-r-r') - K_3^2(v, v-r-r'))\rme^{\mu\max(a, c_0v, v-r-r')}(f(v) + f(v-r))\ol g(v, r+r')\Bigg)\Bigg\}\Bigg]\\
& + \cf(-m_0<v-r<m_0)\int_{\vep(v-r)}^{\infty}\rmd r'\ K_3^2(v-r, v-r-r')\ol\Gamma_{\vep}(v, r)\\
& + \cf(v-r\geq m_0)\Bigg[ \cf(v\leq 3r)\Bigg\{ \int_{v-r}^{\infty}\rmd r'\ K_3^2(v-r, v-r-r')\ol\Gamma_{\vep}(v, r) \Bigg.\Bigg.\\
&\qquad+ \int_{\vep(v-r)}^{v-r}(f(v)+ f(v-r))\left( K_3^2(v-r, v-r-r')\rme^{\mu\max(a, c_0v, v-r)}\ol g(v, r) \right.\Bigg.\Bigg.\\
&\qquad\left.\Bigg. - \rme^{\mu\max(a, c_0v, v-r-r')}(K_3^2(v-r, v-r-r') - K_3^2(v, v-r-r'))\ol g(v, r+r')\right)\Bigg\}\\
&\ \ \ + \cf(v>3r)\Bigg\{  \int\limits_{\vep(v-r)}^{v-r-c_0v}\rmd r'\ K_3^2(v-r, v-r-r')\ol\Gamma_{\vep}(v, r) \Bigg.\Bigg.\\
&\qquad+ \int\limits_{v-r-c_0v}^{\infty}(f(v)+ f(v-r))\left( K_3^2(v-r, v-r-r')\rme^{\mu\max(a, c_0v, v-r)}\ol g(v, r) \right.\Bigg.\Bigg.\\
&\qquad\left.\Bigg. - \rme^{\mu\max(a, c_0v, v-r-r')}(K_3^2(v-r, v-r-r') - K_3^2(v, v-r-r'))\ol g(v, r+r')\right)\Bigg\}\\
& + \tA_1^-[\ol\Gamma_{\vep}^1](v, r) + \tA_2^-[\ol\Gamma_{\vep}^1](v, r) + \tA_3^-[\ol\Gamma_{\vep}^1](v, r),
\end{align*}
with
\begin{align*}
&\tA_1^-[\ol\Gamma_{\vep}^1](v, r) = -\cf(v-r\leq-m_0)\cf(v<-b_0)\cf(\tr>\max(r+\delta_1,\vep(v-r)))\rme^{\mu a}\times\\
&\qquad\quad\times\int\limits_{\max(r+\delta_1,\vep(v-r))}^{\tr}\rmd r'\ \Big( K_3^2(v-r, v-r-r') - K_3^2(v, v-r-r') \Big)\left( f(v-r-r') - f(v-r) \right)\ol g(v, r+r'),
\end{align*}
\begin{align*}
&\tA_2^-[\ol\Gamma_{\vep}^1](v, r) =-\cf(v-r\leq-m_0)\Bigg[ \cf(v<-b_0)\cf(r>\vep(v-r))\rme^{\mu a}\times\Bigg.\\
&\qquad\qquad\ \ \times\int\limits_{\max(\delta_1, \vep(v-r))}^r\rmd r'\Big( K_3^2(v-r, v-r-r') - K_3^2(v, v-r-r') \Big)\left( f(v-r-r') - f(v-r) \right)\ol g(v, r+r')\\
&\qquad \ \Bigg. + \cf(v\geq-b_0)\cf(\tr>\max(\delta_1, \vep(v-r)))\int\limits_{\min(\delta_1,\vep(v-r))}^{\tr}\rmd r'\Big( K_3^2(v-r, v-r-r') - K_3^2(v, v-r-r') \Big)\times\Bigg.\\
&\qquad\qquad\qquad\qquad\times\Bigg.\left( f(v-r-r') - f(v-r) \right)\ol g(v, r+r')\Bigg],
\end{align*}
\begin{align*}
&\tA_3^-[\ol\Gamma_{\vep}^1](v, r) = -\cf(v-r>m_0)\cf(v>3r)\int\limits_{v-r}^{\infty}\rmd r'\Big( K_3^2(v-r, v-r-r') - K_3^2(v, v-r-r') \Big)\times\\
&\qquad\qquad\qquad\qquad\qquad\qquad\qquad\times\rme^{\mu\max(a, c_0v, v-r-r')}\left( f(v-r-r') - f(v-r) \right)\ol g(v, r+r'), 
\end{align*}
and
\begin{align*}
b)\ \ &\tI_3^{(2)}[\ol\Gamma_{\vep}](v, r)=\cf(v\leq-m_0)\Bigg[\int\limits_{\min(r, \max(\vep(v-r),\delta_1))}^{\min(r, -v-b_0)}\rmd r'\ K_3^2(v, v-r')\Big(\ol\Gamma_{\vep}(v, r) - \ol\Gamma_{\vep}(v-r', r-r')\Big)  \Bigg.\\ 
&\ \qquad \Bigg. - \int\limits_{\min(r, -v-b_0)}^r\rmd r'\ K_3^2(v, v-r')\ol\Gamma_{\vep}(v, r)\Bigg] + \cf(-m_0<v<m_0) \int_{\min(\vep(v), r)}^r\rmd r'\ K_3^2(v, v-r')\ol\Gamma_{\vep}(v, r)  \\
& \quad +\cf(v\geq m_0)\Bigg[  \cf(0<v-r<c_0v)\Bigg\{\int\limits_{\max(v-c_0^{-1}(v-r),\vep(v-r))}^r\rmd r'\ K_3^2(v, v-r')\Big( \ol\Gamma_{\vep}(v,r) - \ol\Gamma_{\vep}(v-r', r-r')\Big) \Bigg. \Bigg.\\
&\qquad\quad \Bigg. + \int\limits_{\min(\vep(v-r), v - c_0^{-1}(v-r))}^{v-c_0^{-1}(v-r)}\rmd r'\ K_3^2(v, v-r')\ol\Gamma_{\vep}(v, r) \Bigg\}\\
&\qquad\Bigg. + \cf(v-r\geq c_0v)\int\limits_{\min(r,\max(\delta_1,\vep(v-r)))}^r\rmd r'\ K_3^2(v, v-r') \Big(\ol\Gamma_{\vep}(v, r) - \ol\Gamma_{\vep}(v-r', r-r')\Big)\Bigg].
\end{align*}

\section*{Acknowledgements}

We are deeply grateful to Antti Kupiainen for many illuminating discussions on this problem as well as valuable suggestions and comments on this manuscript.  During the many years over which
we have worked on this project, we have benefited also from discussions with several other colleagues.  In particular, we would like to thank
Cl\'ement Mouhot, Herbert Spohn, and Juan J. L. Vel\'azquez for their comments.
We are also thankful to Miguel Escobedo for correspondence about their
newest work on the topic.
The research has been supported by the Academy of Finland, via
an Academy project (project No.\ 339228), the Finnish Centre of
Excellence in Randomness and Structures (project No.\ 346306) and ERC Advanced Investigator Grants
741487 and
227772. J Bandyopadhyay's work in this paper is intended as a small tribute to their father Raghab Bandyopadhyay.



\bibliographystyle{elsarticle-num}







\end{document}